\documentclass[letterpaper,reqno,10pt,twoside]{amsart}
\usepackage{amsmath,amsthm,amsfonts,amssymb,euscript,mathrsfs,graphics,latexsym,marginnote}
\usepackage[dvips]{graphicx}
\usepackage[left=0.8 in, right=0.8 in,top=0.8 in, bottom=0.8 in]{geometry}
\usepackage[toc,page]{appendix}
\usepackage{float}

\usepackage{xargs}
\usepackage[dvipsnames]{xcolor}

\usepackage{hyperref}
\hypersetup{colorlinks=true, pdfstartview=FitV,linkcolor=blue!70!black,citecolor=red!70!black, urlcolor=green!60!black}
\definecolor{labelkey}{rgb}{0.6,0,0}

\usepackage[colorinlistoftodos,prependcaption,textsize=small]{todonotes}
\allowdisplaybreaks

\makeatletter
\renewcommand \theequation {%
	\ifnum \c@section>\z@ \@arabic\c@section.%
	\fi\@arabic\c@equation} \@addtoreset{equation}{section}
\@namedef{subjclassname@2020}{2020 Mathematics Subject Classification}
\makeatother

\newtheorem{theorem}{Theorem}[section]

\newtheorem{lemma}[theorem]{Lemma}
\newtheorem{proposition}[theorem]{Proposition}

\theoremstyle{definition}
\newtheorem{definition}{Definition}[section]
\theoremstyle{remark}
\newtheorem{remark}{Remark}[section]

\def\dt{\partial_t}
\def\p{\partial}

\def\rt{\rightarrow}

\def\no{\nonumber}

\def\Om{\Omega}
\def\om{\omega}
\def\dive{\mathop{\rm div}\nolimits}

\def\d{\partial}
\def\f{\frac}

\begin{document}
	
	\title{Global Well-Posedness of Contact Lines: 2D Navier-Stokes Flow}
	
	\author[Y. Guo]{Yan Guo}
	\address[Y. Guo]{
		\newline\indent Division of Applied Mathematics, Brown University, Providence, RI 02912, USA}
	\email{yan\_guo@brown.edu}
	\thanks{Y. Guo was supported by NSF Grant DMS-2106650.}
	
	\author[I. Tice]{Ian Tice}
	\address[I. Tice]{
		\newline\indent Department of Mathematical Sciences, Carnegie Mellon University, Pittsburgh, PA 15213, USA}
	\email{iantice@andrew.cmu.edu}
	\thanks{I. Tice was supported by NSF Grant DMS-2508400.}
	
	\author[L. Wu]{Lei Wu}
	\address[L. Wu]{
		\newline\indent Department of Mathematics, Lehigh University, Bethlehem, PA 18015, USA}
	\email{lew218@lehigh.edu}
	\thanks{L. Wu was supported by NSF Grant DMS-2104775 and DMS-2405161.}
	
	\author[X. Yang]{Xiaoding Yang}
	\address[X. Yang]{
		\newline\indent Division of Applied Mathematics, Brown University, Providence, RI 02912, USA}
	\email{xiaoding\_yang@brown.edu}
	\thanks{X. Yang was supported by NSF Grant DMS-2405051.}
	
	\author[Y. Zheng]{Yunrui Zheng}
	\address[Y. Zheng]{
		\newline\indent School of Mathematics, Shandong University, Shandong 250100, Jinan, P. R. China}
	\email{yunrui\_zheng@sdu.edu.cn*}
	\thanks{Y. Zheng was supported by NSFC Grant 11901350 and 12371172.}
	
	\date{}
	
	
	
	\begin{abstract}
		Based on the global a priori estimates in [Guo-Tice, J. Eur. Math. Soc. (2024)], we establish the well-posedness of a viscous fluid model satisfying the dynamic law for the contact line
		\begin{equation*}
			\mathscr{W}(\p_t\zeta(\pm\ell,t))=[\![\gamma]\!]\mp\sigma\frac{\p_1\zeta}{(1+|\p_1\zeta|^2)^{1/2}}(\pm\ell,t)
		\end{equation*}
		in a 2D domain, where $\zeta(x_1,t)$ is a free surface with two contact points $\zeta(\pm\ell,t)$, $[\![\gamma]\!]$ and $\sigma$ are constants characterizing the solid-fluid-gas free energy, and  $\mathscr{W}$ is the contact point velocity response function.  Our approach, which is motivated by the natural energy-dissipation structure of the problem, relies on constructing a pressureless weak solution to the linearized problem for the coupled velocity and free interface, via a Galerkin approximation with a time-dependent basis and an artificial regularization for the capillary operator.
	\end{abstract}
	
	\maketitle

	\pagestyle{myheadings} \thispagestyle{plain} \markboth{GUO, TICE, WU, YANG, ZHENG}{WELL-POSEDNESS OF CONTACT POINTS PROBLEM}

	\setcounter{tocdepth}{2}
	\makeatletter
	\def\l@subsection{\@tocline{2}{0pt}{2.5pc}{5pc}{}}
	\makeatother
	\tableofcontents

	%%%%%%%%%%%%%%%%%%%%%%%%%%%%%%%%%%%%%%%%%%%%%%
	\section{Introduction}
	%%%%%%%%%%%%%%%%%%%%%%%%%%%%%%%%%%%%%%%%%%%%%%

	%%%%%%%%%%%%%%%%%%%%%%%%%%%%%%%%%%%%%%%%%%%%%%
	\subsection{Problem Formulation}
	%%%%%%%%%%%%%%%%%%%%%%%%%%%%%%%%%%%%%%%%%%%%%%

	We consider a $2$D open top vessel as a bounded, connected open set $\mathcal{V}\subseteq\mathbb{R}^2$ which consists of two disjoint sections as well as their intersection, i.e., $\mathcal{V}=\mathcal{V}_{top}\cup\mathcal{V}_{bot}\cup I$.
	We refer to Figure 1 for an example.

	%%%%%%%%%%%%%%%%%%%%%%%%%%%%%%%%%%%%%%%%%%

	\begin{figure}[H]
		\begin{minipage}[t]{0.4\linewidth}
			\centering
			\scalebox{0.9}{\begin{tikzpicture}
					\draw  (-1, 1)--(-1, 5);
					\draw  (1, 1)--(1, 5);
					\filldraw [green!15, draw=black] (-1, 5)--(-1, 1)--(1, 1)--(1, 5);
					\node at(0, 3){$\mathcal{V}_{top}$};
					\filldraw [blue!10, draw=black] (-1, 1) arc(180: 360: 1);
					\node at(0,0.5) {$\mathcal{V}_{bot}$};
					\draw[dashed]  (-1,1)--(1, 1);
					\node at(0, 1){ $I$};
			\end{tikzpicture}}
			\caption{A vessel $\mathcal{V}$.}
		\end{minipage}%
		\begin{minipage}[t]{0.4\linewidth}
			\centering
			\scalebox{0.9}{\begin{tikzpicture}
					\draw  (6, 1)--(6, 5);
					\draw  (8, 1)--(8, 5);
					\filldraw [green!10, draw=black] (6, 4)--(6, 1)--(6, 1) arc(180: 360: 1)--(8, 1)--(8, 4)--(8,4)..controls(7.4, 3.8)and (6.7, 3.9)..(6,4.2);
					\node at(7, 2){$\Om(t)$};
					\node at(7,4){$\Sigma(t)$};
					\node at(8.5, 2){$\Sigma_s(t)$};
			\end{tikzpicture}}
			\caption{The domain $\Om(t)$.}
		\end{minipage}
	\end{figure}
	%%%%%%%%%%%%%%%%%%%%%%%%%%%%%%%%%%%%%%%%%%%%

	Assume that the ``top" part $\mathcal{V}_{top}$ is a rectangular channel
	$
	\mathcal{V}_{top}:=\mathcal{V}\cap\mathbb{R}^2_{+}=\{y\in\mathbb{R}^2: -\ell<y_1<\ell, 0\le y_2<L\}
	$
	for some $\ell$, $L>0$, where $\mathbb{R}^2_{+}$ is the upper half plane $\mathbb{R}^2_{+}=\{y\in\mathbb{R}^2: y_2\ge0\}$. Similarly, we write the ``bottom" part as
	$
	\mathcal{V}_{bot}:=\mathcal{V}\cap\mathbb{R}^2_{-}=\mathcal{V}\cap\{y\in\mathbb{R}^2: y_2\le0\}.
	$
	Clearly, $I=(-\ell, \ell)\times\{0\}$. In addition, we also assume that the boundary $\partial\mathcal{V}$ of $\mathcal{V}$ is $C^2$ away from the points $(\pm\ell, L)$.
	
	We consider a viscous incompressible fluid filling $\mathcal{V}_{bot}$ entirely and $\mathcal{V}_{top}$ partially. More precisely, we assume that the fluid occupies the domain $\Om(t)$ with a top free surface,
	$
	\Om(t):=\mathcal{V}_{bot}\cup\big\{y\in\mathbb{R}^2: -\ell<y_1<\ell,\ 0<y_2<\zeta(y_1,t)\big\},
	$
	where the free surface $\zeta(y_1,t)$ is assumed to be a graph of the function $\zeta: [-\ell, \ell]\times \mathbb{R}_{+}\rightarrow\mathbb{R}$ satisfying $0<\zeta(\pm\ell,t)\le L$ for all $t\in\mathbb{R}_+$, which means the fluid does not spill out of the top domain. For simplicity, we write the free surface as $\Sigma(t)=\{y_2=\zeta(y_1,t)\}$ and the interface between fluid and solid as $\Sigma_s(t)=\partial\Om(t)\backslash\Sigma(t)$.
	We refer to Figure 2 for an illustration of domain.
	
	For each $t>0$, the fluid is described by its velocity and pressure $(u,P):\Om(t)\rightarrow\mathbb{R}^2\times\mathbb{R}$ which is governed by the incompressible Navier-Stokes equations:
	\begin{equation}\label{eq:navier-stokes}
		\left\{
		\begin{aligned}
			&\d_t u+u\cdot\nabla u+\nabla P-\mu\Delta u=0 \quad & \text{in} &\ \Om(t),\\
			&\dive u=0 \quad & \text{in} &\ \Om(t),\\
			&S(P,u)\nu=g\zeta\nu-\sigma \mathcal{H}(\zeta)\nu \quad & \text{on} &\ \Sigma(t),\\
			&\big(S(P,u)\nu-\beta u\big)\cdot\tau=0 \quad & \text{on} &\ \Sigma_s(t),\\
			&u\cdot\nu=0 \quad & \text{on} &\ \Sigma_s(t),\\
			&\d_t\zeta=u\cdot\nu=u_2-u_1\partial_1\zeta \quad & \text{on} &\ \Sigma(t),\\\rule{0ex}{-1.5em}
			&\d_t\zeta(\pm\ell,t)=\mathscr{V}\left([\![\gamma]\!]\mp\sigma\tfrac{\partial_1\zeta}{(1+|\partial_1\zeta|^2)^{1/2}}(\pm\ell,t)\right).
		\end{aligned}
		\right.
	\end{equation}
	
	In the above system \eqref{eq:navier-stokes}, $S(P,u)$ is the viscous stress tensor
	$S(P,u):=PI-\mu\mathbb{D}u$,
	where $I$ is the $2\times2$ identity matrix, $\mu>0$ is the coefficient of viscosity, $\mathbb{D}u=\nabla u+\nabla^T u$ is the symmetric gradient of $u$ for $\nabla^T u$ the transpose of the matrix $\nabla u$.  The pressure unknown $P=\bar{P}+gy_2-P_{atm}$ is obtained by adjusting the actual pressure $\bar{P}$ by the constant atmospheric pressure $P_{atm}$ and the background hydrostatic pressure $gy_2$.  The vector $\nu$ is the outward unit normal, $\tau$ is the unit tangent, $\sigma>0$ is the coefficient of surface tension,
	$
	\mathcal{H}(\zeta)=\p_1\left(\tfrac{\p_1\zeta}{(1+|\p_1\zeta|^2)^{1/2}}\right)
	$
	is twice the mean curvature of the free surface, and $\beta>0$ is the Navier-slip friction coefficient on the vessel walls. The boundary conditions on $\Sigma_s(t)$ are called Navier-slip conditions. $\mathscr{V}:\mathbb{R}\rightarrow\mathbb{R}$ is the contact point velocity response function which is a $C^2$ increasing diffeomorphism satisfying $\mathscr{V}(0)=0$. $[\![\gamma]\!]:=\gamma_{sv}-\gamma_{sf}$ for $\gamma_{sv}$, $\gamma_{sf}\in\mathbb{R}$, where $\gamma_{sv}$, $\gamma_{sf}$ are measures of the free-energy per unit length with respect to the solid-vapor and solid-fluid intersection. We also assume that the Young relation
	$|[\![\gamma]\!]|<\sigma$ (\cite{Yo}),
	which is necessary for the equilibrium state we will consider in the following. For convenience, we introduce the inverse function $\mathscr{W}$ of $\mathscr{V}$ and rewrite the final equation in \eqref{eq:navier-stokes} as
	\begin{equation}\label{eq:contact law}
		\mathscr{W}\big(\p_t\zeta(\pm\ell,t)\big)=[\![\gamma]\!]\mp\sigma\frac{\p_1\zeta}{(1+|\p_1\zeta|^2)^{1/2}}(\pm\ell,t).
	\end{equation}
	
	Contact points in 2D and contact lines in 3D domain have been studied for a long time, and we may refer to \cite[de Gennes]{dG} for an overview. There is an extensive literature dealing with the free boundary problems of contact points or contact lines, and so we will only mention some works closely related to our background settings.
	
	The analysis for Navier-slip conditions with contact angle can be found in \cite[Ren-E]{RE}, \cite[Cox]{C1986} and \cite[Guo-Tice]{GT2020}, thee latter of which we refer to for a lengthier  overview.  Note that the model \eqref{eq:navier-stokes} allows both the contact points and contact angles to move over time.
	The global a priori estimates for \eqref{eq:navier-stokes} are given by \cite[Guo-Tice]{GT2020}. When the contact angle is fixed at $\frac\pi2$, the well-posedness of 2D contact points problem in the Navier-Stokes flow is studied by \cite[Schweizer]{S01} and the 3D contact line problem with an additional periodic direction is studied by \cite[Bodea]{Bo06}.
	
	There are also many results for the stationary problem of contact points or lines. The solvability of stationary Navier-Stokes problem for dynamic contact angles and moving contact lines are studied by \cite[Kr\"oner]{Kr87} and \cite[Socolosky]{Soco93}, respectively. The local well-posedness and stability for the 2D moving contact points in the Stokes flow are studied by \cite[Guo-Tice]{GT18} and \cite[Tice-Zheng]{ZhT17}. The existence results for the Navier-Stokes equations for both static and dynamic contact points and lines in weighted H\"older spaces are proved by \cite[Solonnikov]{Sol95, Sol98}. The 3D contact line problem for the stationary Navier-Stokes equations with fixed angle $\frac{\pi}{2}$ is studied by \cite[Jin]{Jin05}.
	
	The contact angles are also closely related to the droplet problem. The stability problem for droplets governed by the Stokes equations is studied by \cite[Tice-Wu]{TW21}. The well-posedness for droplets governed by Darcy's law is studied by \cite[Kn\"upfer-Masmoudi]{KM15}.
	
	Despite  \cite[Guo-Tice]{GT2020}, there is no construction of solutions with dynamical contact line law \eqref{eq:contact law}, which \textit{must be} compatible with all the other boundary conditions. The purpose of this paper is to establish the first result for the global well-posedness of the 2D Navier-Stokes contact line problem. It is known that \eqref{eq:navier-stokes} admits an equilibrium state with $u=0$, $P=P_0= \text{constant}$ and $\zeta(y_1,t)=\zeta_0(y_1)$ satisfying
	\begin{equation}\label{eq:equibrium}
		P_0=g\zeta_0-\sigma\mathcal{H}(\zeta_0)\quad  \text{on} \ (-\ell,\ell),\quad
		\sigma\frac{\p_1\zeta_0}{\sqrt{1+|\p_1\zeta_0|^2}}(\pm\ell)=\pm[\![\gamma]\!].
	\end{equation}
	For the well-posedness of \eqref{eq:equibrium} we refer to \cite[Guo-Tice]{GT18} or \cite[Finn]{Finn}.
	
	This equilibrium allows us to transform the free surface problem into a fixed domain. Let $\zeta_0\in C^\infty[-\ell,\ell]$ be the equilibrium surface given by \eqref{eq:equibrium}. We then define the equilibrium domain
	$
	\Om:=\mathcal{V}_{b}\cup\big\{x\in\mathbb{R}^2|-\ell<x_1<\ell, 0<x_2<\zeta_0(x_1)\big\} \subset\mathbb{R}^2
	$.
	The boundary $\p\Om$ of the equilibrium $\Om$ is defined by
	$
	\p\Om:=\Sigma\cup\Sigma_s$,
	where
	$
	\Sigma:=\{x\in\mathbb{R}^2|-\ell<x_1<\ell, x_2=\zeta_0(x_1)\}$ and $ \Sigma_s=\p\Om\backslash \Sigma$.
	Here $\Sigma$ is the equilibrium free surface. Denote the corner angle $\om\in(0,\pi)$ of $\Om$ formed by the fluid and solid.
	We assume that the function $\zeta(y_1,t)$ of free surface is the perturbation of $\zeta_0(y_1)$ as
	$
	\zeta(y_1,t)=\zeta_0(y_1)+\eta(y_1,t)$.
	
	Let $\phi\in C^\infty(\mathbb{R})$ be such that $\phi(z)=0$ for $z\le\frac14\min\zeta_0$ and $\phi(z)=z$ for $z\ge\frac12\min\zeta_0$.
	Now we define the mapping $\Phi: \Om\mapsto\Om(t)$ by
	\begin{equation}\label{def:map}
		\Phi(x_1,x_2,t):=\left(x_1, x_2+\frac{\phi(x_2)}{\zeta_0(x_1)}\bar{\eta}(x_1,x_2,t)\right)=(y_1,y_2)\in\Om(t),
	\end{equation}
	with $
	\bar{\eta}(x_1,x_2,t):=\mathcal{P}E\eta(x_1,x_2-\zeta_0(x_1),t)$,
	where $E: H^s(-\ell,\ell)\mapsto H^{s+\frac{1}{2}}(\mathbb{R})$ is a bounded extension operator for all $0\le s\le 3$ and $\mathcal{P}$ is the lower Poisson extension given by
	\[
	\mathcal{P}f(x_1,x_2):=\int_{\mathbb{R}}\hat{f}(\xi)e^{2\pi|\xi| x_2}e^{2\pi ix_1\xi}\,\mathrm{d}\xi.
	\]
	If $\eta$ is sufficiently small (in proper Sobolev spaces), the mapping $\Phi$ is a $C^1$ diffeomorphism of $\Om$ onto $\Om(t)$ that maps the components of $\p\Om$ to the corresponding components of $\p\Om(t)$. The Jacobian matrix $\nabla\Phi$ and the transform matrix $\mathcal{A}$ of $\Phi$ are
	\begin{equation}\label{eq:Jaccobian}
		\nabla\Phi=\left(
		\begin{array}{cc}
			1&0\\
			A&J
		\end{array}
		\right), \quad
		\mathcal{A}=(\nabla\Phi)^{-T}=\left(
		\begin{array}{cc}
			1&-AK\\
			0&K
		\end{array}
		\right),
	\end{equation}
	with entries
	\begin{equation}\label{eq:components}
		A:=W\p_1\bar{\eta}-\frac{\phi}{\zeta_0^2}\p_1\zeta_0\bar{\eta},\quad J:=1+\frac{\phi^\prime}{\zeta_0}\bar{\eta}+W\p_2\bar{\eta},\quad K:=\frac{1}{J},\quad W:=\frac{\phi}{\zeta_0}.
	\end{equation}
	We define the transform operators as follows.
	$
	(\nabla_{\mathcal{A}}f)_i:=\mathcal{A}_{ij}\p_jf,\quad\dive_{\mathcal{A}}X:=\mathcal{A}_{ij}\p_j X_i,\quad \Delta_{\mathcal{A}}f:=\dive_{\mathcal{A}}\nabla_{\mathcal{A}}f
	$
	for appropriate $f$ and $X$. We write the stress tensor
	$
	S_{\mathcal{A}}(P,u):=PI-\mu\mathbb{D}_{\mathcal{A}}u
	$
	where $I$ is the $2\times2$ identity matrix and $(\mathbb{D}_{\mathcal{A}}u)_{ij}:=\mathcal{A}_{ik}\p_ku_j+\mathcal{A}_{jk}\p_ku_i$ is the symmetric $\mathcal{A}$--gradient. Note that if we extend $\dive_{\mathcal{A}}$ to act on symmetric tensors in the natural way, then $\dive_{\mathcal{A}}S_{\mathcal{A}}(P,u)=-\mu\Delta_{\mathcal{A}}u+\nabla_{\mathcal{A}}P$ for vectors fields satisfying $\dive_{\mathcal{A}}u=0$. $\Phi$ is a diffeomorphism justified by \cite{ZhT17}.
	We rewrite the solutions to \eqref{eq:navier-stokes} as a perturbation around the equilibrium state $(0, P_0, \zeta_0)$. By  defining new perturbed unknowns $(u, p, \eta)$ via $u=0+u$, $P=P_0+p$, and $\zeta=\zeta_0+\eta$, we can  use the analysis of \cite{GT2020} to transform the problem \eqref{eq:navier-stokes} to the equilibrium domain $\Om$ for $t\ge0$. For our purpose, we introduce the following  system
	\begin{equation}\label{eq:geometric}
		\left\{
		\begin{aligned}
			&\p_tu-\p_t\bar{\eta}WK\p_2u+u\cdot\nabla_{\mathcal{A}}u+\dive_{\mathcal{A}}S_{\mathcal{A}}(p,u)=0 \quad &\text{in}&\quad \Om(t),\\
			&\dive_{\mathcal{A}}u=0 \quad &\text{in}&\quad \Om(t),\\
			&S_{\mathcal{A}}(p,u)\mathcal{N}=g\eta\mathcal{N}-\sigma\p_1\left(\tfrac{\p_1\eta }{(1+|\p_1\zeta_0|^2)^{3/2}}\right)\mathcal{N}-\sigma\p_1\Big(\mathcal{R}(\p_1\zeta_0,\p_1\eta)\Big)\mathcal{N} \quad &\text{on}&\quad\Sigma(t),\\
			&(S_{\mathcal{A}}(p,u)\nu-\beta u)\cdot\tau=0 \quad &\text{on}&\quad\Sigma_s(t),\\
			&u\cdot\nu=0 \quad &\text{on}&\quad\Sigma_s(t),\\
			&\p_t\eta=u\cdot\mathcal{N} \quad &\text{on}&\quad\Sigma(t),\\
			&\kappa\p_t\eta(\pm\ell,t)+\kappa\hat{\mathscr{W}}(\p_t\eta(\pm\ell,t))=\mp\sigma\left(\tfrac{\p_1\eta}{(1+|\p_1\zeta_0|^2)^{3/2}}+\mathcal{R}(\p_1\zeta_0,\p_1\eta)\right)(\pm\ell,t),
		\end{aligned}
		\right.
	\end{equation}
	where $\kappa:=\mathscr{W}'(0)>0$, $\mathcal{A}$ and $\mathcal{N}:=-\p_1\zeta e_1+e_2$ are still determined in terms of $\zeta=\zeta_0+\eta$. $\mathcal{R}(y,z)$ is a multi-variable function defined as follows:
	\[
	\mathcal{R}(y,z)=\int_{0}^{z}3\frac{(s-z)(s+y)}{(1+|y+s|^{2})^{\frac{5}{2}}}ds.
	\]
	In the following, let $\mathcal{N}_0$ be the non-unit normal vector for the equilibrium surface $\Sigma$, and $\mathcal{N}=\mathcal{N}_0-\p_1\eta e_1$.
	Since all the differential operators in \eqref{eq:geometric} are in terms of $\eta$, \eqref{eq:geometric} is connected to the geometry of the free surface. This geometric structure is essential to control higher-order derivatives.

	%%%%%%%%%%%%%%%%%%%%%%%%%%%%%%%%%%%%%%%%%%%%%%
	\subsection{Main Theorem}
	%%%%%%%%%%%%%%%%%%%%%%%%%%%%%%%%%%%%%%%%%%%%%%

	We use the same notation for the equilibrium contact angle $\om_{eq}$ formed between the surface $\zeta_0$ and the wall $\Sigma_s$, and the parameter $\varepsilon_{max} := \min\{1, -1+ \pi/\om_{eq}\}$ as in \cite{GT2020},  allowing us to choose three parameters $\alpha$, $\varepsilon_-$ and $\varepsilon_+$ that satisfy
	\begin{equation}\label{parameters}
		0<\alpha<\varepsilon_-<\varepsilon_+<\varepsilon_{max},\quad \alpha<\min\left\{\frac{\varepsilon_-}2, \frac{\varepsilon_+-\varepsilon_-}2\right\}, \quad \varepsilon_+\leq \frac{\varepsilon_-+1}2,
	\end{equation}
	\begin{equation}\label{eq:parameter2}
		1< q_-:=\frac{2}{2-\varepsilon_-} < q_+:=\frac{2}{2-\varepsilon_+} <q_{max}:=\frac{2}{2-\varepsilon_{max}} < 2.
	\end{equation}
	
	Define the energy
	\begin{equation}\label{energy}
		\begin{aligned}
			\mathcal{E}(t)&: =\|u\|_{W^{2,q_+}}^2+\sum_{i=0}^{1}\|\p_t^{i}u\|_{1+\varepsilon_-/2}^2+\sum_{i=0}^2 \Big( \|\p_t^iu\|_0^2+\|\p_{t}^{i}p\|_{W^{1,q_+}}^2\Big) + \sum_{i=0}^{1}\|\p_t^{i}p\|_0^2\\&\quad+\|\eta\|_{W^{3-1/q_+, q_+}}^2+\sum_{i=0}^{1}\|\p_t^{i}\eta\|_{H^{3/2+(\varepsilon_--\alpha)/2}}^2+\sum_{j=0}^2\|\p_t^j\eta\|_1^2,
		\end{aligned}
	\end{equation}
	and the dissipation
	\begin{equation}\label{dissipation}
		\begin{aligned}
			\mathcal{D}(t)&:=\Big(\|u\|_{W^{2,q_+}}^2+\|p\|_{W^{1,q_{+}}}^{2}+\|\eta\|_{W^{3-1/q_+, q_+}}^2\Big)+\|\partial_{t}u\|^{2}_{W^{2,q_{-}}}+\sum_{j=0}^2\Big(\|\p_t^ju\|_1^2+\|\p_t^ju\|_{L^2(\Sigma_s)}^2\Big)\\&\quad
			+\|\p_tp\|_{W^{1,q_-}}^2
			+\|\p_t\eta\|_{W^{3-1/q_-, q_-}}^2+\sum_{j=0}^2\Big(\|\p_t^j\eta\|_{H^{3/2-\alpha}}^2+[\p_t^{j+1}\eta]_\ell^2\Big)+\|\p_t^3\eta\|_{1/2-\alpha}^2.
		\end{aligned}
	\end{equation}
	
	\begin{theorem}\label{thm:main}
		Assume that the contact angle $\om_{eq} \in (0, \pi)$ and $\varepsilon_{max} \in (0, 1)$.
		Suppose that the initial data $\Big(u_0, p_0, \eta_0, \p_tu(0), \p_tp(0), \p_t\eta(0), \p_t^2u(0),\p_t^2\eta(0)\Big)$ satisfy the compatibility conditions for $j=0,1,2,$
		\begin{equation}\label{compat_C2}
			\left\{
			\begin{aligned}
				&\dive_{\mathcal{A}_0}D_t^ju(0)=0\quad &\text{in}&\ \Om,\\
				&D_t^ju(0)\cdot\nu=0\quad &\text{on}&\ \Sigma_s,\\
				&D_t^ju(0)\cdot\mathcal{N}(0)=\p_t^{j+1}\eta(0)\quad &\text{on}&\ \Sigma,
			\end{aligned}
			\right.
		\end{equation}
		and the zero-average condition for $k=0, 1, 2$
		\begin{equation}\label{cond:zero}
			\int_{-\ell}^\ell\p_t^k\eta(0)=0.
		\end{equation}
		Then there exists a sufficiently small $\delta_0>0$, such that if $\sqrt{\mathcal{E}(0)} + \|\p_t^2 \eta(0)\|_{W^{2-1/q_+, q_+}} \le \delta_0$, then there exists a unique solution $(u,p,\eta)$ to \eqref{eq:geometric} for $t\in[0,\infty)$ that achieves the initial data and satisfies
		\begin{equation}
			\begin{aligned}
				\sup_{t\ge 0}e^{\lambda t}\mathcal{E}(t)+\int_0^\infty\mathcal{D}(t)\,\mathrm{d}t\le C\mathcal{E}(0)
			\end{aligned}
		\end{equation}
		for some universal constants $C>0$ and $\lambda>0$.
	\end{theorem}
	
	\begin{remark}
		Note that the term $\|\p_t^2\eta(0)\|_{W^{2-1/q_+, q_+}}^2$ is stronger than the initial energy functional. This term is sufficient to guarantee that our initial data set is non-empty (see the discussion in Appendix \ref{sec:initial}). Since $\Phi$ is a $C^1$ diffeomorphism (which follows from $J>0$ as well as the fact that the derivatives of $\eta$ are sufficiently small), we may change coordinates from $\Om$ to $\Om(t)$ to obtain solutions to \eqref{eq:navier-stokes}.
	\end{remark}

	%%%%%%%%%%%%%%%%%%%%%%%%%%%%%%%%%%%%%%%%%%%%%%
	\subsection{Methodology}\label{section:discussion}
	%%%%%%%%%%%%%%%%%%%%%%%%%%%%%%%%%%%%%%%%%%%%%%

	It is well known that the dynamical law \eqref{eq:contact law} for the contact line has been established experimentally, and that this law is a key ingredient in deriving the basic energy--dissipation structure. However, unlike in most free-boundary problems, it is far from clear, from a PDE standpoint, whether such a law can be imposed as an independent boundary condition in a manner consistent with the other standard boundary conditions. In particular, one must construct a solution satisfying all boundary conditions without overdetermining the system.
	
	To construct such a solution, we first introduce a linear system involving prescribed functions and unknowns. We then prove that the nonlinear solution map induced by this system is a contraction on a suitably chosen complete metric space endowed with an appropriate metric. This establishes the local well-posedness of \eqref{eq:navier-stokes}. Combining this local result with the \textit{a priori} estimates established in \cite{GT2020}, we obtain Theorem \ref{thm:main}.
	
	The main part of this paper is devoted to constructing solutions to the linear system. This construction is carried out in the following steps.

	\noindent $\bullet$ \underline{\textsl{Pseudo-linear construction}}
	
	A natural way to construct the solution to \eqref{eq:geometric} is to first consider a linear system and then take temporal derivatives to obtain a higher-order system. In particular, in order to close the energy estimate at the linear level (see Section \ref{sec:strong}), we will not attack \eqref{eq:geometric} directly, and instead begin with the first-order system and equip it with the quasi-linear structure, which reads
	\begin{equation}\label{eq:first}
		\begin{cases}
			\partial_{t}D_{t}u+\operatorname{div}_{\mathcal{A}}S_{\mathcal{A}}(D_{t}u,\p_{t}p)+\p_{t}(Ru)+\dive_{\mathcal{A}}\nabla_{\mathcal{A}}(Ru)=F^{1}(u,p,\eta)~~~&\operatorname{in}~~\Omega,\\
			\operatorname{div}_{\mathcal{A}}D_{t}u=0~~~&\operatorname{in}~~\Omega,\\
			S_{\mathcal{A}}(\p_{t}p,D_{t}u)\mathcal{N}+\nabla_{\mathcal{A}}(Ru)\mathcal{N}=g\p_{t}\eta\mathcal{N}-\sigma\partial_{1}(\frac{\p_1\p_t\eta}{(1+\vert \partial_{1}\zeta_{0}\vert^{2})^{3/2}})\mathcal{N}+\p_{1}I_{1}\mathcal{N}\\
			\quad\quad\quad\quad\quad\quad\quad\quad+\sigma\partial_{1}(\int_{0}^{t}\mathcal{R}_{z}(\partial_{1}\zeta_{0},\partial_{1}\eta)\partial_{1}\partial_{t}^{2}\eta)\mathcal{N}+\sigma\partial_{1}(\int_{0}^{t}\mathcal{R}_{zz}(\partial_{1}\zeta_{0},\partial_{1}\eta)(\partial_{1}\partial_{t}\eta)^{2})\mathcal{N}\\
			\quad\quad\quad\quad\quad\quad\quad\quad+F^{4}(u,p,\eta)~~&\operatorname{on}~~\Sigma,\\
			(S_{\mathcal{A}}(\p_{t}p,D_{t}u)\nu+\nabla_{\mathcal{A}}(\p_{t}(Ru))\nu-\beta D_{t}u)\cdot \tau=F^{5}(u,\eta,p)~~~&\operatorname{on}~~\Sigma_{s},\\
			D_{t}u\cdot \nu=0~~~&\operatorname{on}~~\Sigma_{s},\\
			\partial_{t}^{2}\eta=D_{t}u\cdot \mathcal{N}+(Ru)\cdot \mathcal{N} -\int_{0}^{t}(\partial_{t}u_{1}\cdot\partial_{t}\p_1\eta+u_{1}\partial_{1}\partial_{t}^{2}\eta)+I_{2}~~~&\operatorname{on}~~\Sigma,\\
			\sigma\Big(\mp \frac{\partial_{1}\p_{t}\eta}{(1+\vert \partial_{1}\zeta_{0}\vert^{2})^{\frac{3}{2}}}\pm \int_{0}^{t}\mathcal{R}_{z}(\partial_{1}\zeta_{0},\partial_{1}\eta)\partial_{t}^{2}\partial_{1}\eta\pm\int_{0}^{t}\mathcal{R}_{zz}(\partial_{1}\zeta_{0},\partial_{1}\eta)(\partial_{t}\partial_{1}\eta)^{2})\pm I_{1} \Big)(\pm \ell)\\
			\quad\quad\quad\quad\quad\quad\quad\quad =\kappa (\partial_{t}^{2}\eta)(\pm \ell)-{F}^{7},
		\end{cases}
	\end{equation}
	\noindent where $R$ is defined in \eqref{def:Dt_u}, and
	$ I_{1}=\big(\mathcal{R}_{z}(\p_{1}\zeta_{0},\p_{1}\eta_{0})\big)(\p_{1}\p_{t}\eta_{0}), \
	I_{2}= -(u_{1}\cdot \partial_{t}\mathcal{N})(0)$.
	
	\noindent $\bullet$ \underline{\textsl{Implications of the quasi-linearity}}
	
	We choose not to work directly with the original system \eqref{eq:geometric}, but instead \eqref{eq:first} and further \eqref{eq:quasi_linear_0} below, due to a technical consideration.
	In the energy estimates in second order for $(\partial_t^{2}v,\partial_t^{2}\xi)$,
	the following two integrals arise and cannot be controlled:
	\[
	\int_{-\ell
	}^{\ell} \partial_{1}\partial_{t}^{2}\eta \, \partial_{1}\partial_{t}^{3}\xi \, \mathrm{d}x,
	\qquad
	\int_{-\ell}^{\ell} \partial_{1}^{2}\partial_{t}^{2}\eta \, \partial_{1}\partial_{t}^{2}\xi \, \mathrm{d}x,
	\]
	where $\eta$ is a given background function and $\xi$ is the unknown.
	These problematic terms arise from the nonlinear capillary contribution in \eqref{eq:geometric}, thereby revealing a loss of regularity for the boundary unknown $\xi$.
	
	To overcome this difficulty, we instead consider the first-order time derivative version of \eqref{eq:geometric}, which allows us to modify the nonlinear term in \eqref{eq:first} so that,
	when establishing the energy estimates for $(\partial_{t}^{2}v,\partial_{t}^{2}\xi)$,
	the above integrals are replaced by
	\[
	\int_{-\ell}^{\ell} \partial_{1}\partial_{t}^{2}\xi \, \partial_{1}\partial_{t}^{3}\xi \, \mathrm{d}x,
	\qquad
	\int_{-\ell}^{\ell} \partial_{1}^{2}\partial_{t}^{2}\xi \,
	\partial_{1}\partial_{t}^{2}\xi \, \mathrm{d}x.
	\]
	
	The first integral can now be controlled by integration by parts in time,
	while the second can be bounded via integration by parts in space.
	These bounds allow us to close the energy estimates.
	
	\noindent $\bullet$ \underline{\textsl{Introducing the linear problem}} 
	
	In order to establish the local well-posedness of \eqref{eq:first}, we now introduce
	the following key linear problem (formal) via replacing $(u,p,\eta)$ in \eqref{eq:first} by given functions $(u_g,p_g,\eta_g)$ and unknown functions $(v,q,\xi)$. For simplicity, we continue to write $(u,p,\eta)$ in place of the given functions $(u_g,p_g,\eta_g)$. We have
	\begin{equation}\label{eq:quasi_linear_0}
		\begin{cases}
			\partial_{t}D_{t}v+\operatorname{div}_{\mathcal{A}}S_{\mathcal{A}}(D_{t}v,\p_{t}q)+\p_{t}(Rv)+\dive_{\mathcal{A}}\nabla_{\mathcal{A}}(Rv)=F^{1}(u,p,\eta)~~~&\operatorname{in}~~\Omega,\\
			\operatorname{div}_{\mathcal{A}}D_{t}v=0~~~&\operatorname{in}~~\Omega,\\
			S_{\mathcal{A}}(\p_{t}q,D_{t}v)\mathcal{N}+\nabla_{\mathcal{A}}(Rv)\mathcal{N}=g\p_{t}\xi\mathcal{N}-\sigma\partial_{1}(\frac{\p_1\p_t\xi}{(1+\vert \partial_{1}\zeta_{0}\vert^{2})^{3/2}})\mathcal{N}+\p_{1}I_{1}\mathcal{N}\\
			\quad\quad\quad\quad\quad\quad\quad\quad+\sigma\int_{0}^{t}\mathcal{R}_{zz}(\p_{1}\zeta_{0},\p_{1}\eta)(\p_{1}\p_{t}\xi)\p_{1}\p_{t}\eta ds\mathcal{N}+\sigma\partial_{1}(\int_{0}^{t}\mathcal{R}_{z}(\partial_{1}\zeta_{0},\partial_{1}\eta)\partial_{1}\partial_{t}^{2}\xi)\mathcal{N}\\
			\quad\quad\quad\quad\quad\quad\quad\quad+F^{4}(u,p,\eta)~~~&\operatorname{on}~~\Sigma,\\
			(S_{\mathcal{A}}(\p_{t}q,D_{t}v)\nu+\nabla_{\mathcal{A}}(Rv)\nu-\beta D_{t}v)\cdot \tau=F^{5}(u,\eta,p)~~~&\operatorname{on}~~\Sigma_{s},\\
			D_{t}v\cdot \nu=0~~~&\operatorname{on}~~\Sigma_{s},\\
			\partial_{t}^{2}\xi=D_{t}v\cdot \mathcal{N}+(Rv)\cdot \mathcal{N} - \int_{0}^{t}(\partial_{t}u_{1}\cdot\partial_{t}\p_{1}\xi+u_{1}\partial_{1}\partial_{t}^{2}\xi)+I_{2}~~~&\operatorname{on}~~\Sigma,\\
			\sigma\Big(\mp \frac{\partial_{1}\p_{t}\xi}{(1+\vert \partial_{1}\zeta_{0}\vert^{2})^{\frac{3}{2}}}\pm \int_{0}^{t}\mathcal{R}_{z}(\partial_{1}\zeta_{0},\partial_{1}\eta)\partial_{t}^{2}\partial_{1}\xi \pm\int_{0}^{t}\mathcal{R}_{zz}(\p_{1}\zeta_{0},\p_{1}\eta)(\p_{t}\p_{1}\xi)\p_{t}\p_{1}\eta \pm I_{1} \Big)(\pm \ell)\\
			\quad\quad\quad\quad\quad\quad\quad\quad=\kappa (D_{t}v\cdot \mathcal{N})(\pm \ell)-{F}^{7}.
		\end{cases}
	\end{equation}
	\noindent  Equation \eqref{eq:quasi_linear_0} introduces a linear map from any given functions $(u,p,\eta)$ to unknown functions $(v,q,\xi)$. Here, $\mathcal{K}$ is the gravity-capillary operator defined via $\mathcal{K}(\xi):=g(\xi)-\sigma\p_1\left(\frac{\p_1\xi}{(1+|\p_1\zeta_0|^2)^{3/2}}\right)$, and $\mathcal{N}(\eta)=(-\p_1\zeta_0-\partial_{1}\eta, 1)$.

	For the new capillary equation in \eqref{eq:quasi_linear_0}, we observe that
	\[
	\mathcal{K}(\partial_{t}\xi)\mathcal{N}
	+\partial_{1}\Big(\int_{0}^{t}\mathcal{R}_{z}\Big(\partial_{1}\zeta_{0},\partial_{1}\eta\Big)\partial_{t}^{2}\partial_{1}\xi\Big)\mathcal{N}
	+\partial_{1}\Big(\int_{0}^{t}\mathcal{R}_{z}\Big(\partial_{1}\zeta_{0},\partial_{1}\eta\Big)(\partial_{t}\partial_{1}\xi)^{2}\Big)
	\]
	is not a perfect time derivative. Moreover, the system is equipped with a second-order kinematic boundary condition:
	\[
	\partial_{t}^{2}\xi
	= D_{t}v\cdot \mathcal{N}+(Rv)\cdot \mathcal{N}-\int_{0}^{t}\Big(\partial_{t}u_{1}\cdot\partial_{t}\partial_{1}\xi+u_{1}\partial_{1}\partial_{t}^{2}\xi\Big)+I_{2},
	\]
	which is likewise not a perfect time derivative.
	
	However, when $\xi=\eta$, both expressions become total time derivatives and can therefore be used to recover the corresponding zeroth-order equation. These two observations highlight a structural subtlety: introducing a closed system directly for $D_t v$ is conceptually unnatural without prior knowledge of $v$. While the original variables $(v, q, \xi)$ can still be recovered via (1.23) below, this formulation turns the reconstruction into an indirect process. Specifically, it requires first solving for the first-order derivatives $(\p_t v, \p_t q, \p_t \xi)$, rather than determining the variables directly from a zeroth-order PDE.
	% \red{(I am still thinking that we may be able to find a better way to state this sentence, but do not know how to best do it. It seems a bit confusing for first-time reader. For example, we claim that it is difficult to reconstruct $v$, but clearly (1.23) below can do the job.)}
	% \blue{Xiaoding: I have modified it. Feel free to let me know if this revision is satisfactory. I think the key point is emphasizing that we are not going to recover the 0-order equation system for the linear problem.}
	
	For notational clarity and simplicity, we use $(v_{d},q_{d},\xi_{d})$ as a substitute for $(D_{t}v,\partial_{t}q,\partial_{t}\xi)$ in \eqref{eq:quasi_linear_0} and rewrite the system accordingly:
	\begin{equation}\label{eq:quasi_linear}
		\begin{cases}
			\partial_{t}v_{d}+\operatorname{div}_{\mathcal{A}}S_{\mathcal{A}}(v_{d},q_{d})+\dive_{\mathcal{A}}\nabla_{\mathcal{A}}(Rv)+\p_{t}(Rv)=F^{1}(u,p,\eta)~~~&\operatorname{in}~~\Omega,\\
			\operatorname{div}_{\mathcal{A}}v_{d}=0~~~&\operatorname{in}~~\Omega,\\
			S_{\mathcal{A}}(q_{d},v_{d})\mathcal{N}+\nabla_{\mathcal{A}}(Rv)\mathcal{N}=g\xi_{d}\mathcal{N}-\sigma\partial_{1}(\frac{\p_1\p_t\xi_d}{(1+\vert \partial_{1}\zeta_{0}\vert^{2})^{3/2}})\mathcal{N}+\p_{1}I_{1}\mathcal{N}\\
			\quad\quad\quad\quad\quad\quad\quad+\sigma\partial_{1}(\int_{0}^{t}\mathcal{R}_{zz}(\partial_{1}\zeta_{0},\partial_{1}\eta)(\partial_{1}\p_{t}\xi)\p_{1}\p_{t}\eta)\mathcal{N}+\sigma\partial_{1}(\int_{0}^{t}\mathcal{R}_{z}(\partial_{1}\zeta_{0},\partial_{1}\eta)\partial_{1}\partial_{t}\xi_{d})\mathcal{N}\\
			\quad\quad\quad\quad\quad\quad\quad\quad+F^{4}(u,p,\eta)~~~&\operatorname{on}~~\Sigma,\\
			(S_{\mathcal{A}}(q_{d},v_{d})\nu+\nabla_{\mathcal{A}}(Rv)\nu-\beta v_{d})\cdot \tau=F^{5}(u,\eta,p)~~~&\operatorname{on}~~\Sigma_{s},\\
			v_{d}\cdot \nu=0~~~&\operatorname{on}~~\Sigma_{s},\\
			\partial_{t}\xi_{d}=v_{d}\cdot \mathcal{N}+(Rv)\cdot \mathcal{N}+\int_{0}^{t}(\partial_{t}u_{1}\cdot\partial_{t}\p_{1}\xi+u_{1}\partial_{1}\partial_{t}\xi_{d})+I_{2}~~~&\operatorname{on}~~\Sigma,\\
			\sigma \Big(\mp \frac{\partial_{1}\xi_{d}}{(1+\vert \partial_{1}\zeta_{0}\vert^{2})^{\frac{3}{2}}}\pm \int_{0}^{t}\mathcal{R}_{z}(\partial_{1}\zeta_{0},\partial_{1}\eta)\partial_{t}\partial_{1}\xi_{d}\pm\int_{0}^{t}\mathcal{R}_{zz}(\p_{1}\zeta_{0},\p_{1}\eta)(\p_{t}\p_{1}\xi\p_{t}\p_{1}\eta)\pm I_{1} \Big)(\pm \ell)\\
			\quad\quad\quad\quad\quad\quad\quad=\kappa (v_{d}\cdot \mathcal{N})(\pm \ell)-{F}^{7},
		\end{cases}
	\end{equation}
	\noindent where $v$ and $\xi$ can be recovered via solving the ODEs
	\begin{align}\label{eq:quasi_linear2}
		\begin{aligned}
			\p_{t}v=v_{d}+Rv,\quad \p_{t}\xi=\xi_{d},
		\end{aligned}
	\end{align}
	\noindent with initial data given in Appendix \ref{sec:initial} and \ref{sec:initial_l}. 
	
	We now write down the weak formulation of the system \eqref{eq:quasi_linear}
	\begin{definition} \label{def:weak}  Suppose that $\eta$ is given as well as $\mathcal{A}$, $J$, $\mathcal{N}$, etc. that are determined in terms of $\eta$ as in \eqref{eq:Jaccobian}, \eqref{eq:components} and  $\mathcal{N} = (-\p_1(\zeta_0 + \eta), 1)$. We say that a pair $(v_{d}, \xi_{d}) \in \left(L^\infty([0,T]; H^0)\cap L^2([0,T]; H^1) \right) \times L^\infty([0,T];H^1(-\ell,\ell))$ is a pressureless weak solution to \eqref{eq:quasi_linear}, if $(\p_tv_{d},v_{d}) \in (L^2([0,T];\mathcal{W}^{*}))\times (L^2([0,T];\mathcal{W}))$, and that $(v_{d}, \xi_{d})$ satisfies the weak formulation
		\begin{equation}\label{eq:weak_limit_1}
			\begin{aligned}
				& (\p_tv_{d}, w)_{\mathcal{H}^0}+((v_{d},w))+ (\xi_{d}, w\cdot\mathcal{N})_{1,\Sigma_{0}}+(\mathcal{R}_{z}(\partial_{1}\zeta_{0},\partial_{1}\eta)\partial_{1}\xi_{d},\partial_{1}(w\cdot\mathcal{N}))_{L^{2}}+ [v_{d}\cdot\mathcal{N},w\cdot\mathcal{N}]_\ell
				\\ &=\int_\Om F^1\cdot wJ-\int_{-\ell}^{\ell} F^4\cdot w-\int_{\Sigma_s}F^5(w\cdot\tau)J -[F^7,w\cdot\mathcal{N}]_\ell+(\p_{t}(Rv),w)_{\mathcal{H}^{0}}+((Rv,w))
			\end{aligned}
		\end{equation}
		for each $w\in \mathcal{W}_{\sigma}(t)$ and for a.e. $t\in [0, T]$, provided that $F^1-F^4-F^5\in L^\infty (\mathcal{H}^1)^\ast$, $[F^7]_\ell\in L^2_t$, $F^1\in L^2L^{q_-}$, $F^4\in L^2W^{1-1/q_-,q_-}$, and $F^5\in L^2W^{1-1/q_-,q_-}$ for some $q_-\in (1, 2)$.  All notations here can be found in Section \ref{appendix_spaces}. And $(\cdot, \cdot)_{1,\Sigma_{0}}$ inner product is defined to be:
		\begin{align}
			(f_1,f_{2})_{1,\Sigma_{0}}=&g\int_{-\ell}^{\ell}f_{1}f_{2}\mathrm{d}x+\sigma\int_{-\ell}^{\ell}\frac{\partial_{1}f_{1}\partial_{1}f_{2}}{(1+\vert \partial_{1}\zeta_{0}\vert^{2})^{\frac{3}{2}}} \mathrm{d}x.\label{eq:product}
		\end{align}
		Moreover, equation\eqref{eq:weak_limit_1} should be coupled with the kinematic relation
		\begin{align} \label{eq:kinematic}
			\partial_{t}\xi_{d}=v_{d}\cdot \mathcal{N}+(Rv)\cdot \mathcal{N}+\int_{0}^{t}(\partial_{t}u_{1}\cdot\partial_{t}\p_{1}\xi+u_{1}\partial_{1}\partial_{t}\xi_{d})+I_{2}.
		\end{align}
	\end{definition}
	We note  Definition \ref{def:weak} is motivated naturally from the energy identity, which is symmetric  in $(v_d,\xi_d)$. This key symmetric nature naturally produces a self-adjoint functional structure, which  enables us to invoke a Galerkin formulation in order to construct pressureless weak solutions as well as a pressure, with average zero preserving expected a priori estimates. Thanks to the coupling of the \eqref{eq:kinematic}, with enough regularity, such a pressureless weak solution will satisfy all the boundary conditions (including the dynamical law for the contact line \eqref{eq:contact law}), \textit{except} for the stress-free boundary condition $S_{\mathcal{A}}(q_d, v_d)\mathcal{N}$. In fact, the stress-free boundary condition is satisfied up to a constant, which leads to a complete solution to the original problem with a new modified pressure in Theorem \ref{thm:pressure}.

	\noindent $\bullet$ \underline{\textsl{Galerkin method and the smooth given data}}
	
	We employ the Galerkin method to construct a solution to \eqref{eq:quasi_linear}. However, in doing so, it becomes necessary to estimate the following inner product:
	\[
	\bigl(\partial_{t}\partial_{1}\xi_{d}^{m},
	\partial_{1}u_{1}\,\partial_{t}\partial_{1}\xi_{d}^{m}\bigr)_{L^{2}(\Sigma)},
	\]
	where $\xi_{d}^{m}$ denotes the discrete surface function.
	
	At the Galerkin level, without enhanced $H^{\frac{3}{2}-}$ regularity for $\partial_{t}\xi_{d}^{m}$, this term can only be bounded as follows:
	\[
	\bigl(\partial_{t}\partial_{1}\xi_{d}^{m},
	\partial_{1}u_{1}\,\partial_{t}\partial_{1}\xi_{d}^{m}\bigr)_{L^{2}(\Sigma)}
	\lesssim
	\|\partial_{t}\xi_{d}^{m}\|_{H^{1}}^{2}\,
	\|u\|_{H^{2+}}.
	\]
	However, the $H^{2+}$-norm of $u$ cannot be controlled by the available energy or dissipation functionals. Consequently, this estimate is insufficient to close the energy inequality at the Galerkin approximation level.
	
	To overcome this difficulty, we introduce the following modified system
	with smooth given functions $(u^{k},\eta^{k},\eta^{n})$:
	\begin{equation}{\label{eq:quasi_linear_{s1}}}
		\begin{cases}
			\partial_{t}v_{d}+\operatorname{div}_{\mathcal{A}^{n}}S_{\mathcal{A}^{n}}(v_{d},q_{d})+\dive_{\mathcal{A}^{n}}\nabla_{\mathcal{A}}(R^{n}v)+\p_{t}(R^{n}v)=F^{1}(u,p,\eta)~~~&\operatorname{in}~~\Omega,\\
			\operatorname{div}_{\mathcal{A}^{n}}v_{d}=0~~~&\operatorname{in}~~\Omega,\\
			S_{\mathcal{A}^{n}}(q_{d},v_{d})\mathcal{N}^{n}+\nabla_{\mathcal{A}^{n}}(R^{n}v)\mathcal{N}^{n}=g\xi_{d}\mathcal{N}^{n}-\sigma\partial_{1}(\frac{\p_1\xi_d}{(1+\vert \partial_{1}\zeta_{0}\vert^{2})^{3/2}})\mathcal{N}^{n}+\p_{1}I_{1}^{n,k}\mathcal{N}^{n}\\
			\quad\quad\quad\quad\quad\quad\quad+\sigma\partial_{1}(\int_{0}^{t}\mathcal{R}_{zz}(\partial_{1}\zeta_{0},\partial_{1}\eta^{k})(\partial_{1}\p_{t}\xi\p_{1}\p_{t}\eta))\mathcal{N}^{n}+\sigma\partial_{1}(\int_{0}^{t}\mathcal{R}_{z}(\partial_{1}\zeta_{0},\partial_{1}\eta^{k})\partial_{1}\partial_{t}\xi_{d})\mathcal{N}^{n}\\
			\quad\quad\quad\quad\quad\quad\quad\quad+F^{4}(u,p,\eta)~~~&\operatorname{on}~~\Sigma,\\
			(S_{\mathcal{A}^{n}}(q_{d},v_{d})\nu+\nabla_{\mathcal{A}^{n}}(R^{n}v)\nu-\beta v_{d})\cdot \tau=F^{5}(u,\eta,p)~~~&\operatorname{on}~~\Sigma_{s},\\
			v_{d}\cdot \nu=0~~~&\operatorname{on}~~\Sigma_{s},\\
			\partial_{t}\xi_{d}=v_{d}\cdot \mathcal{N}^{n}+(R^{n}v)\cdot \mathcal{N}^{n}+\int_{0}^{t}(\partial_{t}u_{1}^{k}\cdot\partial_{t}\p_{1}\xi+u_{1}^{k}\partial_{1}\partial_{t}\xi_{d})+I_{2}^{n,k}~~~&\operatorname{on}~~\Sigma,\\
			\sigma\Big(\mp \frac{\partial_{1}\xi_{d}}{(1+\vert \partial_{1}\zeta_{0}\vert^{2})^{\frac{3}{2}}}\pm \int_{0}^{t}\mathcal{R}_{z}(\partial_{1}\zeta_{0},\partial_{1}\eta^{k})\partial_{t}\partial_{1}\xi_{d}\pm\int_{0}^{t}\mathcal{R}_{zz}(\p_{1}\zeta_{0},\p_{1}\eta^{k})(\p_{t}\p_{1}\xi\p_{1}\p_{t}\eta^{k})\pm I_{1}^{n,k} \Big)(\pm \ell)\\
			\quad\quad\quad\quad\quad\quad\quad\quad=\kappa (v_{d}\cdot \mathcal{N}^{n})(\pm \ell)-{F}^{7},
		\end{cases}
	\end{equation}
	\noindent where the $(\eta^{k},u^{k},\eta^{n})$ is defined in Appendix \ref{sec:smooth} by replacing $\varepsilon$ with $\frac{1}{k}$ and $\frac{1}{n}$, respectively. We have the following convergence properties by \eqref{eq:convergence_e}
	\begin{align}\label{temp 2}
		\eta^{k}\rightarrow \eta ~~\operatorname{in}~~L_{t}^{\infty}W^{3-\frac{1}{q_{-}},q_{-}}, \quad
		\partial_{t}\eta^{k}\rightarrow \partial_{t}\eta ~~\operatorname{in}~~L_{t}^{2}W^{3-\frac{1}{q_{-}},q_{-}}~~\operatorname{as}~~k\rightarrow+\infty,
	\end{align}
	\begin{align}\label{temp 3}
		u^{k}\rightarrow u~~\operatorname{in}~~L_{t}^{\infty}W^{2-\frac{1}{q_{-}},q_{-}}(\Sigma),\quad 
		\partial_{t}u^{k}\rightarrow \partial_{t}u~~\operatorname{in}~~L_{t}^{2}W^{2-\frac{1}{q_{-}},q_{-}}(\Sigma)~~\operatorname{as}~~k\rightarrow+\infty,
	\end{align}
	\noindent and
	\begin{align}{\label{convergence_n}}
		\eta^{n}\rightarrow \eta ~~\operatorname{in}~~L_{t}^{\infty}W^{3-\frac{1}{q_{+}},q_{+}}, \quad
		\partial_{t}\eta^{n}\rightarrow \partial_{t}\eta ~~\operatorname{in}~~L_{t}^{2}W^{3-\frac{1}{q_{-}},q_{-}}~~\operatorname{as}~~n\rightarrow+\infty.
	\end{align}
	\noindent Moreover, $\mathcal{N}^{n}=(-\p_{1}\zeta_{0}-\p_{1}\eta^{n},1)$. $\mathcal{A}^{n}$ and $R^{n}$ are defined by substituting $\eta$ with $\eta^{n}$. $I_{1}^{n,k}$ and $I_{2}^{n,k}$ are defined by
	\begin{align}{\label{eq:I}}
		\begin{aligned}
			I_{1}^{n,k}=(\mathcal{R}_{z}(\p_{1}\zeta_{0},\p_{1}\eta^{k}(0)))\p_{1}\xi_{d0}^{n,k},\quad
			I_{2}^{n,k}=u^{k}(0)\p_{1}\xi_{d0}^{n,k}.
		\end{aligned}
	\end{align}
	The construction of these smooth sequences is detailed in Appendix \ref{sec:smooth}, and the full definition of the initial data of $\eta^{k}(0)$, $\xi_{d}^{n,k}(0)$, $u^{k}(0)$ in \eqref{eq:I} for any $n, k$ will be given in Appendix \ref{sec:initial_l}.
	
	These two sequences allow us to first show the well-posedness of the linear problem via assumed higher regularity. Then with the solution in hand, we may introduce the enhanced estimates of $\dt\xi_d$ to further relax the regularity requirements  and close the proof. Finally, passing to the limit $k\rt\infty$ and $n\rt\infty$ leads to the desired bounds.
	
	In detail, here we introduce two different smooth approximating sequences. The first smooth sequence (with index $k$) is used to close the energy estimate at the Galerkin level. The second sequence (with index $n$) is used to enhance the regularity for geometric parameters $R$ and $J$ such that the divergence-free function $D_{t}^{2}v\in H^{1}(\Omega)$.
	
	\noindent $\bullet$ \underline{\textsl{Contraction map and Galerkin method}}
	
	To complete the Galerkin estimates, in addition to introducing smooth prescribed functions, we must also modify the unknowns $\xi$ and $v$ in \eqref{eq:quasi_linear_{s1}}.
	
	In \eqref{eq:quasi_linear_{s1}}, the kinematic boundary condition is 
	\[
	\partial_{t}\xi_{d}
	=v_{d}\cdot \mathcal{N}^{n}
	+(R^{n}v)\cdot \mathcal{N}^{n}
	-\int_{0}^{t}\Bigl(
	\partial_{t}u_{1}^{k}\cdot\partial_{t}\partial_{1}\xi
	+u_{1}^{k}\partial_{1}\partial_{t}\xi_{d}
	\Bigr)
	+I_{2}.
	\]
	The second term on the right hand side of the equation above provides us the following inner product while conducting energy estimate
	\[
	\bigl(\partial_{1}\partial_{t}\xi_{d},
	\partial_{1}(R^{n}v\cdot\mathcal{N}^{n})\bigr)_{L^{2}(\Sigma)}.
	\]
	Without elliptic regularity for $v$, this inner product cannot be controlled at the Galerkin level, since only $v\in H^{1}(\Omega)$ is available. To overcome this difficulty, we replace $(v,\xi)$ in \eqref{eq:quasi_linear_{s1}} by prescribed functions $(v_{l},\xi_{l})$. We then combine the Galerkin scheme with a contraction mapping argument to prove that the prescribed functions coincide with the solution, namely $(v_{l},\xi_{l})=(v,\xi)$.
	
	Accordingly, we reformulate the equation \eqref{eq:quasi_linear_{s1}} to the following system:
	\begin{equation}{\label{eq:quasi_linear_{sk}}}
		\begin{cases}
			\partial_{t}v_{d}+\operatorname{div}_{\mathcal{A}^{n}}S_{\mathcal{A}^{n}}(v_{d},q_{d})+\dive_{\mathcal{A}^{n}}\nabla_{\mathcal{A}}(R^{n}v_{l})+\p_{t}(R^{n}v_{l})=F^{1}(u,p,\eta)~~~&\operatorname{in}~~\Omega,\\
			\operatorname{div}_{\mathcal{A}^{n}}v_{d}=0~~~&\operatorname{in}~~\Omega,\\
			S_{\mathcal{A}^{n}}(q_{d},v_{d})\mathcal{N}^{n}+\nabla_{\mathcal{A}^{n}}(R^{n}v_{l})\mathcal{N}^{n}=g\xi_{d}\mathcal{N}^{n}-\sigma\partial_{1}(\frac{\p_1\xi_d}{(1+\vert \partial_{1}\zeta_{0}\vert^{2})^{3/2}})\mathcal{N}^{n}+\p_{1}I^{n,k}_{1}\mathcal{N}^{n}\\
			\quad\quad\quad\quad\quad\quad\quad+\sigma\partial_{1}(\int_{0}^{t}\mathcal{R}_{zz}(\partial_{1}\zeta_{0},\partial_{1}\eta^{k})(\partial_{1}\p_{t}\xi_{l})\p_{1}\p_{t}\eta^{k})\mathcal{N}^{n}\\
			\quad\quad\quad\quad\quad\quad\quad +\sigma\partial_{1}(\int_{0}^{t}\mathcal{R}_{z}(\partial_{1}\zeta_{0},\partial_{1}\eta^{k})\partial_{1}\partial_{t}\xi_{d})\mathcal{N}^{n}+F^{4}(u,p,\eta)~~~&\operatorname{on}~~\Sigma,\\
			(S_{\mathcal{A}^{n}}(q_{d},v_{d})\nu+\nabla_{\mathcal{A}^{n}}(R^{n}v_{l})\nu-\beta v_{d})\cdot \tau=F^{5}(u,\eta,p)~~~&\operatorname{on}~~\Sigma_{s},\\
			v_{d}\cdot \nu=0~~~&\operatorname{on}~~\Sigma_{s},\\
			\partial_{t}\xi_{d}=v_{d}\cdot \mathcal{N}^{n}+(R^{n}v_{l})\cdot \mathcal{N}^{n}+\int_{0}^{t}(\partial_{t}u_{1}^{k}\cdot\partial_{t}\p_{1}\xi_{l}+u_{1}^{k}\partial_{1}\partial_{t}\xi_{d})+I^{n,k}_{2}~~~&\operatorname{on}~~\Sigma,\\
			\sigma \Big(\mp \frac{\partial_{1}\xi_{d}}{(1+\vert \partial_{1}\zeta_{0}\vert^{2})^{\frac{3}{2}}}\pm \int_{0}^{t}\mathcal{R}_{z}(\partial_{1}\zeta_{0},\partial_{1}\eta^{k})\partial_{t}\partial_{1}\xi_{d}\\
			\quad\quad\quad\quad\quad\quad\quad\quad \pm\int_{0}^{t}\mathcal{R}_{z}(\p_{1}\zeta_{0},\p_{1}\eta^{k})(\p_{t}\p_{1}\xi_{l}\p_{1}\p_{t}\eta^{k})\pm I^{n,k}_{1} \Big) (\pm \ell)=\kappa (v_{d}\cdot \mathcal{N}^{n})(\pm \ell)-{F}^{7},
		\end{cases}
	\end{equation}
	where $v_{l},\xi_{l}$ are given functions such that
	\[
	v_{l}\in L_{t}^{\infty}W^{2,q_{-}}~~~~\operatorname{and}~~~\p_{t}v_{l}\in L_{t}^{2}W^{2,q_{-}},\quad
	\xi_{l}\in L_{t}^{\infty}W^{3-\frac{1}{q_{-}},q_{-}}~~~~\operatorname{and}~~~\p_{t}\xi_{l}\in L_{t}^{2}W^{3-\frac{1}{q_{-}},q_{-}},
	\]
	and
	\[
	v_{l}(0)=v_{0}^{n}~~~~\operatorname{and}~~~~~D_{t}v_{l}(0)=D_{t}v_{0},\quad
	\xi_{l}(0)=\xi_{0}^{n}~~~~\operatorname{and}~~~~~\p_{t}\xi_{l}(0)=\p_{t}\xi_{0}^{n}
	\]
	\noindent for any prescribed $n$. The definition of initial data can be found in Appendix \ref{sec:initial_l}.
	
	Again, to close the estimate at the Galerkin step, in a similar flavor to \eqref{temp 2} and \eqref{temp 3}, we introduce smooth approximations for $(v_{l},\xi_{l})$ denoted as
	$\{(v_{l}^{k},\xi_{l}^{k})\}_{k\ge1}$ such that
	\[
	\begin{aligned}
		v_{l}^{k} &\to v_{l}
		&&\text{in } L_{t}^{\infty}W^{2-\frac{1}{q_{-}},q_{-}}(\Sigma),
		&\quad
		\partial_{t}v_{l}^{k} &\to \partial_{t}v_{l}
		&&\text{in } L_{t}^{2}W^{2-\frac{1}{q_{-}},q_{-}}(\Sigma),
		\\
		\xi_{l}^{k} &\to \xi_{l}
		&&\text{in } L_{t}^{\infty}W^{3-\frac{1}{q_{-}},q_{-}},
		&\quad
		\partial_{t}\xi_{l}^{k} &\to \partial_{t}\xi_{l}
		&&\text{in } L_{t}^{2}W^{3-\frac{1}{q_{-}},q_{-}},
	\end{aligned}
	\]
	as $k\to\infty$. The definitions for $v_{l}^{k},\xi_{l}^{k}$ are given in Appendix \ref{sec:smooth} by replacing $\varepsilon$ by $\frac{1}{k}$. 
	
	Substituting $v_{l},\xi_{l}$ in the kinematic boundary condition and capillary equation by $v_{l}^{k},\xi_{l}^{k}$, the quasi-linear system
	\eqref{eq:quasi_linear_{sk}} is turned into
	\begin{equation}{\label{eq:quasi_linear_{s}}}
		\begin{cases}
			\partial_{t}v_{d}+\operatorname{div}_{\mathcal{A}^{n}}S_{\mathcal{A}^{n}}(v_{d},q_{d})+\dive_{\mathcal{A}^{n}}\nabla_{\mathcal{A}}(R^{n}v_{l})+\p_{t}(R^{n}v_{l})=F^{1}(u,p,\eta)~~~&\operatorname{in}~~\Omega,\\
			\operatorname{div}_{\mathcal{A}^{n}}v_{d}=0~~~&\operatorname{in}~~\Omega,\\
			S_{\mathcal{A}^{n}}(q_{d},v_{d})\mathcal{N}^{n}+\nabla_{\mathcal{A}^{n}}(R^{n}v_{l})\mathcal{N}^{n}=g\xi_{d}\mathcal{N}^{n}-\sigma\partial_{1}(\frac{\p_1\xi_d}{(1+\vert \partial_{1}\zeta_{0}\vert^{2})^{3/2}})\mathcal{N}^{n}+\p_{1}I^{n,k}_{1}\mathcal{N}^{n}\\
			\quad\quad\quad\quad\quad\quad\quad+\sigma\partial_{1}(\int_{0}^{t}\mathcal{R}_{zz}(\partial_{1}\zeta_{0},\partial_{1}\eta^{k})(\partial_{1}\p_{t}\xi_{l}^{k})\p_{1}\p_{t}\eta^{k})\mathcal{N}^{n}\\
			\quad\quad\quad\quad\quad\quad\quad +\sigma\partial_{1}(\int_{0}^{t}\mathcal{R}_{z}(\partial_{1}\zeta_{0},\partial_{1}\eta^{k})\partial_{1}\partial_{t}\xi_{d})\mathcal{N}^{n}+F^{4}(u,p,\eta)~~~&\operatorname{on}~~\Sigma,\\
			(S_{\mathcal{A}^{n}}(q_{d},v_{d})\nu+\nabla_{\mathcal{A}^{n}}(R^{n}v_{l})\nu-\beta v_{d})\cdot \tau=F^{5}(u,\eta,p)~~~&\operatorname{on}~~\Sigma_{s},\\
			v_{d}\cdot \nu=0~~~&\operatorname{on}~~\Sigma_{s},\\
			\partial_{t}\xi_{d}=v_{d}\cdot \mathcal{N}^{n}+(R^{n}v_{l}^{k})\cdot \mathcal{N}^{n}+\int_{0}^{t}(\partial_{t}u_{1}^{k}\cdot\partial_{t}\p_{1}\xi_{l}^{k}+u_{1}^{k}\partial_{1}\partial_{t}\xi_{d})+I^{n,k}_{2}~~~&\operatorname{on}~~\Sigma,\\
			\sigma \Big(\mp \frac{\partial_{1}\xi_{d}}{(1+\vert \partial_{1}\zeta_{0}\vert^{2})^{\frac{3}{2}}}\pm \int_{0}^{t}\mathcal{R}_{z}(\partial_{1}\zeta_{0},\partial_{1}\eta^{k})\partial_{t}\partial_{1}\xi_{d}\\
			\quad\quad\quad\quad\quad\quad\quad\quad\pm\int_{0}^{t}\mathcal{R}_{zz}(\p_{1}\zeta_{0},\p_{1}\eta^{k})(\p_{t}\p_{1}\xi^{k}_{l}\p_{1}\p_{t}\eta^{k})\pm I^{n,k}_{1} \Big)(\pm \ell)=\kappa (v_{d}\cdot \mathcal{N}^{n})(\pm \ell)-{F}^{7},
		\end{cases}
	\end{equation}
	where for notational simplicity, we still use 
	$(v_{d},q_{d},\xi_{d})$ to denote the solution of \eqref{eq:quasi_linear_{s}} for any $k,n$. This will be our starting point to construct the solution.
	
	Notice that system \eqref{eq:quasi_linear_{s}} has the following weak form
	\begin{definition} \label{def:weak1}  Suppose that $\eta^{n}$ is given as well as $\mathcal{A}^{n}$, $J^{n}$, $\mathcal{N}^{n}$, etc. that are determined in terms of $\eta^{n}$ as in \eqref{eq:Jaccobian}, \eqref{eq:components} and  $\mathcal{N}^{n} = (-\p_1(\zeta_0 + \eta^{n}), 1)$. We say that a pair $(v_{d}, \xi_{d}) \in \left(L^\infty([0,T]; H^0)\cap L^2([0,T]; H^1) \right) \times L^\infty([0,T];H^1(-\ell,\ell))$ is a pressureless weak solution to \eqref{eq:quasi_linear_{s}}, if $(\p_tv_{d},v_{d}) \in (L^2([0,T];\mathcal{W}^{*}))\times (L^2([0,T];\mathcal{W}))$, and that $(v_{d}, \xi_{d})$ satisfies the weak formulation
		\begin{equation}\label{eq:weak_limit_0}
			\begin{aligned}
				& (\p_tv_{d}, w)_{\mathcal{H}^0}+((v_{d},w))+ (\xi_{d}, w\cdot\mathcal{N}^{n})_{1,\Sigma_{0}}+(\mathcal{R}_{z}(\partial_{1}\zeta_{0},\partial_{1}\eta^{k})\partial_{1}\xi_{d},\partial_{1}(w\cdot\mathcal{N}^{n}))_{L^{2}}+ [v_{d}\cdot\mathcal{N}^{n},w\cdot\mathcal{N}^{n}]_\ell\\ &
				=\int_\Om F^1\cdot wJ^{n}-\int_{-\ell}^{\ell} F^4\cdot w-\int_{\Sigma_s}F^5(w\cdot\tau)J^{n} -[F^7,w\cdot\mathcal{N}^{n}]_\ell+(\p_{t}(R^{n}v_{l}^{k}),w)_{\mathcal{H}^{0}}+((R^{n}v_{l}^{k},w))\\
				&\quad-(I^{n,k}_{1},\p_{1}(w\cdot \mathcal{N}^{n}))_{L^{2}(-\ell.\ell)}
			\end{aligned}
		\end{equation}
		for each $w\in \mathcal{W}_{\sigma}(t)$ and a.e. $t\in [0, T]$, provided that $F^1-F^4-F^5\in L^\infty (\mathcal{H}^1)^\ast$, $[F^7]_\ell\in L^2_t$, $F^1\in L^2L^{q_-}$, $F^4\in L^2W^{1-1/q_-,q_-}$, and $F^5\in L^2W^{1-1/q_-,q_-}$ for some $q_-\in (1, 2)$.  All notations here can be found in Section \ref{appendix_spaces}. And $(1,\Sigma_{0})$ norm is defined in \eqref{eq:product}. Also, to derive a solution, equation\eqref{eq:weak_limit_0} should be coupled with kinematic relation
		\begin{align} \label{eq:kinematic_0}
			\partial_{t}\xi_{d}=v_{d}\cdot \mathcal{N}^{n}+(R^{n}v_{l}^{k})\cdot \mathcal{N}^{n}+\int_{0}^{t}(\partial_{t}u_{1}^{k}\cdot\partial_{t}\p_{1}\xi_{l}^{k}+u_{1}^{k}\partial_{1}\partial_{t}\xi_{d})+I^{n,k}_{2}.
		\end{align}
	\end{definition}
	
	\noindent $\bullet$ \underline{\textsl{Strong solution and the construction of the pressure}}
	
	Once the weak solution $(v_{d},\xi_{d})$ to \eqref{eq:weak_limit_0} has been derived, we construct the pressure $\overset{\circ}q_{d}:=q_{d}-\bar{q}_{d}$ from the pressureless $(v_{d},\xi_{d})$. Furthermore, $(v_{d}, \overset{\circ}q_{d}
	, \xi_{d})$ is a strong solution to \eqref{eq:quasi_linear_{s}} and the difference between the real
	pressure $q$ and $\overset{\circ}q$
	is a constant depending only on time. We next determine the constant needed for our pressuresless weak solutions to satisfy the capillary equation.  Due to the structure of the problem, we may choose such a constant as the average of the final pressure, solving the original full set $(v_{d},q_{d},\xi_{d})$ of equations.
	
	\noindent $\bullet$ \underline{\textsl{Reconstruction of $(v,q,\xi)$}}
	
	After deriving the regularity for $(v_{d},q_{d},\xi_{d},\p_{t}v_{d},\p_{t}\xi_{d})$, it remains to establish the regularity of the zero-order terms $(v,q,\xi)$.  We recover $v,q,\xi$ (also depending on $k,n$) by solving the following ODE systems
	\begin{align}{\label{eq:ODE}}
		\begin{aligned}
			&v(t)=v(0)+\int_0^t\dt v(s)\mathrm{d} s=v(0)+\int_0^t\big(v_d(s)+R^{n}v(s)\big)\mathrm{d}s,\\   &\xi(t)=\xi(0)+\int_{0}^{t}\xi_d(s)\mathrm{d}s,\quad
			q(t)=q(0)+\int_{0}^{t}q_d(s)\mathrm{d}s,
		\end{aligned}
	\end{align}
	\noindent with the initial conditions
	\[
	D_{t}v(0)=v^{n,k}_{d0},~~~~v(0)=v_{0}^{n,k},~~~~\xi_{d}(0)=\xi_{d0}^{n,k},~~~~\xi(0)=\xi_{0}^{n,k},~~~~q_{d}(0)=q_{d0}^{n,k},~~~~q(0)=q_{0}^{n,k},
	\]
	\noindent where the definition of $v_{d0}^{n,k},v_{0}^{n,k},\xi^{n,k}_{d0},\xi_{0}^{n,k},q_{d0}^{n,k},q_{0}^{n,k}$ can all be found in Appendix \ref{sec:initial_l}.
	
	After integrating in time and using the standard estimate
	\[
	\|f\|_{L_{t}^{\infty}\mathcal{B}}\lesssim \|f(0)\|_{\mathcal{B}}+\int_{0}^{t}\|\p_{t}f\|_{\mathcal{B}}ds\lesssim \|f(0)\|_\mathcal{B}+t^{\frac{1}{2}}\|\p_{t}f\|_{L_{t}^{2}\mathcal{B}},
	\]
	where $\mathcal{B}$ can represent the norm of an arbitrary Sobolev space, we obtain the desired regularity for $(v,q,\xi)$ and $(\partial_{t}v,\partial_{t}q,\partial_{t}\xi)$ directly from the regularity of $(v_{d},q_{d},\xi_{d})$ after integration. Moreover, this procedure preserves the boundedness of the corresponding norms.
	
	\noindent $\bullet$ \underline{\textsl{Passing to the limit  $k,n \to \infty$} and identifying a contractive mapping} 
	
	After establishing the existence of a strong solution to \eqref{eq:quasi_linear_{s}}, our next step is to derive bounds uniform in $k$, for fixed $n$ and prescribed functions $(u,\eta,p,\eta^{n},v_{l},\xi_{l})$.
	
	With the uniform bounds in hand, letting $k\to +\infty$, the solution to \eqref{eq:quasi_linear_{s}} converges to the solution to the following limiting system:
	\begin{equation}{\label{eq:quasi_linear_{k}}}
		\begin{cases}
			\partial_{t}v_{d}+\operatorname{div}_{\mathcal{A}^{n}}S_{\mathcal{A}^{n}}(v_{d},q_{d})+\dive_{\mathcal{A}^{n}}\nabla_{\mathcal{A}}(R^{n}v_{l})+\p_{t}(R^{n}v_{l})=F^{1}(u,p,\eta)~~~&\operatorname{in}~~\Omega,\\
			\operatorname{div}_{\mathcal{A}^{n}}v_{d}=0~~~&\operatorname{in}~~\Omega,\\
			S_{\mathcal{A}^{n}}(q_{d},v_{d})\mathcal{N}^{n}+\nabla_{\mathcal{A}^{n}}(R^{n}v_{l})\mathcal{N}^{n}=g\xi_{d}\mathcal{N}^{n}-\sigma\partial_{1}(\frac{\p_1\xi_d}{(1+\vert \partial_{1}\zeta_{0}\vert^{2})^{3/2}})\mathcal{N}^{n}+\p_{1}I^{n}_{1}\mathcal{N}^{n}\\
			\quad\quad\quad\quad\quad\quad\quad+\sigma\partial_{1}(\int_{0}^{t}\mathcal{R}_{zz}(\partial_{1}\zeta_{0},\partial_{1}\eta)(\partial_{1}\p_{t}\xi_{l})\p_{1}\p_{t}\eta)\mathcal{N}^{n}+\sigma\partial_{1}(\int_{0}^{t}\mathcal{R}_{z}(\partial_{1}\zeta_{0},\partial_{1}\eta)\partial_{1}\partial_{t}\xi_{d})\mathcal{N}^{n}\\
			\quad\quad\quad\quad\quad\quad\quad+F^{4}(u,p,\eta)~~~&\operatorname{on}~~\Sigma,\\
			(S_{\mathcal{A}^{n}}(q_{d},v_{d})\nu+\nabla_{\mathcal{A}^{n}}(R^{n}v_{l})\nu-\beta v_{d})\cdot \tau=F^{5}(u,\eta,p)~~~&\operatorname{on}~~\Sigma_{s},\\
			v_{d}\cdot \nu=0~~~&\operatorname{on}~~\Sigma_{s},\\
			\partial_{t}\xi_{d}=v_{d}\cdot \mathcal{N}^{n}+(R^{n}v_{l})\cdot \mathcal{N}^{n}+\int_{0}^{t}(\partial_{t}u_{1}\cdot\partial_{t}\p_{1}\xi_{l}+u_{1}\partial_{1}\partial_{t}\xi_{d})+I^{n}_{2}~~~&\operatorname{on}~~\Sigma,\\
			\sigma \Big(\mp \frac{\partial_{1}\xi_{d}}{(1+\vert \partial_{1}\zeta_{0}\vert^{2})^{\frac{3}{2}}}\pm \int_{0}^{t}\mathcal{R}_{z}(\partial_{1}\zeta_{0},\partial_{1}\eta)\partial_{t}\partial_{1}\xi_{d}\pm\int_{0}^{t}\mathcal{R}_{zz}(\p_{1}\zeta_{0},\p_{1}\eta^{k})(\p_{t}\p_{1}\xi_{l}\p_{1}\p_{t}\eta^{k})\pm I^{n}_{1} \Big)(\pm \ell)\\
			\quad\quad\quad\quad\quad\quad\quad\quad=\kappa (v_{d}\cdot \mathcal{N}^{n})(\pm \ell)-{F}^{7}.
		\end{cases}
	\end{equation}
	
	Notice that for fixed $n$ and $(u,p,\eta,\eta^{n},v_{l},\xi_{l})$, system \eqref{eq:quasi_linear_{k}} defines a linear mapping from $(v_{l},\xi_{l})$ to $(v_{d},\xi_{d})$. After reconstructing $(v,\xi)$ from $(v_{d},\xi_{d})$, we prove in Section \ref{sec:strong} that this mapping is a contraction from $(v_{l},\xi_{l})$ to $(v,\xi)$. Consequently, there exists a solution to the following system:
	\begin{equation}{\label{eq:quasi_linear_n}}
		\begin{cases}
			\partial_{t}^{2}v
			+\operatorname{div}_{\mathcal{A}^{n}}
			S_{\mathcal{A}^{n}}(\partial_{t}v,\partial_{t}q)
			=F^{1}(u,p,\eta)
			& \text{in } \Omega, \\[0.4em]
			\operatorname{div}_{\mathcal{A}^{n}}D_{t}v=0
			& \text{in } \Omega, \\[0.4em]
			S_{\mathcal{A}^{n}}(\partial_{t}q,\partial_{t}v)\mathcal{N}^{n}
			= g\partial_{t}\xi\,\mathcal{N}^{n}
			-\sigma\partial_{1}\!\left(
			\frac{\p_1\p_t\xi}{(1+\vert \partial_{1}\zeta_{0}\vert^{2})^{3/2}}
			\right)\mathcal{N}^{n} \\[0.2em]
			\qquad  \quad\quad\quad\quad\quad\quad\quad
			+\sigma\partial_{1}\!\bigl(
			\mathcal{R}_{z}(\partial_{1}\zeta_{0},\partial_{1}\eta)
			\,\partial_{1}\partial_{t}\xi
			\bigr)\mathcal{N}^{n}
			+F^{4}(u,p,\eta)
			& \text{on } \Sigma, \\[0.4em]
			\bigl(
			S_{\mathcal{A}^{n}}(\partial_{t}q,\partial_{t}v)\nu
			-\beta\,\partial_{t}v
			\bigr)\cdot\tau
			=F^{5}(u,\eta,p)
			& \text{on } \Sigma_{s}, \\[0.4em]
			\partial_{t}v\cdot\nu=0
			& \text{on } \Sigma_{s}, \\[0.4em]
			\partial_{t}^{2}\xi
			=\partial_{t}v\cdot\mathcal{N}^{n}
			+\partial_{t}u_{1}\partial_{1}\xi
			& \text{on } \Sigma, \\[0.4em]
			\sigma\!\left(
			\mp\frac{\partial_{1}\partial_{t}\xi}
			{(1+|\partial_{1}\zeta_{0}|^{2})^{3/2}}
			\pm
			\mathcal{R}_{z}(\partial_{1}\zeta_{0},\partial_{1}\eta)
			\partial_{1}\partial_{t}\xi
			\right)(\pm \ell)
			=\kappa(\partial_{t}v\cdot\mathcal{N}^{n})(\pm \ell)-F^{7}.
		\end{cases}
	\end{equation}
	
	Finally, with the help of uniform bounds in $n$, we show that as $n\to +\infty$, the solution to \eqref{eq:quasi_linear_n} converges to the solution to \eqref{eq:quasi_linear}. This completes the construction and yields a solution to the original linear problem \eqref{eq:quasi_linear}.
	
	% \noindent $\bullet$ \underline{\textsl{Regularity for solution of \eqref{eq:quasi_linear}}}
	
	% As $n,k\rightarrow +\infty$ and considering the contraction mapping argument, we know that the solutions to \eqref{eq:quasi_linear_{s}} converge toward a solution of \eqref{eq:quasi_linear}. 
	
	Also, the regularity of the limiting solution can be deduced from the uniform estimates for \eqref{eq:quasi_linear_{s}} and the convergence argument to further \eqref{eq:quasi_linear_{k}}, \eqref{eq:quasi_linear_n} and \eqref{eq:quasi_linear}.
	
	\noindent  $\bullet$ \underline{\textsl{Construction of initial data} }
	All the pseudo-linear and full nonlinear systems are initial--boundary value problems, and thus they rely on suitably prescribed initial data. However, from the \textit{a priori estimates} in \cite{GT2020}, the energy does not propagate in time, preventing us from constructing the initial data in the usual way. 
	
	Usually, given $(u_0, \eta_0)$, by the Navier-Stokes system, we construct $(\p_tu(0), p_0, \p_t\eta(0))$, and subsequently obtain $(\p_t^2u(0), \p_tp(0)$, $\p_t^2\eta(0))$ by differentiating the Navier-Stokes system in time. 
	
	However, here we need a different yet more sophisticated approach. In Appendix \ref{sec:initial} -- \ref{sec:initial_l}, we apply the backward in time derivative method (for instance, given $(\p_t^2u(0), \p_t^2\xi(0))$, then construct $(\p_tu(0), \p_tp(0), \p_t\xi(0), u_0, p_0, \xi_0)$). In particular, $\p_t\eta(0) \in W^{2-1/q_{+}, q_{+}}((-\ell, \ell))$ appears naturally as part of boundary conditions. Hence, although this quantity is not included in the energy itself, it is necessary to prove that the admissible set of initial data with finite initial energy is nonempty. This constitutes the final step needed to formulate the nonlinear problem within the functional framework used in this paper.
	
	%%%%%%%%%%%%%%%%%%%%%%%%%%%%%%%%%%%%%%%%%%%%%%%%%%%%%%%%%%%%%%%%%%%%%%%%%
	%\subsection{Appendix and Notation}
	%%%%%%%%%%%%%%%%%%%%%%%%%%%%%%%%%%%%%%%%%%%%%%%%%%%%%%%%%%%%%%%%%%%%%%%%%
	%
	%
	%Due to the complexity of the problem, we put all the notation, including definitions of norms, inner products etc, in Appendix for the convenience of
	%the reader.

	% %%%%%%%%%%%%%%%%%%%%%%%%%%%%%%%%%%%%%%%%%%%%%%
	% \subsection{Organization of the paper}
	% %%%%%%%%%%%%%%%%%%%%%%%%%%%%%%%%%%%%%%%%%%%%%%

	% In Section 2, we developed the machinery of time-dependent function spaces with the $\dive_{\mathcal{A}}$--free vector fields and many basic estimates used in the nonlinear energy estimates. In Section 3, we proved that initial data for approximated solutions in Galerkin method really converge to the initial data for original nonlinear system. In section 4, we constructed the solutions to the linear system by time-dependent Galerkin method based on the convergence of initial data in section 3. In section 5, we studied the local well-posedness of $\varepsilon$--linear problem for a new unknown $(u, p, \xi)$ with the fixed point theory based on the linear iteration system in section 4. In section 6, we gave the uniform $\varepsilon$ estimates and then gave the main theorem in section 7.

	%%%%%%%%%%%%%%%%%%%%%%%%%%%%%%%%%%%%%%%%%%%%%%
	\section{Functional Settings and Basic Estimates} \label{appendix_spaces}
	%%%%%%%%%%%%%%%%%%%%%%%%%%%%%%%%%%%%%%%%%%%%%%
	
	In this section, we introduce the functional framework needed for the implementation of the Galerkin scheme. For simplicity, all definitions are formulated with respect to the surface function 
	$\eta$. The corresponding functional spaces associated with the smooth surface function 
	$\eta^{n}$
	are defined analogously by replacing $\eta$ with $\eta^{n}$.

	%%%%%%%%%%%%%%%%%%%%%%%%%%%%%%%%%%%%%%%%%%%%%%
	\subsection{Constants and Sobolev Norms}\label{sec:notation}
	%%%%%%%%%%%%%%%%%%%%%%%%%%%%%%%%%%%%%%%%%%%%%%

	Let $C>0$ denote a universal constant that only depends on the parameters of the problem, $N$ and $\Om$, but is independent of the data. This constant is allowed to change from line to line. We will write $C=C(z)$ to indicate that the constant $C$ depends on $z$. We write $a\lesssim b$ if $a\le C b$ for a universal constant $C>0$.
	
	We will write $H^k$ for $H^k(\Om)$ for $k\ge0$, and $H^s(\Sigma)$ with $s\in\mathbb{R}$ for the usual Sobolev spaces. Typically, we will write $H^0=L^2$, with the exception that we will use $L^2([0,T];H^k)$ (or $L^2([0,T];H^s(\Sigma))$) to denote the space of temporal square--integrable functions with values in $H^k$ (or $H^s(\Sigma)$).
	
	Sometimes we will write $\|\cdot\|_k$ instead of $\|\cdot\|_{H^k(\Om)}$ or $\|\cdot\|_{H^k(\Sigma)}$. We assume that functions belong to their respective natural spaces. For example, the functions $u$, $p$ and $\bar{\eta}$ live on $\Om$, while $\eta$ lives on $\Sigma$. So we may write $\|\cdot\|_{H^k}$ for the norms of $u$, $p$ and $\bar{\eta}$ in $\Om$, and $\|\cdot\|_{H^s}$ for norms of $\eta$ on $\Sigma$.

	%%%%%%%%%%%%%%%%%%%%%%%%%%%%%%%%%%%%%%%%%%%%%%
	\subsection{Time-Independent Inner Products and Function Spaces}
	%%%%%%%%%%%%%%%%%%%%%%%%%%%%%%%%%%%%%%%%%%%%%%

	We define the time-independent spaces
	\[
	{}_0H^1(\Om):=\Big\{u\in L^2(\Om)\Big|\|u\|_{{}_0H^1}<\infty, u\cdot\nu=0\ \text{on}\ \Sigma_s\Big\},
	\]
	endowed with the norm 
	\[\|u\|_{{}_0H^1}^{2}:=\|\nabla u\|_{L^{2}(\Omega)}^{2}+\beta\|u\cdot \tau\|^{2}_{L^{2}(\Sigma_{s})}+[u\cdot \mathcal{N}_{0}]_{\ell}^2\]
	and 
	\[
	\mathring{H}^k(U):=\left\{f\in H^k(U)\Big| \int_Uf=0\right\},
	\]
	where $U=\Om$ or $(-\ell, \ell)$ and $k\in\mathbb{N}$,
	\[
	W:=\Big\{u\in{}_0H^1(\Om)\Big|u\cdot\mathcal{N}_0\in \mathring{H}^1(-\ell, \ell)\Big\},
	\]
	and
	\[
	W_\sigma:=\Big\{u\in W\Big| \dive u=0\Big\}.
	\]
	
	Throughout the paper we will often utilize the following Korn-type inequality. 
	
	\begin{lemma}\label{lem:korn}
		For any $u\in{}_0H^1(\Om)$ and $u \in H^0(\Om)$, it holds that
		\begin{equation}\label{eq:korn_1}
			\|u\|_1^2\lesssim \|\mathbb{D}u\|_0^2+\|u\|_0^2 \text{ for all } u \in H^1(\Om),
		\end{equation}
	\end{lemma}
	\begin{proof}
		The inequality \eqref{eq:korn_1} may be proved in various ways.  See \cite{N} for a direct proof.  It can  also be derived from the Ne\^{c}as inequality  \cite[Lemma IV.7.6 ]{boyer_fabrie}.
	\end{proof}
	Using this lemma, we establish the following Poincar\'e-type inequality
	\begin{theorem}{\label{thm:poin}}
		For any $u\in{}_0H^1(\Om)$, it holds that
		\[
		\|u\|_{1}^{2}\lesssim \|\mathbb{D}u\|_{0}^{2}
		\]
		
	\end{theorem}
	\begin{proof}
		From Lemma \ref{lem:korn}, it suffices to show that
		\begin{align}{\label{equ:poin_1}}
			\|u\|_{L^{2}}^{2}\lesssim \|\mathbb{D}u\|_{0}^{2}
		\end{align}
		for any $u\in {}_0H^{1}$. We derive this inequality via a contradiction argument.
		
		Suppose, for the sake of contradiction, that \eqref{equ:poin_1} is false. Then there exists a sequence $u_{n}$ such that
		\begin{align}{\label{equ:poin_2}}
			\|u_{n}\|_{L^{2}}=1,
		\end{align}
		\noindent and
		\begin{align}{\label{equ:poin_3}}
			\|\mathbb{D}u_{n}\|_{0}^{2}\leq \frac{1}{n}.
		\end{align}
		By Lemma \ref{lem:korn} again, $u_{n}$ is uniformly bounded in $H^{1}$. Therefore, $u_{n}\rightarrow u_{0}$ in $H^{1}$ as $n\rightarrow \infty$. By equation \eqref{equ:poin_3}, we have
		\begin{align}{\label{equ:poin_4}}
			\|\mathbb{D}u_{0}\|_{0}^{2}=0
		\end{align}
		\noindent which implies that $u_{0}=Ax+b$ for some skew-symmetric matrix $A$ and constant vector $b$.
		
		Since $u_{n}\in {}_0H^{1}$, we have $u_{0}\cdot \nu =0$ on $\Sigma_{s}$ by trace theorem. Therefore, we obtain
		\begin{align}{\label{equ:point_5}}
			(A(x,0)^{T}+b)\cdot (0,1)=0
		\end{align}
		\noindent for all $x\in (-\ell,\ell)$. Suppose that 
		\[
		A=
		\begin{pmatrix}
			0 &a_{2}\\
			-a_{2} &0
		\end{pmatrix},
		\]
		and that $b=(b_{1},b_{2})$. We derive $a_{2}=b_{2}=0$ from \eqref{equ:point_5}. Hence, $A=0$ and $u_{0}=b=(b_{1},0)$.
		
		Again, using the fact that $u_{0}\cdot {N}=0$ on $\Sigma_{s}$, we have
		\[
		u_{0}\cdot (1,0)=0
		\]
		\noindent for all $y\in (0,\zeta(\ell))$. Therefore
		\[
		b_{1}=0.
		\]
		This implies that $u_{0}=0$, which contradicts the fact that $\|u_{0}\|_{L^{2}}=0$.
	\end{proof}

	%%%%%%%%%%%%%%%%%%%%%%%%%%%%%%%%%%%%%%%%%%%%%%
	\subsection{Time-Dependent Inner Products and Function Spaces}
	%%%%%%%%%%%%%%%%%%%%%%%%%%%%%%%%%%%%%%%%%%%%%%

	Suppose that $\eta$ is given, and that $\mathcal{A}$, $J$ and $\mathcal{N}$, etc., are determined in terms of $\eta$. Let us define
	\[
	((u,v)):=\int_{\Om}\frac{\mu}{2}\mathbb{D}_{\mathcal{A}}u:\mathbb{D}_{\mathcal{A}}vJ+\int_{\Sigma_s}\beta(u\cdot\tau)(v\cdot\tau)J.
	\]
	We also define
	\begin{equation}{\label{eq:inner_product}}
		(\phi,\psi)_{1,\Sigma}:=\int_{-\ell}^{\ell}g\phi\psi+\sigma\frac{\p_1\phi\p_1\psi}{(1+|\p_1\zeta_0|^2)^{3/2}}+\int_{-\ell}^{\ell}\mathcal{R}_{z}(\partial_{1}\zeta_{0},\partial_{1}\eta)\partial_{1}\phi\partial_{1}\psi,
	\end{equation}
	\begin{equation}{\label{eq:inner_product——1}}
		(\phi,\psi)_{1,\Sigma_{k}}:=\int_{-\ell}^{\ell}g\phi\psi+\sigma\frac{\p_1\phi\p_1\psi}{(1+|\p_1\zeta_0|^2)^{3/2}}+\int_{-\ell}^{\ell}\mathcal{R}_{z}(\partial_{1}\zeta_{0},\partial_{1}\eta^{k})\partial_{1}\phi\partial_{1}\psi,
	\end{equation}
	\[
	(\phi,\psi)_{1,\Sigma_{0}}:=\int_{-\ell}^{\ell}g\phi\psi+\sigma\frac{\p_1\phi\p_1\psi}{(1+|\p_1\zeta_0|^2)^{3/2}},
	\]
	\begin{equation}\label{sum_point}
		[a,b]_\ell:=\kappa\Big(a(\ell)b(\ell)+a(-\ell)b(-\ell)\Big).
	\end{equation}
	We denote $\|\xi\|_{1,\Sigma}:=\sqrt{(\xi,\xi)_{1,\Sigma}}$ and $[\phi]_\ell :=\sqrt{[\phi,\phi]_\ell}$.
	
	For convenience, let us define some time-dependent spaces
	\[
	\mathcal{H}^0(t):=\Big\{u(t):\Om\rightarrow\mathbb{R}^2\Big|\sqrt{J}u\in H^0(\Om)\Big\},
	\]
	\[
	\mathcal{H}^1(t):=\Big\{u(t):\Om\rightarrow\mathbb{R}^2\Big|((u,u))+[u\cdot \mathcal{N}]_{\ell}^2<\infty \Big\},
	\]
	\[
	{}_0\mathcal{H}^1(t):=\Big\{u(t)\in \mathcal{H}^1(t)\Big|u\cdot\nu=0\ \text{on}\ \Sigma_s\Big\},
	\]
	endowed with the norm $\|u\|_{{}_0\mathcal{H}^1}:= \sqrt{ ((u,u))+[u\cdot \mathcal{N}]_{\ell}^2 }$, and 
	\[
	\mathring{\mathcal{H}}^k(t):=\left\{f\in \mathcal{H}^k(t)\Big| \int_\Om f=0\right\},
	\]
	where $k\in\mathbb{N}$.
	
	Finally, we define the inner products on $L^2([0, T]; H^k(\Om))$ for $k=0,1$ as
	$
	(u,v)_{\mathcal{H}^1_T}=\int_0^T\big(u(t),v(t)\big)_{\mathcal{H}^1}\,\mathrm{d}t$,
	and write $\mathcal{H}^1_T = L^2([0, T]; \mathcal{H}^1(t))$ as the corresponding spaces with the corresponding norms $\|u\|_{\mathcal{H}^1_T}$.
	
	For some forces $\mathcal{F}^j$ and $\psi \in \mathcal{H}^1$, denote the duality
	\begin{align}\label{duality}
		\left<\mathcal{F}^j,\psi\right>_{(\mathcal{H}^1_T)^\ast} :=\int_0^T\left<\mathcal{F}^j,\psi\right>_{(\mathcal{H}^1)^\ast},\quad \left<\mathcal{F}^j,\psi\right>_{(\mathcal{H}^1)^\ast}: =\int_{\Om}F^{1,j}\cdot \psi J-\int_{-\ell}^{\ell}F^{4,j}\cdot \psi-\int_{\Sigma_s}F^{5,j}(\psi\cdot\tau)J.
	\end{align}
	
	The following lemma implies that $\mathcal{H}^1(t)$ is equivalent to ${}_0H^1(\Om)$.
	\begin{lemma}\label{lem:equivalence_norm}
		There exists a small universal $\delta_0>0$ such that if $\sup_{0\le t\le T}\|\eta(t)\|_{W^{3-1/q_-, q_-}}<\delta_0$,
		then
		$
		\frac{1}{\sqrt{2}}\|u\|_k\le\|u\|_{\mathcal{H}^k}\le\sqrt{2}\|u\|_k$
		for $k=0, 1$ and for all $t\in[0, T]$. As a consequence, for $k=0, 1$,
		$
		\|u\|_{L^2H^k(\Om)}\le \|u\|_{\mathcal{H}^k_T}\le \|u\|_{L^2H^k(\Om)}$.
	\end{lemma}
	\begin{proof}
		For the case of $k=0$, it is directly the result in \cite{GT1}. When $k=1$, we might use the similar calculation as the proof of Lemma 2.1 in \cite{GT1} together with the Korn-type inequality in Lemma \ref{lem:korn} and trace theory to derive the results.
	\end{proof}

	Now, we want to show that the time-independent spaces are related to the time-dependent spaces.
	We consider the matrix
	\begin{equation}\label{def:M}
		M:=M(t)=K\nabla\Phi=(J\mathcal{A}^T)^{-1},
	\end{equation}
	which induces a linear operator $\mathcal{M}_t : u\mapsto M(t)u$.
	
	\begin{proposition}\label{prop:isomorphism}
		Assume that $\eta\in W^{3-1/q_+, q_+}(-\ell, \ell)$. We have:
		$(i)$. For each $t\in [0,T]$, $\mathcal{M}_t $ is a bounded isomorphism from $W^{k,q_+}(\Om)$ to $W^{k,q_+}(\Om(t))$ for $k=0, 1,2$; from $W^{k,q_-}(\Om(0))$ to $W^{k,q_-}(\Om)$ for $k=0, 1,2$.
		$(ii)$. For each $t\in [0,T]$, $\mathcal{M}_t $ is a bounded isomorphism from ${}_0H^1(\Om)$ to ${}_0\mathcal{H}^1(\Om)$. Moreover,
		$
		\|Mu\|_{\mathcal{H}^1}\lesssim (1+\|\eta\|_{W^{3-1/q_-, q_-}})\|u\|_1$.
		$(iii)$. For each $t
		\in [0,T]$, $\mathcal{M}_t $ is a bounded isomorphism from $H^{1+\varepsilon_-/2}(\Om)$ to $H^{1+\varepsilon_-/2}(\Om)$. $(iv)$. Let $u\in H^1(\Om)$. Then $\dive u=p$ if and only if $\dive_{\mathcal{A}}(Mu)=Kp$.
		$(v)$ $ M^\top\mathcal{N}=\mathcal{N}_0\ \text{on}\ \Sigma$.
	\end{proposition}
	\begin{proof}
		All the results are referred to \cite{GT2020}. The equality $J\mathcal{A}\mathcal{N}_0=\mathcal{N}$
		gives $M^\top\cdot\mathcal{N}= K\nabla \Phi^\top \cdot \mathcal{N}=\mathcal{N}_0$.
	\end{proof}
	The following proposition is also useful.
	\begin{proposition}\label{prop:solid_boundary}
		If $u\cdot\nu=0$ on $\Sigma_s$, then $Ru\cdot\nu=0$ on $\Sigma_s$, where $R:=\p_tMM^{-1}$.
	\end{proposition}
	\begin{proof}
		From \cite{GT18}, it is known that $Mu\cdot\nu=0\Leftrightarrow u\cdot\nu=0$ on $\Sigma_s$, which implies that
		$M^{-1}u\cdot\nu=0\Leftrightarrow u\cdot\nu=0$ on $\Sigma_s$.
		Then by definition of $R$,
		\[
		Ru\cdot\nu=\p_tMM^{-1}u\cdot\nu=-M\p_t(M^{-1}u)\cdot\nu=0,
		\]
		since $\p_t(M^{-1}u)\cdot\nu=\p_t(M^{-1}u\cdot\nu)=0$.
	\end{proof}
	
	Define the time-dependent spaces:
	\begin{align}\label{def:w}
		\mathcal{W}(t):=\{u\in{}_0\mathcal{H}^1(t): \ u\cdot\mathcal{N}\in \mathring{\mathcal{H}}^1(-\ell, \ell)\}
	\end{align}
	endowed with the norm $\|u\|_{\mathcal{W}}=\|u\|_{{}_{0}\mathcal{H}^{1}}+\|u\cdot \mathcal{N}\|_{H^{1}}$ and the subspace $\mathcal{W}_\sigma(t)$
	\begin{align}\label{def:w-sigma}
		\mathcal{W}_\sigma(t):=\{u\in \mathcal{W}(t): \dive_{\mathcal{A}}u=0\}.
	\end{align}
	
	Define the operator $D_t$ as
	\begin{equation}\label{def:Dt_u}
		D_tu:=\p_tu-Ru\quad\text{for}\quad R:=\p_tMM^{-1},
	\end{equation}
	with $M=K\nabla\Phi$, where $K$ and $\Phi$ are defined as in \eqref{eq:components} and \eqref{def:map}, respectively. It is easy to see that $D_t$ preserves the $\dive_{\mathcal{A}}$--free condition, since
	\begin{equation}\label{eq:dive_dt}
		J\dive_{\mathcal{A}}(D_tv)=J\dive_{\mathcal{A}}(M\p_t(M^{-1}v))=\dive(\p_t(M^{-1}v))=\p_t\dive(M^{-1}v)=\p_t(J\dive_{\mathcal{A}}v),
	\end{equation}
	where in the second and last equalities, we used the fact that $J\dive_{\mathcal{A}}v=\dive(M^{-1}v)$, which is proved, according to \cite[Lemma A]{ZhT17} and the definition \eqref{def:M} of $M$, as
	\begin{equation}\label{eq:div-a-div}
		J\dive_{\mathcal{A}}v=J\mathcal{A}_{ij}\p_jv_i=\p_j(J\mathcal{A}_{ij}v_i)=\dive(J\mathcal{A}^\top v)=\dive(M^{-1}v).
	\end{equation}

	%%%%%%%%%%%%%%%%%%%%%%%%%%%%%%%%%%%%%%%%%%%%%%
	\section{Preliminaries for the Linear System}\label{sec:linear}
	%%%%%%%%%%%%%%%%%%%%%%%%%%%%%%%%%%%%%%%%%%%%%%

	In this section, we will introduce several preliminary results regarding the linear system \eqref{eq:quasi_linear} under the assumption that the forcing terms $F^i$ are given and independent of the unknowns. In particular, we will show that the weak solution $(v_d,\xi_d)$ can be used to recover the pressure $q_d$ and further be upgraded to the strong solution under proper regularity assumptions.

	%%%%%%%%%%%%%%%%%%%%%%%%%%%%%%%%%%%%%%%%%%%%%%
	\subsection{Uniqueness of Weak Solution}
	%%%%%%%%%%%%%%%%%%%%%%%%%%%%%%%%%%%%%%%%%%%%%%

	The definition of a weak solution was given in Definition \ref{def:weak}.
	In Section \ref{sec:app_initial}, we will see that weak solutions to \eqref{eq:weak_limit_1} will arise  as a byproduct of the construction of strong solutions to \eqref{eq:weak_limit_1}. Hence, we temporarily assume the existence of weak solutions and focus first on establishing uniqueness.
	
	\begin{lemma}{\label{elemma}}
		Suppose that $(v_{d}, \xi_{d})$ is a pressureless weak solution to \eqref{eq:weak_limit_1}. Then for almost every $t\in [0,T]$,
		\begin{equation}{\label{bound}}
			\begin{aligned}
				&\sup_{0\le t\le T}(\|v_{d}\|_{0}^2+\|\xi_{d}\|_1^2)+\|v_{d}\|_{L^\infty H^1}^2+\|v_{d}\|_{L^\infty H^0(\Sigma_s)}^2+\|[v_{d}\cdot\mathcal{N}]_\ell\|_{L^\infty([0,T])}^2
				\lesssim\exp(C_{0}T)\\
				&\quad\times\bigg\{\|u_0\|^{2}_{W^{2,q_+}(\Om)}+\|D_{t}u(0)\|^{2}_{H^{1+\frac{\varepsilon_{-}}{2}}}+\|\xi_0\|^{2}_{W^{3-1/q_+, q_+}}+\|\p_t\xi(0)\|^{2}_{H^{3/2+(\varepsilon_--\alpha)/2}}\\
				&\quad+\|(F^1-F^4-F^5)(0)\|^{2}_{(\mathcal{H}^1)^\ast}+\big(1+\|\p_t\eta\|_{L^\infty H^{3/2+\varepsilon_-/2}}^2\big)\big(\|F^1\|_{L^2L^{q_-}}^2+\|F^4\|_{L^2W^{1-1/q_-,q_-}}^2\\
				&\quad+\|F^5\|_{L^2W^{1-1/q_-,q_-}}^2\big)+\big(1+\|\p_t\eta\|_{L^\infty H^{3/2+\varepsilon_-/2}}^2\big)\|\p_t(F^1-F^4-F^5)\|_{L^{2}(\mathcal{H}^1_T)^{\ast}}^2
				+\sum_{j=0}^1\|[\p_{t}^{j}F^{7}]_\ell\|_{L^2}^2,
			\end{aligned}
		\end{equation}
		where $C_0(\eta):=\sup_{0\le t\le T}\|\p_tJK\|_{L^\infty}+\sup_{0\le t\le T}\|u\|_{W^{2,q_{+}}}+\sup_{0\le t\le T}\|\eta\|_{W^{3-\frac{1}{q_{+}},q_{+}}}$.
	\end{lemma}
	\begin{proof}
		The bound \eqref{bound} can be directly derived by choosing $v_{d}$ as the test function of equation \eqref{eq:weak_limit_1}. The detailed computation is given in \cite{GT2020}
	\end{proof}
	\begin{proposition}\label{prop:unique}
		The weak solution to \eqref{eq:weak_limit_1} is unique.
	\end{proposition}
	\begin{proof}
		If $(v_{d}^1,\xi_{d}^1)$ and $(v_{d}^2,\xi_{d}^2)$ are both weak solutions to \eqref{eq:weak_limit_1}, then $(v=v^1_d-v^2_d,\xi_d=\xi^1_d-\xi^2_d)$ is a weak solution to \eqref{eq:weak_limit_1} with $F^1=F^2=F^4=F^5=F^7=0$ and the initial data $v_d(0)=\xi_d(0)=u(0)= D_tu(0) = \p_t\xi(0)=0$. The bound \eqref{bound} implies that $v=0$, $\xi=0$. 
	\end{proof}

	%%%%%%%%%%%%%%%%%%%%%%%%%%%%%%%%%%%%%%%%%%%%%%
	\subsection{Pressure Estimates}
	%%%%%%%%%%%%%%%%%%%%%%%%%%%%%%%%%%%%%%%%%%%%%%

	Suppose that we have established the weak solution $(v_{d}, \xi_{d}) \in L^\infty H^0 \times L^\infty H^1$ such that $(\p_tv_{d}, \p_t\xi_{d}) \in L^\infty H^0 \times \left( L^\infty H^1 \cap L^2H^{3/2-\alpha} \right)$ in Theorem \ref{thm:linear_low}. We now introduce the pressure by treating it as a Lagrange multiplier and simultaneously get higher regularity of $(v_{d}, \xi_{d})$ to show that the weak solution coupled with the pressure is a strong solution.
	
	Before presenting the main theorem, we establish a lemma concerning the $\|\cdot\|_{1,\Sigma}$ norm defined by \eqref{eq:inner_product}. 
	
	\begin{lemma}{\label{lem:lemma1}}
		Suppose that $\|\eta^{k}\|\lesssim\vert \vert \eta\vert \vert_{L^{\infty}W^{3-\frac{1}{q_{+}},q_{+}}}\leq \delta\ll1$. We have 
		\begin{align}
			\|\p_{1}\mathcal{R}_{z}(\p_{1}\zeta_{0},\p_{1}\eta^{k})\|_{L_{t}^{\infty}W^{1-\frac{1}{q_{+}},q_{+}}}\lesssim\vert \vert \partial_{1}\mathcal{R}_{z}(\partial_{1}\zeta_{0},\partial_{1}\eta)\vert \vert_{L_{t}^{\infty}W^{1-\frac{1}{q_{+}},q_{+}}}\lesssim \delta.
		\end{align}
	\end{lemma}
	
	\begin{proof}
		
		Using \cite[Theorem A.1]{GT2020}, we have
		
		\[
		\begin{aligned}
			\vert \vert \partial_{1}\mathcal{R}_{z}(\partial_{1}\zeta_{0},\partial_{1}\eta)\vert \vert_{L_{t}^{\infty}W^{1-\frac{1}{q_{+}},q_{+}}}\lesssim& \vert\vert \partial_{1}^{2}\zeta_{0}\partial_{1}\eta\vert\vert_{L_{t}^{\infty}W^{1-\frac{1}{q_{+}},q_{+}}}+\vert \vert \partial_{1}^{2}\eta\vert \vert_{W^{1-\frac{1}{q_{+}},q_{+}}}\notag\\\lesssim& \vert \vert \zeta_{0}\vert \vert_{W^{3,q_{+}}}\vert \vert \eta\vert \vert_{W^{2-\frac{1}{q_{+}},q_{+}}}+\vert \vert \eta \vert \vert_{W^{3-\frac{1}{q_{+}},q_{+}}}\lesssim \delta
		\end{aligned}
		\]
		
		\noindent which finishes the proof.
	\end{proof}
	
	We then introduce the following lemma which is a simple corollary of Lemma \ref{lem:lemma1}.  
	
	\begin{lemma}{\label{cor:cor1}}
		Suppose that $\vert \vert \eta\vert \vert_{L^{\infty}W^{3-\frac{1}{q_{+}},q_{+}}}\leq \delta\ll1$ for a sufficiently small $\delta$ to be determined. Then the following inequality holds for any $\phi\in H^{1}(\Sigma)$.
		\[
		(\phi,\phi)_{1,\Sigma}\geq C(\delta)\vert \vert \phi\vert \vert_{H^{1}}.
		\]
	\end{lemma}
	
	\begin{proof}
		Using the definition of $(\cdot)_{1,\Sigma}$ norm, we have the following computation:
		\[
		(\phi,\phi)_{1,\Sigma}=\int_{-\ell}^{\ell}g\phi^2+\sigma\frac{(\p_1\phi)^2}{(1+|\p_1\zeta_0|^2)^{3/2}}+\sigma\int_{-\ell}^{\ell}\mathcal{R}_{z}(\partial_{1}\zeta_{0},\partial_{1}\eta)(\partial_{1}\phi)^2.\label{eq:cor11}
		\]
		\noindent From Lemma \ref{lem:lemma1}, we have:
		\[
		\vert \vert \mathcal{R}_{z}(\partial_{1}\zeta_{0},\partial_{1}\eta)\vert \vert_{L^{\infty}}\lesssim \vert \vert \mathcal{R}_{z}(\partial_{1}\zeta_{0},\partial_{1}\eta)\vert \vert_{W^{2-\frac{1}{q_{-}},q_{-}}}\lesssim \vert \vert \eta\vert \vert_{W^{3-\frac{1}{q_{-}},q_{-}}}\lesssim \delta.
		\]
		\noindent Therefore, we can choose a sufficiently small $\delta$ such that $
		\vert \vert \mathcal{R}_{z}(\partial_{1}\zeta_{0},\partial_{1}\eta)\vert \vert_{L^{\infty}}\lesssim \frac{1}{2}\inf_{x\in (-\ell,\ell)}\frac{1}{(1+\vert \partial_{1}\zeta_{0}\vert^{2})^{\frac{3}{2}}}$. Then we use this relation in the equation \eqref{eq:cor11} to show that:
		\[
		(\phi,\phi)_{1,\Sigma}\geq \int_{-\ell}^{\ell}g\phi\phi+\frac{1}{2}\sigma\frac{\p_1\phi\p_1\phi}{(1+|\p_1\zeta_0|^2)^{3/2}}dx-\frac{1}{2}\inf_{x\in (-\ell,\ell)}\frac{1}{(1+\vert \partial_{1}\zeta_{0}\vert^{2})^{\frac{3}{2}}}\|\phi\|_{H^{1}}^{2}\gtrsim \vert \vert \phi\vert \vert_{H^{1}}
		\]
		\noindent which finishes the proof.
	\end{proof}
	
	\begin{theorem}\label{thm:pressure}
		Suppose that we have the weak solution $( v_{d},  \xi_{d}) \in L^\infty H^0 \times \left( L^2 H^1 \cap L^2H^{3/2-\alpha} \right)$ to the system \eqref{eq:weak_limit_1} on the finite time interval $[0, T]$, $j=0, 1$ .  Then there exists a unique pressure $q_{d} \in L^\infty ([0, T]; H^0(\Om)) \cap L^2 ([0, T]; W^{1, q_-})$, such that
		\begin{equation}\label{est:avg_q1}
			\begin{aligned}
				\|q_{d}\|_{L^2 H^0}^2 \lesssim&\|\p_tv_{d}\|_{L_{t}^2H^0}^2  +\|\xi_{d}\|_{L_{t}^2 H^1}^2 + \|v_{d}\|_{L_{t}^2H^1}^2 + \|[v_{d}\cdot \mathcal{N}]_\ell\|_{L^2_t}^2\\
				&+\|[F^7]_\ell\|_{L^2_t}^2 + \|F^1\|_{L^2L^{q_-}}^2 + \|F^4\|_{L^2W^{1-1/q_-,q_-}}^2 + \|F^5\|_{L^2W^{1-1/q_-,q_-}}^2
			\end{aligned}
		\end{equation}
		and
		\begin{equation}\label{est:avg_q2}
			\begin{aligned}
				\|q_{d}\|_{L^\infty H^0}^2
				&\lesssim\|\p_tv_{d}\|_{L^\infty H^0}^2 +\|\xi_{d}\|_{L^\infty H^1}^2 + \|v_{d}\|_{L^\infty H^1}^2 +\|[v_{d}\cdot \mathcal{N}]_\ell\|_{L^\infty_t}^2+\\
				&\quad+\|[F^7]_\ell\|_{L^\infty_t}^2 + \|F^1-F^4-F^5\|_{L^\infty (\mathcal{H}^1)^\ast}^2.
			\end{aligned}
		\end{equation}
		\noindent Moreover, after deriving $q_{d}$, we may apply the ODE \eqref{eq:ODE} to recover $(v,q,\xi)$, together with the following elliptic estimate.
		\begin{equation}\label{est:diss_1}
			\begin{aligned}
				&\| v_{d}\|_{L^2W^{2,q_-}}^2 + \|q_{d}\|_{L^2W^{1,q_-}}^2 + \|\xi_{d}\|_{L^2W^{3-1/q_-,q_-}}^2 \\
				&\lesssim  \|\p_tv_{d}\|_{L^2H^0}^2 +\|\xi_{d} \|_{L^2H^1}^2 + \|v_{d}\|_{L^2H^1}^2 + \|\xi_{d}\|_{L^2H^{3/2-\alpha}}^2 + \|[v_{d}\cdot \mathcal{N}]_\ell\|_{L^2_t}^2\\
				&\quad + \|F^1\|_{L^2L^{q_-}}^2 + \|F^4\|_{L^2W^{1-1/q_-,q_-}}^2 + \|F^5\|_{L^2W^{1-1/q_-,q_-}}^2 + \|[F^7]_\ell\|_{L_t^{2}}^2.
			\end{aligned}
		\end{equation}
	\end{theorem}
	\begin{proof}
		\
		\paragraph{\underline{Step 1 -- Construction of Pressure $\overset{\circ}q_{d}$}}
		
		We state the weak formulation of the system \eqref{eq:weak_limit_1}:
		\begin{equation}{\label{eq:less_weak}}
			\begin{aligned}
				& (\p_tv_{d}, w)_{\mathcal{H}^0}+((v_{d},w))+ (\xi_{d}, w\cdot\mathcal{N})_{1,\Sigma}+ [v_{d}\cdot\mathcal{N},w\cdot\mathcal{N}]_\ell\\
				&=\int_\Om F^1\cdot wJ-\int_{-\ell}^{\ell} F^4\cdot w-\int_{\Sigma_s}F^5(w\cdot\tau)J -[F^7,w\cdot\mathcal{N}]_\ell+(\p_{t}(Rv),w)_{\mathcal{H}^{0}}+((Rv,w))
			\end{aligned}
		\end{equation}
		for any $w\in\mathcal{W}_{\sigma}$. Define the functional $\Lambda_t$ so that $\Lambda_t(w)$ is defined as the difference between the left-hand side and the right-hand side of \eqref{eq:less_weak} with $w\in\mathcal{W}_{\sigma}(t)$. By the energy estimate established in Lemma \ref{elemma}, we obtain that $\Lambda_{t}\in \mathcal{W}^{\ast}$ and $\Lambda_t(w)=0$ for all $w\in\mathcal{W}_\sigma(t)$.  By \cite[Theorem 4.6]{GT18}, there exists a unique $\overset{\circ}q_{d}(t)\in\mathring{H}^0(t)$ such that $(\overset{\circ}q_{d}(t),\dive_{\mathcal{A}}w)_{\mathcal{H}^0}=\Lambda_t(w)$ for all $w\in\mathcal{W}(t)$, which is equivalent to
		\begin{equation}\label{eq:weak_pressure_1}
			\begin{aligned}
				&(\p_tv_{d}, w)_{\mathcal{H}^0} + ((v_{d},w)) + (\xi_{d}, w\cdot\mathcal{N})_{1,\Sigma}-(\overset{\circ}q_{d},\dive_{\mathcal{A}}w)_{\mathcal{H}^0}
				+[v_{d}\cdot \mathcal{N},w\cdot\mathcal{N}]_\ell\\
				&=\int_\Om F^1\cdot wJ-\int_{-\ell}^{\ell} F^4\cdot w-\int_{\Sigma_s}F^5(w\cdot\tau)J -[F^7,w\cdot\mathcal{N}]_\ell+(\p_{t}(Rv),w)_{\mathcal{H}^{0}}+((Rv,w)).
			\end{aligned}
		\end{equation}
		Moreover, by Ne$\check{\text{c}}$as inequality (for instance, see \cite{boyer_fabrie}), we have
		\begin{equation}\label{est:pressure_3}
			\|\overset{\circ}q_{d}\|_0^2\lesssim \|v\|_{1}^{2}+ \|v_{d}\|_1^2+\|\p_tv_{d}\|^{2}_0 +\|\xi_{_{d}}\|_{1}^2 + \|F^1-F^{4}-F^{5} \|^{2}_{(\mathcal{H}^1)^\ast}+\sup_{0\leq s\leq t}\|\xi_{d}(s)\|_{H^{1}}^{2}+|[F^{7}]_{\ell}|^{2}.
		\end{equation}
		
		\paragraph{\underline{Step 2 -- Elliptic Estimates for $(v_d,\overset{\circ}q_{d},\xi_d)$}}
		For a.e. $t\in[0,T]$, the triple $(v_{d}(t),\overset{\circ}q_{d}(t),\xi_{d}(t))$ forms the unique weak solution to the elliptic problem.
		\begin{equation}{\label{eq:e-1}}
			\left\{
			\begin{aligned}
				&-\mu \Delta_{\mathcal{A}} v_{d} + \nabla_{\mathcal{A}} \overset{\circ}q_{d} = -\p_t v_{d}-\p_{t}(Rv)-\operatorname{div}_{\mathcal{A}}\nabla_{\mathcal{A}}(Rv) + F^1 \quad &\text{in}& \ \Om,\\
				& \dive_{\mathcal{A}} v_{d} =0 \quad &\text{in}& \ \Om,\\
				& S_{\mathcal{A}}(\overset{\circ}q_{d},v_{d})\mathcal{N}\cdot \mathcal{T}=-\nabla_{\mathcal{A}}(Rv)\mathcal{N}\cdot \mathcal{T}+F^{4}(u,p,\eta)\cdot \mathcal{T} \quad &\text{on}&~  \Sigma,\\
				& v_{d}\cdot \mathcal{N} = \p_t\xi_{d}-(Rv)\cdot \mathcal{N}-\p_{t}u_{1}\p_{1}\xi_{d} \quad &\text{on}& \ \Sigma,\\
				&(S_{\mathcal{A}}(\overset{\circ}q_{d},v_{d})\nu+\nabla_{\mathcal{A}}(Rv)-\beta v_{d})\cdot \tau=F^{5}(u,\eta,p) \quad &\text{on}& \ \Sigma_s,\\
				& v_{d} \cdot \nu =0 \quad &\text{on}& \ \Sigma_s,
			\end{aligned}
			\right.
		\end{equation}
		where $\mathcal{T} = (1, \p_1(\zeta_0 + \eta))$ is the tangential vector of free surface.
		
		Since $F^1(t)\in L_{t}^{2}L^{q_-}(\Om)$, $F^{4}(t)\in L_{t}^{2}W^{1-1/q_-, q_-}$ , $F^5(t)\in L_{t}^{2}W^{1-1/q_-, q_-}$, and according to the elliptic theory of \cite[Theorem 4.7]{GT2020}, the following elliptic problem
		\begin{equation}{\label{eq:e2}}
			\left\{
			\begin{aligned}
				&-\mu \Delta_{\mathcal{A}} v_{d} + \nabla_{\mathcal{A}} \overset{\circ}q_{d} = -\p_t v_{d}+ F^1 \quad &\text{in}& \ \Om,\\
				& \dive_{\mathcal{A}} v_{d} = 0 \quad &\text{in}& \ \Om,\\
				& -\mu \mathbb{D}_{\mathcal{A}} v_{d}\mathcal{N} \cdot \mathcal{T} = F^{4} \cdot \mathcal{T} \quad &\text{on}& \ \Sigma,\\
				& v_{d}\cdot \mathcal{N} = \p_t\xi_{d}+F^{4_{+}} \quad &\text{on}& \ \Sigma,\\
				&(-\mu \mathbb{D}_{\mathcal{A}} v_{d} \nu - \beta v_{d} ) \cdot \tau =F^{5} \quad &\text{on}& \ \Sigma_s,\\
				& v_{d} \cdot \nu =0 \quad &\text{on}& \ \Sigma_s,
			\end{aligned}
			\right.
		\end{equation}
		admits a unique strong solution $(v_{d}, \overset{\circ}q_{d})$ obeying
		\begin{equation}\label{est:bound_vq}
			\begin{aligned}
				&\|v_{d}(t)\|_{L_{t}^{2}W^{2,q_-}}^2+\|\nabla \overset{\circ}q_{d}(t)\|_{L_{t}^{2}L^{q_-}}^2 + \|v_{d}(t)\|_{L_{t}^{2}H^{1+\varepsilon_-/2}}^2 + \|\overset{\circ}q_{d}(t)\|_{L_{t}^{2}H^{\varepsilon_-/2}}^2\\
				&\lesssim \|\p_tv_{d}(t)\|_{L_{t}^{2}L^{q_-}}^2 +\|\p_t\xi_{d}(t)\|_{L_{t}^{2}W^{2-1/q_-,q_-}}^2+\|F^1\|_{L_{t}^{2}L^{q_a}}^2+\|F^{4}\|_{L_{t}^{2}W^{1-1/q_-,q_-}}^2\\
				& +\|F^5\|_{L_{t}^{2}W^{1-1/q_-,q_-}}^2+\|F^{4_{+}}\|_{L_{t}^{2}W^{2-\frac{1}{q_{-}},q_{-}}}.
			\end{aligned}
		\end{equation}
		\noindent In other words, the linear map $\Phi$ induced by the elliptic system 
		\[
		\Phi: X\ni(v_d,\overset{\circ}q_d)\rightarrow (F^{1},F^{4},F^{4_{+}},F^{5})\in Y,
		\]
		has a bounded inverse (The definitions of $X$ and $Y$ are given in Appendix \eqref{eq:def1} and \eqref{eq:def2}).
		
		Therefore, the system \eqref{eq:e-1} can be considered as a perturbation of \eqref{eq:e2}. Choosing $F^{i}$ to be the same in \eqref{eq:e-1} and \eqref{eq:e2} for $i=1,4,5$, and $F^{4_{+}}=0$, we denote the resulting perturbation by $\Phi_{\varepsilon}$, which has the following form:
		\[
		\Phi_{\varepsilon}(v_{d},\overset{\circ}q_{d})=\Phi(v_{d},\overset{\circ}q_{d})+(L_{1},L_{2},L_{3},L_{4}).
		\]
		\noindent where
		\[
		L_{1}=\p_{t}(Rv)-\operatorname{div}_{\mathcal{A}}\nabla_{\mathcal{A}}(Rv),\
		L_{2}=-\nabla_{\mathcal{A}}(Rv)\mathcal{N}\cdot \mathcal{T},\
		L_{3}=-(Rv)\cdot \mathcal{N}-\p_{t}u_{1}\p_{1}\xi_{d},\
		L_{4}=(\nabla_{\mathcal{A}}(Rv)\nu)\cdot \tau.
		\]
		\noindent Using H\"older's inequality and trace theorem, we have the following estimates
		\[
		\begin{aligned}
			& \|L_{1}\|_{L_{t}^{2}L^{q_{-}}}\lesssim T\|\p_{t}{\eta}\|_{L_{t}^{2}W^{3-\frac{1}{q_{-}},q_{-}}}\|v_{d}\|_{L_{t}^{2}W^{2,q_{-}}}+T\|\p_{t}^{2}{\eta}\|_{L_{t}^{2}H^{\frac{3}{2}-\alpha}}\|v_{d}\|_{L_{t}^{2}W^{2,q_{-}}}+\|\p_{t}\eta\|_{L_{t}^{\infty}H^{\frac{3}{2}-\alpha}}\|v_{d}\|_{L_{t}^{2}W^{2,q_{-}}},\\
			& \|L_{2}\|_{L_{t}^{2}W^{1-\frac{1}{q_{-}},q_{-}}(\Sigma)}\lesssim T\|\p_{t}\eta\|_{L_{t}^{\infty}W^{3-\frac{1}{q_{-}},q_{-}}}\|v_{d}\|_{L_{t}^{2}W^{2,q_{-}}},\\
			&\|L_{3}\|_{L_{t}^{2}W^{2-\frac{1}{q_{-}},q_{-}}(\Sigma)}\lesssim T\|\p_{t}\eta\|_{L_{t}^{2}W^{3-\frac{1}{q_{-}},q_{-}}}\|v_{d}\|_{L_{t}^{2}W^{2,q_{-}}}+\|\p_{t}u\|_{L_{t}^{\infty}H^{1+\frac{\varepsilon_{-}}{2}}}\|\xi_{d}\|_{L_{t}^{2}W^{3-\frac{1}{q_{-}},q_{-}}},\\
			&\|L_{4}\|_{L_{t}^{2}W^{1-\frac{1}{q_{-}},q_{-}}(\Sigma)}\lesssim T\|\p_{t}\eta\|_{L_{t}^{\infty}W^{3-\frac{1}{q_{-}},q_{-}}}\|v_{d}\|_{L_{t}^{2}W^{2,q_{-}}}.
		\end{aligned}
		\]
		\noindent
		Using the estimates above, there is an $\varepsilon_{0}$ such that $\Phi_{\varepsilon}$ also has a bounded inverse when $\varepsilon\leq \varepsilon_{0}$, given that $T\vert \vert \eta\vert \vert_{L_{t}^{\infty}W^{3-\frac{1}{q_{-}},q_{-}}}+T\|\p_{t}\eta\|_{L_{t}^{2}W^{3-\frac{1}{q_{-}},q_{-}}}+\|\p_{t}\eta\|_{L_{t}^{\infty}H^{\frac{3}{2}-\alpha}}\leq \delta \leq\varepsilon_{0}$. Therefore, the solution to \eqref{eq:e-1} also obeys the same elliptic estimate as \eqref{eq:e2}:
		\begin{equation}\label{est:bound_vq1}
			\begin{aligned}
				&\|v_{d}(t)\|_{L_{t}^{2}W^{2,q_-}}^2+\|\nabla \overset{\circ}q_{d}(t)\|_{L_{t}^{2}L^{q_-}}^2 + \|v_{d}(t)\|_{L_{t}^{2}H^{1+\varepsilon_-/2}}^2 + \|\overset{\circ}q_{d}(t)\|_{L_{t}^{2}H^{\varepsilon_-/2}}^2\\
				&\lesssim \|\p_tv_{d}(t)\|_{L_{t}^{2}L^{q_-}}^2 +\|\p_t\xi_{d}(t)\|_{L_{t}^{2}W^{2-1/q_-,q_-}}^2+\|F^1\|_{L_{t}^{2}L^{q_-}}^2+\|F^{4}\|_{L_{t}^{2}W^{1-1/q_-,q_-}}^2
				+\|F^5\|_{L_{t}^{2}W^{1-1/q_-,q_-}}^2.
			\end{aligned}
		\end{equation}
		{Then using Sobolev embedding $W^{2-1/q_a,q_a}(-\ell, \ell) \hookrightarrow H^{3/2-\alpha}(-\ell, \ell)$ and Poincar\'e inequality, the elliptic estimate \eqref{est:bound_vq1} can  be rewritten as
			\[
			\begin{aligned}
				&\|v_{d}(t)\|_{L_{t}^{2}W^{2,q_-}}^2+\|\nabla \overset{\circ}q_{d}(t)\|_{L_{t}^{2}L^{q_-}}^2 + \|v_{d}(t)\|_{L_{t}^{2}H^{1+\varepsilon_-/2}}^2 + \|\overset{\circ}q_{d}(t)\|_{L_{t}^{2}H^{\varepsilon_-/2}}^2\\
				&\lesssim \|\p_tv_{d}(t)\|_{L_{t}^{2}L^{q_-}}^2 +\|\p_t\xi_{d}(t)\|_{L_{t}^{2}H^{\frac{3}{2}-\alpha}}^2+\|F^1\|_{L_{t}^{2}L^{q_-}}^2+\|F^{4}\|_{L_{t}^{2}W^{1-1/q_-,q_-}}^2 +\|F^5\|_{L_{t}^{2}W^{1-1/q_-,q_-}}^2.
			\end{aligned}
			\]
		}
		
		\paragraph{\underline{Step 3 -- Elliptic Estimates for $\xi_{d}$}}
		The weak formulation of the stress-gravity-capillary equation
		\begin{align}\label{eq:sgc}
			\left\{
			\begin{aligned}
				&g(\xi_{d}) - \sigma \p_1 \left( \frac {\p_1(\xi_{d} )} {(1+|\p_1\zeta_0|^2)^{3/2}}\right) -\sigma \partial_{1}(\mathcal{R}_{z}(\partial_{1}\zeta_{0},\partial_{1}\eta)\partial_{1}\xi_{d})\\
				&\qquad\qquad\qquad\qquad\qquad\qquad= \left( S_{\mathcal{A}} (\overset{\circ}q_{d}, v_{d})\mathcal{N}\right)\cdot \frac{\mathcal{N}} {|\mathcal{N}|^2}+F^{4}\cdot \mathcal{N}
				+\nabla_{\mathcal{A}}(Rv_{d})\mathcal{N}\cdot \frac{\mathcal{N}}{|\mathcal{N}|^{2}} \quad \text{in} \ (-\ell, \ell),\\
				&\mp\sigma\frac{\p_1(\xi_{d})}{(1+|\p_1\zeta_0|^2)^{3/2}}(\pm\ell)\pm\sigma \mathcal{R}_{z}(\partial_{1}\zeta_{0},\partial_{1}\eta)\partial_{1}\xi_{d}({\pm\ell})=\kappa v_{d}\cdot \mathcal{N}(\pm\ell)-F^7(\pm\ell),
			\end{aligned}
			\right.
		\end{align}
		is given by
		\begin{equation}\label{eq:w_capillary}
			\begin{aligned}
				&\int_{-\ell}^\ell {g} (\xi_{d}) \vartheta + \sigma \frac{\p_1(\xi_{d}) \p_1\vartheta}{(1+|\p_1\zeta_0|^2)^{3/2}} +\sigma \int_{-\ell}^{\ell}(\mathcal{R}_{z}(\partial_{1}\zeta_{0},\partial_{1}\eta)\partial_{1}\xi_{d} ds)\partial_{1}\vartheta\\
				&=   \int_{-\ell}^\ell S_{\mathcal{A}}(\overset{\circ}q_{d},v_{d})\mathcal{N} \cdot \frac{\mathcal{N}}{|\mathcal{N}|^2} \vartheta+\int_{-\ell}^{\ell}\nabla_{\mathcal{A}}(Rv)\mathcal{N}\cdot \frac{\mathcal{N}}{|\mathcal{N}|^{2}}
				-[ v_{d}\cdot \mathcal{N}, \vartheta]_\ell
				+[F^7, \vartheta]_\ell
			\end{aligned}
		\end{equation}
		for each $\vartheta \in H^1((-\ell, \ell))$. Just like Step 2, system \eqref{eq:sgc} can be viewed as a linear perturbation of the elliptic system:
		\begin{align}\label{eq:sgc1}
			\left\{
			\begin{aligned}
				& g(\xi_{d}) - \sigma \p_1 \left( \frac {\p_1(\xi_{d} )} {(1+|\p_1\zeta_0|^2)^{3/2}}\right)= \left( S_{\mathcal{A}} (\overset{\circ}q_{d}, v_{d})\mathcal{N}\right)\cdot \frac{\mathcal{N}} {|\mathcal{N}|^2}+F^{4}\cdot \mathcal{N}
				+\nabla_{\mathcal{A}}(Rv_{d})\mathcal{N}\cdot \frac{\mathcal{N}}{|\mathcal{N}|^{2}} \quad \text{in} \ (-\ell, \ell),\\
				& \mp\sigma\frac{\p_1(\xi_{d})}{(1+|\p_1\zeta_0|^2)^{3/2}}(\pm\ell)=\kappa v_{d}\cdot \mathcal{N}(\pm\ell)-F^7(\pm\ell),
			\end{aligned}
			\right.
		\end{align}
		By the elliptic estimates of \eqref{eq:sgc1},
		we have
		\begin{align}\label{est:bound_theta1}
			\begin{aligned}
				\|\xi_{d}(t)\|_{L^{2}W^{3-1/q_-,q_-}}^2 & \lesssim \|\overset{\circ}q_{d}\|_{L^{2}W^{1-1/q_-, q_-}(-\ell, \ell)}^2 + \|v_{d}\|_{L^{2}W^{2-1/q_-, q_-}(-\ell, \ell)}^2   + \vert \vert[F^7]_\ell\vert \vert_{L_{t}^{2}}^{2}.
			\end{aligned}
		\end{align}
		\noindent For the perturbation terms, by H\"older's inequality, we have
		\[
		\|\p_{1}(\mathcal{R}_{z}(\p_{1}\zeta_{0},\p_{1}\eta))\p_{1}\xi_{d}\|_{L_{t}^{2}W^{3-\frac{1}{q_{-}},q_{-}}}\lesssim \|\eta\|_{L_{t}^{\infty}W^{3-\frac{1}{q_{-}},q_{-}}}\|\xi_{d}\|_{L_{t}^{2}W^{3-\frac{1}{q_{-}},q_{-}}},
		\]
		\noindent and similarly
		\[
		[\mathcal{R}_{z}(\p_{1}\zeta_{0},\p_{1}\eta)\p_{1}\xi_{d}]_{\ell}\lesssim \|\eta\|_{L_{t}^{\infty}W^{3-\frac{1}{q_{-}},q_{-}}}\|\xi_{d}\|_{L_{t}^{2}W^{3-\frac{1}{q_{-}},q_{-}}}.
		\]
		\noindent Therefore, using the similar argument as in Step 2 by regarding \eqref{eq:w_capillary} as a perturbation of \eqref{eq:sgc1}, and the smallness of $\vert \vert \eta\vert \vert_{L^{2}W^{3-\frac{1}{q_{-}},q_{-}}}$, we deduce that $\xi$ is a strong solution to \eqref{eq:sgc} which obeys \eqref{est:bound_theta1}.
		
		\paragraph{\underline{Step 4 --The Estimate of $\bar{q}_{d}$}}
		
		{From the previous two steps, in order to complete the elliptic estimate, it remains to derive the boundedness for $\bar{q}_{d}$ which is the mean value of $q_{d}$. We aim to derive the estimate for $\bar{q}_{d}$ from the capillary equation
			\[
			S_{\mathcal{A}}(q_{d},v_{d})\mathcal{N}=\mathcal{K}(\xi)\mathcal{N}+(\p_{1}(\mathcal{R}_{z}(\p_{1}\zeta_{0},\p_{1}\eta)\p_{1}\xi_{d}))\mathcal{N}+F^{4}+\nabla_{\mathcal{A}}(Rv)\mathcal{N}.
			\]
			Testing this equation with the function $\psi=M\nabla\tilde{\phi}$ with $\tilde{\phi}$ solving the following system
			\begin{align}{\label{equ:ellip_d}}
				\Delta\tilde{\phi}=g\frac{2\ell}{|\Omega|}\bar{q}_{d} ~\operatorname{in}~\Omega,\quad\tilde{\phi}\cdot\nu=\frac{g\bar{q}_{d}}{|\mathcal{N}_{0}|}~\operatorname{on}~\Sigma,\quad\tilde{\phi}\cdot \nu=0~\operatorname{on}~\Sigma_{s},
			\end{align}
			and using integration by parts, we obtain
			\begin{align}{\label{equ:ODE_1}}
				2g\ell|\bar{q}_{d}|^{2}=(\xi_{d},\psi\cdot \mathcal{N})_{1,\Sigma}-(S_{\mathcal{A}}(q_{d}-\bar{q}_{d},v_{d}+Rv)\mathcal{N},\psi)_{L^{2}}+\kappa[v\cdot \mathcal{N},\psi\cdot \mathcal{N}]_{\ell}-[F^{7},\psi\cdot \mathcal{N}]_{\ell}.
			\end{align}
			For the term $(S_{\mathcal{A}}(q_{d}-\bar{q}_{d},v_{d}+Rv)\mathcal{N},\psi)_{L^{2}}$, we obtain the following from equation $\p_{t}v_{d}+\operatorname{div}_{\mathcal{A}}(S_{\mathcal{A}}  (q_{d},v_{d}+Rv))=F^{1}$ and the fact that $v_{d}+Rv=\p_{t}v$
			\begin{align}{\label{equ:ODE_3}}
				\begin{aligned}
					(S_{\mathcal{A}}(q_{d}-\bar{q}_{d},\p_{t}v)\mathcal{N},\psi)_{L^{2}}=-\int_{\Omega}(\p_{t}^{2}v)\psi J+\int_{\Omega} JF^{1}\psi+\int_{\Omega} S_{\mathcal{A}}(q_{d}-\bar{q}_{d},\p_{t}v):D_{\mathcal{A}^{n}}\psi J\\
					-\int_{\Sigma_{s}}(S_{\mathcal{A}^{n}}(q_{d}-\bar{q}_{d},\p_{t}v)\nu\cdot \tau)\psi\cdot \tau J.
				\end{aligned}
			\end{align}
			The elliptic estimate for \eqref{equ:ellip_d} yields that
			\[
			\|\psi\|_{H^{1}(\Omega)}\lesssim |\bar{q}_{d}|.
			\]
			Therefore, applying this result and \eqref{equ:ODE_3} to equation \eqref{equ:ODE_1}, and using trace theorem, we obtain
			\[\begin{aligned}
				2g\ell|\bar{q}_{d}|^{2}\lesssim& g^{2}|\bar{q}_{d}|\int_{-\ell}^{\ell}|\xi_{d}|dx+|\bar{q}_{d}|(\|v_{d}\|_{H^{1}}+\|\p_{t}v_{d}\|_{L^{2}}+\|q_{d}-\bar{q}_{d}\|_{L^{2}}+\|F^{1}\|_{L^{q_-}})\notag\\
				&+|[v_{d}\cdot \mathcal{N}]_{\ell}||\bar{q}_{d}|+|[F^{7}]_{\ell}||\bar{q}_{d}|,
			\end{aligned}\]
			\noindent which implies that
			\begin{align}{\label{equ:ODE_2}}
				2g\ell|\bar{q}_{d}|\lesssim g^{2}\int_{-\ell}^{\ell}|\xi_{d}|dx+(\|v_{d}\|_{H^{1}}+\|\p_{t}v_{d}\|_{L^{2}}+\|q_{d}-\bar{q}_{d}\|_{W^{1,q_{-}}}+\|F^{1}\|_{L^{q_-}})
				&+[v_{d}\cdot \mathcal{N}]_{\ell}+[F^{7}]_{\ell}.
			\end{align}
			Therefore, we have
			\[
			|\bar{q}_{d}|_{L_{t}^{2}}\lesssim \|\xi_{d}\|_{L^{2}H^{1}}+(\|v_{d}\|_{L_{t}^{2}H^{1}}+\|\p_{t}v_{d}\|_{L_{t}^{2}L^{2}}+\|q_{d}-\bar{q}_{d}\|_{L_{t}^{2}L^{2}}+\|F^{1}\|_{L_{t}^{2}L^{q_-}})
			+\|[v_{d}\cdot \mathcal{N}]_{\ell}\|_{L_{t}^{2}}+\||[F^{7}]_{\ell}|\|_{L_{t}^{2}},
			\]
			\noindent and
			\[
			|\bar{q}_{d}|_{L_{t}^{\infty}}\lesssim \|\xi_{d}\|_{L^{\infty}H^{1}}+(\|v_{d}\|_{L_{t}^{\infty}H^{1}}+\|\p_{t}v_{d}\|_{L_{t}^{\infty}L^{2}}+\|q_{d}-\bar{q}_{d}\|_{L_{t}^{\infty}L^{2}}+\|F^{1}\|_{L_{t}^{\infty}(\mathcal{H}^1)^{*}})+\|[v_{d}\cdot \mathcal{N}]_{\ell}\|_{L_{t}^{\infty}}+\||[F^{7}]_{\ell}|\|_{L_{t}^{\infty}}.
			\]
			
			Combining Step 2 -- Step 4, we obtain the elliptic estimate \eqref{est:diss_1}, and the estimates \eqref{est:avg_q1} and \eqref{est:avg_q2} for $q_{d}$.  
		}
		\paragraph{\underline{Step 5 -- Strong Solution to \eqref{eq:quasi_linear_{s}}}}
		We now use the weak formulation with pressure to derive the strong formulation. Since we already have the estimate for $v_{d}$, we then solve the ODE \eqref{eq:ODE} to recover $(v,q,\xi)$ and use $(D_{t}v,\p_{t}\xi,\p_{t}p)$ to substitute $(v_{d},\xi_{d},q_{d})$. We first integrate $((\p_{t}v,\psi))-(\p_{t}q,\dive_{\mathcal{A}}\psi)_{\mathcal{H}^0}$ by parts to deduce that
		\begin{align}\label{eq:w_ibp1}
			\begin{aligned}
				&(\p_t^{2}v,\psi)_{\mathcal{H}^0}-\mu(\Delta_{\mathcal{A}}\p_{t}v,\psi)_{\mathcal{H}^0}+(\nabla_{\mathcal{A}}\p_{t}q, \psi)_{\mathcal{H}^0}+(\p_{t}\xi,\psi\cdot\mathcal{N})_{1,\Sigma_{0}}+(\mathcal{R}_{z}(\partial_{1}\zeta_{0},\partial_{1}\eta)\partial_{1}\partial_{t}\xi,\partial_{1}(\psi\cdot\mathcal{N}))_{L^{2}}\\&+ [\p_{t}v\cdot \mathcal{N}^{n},\psi\cdot\mathcal{N}]_{\ell}
				=\int_{\Om}F^1\cdot \psi J^{n}+\int_{-\ell}^\ell S_{\mathcal{A}}(\p_{t}q, \p_{t}v)\mathcal{N} \cdot \psi\\
				&\quad+\int_{\Sigma_s}\left(S_{\mathcal{A}}(\p_{t}q, \p_{t}v)\nu\cdot\tau\right) (\psi\cdot\tau)J-\beta\int_{\Sigma_s}(\p_{t}v\cdot\tau)(\psi\cdot\tau)J^{n}-[\psi\cdot\mathcal{N},F^7]_\ell.
			\end{aligned}
		\end{align}
		If we restrict $\psi \in C_c^1(\Om) \subset \mathcal{W}(t)$, all boundary terms in \eqref{eq:w_ibp1} vanish, so that $\p_t^{2}v + \nabla_{\mathcal{A}}\p_{t}q - \Delta_{\mathcal{A}}\p_{t}v= F^1$ in $\Om$. By plugging this into \eqref{eq:w_ibp1} and integrating by parts over $(-\ell, \ell)$, we have
		\begin{align}\label{eq:w_ibp2}
			\begin{aligned}
				&\bigg(g(\p_{t}\xi )-\sigma\p_1\Big(\frac{\p_1(\p_{t}\xi)}{(1+(\p_1\zeta_0)^2)^{3/2}}\Big)-\sigma\partial_{1}(\mathcal{R}_{z}(\partial_{1}\zeta_{0},\partial_{1}\eta)\partial_{1}\partial_{t}\xi),\psi\cdot\mathcal{N} \bigg)_{L^{2}} + \sigma\frac{\p_1(\p_{t}\xi)}{(1+(\p_1\zeta_0)^2)^{3/2}}(\psi\cdot\mathcal{N})\Big|_{-\ell}^\ell\\& +[\p_{t}v\cdot \mathcal{N},\psi\cdot\mathcal{N}]_{\ell}-[\mathcal{R}_{z}(\partial_{1}\zeta_{0},\partial_{1}\eta)\partial_{1}\partial_{t}\xi,\psi\cdot\mathcal{N}]_{\ell}
				=\int_{-\ell}^\ell S_{\mathcal{A}}(\p_{t}q, \p_{t}v)\mathcal{N} \cdot \psi- F^4\cdot \psi \\
				&\quad+\int_{\Sigma_s}\left(S_{\mathcal{A}}(\p_{t}q, \p_{t}v)\nu\cdot\tau\right) (\psi\cdot\tau)J-\int_{\Sigma_s}F^5(\psi\cdot\tau)J-\beta\int_{\Sigma_s}(\p_{t}v\cdot\tau)(\psi\cdot\tau)J-[\psi\cdot\mathcal{N},F^7]_\ell.
			\end{aligned}
		\end{align}
		Restricting $\psi \in \mathcal{W}_{\sigma}(t) \cap \{\psi| \psi \cdot\mathcal{N} \in C^1_c(-\ell, \ell), \psi =0 \ \text{on}\ \Sigma_s\}$, we deduce stress tensor boundary condition $S_{\mathcal{A}}(\p_{t}q,\p_{t}v)\mathcal{N}=\left(g(\p_{t}\xi)-\sigma\p_1\Big(\frac{\p_1(\p_{t}\xi)}{(1+(\p_1\zeta_0)^2)^{3/2}}\Big)-\sigma\partial_{1}(\mathcal{R}_{z}(\partial_{1}\zeta_{0},\partial_{1}\eta)\partial_{1}\partial_{t}\xi)\right)\mathcal{N}$. Here $\psi$ obeys the zero average constraint $\int_{-\ell}^\ell (\psi\cdot\mathcal{N})=0$, which results in a constant coming from the term $(\p_{t}\xi,\psi\cdot\mathcal{N})_{1,\Sigma_{0}}$ after integration by parts on $(-\ell, \ell)$. Similarly, we can restrict $\psi \in \mathcal{W}(t) \cap \{\psi | \psi|_{\Sigma_s} \in C^1_c(\Sigma_s)\}$ to deduce the boundary condition $S_{\mathcal{A}}(\p_{t}q, \p_{t}v)\nu\cdot\tau -\beta(\p_{t}v\cdot\tau) =F^{5}$ on $\Sigma_s$.
		
		Finally, the remaining boundary conditions at contact points in \eqref{eq:w_ibp2} are
		\begin{align}\label{eq:w_ibp3}
			\sigma\frac{\p_1(\p_{t}\xi )}{(1+(\p_1\zeta_0)^2)^{3/2}}(\psi\cdot\mathcal{N})\Big|_{-\ell}^\ell +[\p_{t}v\cdot\mathcal{N},\psi\cdot\mathcal{N}]_{\ell}-[\mathcal{R}_{z}(\partial_{1}\zeta_{0},\partial_{1}\eta)\partial_{1}\partial_{t}\xi,\psi\cdot\mathcal{N}]_{\ell} = -[\psi\cdot\mathcal{N},F^7]_\ell.
		\end{align}
		If we restrict $\psi \in \mathcal{W}_{\sigma}(t)$ and $\psi \cdot\mathcal{N} $ supported on $-\ell$, \eqref{eq:w_ibp3} gives the contact point condition on $-\ell$. Moreover, \eqref{eq:w_ibp3} with restriction of $\psi \in \mathcal{W}(t)$ and $\psi \cdot\mathcal{N} $ supported on $\ell$ gives the contact point condition on $\ell$.
		
		Using the definition of $v$, we derive the strong form of system \eqref{eq:quasi_linear_{s}} for $(v_{d},q_{d},\xi_{d})$. This finishes the proof. 
	\end{proof}
	
	From Theorem \ref{thm:pressure}, we have the following corollary.
	\begin{theorem}{\label{thm:pressure_s}}
		Suppose that we have the weak solution $( v_{d},  \xi_{d}) \in L^\infty H^0 \times \left( L^2 H^1 \cap L^2H^{3/2-\alpha} \right)$ to the smooth given data system \eqref{eq:weak_limit_0} on the finite time interval $[0, T]$, $j=0, 1$. Then there exists a unique pressure $q_{d} \in L^\infty ([0, T]; H^0(\Om)) \cap L^2 ([0, T]; W^{1, q_-})$ satisfying the estimate \eqref{est:avg_q1} and \eqref{est:avg_q2}. Moreover $(v_{d},q_{d},\xi_{d})$ is the strong solution to \eqref{eq:quasi_linear_{s}}, and obeys the following estimate 
		\begin{equation}\label{est:diss_100}
			\begin{aligned}
				&\| v_{d}\|_{L^2W^{2,q_-}}^2 + \|q_{d}\|_{L^2W^{1,q_-}}^2 + \|\xi_{d}\|_{L^2W^{3-1/q_-,q_-}}^2 \\
				&\lesssim  \|\p_tv_{d}\|_{L^2H^0}^2 +\|\xi_{d} \|_{L^2H^1}^2 + \|v_{d}\|_{L^2H^1}^2 + \|\xi_{d}\|_{L^2H^{3/2-\alpha}}^2 + \|[v_{d}\cdot \mathcal{N}]_\ell\|_{L^2_t}^2+ \|F^1\|_{L^2L^{q_-}}^2 + \|F^4\|_{L^2W^{1-1/q_-,q_-}}^2 \\
				&\quad + \|F^5\|_{L^2W^{1-1/q_-,q_-}}^2 + \|[F^7]_\ell\|_{L_t^{2}}^2+\|\p_{t}\eta^{n}\|^{2}_{L_{t}^{\infty}W^{3-\frac{1}{q_{-}},q_{-}}}\|v_{l}^{k}\|_{L_{t}^{2}W^{2,q_{-}}}^{2}+\|\p_{t}u^{k}\|_{L_{t}^{2}W^{2,q_{-}}}^{2}\|\p_{t}\xi_{l}^{k}\|_{L_{t}^{2}W^{3-\frac{1}{q_{-}},q_{-}}}^{2}\\
				&\quad+\|\p_{t}\eta^{n}\|_{L_{t}^{2}W^{3-\frac{1}{q_{-}},q_{-}}}^{2}\|\p_{t}\xi_{l}^{k}\|_{L_{t}^{2}W^{3-\frac{1}{q_{-}},q_{-}}}^{2}.
			\end{aligned}
		\end{equation}
	\end{theorem}
	\begin{proof}
		The proof is similar to that of Theorem \ref{thm:pressure}. We omit the details here.
	\end{proof}

	%%%%%%%%%%%%%%%%%%%%%%%%%%%%%%%%%%%%%%%%%%%%%%
	\subsection{Elliptic Estimates}
	%%%%%%%%%%%%%%%%%%%%%%%%%%%%%%%%%%%%%%%%%%%%%%

	In order to pass to the limit $n \to +\infty$ and recover the original system \eqref{eq:quasi_linear}, we need to establish $q_{+}$-elliptic estimates for the zero-order terms $(v,q,\xi)$. This is achieved in the following theorem.
	
	\begin{theorem}{\label{thm:pressure_+}}
		Suppose that $(\p_{t}v,\p_{t}q,\p_{t}\xi)\in L_{t}^{2}W^{2,q_{-}}(\Omega)\times L_{t}^{2}W^{1,q_{-}}(\Omega)\times L_{t}^{2}W^{3-\frac{1}{q_{-}},q_{-}}$, and that $(v,q,\xi)$ is the strong solution to the following equation
		\begin{equation}\label{eq:quasi_linear_2}
			\begin{cases}
				\partial_{t}^{2}v+\operatorname{div}_{\mathcal{A}^{n}}S_{\mathcal{A}^{n}}(\p_{t}v,\p_{t}q)=F^{1}(u,p,\eta)~~~&\operatorname{in}~~\Omega(t),\\
				\operatorname{div}_{\mathcal{A}^{n}}D_{t}v=0~~~&\operatorname{in}~~\Omega(t),\\
				S_{\mathcal{A}^{n}}(\p_{t}q,\p_{t}v)\mathcal{N}^{n}=g\p_{t}\xi\mathcal{N}^{n}-\sigma\partial_{1}\left(\frac{\p_1\p_t\xi}{(1+\vert \partial_{1}\zeta_{0}\vert^{2})^{3/2}}\right)\mathcal{N}^{n}\\
				\quad\quad\quad\quad\quad\quad\quad\quad\quad+\sigma\partial_{1}(\mathcal{R}_{z}(\partial_{1}\zeta_{0},\partial_{1}\eta)\partial_{1}\partial_{t}\xi)\mathcal{N}^{n}+F^{4}(u,p,\eta)~~~&\operatorname{on}~~\Sigma(t),\\
				(S_{\mathcal{A}^{n}}(\p_{t}q,\p_{t}v)\nu-\beta \p_{t}v)\cdot \tau=F^{5}(u,\eta,p)~~~&\operatorname{on}~~\Sigma_{s}(t),\\
				\p_{t}v\cdot \nu=0~~~&\operatorname{on}~~\Sigma_{s}(t),\\
				\partial_{t}^{2}\xi=\p_{t}v\cdot \mathcal{N}^{n}+u_{1}\partial_{1}\partial_{t}\xi~~~&\operatorname{on}~~\Sigma(t),\\
				\sigma(\mp \frac{\partial_{1}\p_{t}\xi}{(1+\vert \partial_{1}\zeta_{0}\vert^{2})^{\frac{3}{2}}}\pm \mathcal{R}_{z}(\partial_{1}\zeta_{0},\partial_{1}\eta)\partial_{t}\partial_{1}\xi)(\pm \ell)=\kappa (\p_{t}v\cdot \mathcal{N}^{n})(\pm \ell)-{F}^{7}.
			\end{cases}
		\end{equation}
		Given that $\int_{0}^{t}F^{1}(u,p,\eta)\in L_{t}^{\infty}L^{q_{+}}(\Omega)$, $\int_{0}^{t}F^{4}(u,p,\eta)\in L_{t}^{\infty}W^{1-\frac{1}{q_{+}},q_{+}}(\Sigma)$, $\int_{0}^{t}F^{5}(u,p,\eta)\in L_{t}^{\infty}W^{1-\frac{1}{q_{+}},q_{+}}(\Sigma)$, $\int_{0}^{t}[F^{7}(u,\eta,p)]_{\ell}\in L_{t}^{\infty}$, 
		and $F^{1,0}(u_{0},p_{0},\eta_{0})\in L^{q_{+}}(\Omega)$,
		where the definitions of \(F^{i}\) are given in Appendix \ref{sec:dive_forcing}, and the definition of \(F^{1,0}\) is given by \eqref{def:forcing_0} in Appendix \ref{sec:initial}, it follows that \((\partial_t v,\partial_t q,\partial_t \xi)\) satisfies the following estimate.
		\begin{equation}\label{est:diss_10}
			\begin{aligned}
				&\| v\|_{L^\infty W^{2,q_+}}^2 + \|q\|_{L^\infty W^{1,q_+}}^2 + \|\xi\|_{L^\infty W^{3-1/q_+,q_+}}^2 \\
				&\lesssim  \|\p_tv\|_{L^\infty H^0}^2 +\|\p_{t}\xi \|_{L^\infty H^1}^2 + \|v\|_{L^\infty H^1}^2 + \|\xi\|_{L^\infty H^{3/2-\alpha}}^2 + \|[v\cdot \mathcal{N}]_\ell\|_{L^\infty_t}^2+ \left\|\int_{0}^{t}F^1 \right\|_{L^\infty 
					L^{q_+}}^2\\
				&\quad  + \left\|\int_{0}^{t}F^4 \right\|_{L^\infty W^{1-1/q_+,q_+}}^2 + \left\|\int_{0}^{t}F^5 \right\|_{L^\infty W^{1-1/q_+,q_+}}^2 + \left\|\int_{0}^{t}[F^7]_\ell \right\|_{L_t^{\infty}}^2 +\|F^{1,0}\|_{L^{q_{+}}}^{2}+\|\eta_{0}\|_{W^{3-\frac{1}{q_{+}},q_{+}}}^{4}.
			\end{aligned}
		\end{equation}
	\end{theorem}
	
	\begin{proof}
		Integrating the equation \eqref{eq:quasi_linear_2} from $0$ to $t$, and using the compatibility condition for the initial condition, we obtain the following equation
		\begin{equation}\label{eq:quasi_linear_3}
			\begin{cases}
				\partial_{t}v+\operatorname{div}_{\mathcal{A}^{n}}S_{\mathcal{A}^{n}}(v,q)=\int_{0}^{t}F^{1}(u,p,\eta)+\int_0^{t}\operatorname{div}_{\p_{t}\mathcal{A}^{n}}S_{\mathcal{A}^{n}}(v,q)+\int_0^{t}\operatorname{div}_{\mathcal{A}^{n}}S_{\p_{t}\mathcal{A}^{n}}(v,q)\\
				\qquad\qquad\qquad\qquad\qquad+F^{1,0}(u(0),p(0),\eta(0))~~~&\operatorname{in}~~\Omega(t),\\
				\operatorname{div}_{\mathcal{A}^{n}}v=0~~~&\operatorname{in}~~\Omega(t),\\
				S_{\mathcal{A}^{n}}(q,v)\mathcal{N}^{n}=g\xi\mathcal{N}^{n}-\sigma\partial_{1}\left(\frac{\p_1\xi}{ (1+\vert \partial_{1}\zeta_{0}\vert^{2})^{3/2}} \right)\mathcal{N}^{n}+\int_{0}^{t}\mathcal{K}(\xi)\p_{t}\mathcal{N}^{n}+\int_{0}^{t}S_{\mathcal{A}^{n}}(v,q)\p_{t}\mathcal{N}^{n}\\
				\quad\quad\quad\quad\quad\quad\quad +\p_{1}(\mathcal{R}(\p_{1}\zeta_{0},\p_{1}\eta_{0}))+\sigma\int_{0}^{t}\partial_{1}(\mathcal{R}_{z}(\partial_{1}\zeta_{0},\partial_{1}\eta)\partial_{1}\partial_{t}\xi)\mathcal{N}^{n}+\int_{0}^{t}F^{4}(u,p,\eta)\\
				\quad\quad\quad\quad\quad\quad\quad +\int_{0}^{t}S_{\p_{t}\mathcal{A}^{n}}(q,v)\mathcal{N}^{n}~~~&\operatorname{on}~~\Sigma(t),\\
				(S_{\mathcal{A}^{n}}(q,v)\nu-\beta v)\cdot \tau=\int_{0}^{t}F^{5}(u,\eta,p)+\int_{0}^{t}(S_{\p_{t}\mathcal{A}^{n}}(q,v)\nu)\cdot \tau~~~&\operatorname{on}~~\Sigma_{s}(t),\\
				v\cdot \nu=0~~~&\operatorname{on}~~\Sigma_{s}(t),\\
				\partial_{t}\xi=v\cdot \mathcal{N}^{n}-\int_{0}^{t}v\cdot \p_{t}\mathcal{N}^{n}+\int_{0}^{t}(u_{1}\p_{1}\p_{t}\xi)~~~&\operatorname{on}~~\Sigma(t),\\
				\sigma(\mp \frac{\partial_{1}\xi}{(1+\vert \partial_{1}\zeta_{0}\vert^{2})^{\frac{3}{2}}}\mp\int_{0}^{t}\mathcal{R}_{z}(\partial_{1}\zeta_{0},\partial_{1}\eta)\partial_{t}\partial_{1}\xi(\pm \ell))
				=\kappa (v\cdot \mathcal{N}^{n})(\pm \ell)-\int_{0}^{t}{F}^{7}.
			\end{cases}
		\end{equation}
		
		Following an argument similar to that in the proof of Theorem \ref{thm:pressure}, we regard system \eqref{eq:quasi_linear_3} as a perturbation of the following system
		\[
		\begin{cases}
			\partial_{t}v+\operatorname{div}_{\mathcal{A}^{n}}S_{\mathcal{A}^{n}}(v,q)=\int_{0}^{t}F^{1}(u,p,\eta)+F^{1,0}(u_{0},\eta_{0},p_{0})~~~&\operatorname{in}~~\Omega(t),\\
			\operatorname{div}_{\mathcal{A}^{n}}v=0~~~&\operatorname{in}~~\Omega(t),\\
			S_{\mathcal{A}^{n}}(q,v)\mathcal{N}^{n}=g\xi\mathcal{N}^{n}-\sigma\partial_{1}\left(\frac{\p_1\xi}{(1+\vert \partial_{1}\zeta_{0}\vert^{2})^{3/2}} \right)\mathcal{N}^{n}
			+\int_{0}^{t}F^{4}(u,p,\eta)+\p_{1}(\mathcal{R}(\p_{1}\zeta_{0},\p_{1}\eta_{0}))~~~&\operatorname{on}~~\Sigma(t),\\
			(S_{\mathcal{A}^{n}}(q,v)\nu-\beta v)\cdot \tau=\int_{0}^{t}F^{5}(u,\eta,p)~~~&\operatorname{on}~~\Sigma_{s}(t),\\
			v\cdot \nu=0~~~&\operatorname{on}~~\Sigma_{s}(t),\\
			\partial_{t}\xi=v\cdot \mathcal{N}^{n}~~~&\operatorname{on}~~\Sigma(t),\\
			\sigma(\mp \frac{\partial_{1}\xi}{(1+\vert \partial_{1}\zeta_{0}\vert^{2})^{\frac{3}{2}}})(\pm \ell)
			=\kappa (v\cdot \mathcal{N}^{n})(\pm \ell)-\int_{0}^{t}{F}^{7}.
		\end{cases}
		\]
		Therefore, it remains to estimate the additional terms and show that they act as small perturbations. We begin with the Navier–Stokes equation, for which the following estimate follows from Hölder's inequality:
		\[
		\begin{aligned}
			\left     \|\int_{0}^{t}\operatorname{div}_{\p_{t}\mathcal{A}^{n}}S_{\mathcal{A}^{n}}(v,q) \right\|_{L_{t}^{\infty}L^{q_{+}}}\lesssim& T\|\p_{t}\eta^{n}\|_{L_{t}^{\infty}W^{3-\frac{1}{q_{-}},q_{-}}}(1+\|\eta^{n}\|_{L_{t}^{\infty}W^{3-\frac{1}{q_{+}},q_{+}}})(\|v\|_{L_{t}^{\infty}W^{2,q_{+}}}+\|q\|_{L_{t}^{\infty}W^{1,q_{+}}}),\\
			\left\|\int_{0}^{t}\operatorname{div}_{\mathcal{A}^{n}}S_{\p_{t}\mathcal{A}^{n}}(v,q) \right\|_{L_{t}^{\infty}L^{q_{+}}}
			\lesssim& T\|\p_{t}\eta^{n}\|_{L_{t}^{\infty}W^{3-\frac{1}{q_{-}},q_{-}}}(1+\|\eta^{n}\|_{L_{t}^{\infty}W^{3-\frac{1}{q_{+}},q_{+}}})(\|v\|_{L_{t}^{\infty}W^{2,q_{+}}}+\|q\|_{L_{t}^{\infty}W^{1,q_{+}}}).
		\end{aligned}
		\]
		For the capillary boundary condition, by temporal integration by part, we have
		\[
		\begin{aligned}
			& \left \|\int_{0}^{t}\mathcal{K}(\xi)\p_{t}\mathcal{N}^{n} \right\|_{L_{t}^{\infty}W^{1-\frac{1}{q_{+}},q_{+}}}\lesssim \|\xi\|_{L_{t}^{\infty}W^{3-\frac{1}{q_{+}},q_{+}}}\|\p_{t}\eta^{n}\|_{L_{t}^{\infty}W^{3-\frac{1}{q_{-}},q_{-}}}\lesssim \|\xi\|_{L_{t}^{\infty}W^{3-\frac{1}{q_{+}},q_{+}}}\|\p_{t}\eta\|_{L_{t}^{\infty}W^{3-\frac{1}{q_{-}},q_{-}}},\\
			& \left \|\int_{0}^{t}S_{\mathcal{A}^{n}}(v,q)\p_{t}\mathcal{N}^{n} \right\|_{L_{t}^{\infty}W^{1-\frac{1}{q_{+}},q_{+}}}\lesssim (\|v\|_{L_{t}^{\infty}W^{2,q_{+}}}+\|q\|_{L_{t}^{\infty}W^{1,q_{+}}})\|\p_{t}\eta^{n}\|_{L_{t}^{2}W^{3-\frac{1}{q_{-}},q_{-}}},\\
			& \left\|\int_{0}^{t}S_{\p_{t}\mathcal{A}^{n}}(v,q)\mathcal{N}^{n} \right\|_{L_{t}^{\infty}W^{1-\frac{1}{q_{+}},q_{+}}}\lesssim (\|v\|_{L_{t}^{\infty}W^{2,q_{+}}}+\|q\|_{L_{t}^{\infty}W^{1,q_{+}}})\|\p_{t}\eta^{n}\|_{L_{t}^{\infty}W^{3-\frac{1}{q_{-}},q_{-}}},
		\end{aligned}
		\]
		and
		\[
		\begin{aligned}
			&\left \|\int_{0}^{t}\sigma\p_{1}(\mathcal{R}_{z}(\p_{1}\zeta_{0},\p_{1}\eta)\p_{1}\p_{t}\xi)\mathcal{N}^{n}\right\|_{L_{t}^{\infty}W^{1-\frac{1}{q_{+}},q_{+}}}\\&\lesssim \left\|\int_{0}^{t}(\mathcal{R}_{z}(\p_{1}\zeta_{0},\p_{1}\eta)\p_{1}^{2}\p_{t}\xi) \right\|_{L_{t}^{\infty}W^{1-\frac{1}{q_{+}},q_{+}}}+\left\|\int_{0}^{t}(\mathcal{R}_{zz}(\p_{1}\zeta_{0},\p_{1}\eta)\p_{1}\p_{t}\xi\p_{1}^{2}\eta) \right\|_{L_{t}^{\infty}W^{1-\frac{1}{q_{+}},q_{+}}}\\&\lesssim \left\|\sigma\int_{0}^{t}(\mathcal{R}_{zz}(\p_{1}\zeta_{0},\p_{1}\eta)\p_{1}^{2}\xi\p_{1}\p_{t}\eta)\mathcal{N}^{n} \right\|_{L_{t}^{\infty}W^{1-\frac{1}{q_{+}},q_{+}}}+\|\sigma\p_{1}(\mathcal{R}_{z}(\p_{1}\zeta_{0},\p_{1}\eta)\p_{1}\xi)\mathcal{N}^{n}(0)\|_{L_{t}^{\infty}W^{1-\frac{1}{q_{+}},q_{+}}}\\
			&\quad+\|\sigma\p_{1}(\mathcal{R}_{z}(\p_{1}\zeta_{0},\p_{1}\eta)\p_{1}\xi)\mathcal{N}^{n}(T)\|_{L_{t}^{\infty}W^{1-\frac{1}{q_{+}},q_{+}}}+\left\|\int_{0}^{t}(\mathcal{R}_{zz}(\p_{1}\zeta_{0},\p_{1}\eta)\p_{1}\p_{t}\xi\p_{1}^{2}\eta)\right\|_{L_{t}^{\infty}W^{1-\frac{1}{q_{+}},q_{+}}}\\
			&\lesssim T\|\p_{t}\p_{1}\xi\|_{L_{t}^{2}W^{\max(\frac{1}{q_{+}},1-\frac{1}{q_{+}}),q_{+}}}\|\p_{1}\eta\|_{L_{t}^{\infty}W^{2-\frac{1}{q_{+}},q_{+}}}+T\|\p_{1}^{2}\xi\|_{L_{t}^{\infty}W^{1-\frac{1}{q_{+}},q_{+}}}\|\p_{1}\p_{t}\eta\|_{L_{t}^{2}W^{\max(\frac{1}{q_{+}},1-\frac{1}{q_{+}}),q_{+}}}\\
			&\quad+\|\p_{1}\xi\|_{L_{t}^{\infty}W^{2-\frac{1}{q_{+}},q_{+}}}\|\p_{1}\eta\|_{L_{t}^{\infty}W^{2-\frac{1}{q_{+}},q_{+}}}\\
			&\lesssim \left( T\|\p_{t}\xi\|_{L_{t}^{2}W^{3-\frac{1}{q_{-}},q_{-}}} +\|\xi\|_{L_{t}^{\infty}W^{3-\frac{1}{q_{+}},q_{+}}} \right)\|\eta\|_{L_{t}^{\infty}W^{3-\frac{1}{q_{+}},q_{+}}} +T\|\xi\|_{L_{t}^{\infty}W^{3-\frac{1}{q_{+}},q_{+}}}\|\p_{t}\eta\|_{L_{t}^{2}W^{3-\frac{1}{q_{-}},q_{-}}}.
		\end{aligned}
		\]
		For the boundary condition on $\Sigma_{s}$, we have
		\[
		\left\|\int_{0}^{t}(S_{\p_{t}\mathcal{A}^{n}}(q,v)\nu)\cdot \tau \right\|_{L_{t}^{\infty}W^{1-\frac{1}{q_{+}},q_{+}}}\lesssim \|\p_{t}\eta\|_{L_{t}^{\infty}W^{3-\frac{1}{q_{-}},q_{-}}}(\|q\|_{L_{t}^{\infty}W^{1,q_{+}}}+\|v\|_{L_{t}^{\infty}W^{2,q_{+}}}).
		\]
		For the kinematic boundary condition, by integration by part, we have
		\[
		\begin{aligned}
			&\|\int_{0}^{t}v\cdot \p_{t}\mathcal{N}^{n}\|_{L_{t}^{\infty}W^{2-\frac{1}{q_{+}},q_{+}}}\lesssim \|\int_{0}^{t}\p_{1}v\cdot \p_{t}\mathcal{N}^{n}\|_{L_{t}^{\infty}W^{1-\frac{1}{q_{+}},q_{+}}}+\|\int_{0}^{t}\p_{t}v\p_{1}^{2}\eta^{n}\|_{L_{t}^{\infty}W^{1-\frac{1}{q_{+}},q_{+}}}\\
			&\qquad\qquad+\|v\p_{1}^{2}\eta^{n}(0)\|_{L_{t}^{\infty}W^{1-\frac{1}{q_{+}},q_{+}}}+\|v\p_{1}^{2}\eta^{n}(T)\|_{L_{t}^{\infty}W^{1-\frac{1}{q_{+}},q_{+}}}\\&\lesssim \|v\|_{L_{t}^{\infty}W^{2,q_{+}}}\|\p_{t}\eta^{n}\|_{L_{t}^{\infty}W^{3-\frac{1}{q_{-}},q_{-}}}+\|\p_{t}v\|_{L_{t}^{\infty}W^{2,q_{-}}}\|\eta^{n}\|_{L_{t}^{\infty}W^{3-\frac{1}{q_{+}},q_{+}}},\\
			&\|\int_{0}^{t}(u_{1} \p_{t}\p_{1}\xi)\|_{L_{t}^{\infty}W^{2-\frac{1}{q_{+}},q_{+}}}\lesssim \|u\|_{L_{t}^{\infty}W^{2,q_{+}}}\|\xi\|_{L_{t}^{\infty}W^{3-\frac{1}{q_{+}},q_{+}}}+\|u\|_{L_{t}^{\infty}W^{2-\frac{1}{q_{+}},q_{+}}(\Sigma)}\|\p_{1}\p_{t}\xi\|_{L_{t}^{2}W^{\max(\frac{1}{q_{+}},1-\frac{1}{q_{+}}),q_{+}}}\\
			&\qquad\qquad+\|\p_{t}u\|_{L_{t}^{2}W^{\max(\frac{1}{q_{+}},1-\frac{1}{q_{+}}),q_{+}}(\Sigma)}\|\p_{t}\xi\|_{L_{t}^{\infty}W^{3-\frac{1}{q_{-}},q_{-}}}\\
			&\lesssim \|u\|_{L_{t}^{\infty}W^{2,q_{+}}}\|\xi\|_{L_{t}^{\infty}W^{3-\frac{1}{q_{+}},q_{+}}}+\|u\|_{L_{t}^{\infty}W^{2,q_{+}}}\|\p_{t}\xi\|_{L_{t}^{2}W^{3-\frac{1}{q_{-}},q_{-}}}+\|\p_{t}u\|_{L_{t}^{2}W^{2,q_{-}}}\|\xi\|_{L_{t}^{\infty}W^{3-\frac{1}{q_{-}},q_{-}}}.
		\end{aligned}
		\]
		Finally, for the contact point condition, using the similar computation as for the capillary equation, we obtain
		\[
		\left \|\int_{0}^{t}\mathcal{R}_{z}(\p_{1}\zeta_0,\p_{1}\eta)\p_{t}\p_{1}\xi(\pm\ell) \right\|_{L_{t}^{\infty}}\lesssim \|\xi\|_{L_{t}^{\infty}W^{3-\frac{1}{q_{+}},q_{+}}}\|\p_{t}\eta\|_{L_{t}^{2}W^{3-\frac{1}{q_{-}},q_{-}}}.
		\]
		\noindent Therefore, given that $\|\p_{t}\eta\|_{L_{t}^{2}W^{3-\frac{1}{q_{-}},q_{-}}}+\|\eta\|_{L_{t}^{\infty}W^{3-\frac{1}{q_{+}},q_{+}}}+\|u\|_{L_{t}^{\infty}W^{2,q_{+}}}+\|\p_{t}u\|_{L_{t}^{2}W^{2,q_{-}}}\leq\delta\ll 1$. We obtain the estimate \eqref{est:diss_10}.
	\end{proof}

	%%%%%%%%%%%%%%%%%%%%%%%%%%%%%%%%%%%%%%%%%%%%%%
	\section{Construction of Strong Solutions to the Linear System}\label{sec:strong}
	%%%%%%%%%%%%%%%%%%%%%%%%%%%%%%%%%%%%%%%%%%%%%%

	In this section, we aim to construct the solution to \eqref{eq:quasi_linear} under the assumptions that the forcing terms $F^i$ are given and independent of the unknowns. To begin with, we derive a solution of \eqref{eq:quasi_linear_{s}} and then solve \eqref{eq:quasi_linear} by passing to the limit. To solve \eqref{eq:quasi_linear_{s}}, we suppose $(u,\eta,p,v_{l},\xi_{l})$, $n$, $k$, as well as $\mathcal{A}(\eta^{n})$, $J(\eta^{n})$, $\mathcal{N}(\eta^{n})$, etc. are given.

	%%%%%%%%%%%%%%%%%%%%%%%%%%%%%%%%%%%%%%%%%%%%%%
	\subsection{Galerkin Setup}
	%%%%%%%%%%%%%%%%%%%%%%%%%%%%%%%%%%%%%%%%%%%%%%
	
	Our goal is to use the Galerkin method to establish the existence and uniqueness of solutions to \eqref{eq:quasi_linear_{s}}, in which the unknowns are $(v_{d}, q_{d}, \xi_{d})$ and their derivatives. We omit the subscripts $(k, n)$ when there is little chance of confusion. In order to utilize the Galerkin method, we must first construct a countable basis of $\mathcal{W}(t)$ for each $t\in[0, T]$. It is evident that $\mathcal{W}_{\sigma}(t)\hookrightarrow \mathcal{H}_{\sigma}^0(t)\hookrightarrow (\mathcal{W}_{\sigma}(t))^\ast$.
	Due to the time-dependent incompressible condition $u\cdot \mathcal{N}^{n}\in H^{1}(\Sigma)$, the basis should also be time-dependent.
	
	We note a key technical difficulty. We must ensure that the initial data for the sequence of Galerkin approximate solutions $(v_{d}^{m}(0),D_{t}v_{d}^{m}(0))$ converge as $m\to\infty$ to the initial data of the PDE \eqref{eq:quasi_linear} or equivalently \eqref{eq:quasi_linear_{s}}. The initial data
	\[
	(v_{d}(0),D_{t}v_{d}(0))=( v_{d0}^{n,k}, D_t^2u(0))\in W^{2,q_{+}}(\Omega)\times H^{1}(\Omega),
	\]
	are compatible with the dynamical energy. The definitions of the initial functions $v_{d0}^{n,k}, D_t^2u(0)$ are given in Appendix \ref{sec:initial} and \ref{sec:initial_l}.  
	
	However, the regularity $v_{d}(0)\in W^{2,q_{+}}(\Omega)$ is not strong enough to ensure the following property
	\[
	v_{d}(0)\cdot \mathcal{N}^{n}(0)\in H^{1}(\Sigma),
	\]
	which is required in the definition of space $\mathcal{W}(0)$ as \eqref{def:w}. Hence, it is not suitable to construct $v_{d}^{m}(0)$ by directly projecting $v_{d}(0)$ to some finite dimensional subspace of $\mathcal{W}$. This presents the primary technical difficulty in the Galerkin scheme.
	
	In addition, both terms in bulk and terms on the boundary are contained in the Galerkin approximated weak formulation, as \eqref{disc_1}. Consequently, it is challenging to discretize both $v_{d}$ and $\xi_{d}$ simultaneously to deduce the convergence from 
	$
	(v^{m}_{d}(0),D_{t}v^{m}_{d}(0),\xi_{d}^{m}(0),\p_{t}\xi_{d}^{m}(0)),$
	to $
	(v_{d}(0),D_{t}v_{d}(0),\xi_{d}(0),\p_{t}\xi_{d}(0))=(v_{d0}^{n,k}, D_t^2 u(0), \xi_{d0}^{n,k}, \p_t^2\eta(0)) $.
	
	Recalling the classical Galerkin method for solving Navier-Stokes equations (for instance, see \cite{boyer_fabrie} and \cite{Te}), we must choose a special basis for our Galerkin process, to cancel the influence resulting from the boundary terms. Our construction for the basis is much more subtle because of the time-dependence and the convergence of approximated initial data as stated above. 
	
	We plan to first construct a basis of $\mathcal{H}_{\sigma}^0(0)$ and $\mathcal{W}_{\sigma}(0)$, then we derive the time-dependent basis. For this purpose,
	we define the map $\mathcal{J}:\mathcal{W}_{\sigma}(0)\rightarrow \mathcal{W}^{\ast}_{\sigma}(0)$ via
	$\left<\mathcal{J}v, w\right>_\ast=(v, w)_{\mathcal{W}},\quad \forall v, w\in \mathcal{W}(0)$.
	By Riesz representation theorem, $\mathcal{J}$ is an isometric isomorphism. This enables us to define $A=\mathcal{J}^{-1}: \mathcal{W}_{\sigma}(0)^\ast\rightarrow \mathcal{W}_{\sigma}(0)$ so that for any $f\in \mathcal{W}_{\sigma}(0)^\ast$, $Af\in \mathcal{W}_{\sigma}(0)$ is uniquely determined via
	$\left<f, w\right>_\ast=(Af, w)_{\mathcal{W}(0)},\quad \forall w\in \mathcal{W}_{\sigma}(0)$,
	which is equivalent to saying that $Af=v$ is a weak solution to $\mathcal{J}v=f$.
	
	\begin{lemma}\label{lem:basis_initial}
		The operator $A$ restricted on $\mathcal{H}_{\sigma}^0(0)$ is a compact, positive and self-adjoint operator. Thus, the eigenvalues of $A$ are $\{\lambda_i\}_{i=1}^\infty$ such that $
		0<\lambda_1\le\lambda_2\le\cdots$
		and
		$\lambda_i\to\infty$, as $ i\to\infty$.
		Finally, there exists an orthonormal basis $\{\psi_i\}_{i=1}^\infty$ of $\mathcal{H}_{\sigma}^0(0)$, where $\psi_i\in \mathcal{W}_{\sigma}(0)$ is an eigenfunction of $\mathcal{J}$  corresponding to $\lambda_i$:
		$
		\mathcal{J}\psi_i=\lambda_i\psi_i.
		$
		Moreover, $\left\{\frac{\psi_i}{\sqrt{\lambda_i}}\right\}_{i=1}^\infty$ is an orthonormal basis of $\mathcal{W}_{\sigma}(0)$.
	\end{lemma}
	\begin{proof}
		Note that $\mathcal{W}_{\sigma}(0)\subset\subset \mathcal{H}_{\sigma}^0(0)\hookrightarrow\mathcal{W}_{\sigma}^\ast(0)$.
		The restriction $A|_{\mathcal{H}_{\sigma}^0(0)}: \mathcal{H}_{\sigma}^0(0)\rightarrow\mathcal{W}_{\sigma}(0)$ implies that $A|_{\mathcal{H}_{\sigma}^0(0)}: \mathcal{H}_{\sigma}^0(0)\rightarrow \mathcal{H}_{\sigma}^0(0)$ is compact.
		
		We claim that $A$ is also a positive, self-adjoint operator.  For $f, g\in \mathcal{H}^0(0)$, let $u=Af$ and $v=Ag$. Then
		$(u, v)_{\mathcal{W}^{*}(0)}=(Af , v)_{\mathcal{W}(0)}=(f, v)_{\mathcal{H}^0(0)}=(f,Ag)_{\mathcal{H}^0(0)}$,
		and
		$(v, u)_{\mathcal{W}(0)}=(Ag , u)_{\mathcal{W}(0)}=(g, u)_{\mathcal{H}^0(0)}=(g,Af)_{\mathcal{H}^0(0)}$.
		By symmetry, $(u, v)_{\mathcal{W}(0)}=(v, u)_{\mathcal{W}(0)}$. Thus,
		$(f,Ag)_{\mathcal{H}^0(0)}=(g,Af)_{\mathcal{H}^0(0)}=(Af, g)_{\mathcal{H}^0(0)}$.
		Also, $(Af, f)_{\mathcal{H}^0(0)}=(f, f)_{\mathcal{W}(0)}\ge0$. Therefore, our claim is valid.
		
		Consequently, by Riesz-Schauder theorem, there exists a complete orthonormal basis $\{\psi_i\}_{i=1}^\infty\subseteq \mathcal{H}_{\sigma}^0(0)$ such that
		$A\psi_i=\mu_i\psi_i$, and $\mu_i\to0$, as $i\to\infty$.
		But
		$A\psi_i=\mu_i\psi_i\in \mathcal{W}_{\sigma}(0)$, for all $\psi_i\in \mathcal{H}_{\sigma}^0(0)$,
		so
		$\{\psi_i\}_{i=1}^\infty\subseteq \mathcal{W}_{\sigma}(0)$,
		and
		$\mu_i(\psi_i, w)_{\mathcal{W}(0)}=(A\psi_i, w)_{\mathcal{W}(0)}=(\psi_i, w)_{\mathcal{H}^0(0)}, \forall w\in \mathcal{W}(0)$.
		Then we have
		$(\psi_i, w)_{\mathcal{H}^{0}(0)}=\lambda_i(\psi_i, w)_{\mathcal{W}(0)}$ with  $\lambda_i=\frac1{\mu_i}\to\infty$, as $k\to\infty$.
		Thus, $(\psi_i, \psi_j)_{\mathcal{W}(0)}=\lambda_i(\psi_i, \psi_j)_{\mathcal{H}^0(0)}=\lambda_i\delta_{ji}$ implies $\{\psi_i/\sqrt{\lambda_i}\}$ is a complete orthonormal basis of $\mathcal{W}_{\sigma}(0)$.
	\end{proof}
	
	Let $\phi^j=M(0)^{-1}\psi_j$. By Lemma \ref{lem:equivalence_norm} and \eqref{eq:div-a-div}, $L^2(\Om)$ is isomorphic to $\mathcal{H}_{\sigma}^0(0)$, so $\{\phi^j\}_{j=1}^\infty$ is a basis of $L^2(\Om)$. Moreover, ${W}_{\sigma}(0)$ is isomorphic to $\mathcal{W}_{\sigma}(t)$ and $\{\frac{\phi^j}{\sqrt{\lambda_j}}\}_{j=1}^\infty$ is a basis of ${W}_{\sigma}(t)$. With the above preparations, we construct the time-dependent basis in the following proposition.
	\begin{proposition}\label{prop:basis_galerkin}
		For each $i>0$, $\phi^{i}$ is defined above. Then for each $t\ge0$, $\{\omega^j(t)=M(t)\frac{\phi^j}{\sqrt{\lambda_j}}\}_{j=1}^\infty$ is a basis of $\mathcal{W}_{\sigma}(t)$ (defined by \eqref{def:w}) and $\sqrt{\lambda_j}\omega^j(t)$ is a basis of $\mathcal{H}_{\sigma}^0(t)$.
	\end{proposition}
	\begin{proof}
		By the definition $\omega^j(t)=M(t)\frac{\phi^j}{\sqrt{\lambda_j}}$, we have $\omega^j(0)=M(0)\frac{\phi^j(0)}{\sqrt{\lambda_j}}=\frac{\psi_j}{\sqrt{\lambda_j}}$ is an orthonormal basis of $\mathcal{W}_{\sigma}(0)$. We use Lemma \ref{lem:equivalence_norm} and \eqref{eq:div-a-div} to deduce that
		$\{\omega^j(t)\}_{j=1}^\infty$ is a basis of $\mathcal{W}_{\sigma}(t)$ and $\sqrt{\lambda_j}\omega^j(t)$ is a basis of $\mathcal{H}_{\sigma}^0(t)$.
	\end{proof}
	
	%%%%%%%%%%%%%%%%%%%%%%%%%%%%%%%%%%%%%%%%%%%%%%
	\subsection{Galerkin Approximation for Initial Data}\label{sec:app_initial}
	%%%%%%%%%%%%%%%%%%%%%%%%%%%%%%%%%%%%%%%%%%%%%%
	
	We note that $\xi_{d}(0)$ can be represented by the normal trace of $v_{d0}$ due to the kinematic boundary condition. In the Galerkin method for $t>0$, we construct the approximate solutions $v_{d}^m$ and $D_tv_{d}^m$ and then pass to the limit to find $v_{d}$ and $D_tv_{d}$, and require the initial data $(v_{d}^m(0), D_tv_{d}^m(0))$ to converge to $(v_{d0}^{n,k}(0), D_t^{2}v(0))$. Specifically, we can prove that
	\begin{theorem}\label{thm:initial_convergence}
		Suppose that $\{\omega^j(t)\}_{j=1}^\infty$ is the basis defined as in Proposition \ref{prop:basis_galerkin} for the Hilbert space $\mathcal{W}_{\sigma}(t)$ for each $t\ge0$. When $t=0$, $\mathcal{W}_{\sigma}^m(0):=\operatorname{span} \{\omega^1(0), \ldots, \omega^m(0)\}$, the Galerkin approximated initial data $(v_{d}^{m}(0), D_tv_{d}^m(0))\in \mathcal{W}_{\sigma}^m(0)$ obeys the estimates for any $n,k>0$
		\[
		\|D_{t}v_{d}^m(0)\|_{\mathcal{H}^0}  \lesssim \|D_{t}^{2}u(0)\|_{\mathcal{H}^0}+\|v_{d0}^{n,k}\|_{\mathcal{H}^1},
		\
		\|v_{d}^m(0)\|_{{}_{0}\mathcal{H}^1}  \lesssim \|D_{t}^{2}u(0)\|_{\mathcal{H}^0}+\|v_{d0}^{n,k}\|_{\mathcal{H}^1}.
		\]
		Moreover, $(v_{d}^m(0), D_tv_{d}^m(0))$ converge to $(v_{d0}^{n,k}(0), D_t^{2}u(0))$ weakly in $\mathcal{W}\times\mathcal{H}^0$.
	\end{theorem}
	This theorem follows directly from Proposition \ref{prop:basis_galerkin} and the subsequent analysis.
	
	With the basis in Proposition \ref{prop:basis_galerkin}, we use the Galerkin method to construct the approximate solutions $D_tv_{d}^m(t)$ and $v_{d}^{m}(t)$, under the space $\mathcal{W}_{\sigma}^m(t)=\text{span}\{\omega^j(t)\}_{j=1}^m$ with $t\ge0$. In order to do that, we need to construct the initial data $(D_tv_{d}^m(0),v_{d}^{m}(0))$, and prove that they are uniformly bounded.
	
	We first construct the initial data for the approximated solution $v_{d}^m(0)$. Instead of projecting $v_{d}(0)$ into the space $\mathcal{W}_{\sigma}^{m}(0)$, we introduce a novel scheme. 
	
	Suppose that $D_{t}v_{d}^{m}(0)$ is the $L^{2}$ projection of $D_{t}^{2}u(0)$ onto the space $\mathcal{W}_{\sigma}^{m}(0)$. This implies that $D_{t}v_{d}^{m}(0)\rightarrow D_{t}v_{d}(0)$ strongly in $H^{0}$. 
	Then we construct $v_{d}^{m}(0)$ via the following equation
	\begin{align}{\label{disc_1}}
		\begin{aligned}
			(D_{t}v_{d}^{m}(0),\psi)_{\mathcal{H}^{0}}+((v_{d}^m(0),\psi))
			+[(v_{d}^{m}\cdot \mathcal{N}^{n})(0),\psi\cdot\mathcal{N}^{n}]_\ell
			+(\xi_{d}(0),\psi\cdot\mathcal{N}^{n})_{1,\Sigma_{k}}=\mathscr{F}_1(\psi)-((R^{n}v_{d}^{m})(0),\psi)_{\mathcal{H}^{0}}
		\end{aligned}
	\end{align}
	for any $\psi\in \mathcal{W}^{m}$, where
	\begin{align}{\label{def:F_1}}
		\begin{aligned}
			\mathscr{F}_1(\psi)=\int_{\Omega}F^{1}\cdot \psi J^{n}-\int_{-\ell}^{\ell}F^{4}\cdot \psi-\int_{\Sigma_{s}}F^{5}(\psi\cdot \tau)J^{n}-[F^{7},\psi\cdot \mathcal{N}^{n}]_{\ell}+((R^{n}(v_{0}^{n}),\psi))+(\p_{t}(R^{n}v_{0}^{n}),\psi)_{\mathcal{H}^{0}(0)}.
		\end{aligned}
	\end{align}
	
	Assuming $v_{d}^{m}(0)=\sum_{k=1}^{m}d_{k}^{m}(0)\omega^{k}$, it satisfies the matrix equation induced by equation \eqref{disc_1}:
	\[
	\mathcal{B}\vec{d}^T-\mathscr{R}\vec{d}=\vec{F}^T-\vec{\xi}^T-\vec{Y}^T
	\]
	\noindent where $\vec{d}=(d_1(0),\ldots,d_{m}(0))$, $\vec{F}=(\mathscr{F}_{1}(\omega^{1}),\ldots,\mathscr{F}_{1}(\omega^{m}))$,$\vec{\xi}=((\xi_{d}(0),\omega^{1}\cdot \mathcal{N}^{n}),\ldots,(\xi_{d}(0),\omega^{m}\cdot \mathcal{N}^{n}))$, $\vec{Y}=(((D_{t}v_{d}^{m}(0)),\omega^{1})_{\mathcal{H}^{0}},\ldots,((D_{t}v_{d}^{m}(0)),\omega^{m})_{\mathcal{H}^{0}})$, $\mathscr{R}_{i,j}=(R^{n}\omega^{i},\omega^{j})_{\mathcal{H}^{0}}$. $\mathcal{B}$ is an invertible matrix matrix with elements $\mathcal{B}_{i,j}=(\omega^{i}(0),\omega^{j}(0))_{{}_{0}\mathcal{H}^{1}}$. Using the smallness of $\|\mathscr{R}\|\lesssim \|\p_{t}\eta(0)\|_{W^{3-\frac{1}{q_{-}},q_{-}}}$ and a Picard iteration, we can solve for $d_{j}^{m}$.
	
	Moreover, Subtracting the weak form equation for the initial condition from equation \eqref{disc_1}, we obtain:
	\begin{align}{\label{equ:disc_2}}
		\begin{aligned}
			((D_{t}v_{d}^{m}(0))-D_{t}v_{d}(0),\psi)_{\mathcal{H}^{0}}+((v_{d}^m(0)-v_{d}(0),\psi))+[v_{d}^{m}(0)\cdot \mathcal{N}^{n}-v_{d}(0)\cdot \mathcal{N}^{n},\psi\cdot\mathcal{N}^{n}]_\ell&\\
			+(R^{n}v_{d}^{m}(0)-R^{n}v_{d}(0),\psi)&=0
		\end{aligned}
	\end{align}
	for any $\psi\in \mathcal{W}^{m}$. Setting $\psi=v_{d}^{m}$ in equation \eqref{equ:disc_2}, we rewrite it as follows:
	\[\begin{aligned}
		&((D_{t}v_{d}^{m}(0)),v_{d}^{m}(0))_{\mathcal{H}^{0}}+((v_{d}^m(0),v_{d}^{m}(0)))
		+[v_{d}^{m}(0)\cdot \mathcal{N}^{n},v_{d}^{m}(0)\cdot\mathcal{N}^{n}]_\ell\notag\\
		&=((D_{t}v_{d}(0)),v_{d}^{m})_{\mathcal{H}^{0}}+((v_{d}(0),v_{d}^{m}(0)))
		+[v_{d}(0)\cdot \mathcal{N}^{n},v_{d}^{m}(0)\cdot\mathcal{N}^{n}]_\ell-(R^{n}v_{d}^{m}(0)-R^{n}v_{d}(0),v_{d}^{m}(0)).
	\end{aligned}\]
	\noindent Using Cauchy's inequality and the fact that $\|D_{t}v_{d}^{m}(0)\|_{H^{0}}\lesssim\|D_{t}v_{d}(0)\|$, we obtain the following uniform estimate for the initial data $v_{d}^{m}(0)$:
	\[
	((v_{d}^m(0),v_{d}^{m}(0)))
	+[v_{d}^{m}(0)\cdot \mathcal{N}^{n},v_{d}^{m}(0)\cdot\mathcal{N}^{n}]_\ell
	\leq \|\p_{t}v_{d}(0)\|^{2}_{\mathcal{H}^{0}}+\frac{1}{2}\|v_{d}^{m}(0)\|^{2}_{{}_0\mathcal{H}^{1}}+2\|v_{d}(0)\|^{2}_{{}_0\mathcal{H}^{1}}+\|v_{d}(0)\|_{\mathcal{H}^{0}}^{2}.
	\]
	\noindent Therefore, we derive the following inequality
	\begin{align}{\label{bound_dtu_0}}
		((v_{d}^m(0),v_{d}^{m}(0)))
		+[v_{d}^{m}(0)\cdot \mathcal{N}^{n},v_{d}^{m}(0)\cdot\mathcal{N}^{n}]_\ell
		\lesssim \|\p_{t}v_{d}(0)\|^{2}_{\mathcal{H}^{0}}+\|v_{d}(0)\|^{2}_{{}_0\mathcal{H}^{1}},
	\end{align}
	\noindent which leads to the following weak convergence theorem.
	\begin{proposition}\label{prop:converge_dtu2}
		It holds that $v_{d}^m(0)$ converges to $v_{d}(0)$ weakly in ${}_0\mathcal{H}^1$ as $m\to \infty$.
	\end{proposition}
	
	\begin{proof}
		The uniform bound for $v_{d}^m(0)$ in \eqref{bound_dtu_0} allows us to extract a weakly convergent subsequence, still denote $v_d^m(0)$, such that
		$v_{d}^m(0)\rightharpoonup w\quad\text{in}~~{}_0\mathcal{H}^1$,
		for some $w\in {}_0\mathcal{H}^1$, which implies that $\operatorname{div}_{\mathcal{A}^{n}}w=0$.
		
		Passing to the limit $m\to \infty$ in \eqref{bound_dtu_0} and
		using the fact that $(D_{t}v_{d}^{m}(0))\to D_{t}v_{d}(0)$ in $\mathcal{H}^0$,
		we obtain the following equation
		\begin{equation}\label{eq:in_dtu_1}
			(w-v_{d}(0), \psi)_{{}_0\mathcal{H}^1}+(R^{n}w-R^{n}v_{d}(0),\psi)_{\mathcal{H}^{0}}=0
		\end{equation}
		for any $\psi\in\mathcal{W}(0)$.  We now pull back to the initial state of the original deformable domain $\Om(0)$ by the inverse mapping $\Phi(0)^{-1}$ and let 
		\[
		\tilde{w}\circ \Phi(0)=w,\  \widetilde{v_{d}}(0)\circ\Phi(0)=v_{d}(0),
		\]
		so that $\tilde{w}\in L^2(\Om(0)):=\{u\in L^2(\Om(0)) \}$ and
		$
		\widetilde{v_{d}}(0)\in {}_{0}{H}^{1}(\Om(0)):=\{u\in L^2(\Om(0)): \|u\|_{\mathcal{H}^{1}(\Om(0))}^2+[u\cdot\mathcal{N}^{n}]_\ell^2<\infty\}$.
		We choose $\psi$ in the closure of $C_{c}^\infty(\Om)$ under $\mathcal{W}$ norm, that is a subset of $\mathcal{W}_{\sigma}$, so that \eqref{equ:disc_2} is reduced to
		\[
		(\tilde{w}-\widetilde{v_{d}}(0), \psi)_{H^{1}(\Om(0))}+(R^{n}\tilde{w}-R^{n}\tilde{v_{d}}(0),\psi)_{L^{2}(\Omega)}=0.
		\]
		Due to the smallness of $\|R^{n}(0)\|_{L^{\infty}(\Omega)}\lesssim\|\p_{t}\eta^{n}(0)\|_{W^{3-\frac{1}{q_{-}},q_{-}}}$, the inner product 
		$(\phi, \psi)_{H^{1}(\Om(0))}+(R^{n}\phi,\psi)_{L^{2}(\Omega)}$ is positive when $\phi=\psi$. Since $C_{c}^\infty(\Om)$ is dense in $H^1(\Om)$, we have
		$\tilde{w}-\widetilde{v_{d}}(0)=0$ and $w=v_{d}(0)$ by $\Phi(0)^{-1}$.
	\end{proof}

	%%%%%%%%%%%%%%%%%%%%%%%%%%%%%%%%%%%%%%%%%%%%%%
	\subsection{Construction of Strong Solutions}
	%%%%%%%%%%%%%%%%%%%%%%%%%%%%%%%%%%%%%%%%%%%%%%

	Suppose that all the forcing terms $F^{i}$ are in the spaces satisfying estimate \eqref{est:bound_linear}. We say that
	$(v_{d},q_{d},\xi_{d})\in L^{2}W^{2,q_{+}}\times L^{2}W^{1,q_{+}}\times L^{2}W^{3-1/q_{+},\,q_{+}}$
	is a strong solution to \eqref{eq:quasi_linear} if, for each \(j=0,1,2\), $
	\partial_t^j v \in L^\infty H^0 \cap L^2 H^1$ and $
	\partial_t^j \xi \in L^\infty H^1 \cap L^2 H^{3/2-\alpha}$,
	and if \((v_{d},q_{d},\xi_{d})\) satisfies \eqref{eq:quasi_linear}.
	
	In the following theorem, we construct strong solutions to \eqref{eq:quasi_linear} by first solving \eqref{eq:quasi_linear_{s}} via the Galerkin method. For fixed $n$, $k$ and prescribed functions $(v_{l},\xi_{l})$, the construction proceeds in four main steps:
	\begin{itemize}
		\item 
		Step 1. Using a basis of $\mathcal{W}_{\sigma}(t)$ and ODE theory, we construct the Galerkin solutions $(v_{d}^m, \xi_{d}^m)$ to \eqref{eq:quasi_linear_{s}} in an $m$-dimensional subspace.
		\item 
		Step 2. We differentiate the system in time to derive equations for $(D_t v_{d}^m, \partial_t \xi_{d}^m)$. Energy estimates yield uniform bounds for these quantities. Passing to the limit as $m \to \infty$ using functional analytic arguments, we recover the weak formulation and improve the regularity of $(\xi_{d}, \partial_t \xi_{d})$.
		\item 
		Step 3. We recover the pressure using Theorem~\ref{thm:pressure_s}. Furthermore, Theorem~\ref{thm:pressure} implies that $(v_{d}, \xi_{d})$ is in fact a strong solution to \eqref{eq:quasi_linear_{s}}.
		\item 
		Step 4. We derive uniform bounds for the solutions of \eqref{eq:quasi_linear_{s}} that are independent of $k$.
	\end{itemize}
	
	Once the solution to \eqref{eq:quasi_linear_{s}} and the corresponding $k$-independent uniform estimates are established, we pass the limit as $k \to +\infty$. We then apply the contraction mapping theorem to derive a fixed point satisfying $(v_{l}, \xi_{l}) = (v, \xi)$, which yields a solution to \eqref{eq:quasi_linear_n}.
	
	It then remains to pass to the limit $\eta^{n} \to \eta$. To this end, we first invoke Theorem~\ref{thm:pressure_+} to establish a $q_{+}$-elliptic estimate for the zeroth-order quantities $(v, q, \xi)$. We then apply the a priori estimates developed in \cite{GT2020} to derive uniform bounds independent of $n$, which enable us to pass to the limit $n \to \infty$ and thereby complete the proof
	
	Before constructing the solution, we first introduce the following definitions: 
	\begin{equation}\label{def:dissipation}
		\begin{aligned}
			\mathscr{D}(u,p,\eta)&:=\Big(\|u\|_{L^{2}W^{2,q_{+}}}^2+\|p\|_{L^{2}W^{1,q_{+}}}^{2}+\|\eta\|_{L^{2}W^{3-1/q_+, q_+}}^2\Big)+\sum_{j=0}^2\Big(\|\p_t^ju\|_{L^{2}H^{1}}^2+\|\p_t^ju\|_{L^{2}L^2(\Sigma_s)}^2\Big)\\&\quad
			+\sum_{j=0}^2\Big(\|\p_t^j\eta\|_{L^{2}H^{3/2-\alpha}}^2+\|[\p_t^{j}u\cdot \mathcal{N}]_\ell^2\|_{L_{t}^{2}}\Big)+\|\p_t^3\eta\|_{L^{2}H^{1/2-\alpha}}^2\\
			&\quad+\Big(\|\p_tu\|_{L^{2}W^{2,q_{-}}}^2+\|\p_{t}p\|_{L^{2}W^{1,q_{-}}}^{2}+\|\p_{t}\eta\|_{L^{2}W^{3-1/q_-, q_-}}^2\Big),
		\end{aligned}
	\end{equation}
	\begin{equation}\label{def:energy}
		\begin{aligned}
			\mathscr{E}(u,p,\eta)&: =\|u\|_{L^{\infty}W^{2,q_+}}^2+\sum_{i=0}^{1}\|\p_t^{i}u\|_{L_{t}^{\infty}H^{1+\varepsilon_-/2}}^2+\sum_{k=0}^2\|\p_t^ku\|_{L_{t}^{\infty}L^{2}}^2+\|p\|_{L_{t}^{\infty}W^{1,q_-}}^2+\sum_{i=0}^{1}\|\p_t^{i}p\|_{L_{t}^{\infty}H^{0}}^2\\&\quad+\|\p_{t}^{i}\eta\|_{L_{t}^{\infty}W^{3-1/q_-, q_-}}^2+\sum_{i=0}^{1}\|\p_t^{i}\eta\|_{L_{t}^{\infty}H^{3/2+(\varepsilon_--\alpha)/2}}^2+\sum_{j=0}^2\Big(\|\p_t^j\eta\|_{L_{t}^{\infty}H^{1}}^2+\|[\p_{t}^{j}u\cdot \mathcal{N}]_{\ell}\|_{L_{t}^{\infty}}^{2}\Big),
		\end{aligned}
	\end{equation}
	\begin{equation}\label{def:DEK_0}
		\mathfrak{K}(u,p,\eta):=\mathscr{E}(u,p,\eta)+\mathscr{D}(u,p,\eta).
	\end{equation}
	\noindent Compared to \eqref{energy} and \eqref{dissipation}, we add only regularity in time here.
	
	Now, we state our main theorem in this section as follows.
	\begin{theorem}\label{thm:linear_low}
		Assume that the initial data are constructed in Section \ref{sec:initial_linear}.
		Suppose that $\mathfrak{K}(u,\eta,p)\le\delta$ (as defined in \eqref{def:DEK_0}) is smaller than the constant $\delta_0$ defined in Lemma \ref{lem:lemma1}, Lemma \ref{lem:equivalence_norm} and Theorem 4.7 in \cite{GT2020}.
		Then there exists a unique strong solution $(v_{d},q_{d},\xi_{d}, v, q, \xi)$ solving \eqref{eq:quasi_linear} and \eqref{eq:quasi_linear2}. Moreover, the solution obeys the estimate
		\begin{align}\label{est:bound_linear}
			\begin{aligned}
				\mathfrak{K}(v,q,\xi)&\lesssim\exp\big\{T(\|\p_t\eta\|_{L^\infty H^{3/2+(\varepsilon_--\alpha)/2}}+\|\eta\|_{L^{\infty}W^{3-\frac{1}{q_{-}},q_{-}}})\big\}\\
				&\quad\quad\bigg\{ \mathcal{E}(u_{0},p_{0},\xi_{0}) +\|(F^1-F^4-F^5)(0)\|^{2}_{(\mathcal{H}^1)^\ast}+\mathfrak{F}+\mathcal{Z}+\mathfrak{K}(u,p,\eta)\bigg\},
			\end{aligned}
		\end{align}
		where
		\begin{align}
			\begin{aligned}
				\mathfrak{F}=\big(1+\|\p_t\eta\|_{L^\infty H^{3/2+\varepsilon_-/2}}^2\big)\Big(\|F^1\|_{L^2L^{q_-}}^2+\|F^4\|_{L^2W^{1-1/q_-,q_-}}^2
				+\|F^5\|_{L^2W^{1-1/q_-,q_-}}^2\Big) \\
				+\big(1+\|\p_t\eta\|_{L^\infty H^{3/2+\varepsilon_-/2}}^2\big)\|\p_t(F^1-F^4-F^5)\|_{(\mathcal{H}^1_T)^{\ast}}^2
				+\sum_{j=0}^1\|[F^{7,j}]_\ell\|_{L^2}^2,
			\end{aligned}
		\end{align}
		and
		\begin{align}
			\mathcal{Z}=\|\int_{0}^{t}F^1\|_{L^\infty 
				L^{q_+}}^2 + \|\int_{0}^{t}F^4\|_{L^\infty W^{1-1/q_+,q_+}}^2 + \|\int_{0}^{t}F^5\|_{L^\infty W^{1-1/q_+,q_+}}^2 + \|\int_{0}^{t}[F^7]_\ell\|_{L_t^{\infty}}^2
		\end{align}
		where "$\lesssim$" omits a constant depending on $(g, \sigma, \ell, |\Om|)$.
	\end{theorem}
	\begin{proof}
		$~$
		
		\paragraph{\underline{Step 1 -- Galerkin Setup}}
		We begin by constructing the solution to \eqref{eq:quasi_linear_{s}} for any given $n,k$ and prescribed functions $(u,\eta,p,\eta^{k},u^{k})$, $(\eta^{n},v_{l}^{k},\xi_{l}^{k})$. 
		
		In order to utilize the Galerkin method, we choose the same basis $\{\omega^j(t)\}_{j=1}^\infty$ of ${}\mathcal{W}_{\sigma}$ constructed in Proposition \ref{prop:basis_galerkin} for each $t\in[0, T]$, that satisfies the time-dependent requirement of condition $u\cdot \mathcal{N}^{n}\in H^{1}(-\ell,\ell)$. Furthermore, tthe time derivatives of these basis vectors are expressed as a finite linear combination of these basis vectors. In particular, $\p_t\omega^j=R^{n}\omega^j$,
		where $R(t)$ is defined by $R^{n}(t):=\p_tM^{n}(t)(M^{n})^{-1}(t)$.
		For any integer $m\ge1$, we define the finite dimensional space
		$
		\mathcal{W}_{\sigma}^m(t):=\text{span}\{\omega^1(t), \cdots, \omega^m(t)\}\subseteq \mathcal{W}_{\sigma}(t)$,
		and we write
		$
		\mathcal{P}^m_t: \mathcal{W}_{\sigma}(t)\rightarrow\mathcal{W}_{\sigma}^m(t)
		$
		for the $\mathcal{W}$ orthogonal projection onto $\mathcal{W}^m(t)$. Clearly, for each $\omega\in \mathcal{W}_{\sigma}(t)$, $\mathcal{P}^m_t\omega\rightarrow\omega$ as $m\rightarrow\infty$.

		\paragraph{\underline{Step 2 -- Solving the Approximate Problem}} In this step, we fix $n,k$ and the given functions $v_{l},\xi_{l}$. 
		
		For each $m\ge1$, we define an approximate solution $
		v_{d}^m(t):=d^m_j(t)\omega^j(t)$, with $d^m_j: [0, T]\rightarrow\mathbb{R}\ \text{for}\ j=1, \dots, m$,
		where, as usual, we use the Einstein convention of summation of the repeated index $j$. We define $\xi_{d}^{m}$ similarly via the following transport equation:
		\[
		\partial_{t}^{2}\xi_{d}^{m}=D_{t}v_{d}^{m}\cdot \mathcal{N}^{n}
		+D_{t}(R^{n}v_{l}^{k})\cdot \mathcal{N}^{n}+u_{1}^{k}\partial_{1}\partial_{t}\xi_{d}^{m}+\p_{t}u_{1}^{k}\partial_{1}\partial_{t}\xi_{l}^{k}.
		\]
		\noindent This system admits the following solution:
		\begin{align}
			\partial_{t}\xi_{d}^{m}(s,X(s,x))=&\partial_{t}\xi_{d}(0)(x)+\int_{0}^{t}((D_{t}v_{d}^{m})\cdot \mathcal{N}^{n}) (y,X(y,x))\mathrm{d}y\no\\
			&+\int_{0}^{t}(\partial_{t}u_{1}^{k}\partial_{1}\partial_{t}\xi_{l}^{k}+D_{t}(R^{n}v_{l}^{k})\cdot \mathcal{N}^{n}))(y,X(y,x))\mathrm{d}y \label{eq:c1},
		\end{align}
		\noindent where $X$ satisfies the following ODE:
		\[
		\frac{d}{dt}X(t,x)=u_{1}^{k}(t,x)~~~\operatorname{and}~~~X(0,x)=x.
		\]
		\noindent Then the equation \eqref{eq:c1} is equivalent to:
		\begin{align}\label{eq:trans}
			\begin{aligned}
				\partial_{t}\xi_{d}^{m}(s,x)=&\partial_{t}\xi_{d}(0)(X(-s,x))+\int_{0}^{s}((D_{t}v_{d}^{m})\cdot \mathcal{N}^{n}) (y,X(y,X(-s,x)))\mathrm{d}y\\
				&+\int_{0}^{s}(\partial_{t}u_{1}^{k}\partial_{1}\partial_{t}\xi_{l}^{k}+D_{t}(R^{n}v_{l}^{k})\cdot \mathcal{N}^{n})(y,X(y,X(-s,x)))\mathrm{d}y.
			\end{aligned}
		\end{align}
		\noindent Using the fact that
		$
		\frac{d}{dt}f(t,X(t,X(-s,x)))=\partial_{t}f-u_{1}^{k}\partial_{x}f,
		$
		\noindent integrating equation \eqref{eq:trans} from $0$ to $t$, and using integration by part, we obtain the following equation for $\xi_{d}^{m}$
		\begin{align}{\label{eq:xi_d}}
			\xi_{d}^{m}(t,x)=&\int_{0}^{t}\partial_{t}^{2}\eta_{0}(X(-s,x))ds+\int_{0}^{t}\int_{0}^{s}(\partial_{t}u_{1}^{k}\partial_{1}\partial_{t}\xi_{l}^{k}+D_{t}(R^{n}v_{l}^{k})\cdot \mathcal{N}^{n})(y,X(y,X(-s,x)))\mathrm{d}y\,\mathrm{d}s\notag\\
			&+\xi_{d}(0,x)-\int_{0}^{t}((v_{d}^{m})\cdot \mathcal{N}^{n})(0,X(-s,x))\mathrm{d}s+\int_{0}^{t} ((v_{d}^{m})\cdot \mathcal{N}^{^{n}})(s,x)\mathrm{d}s
			\\
			&-\int_{0}^{t}\int_{0}^{s}u_{1}^{k}\partial_{1}((v_{d}^{m})\cdot \mathcal{N}^{n})(y,X(y,X(-s,x)))\mathrm{d}y\,\mathrm{d}s\no,
		\end{align}
		where $v_{d}^m(s)$ and $u_{1}(t,x)$ denote their traces on $\Sigma$. Using the fact that $u_{1}(\pm \ell)=0$ and $W^{2,q_{-}}(\Omega)\hookrightarrow C^{0}(-\ell,\ell)$, it follows that $X(t,x)$ is well defined and
		\[
		X(t,\pm \ell)=\pm \ell~\operatorname{for}~\operatorname{any}~t\geq 0,
		\]
		\noindent ensuring that the flow $X$ induced by $u_{1}$ fixes the boundary.
		
		We choose the coefficients $d^m_j(t)$ so that
		\begin{equation}
			\begin{aligned}
				&(\p_tv_{d}^m,w)_{\mathcal{H}^0}+((v_{d}^m,w))+(\xi_{d}^m,w\cdot\mathcal{N}^{n})_{1,\Sigma_{0}}+(\int_{0}^{t}\mathcal{R}_{z}(\partial_{1}\zeta_{0},\partial_{1}\eta^{k})\partial_{1}\partial_{t}\xi_{d}^{m},\partial_{1}(w\cdot\mathcal{N}^{n}))_{L^{2}}+[v_{d}^{m}\cdot \mathcal{N}^{n},w\cdot\mathcal{N}^{n}]_{\ell}\\&
				=\int_{\Om}F^1\cdot wJ^{n}+\int_{-\ell}^{\ell}F^4\cdot w-\int_{\Sigma_s}F^5(w\cdot\tau)J^{n}-[F^7,w\cdot\mathcal{N}^{n}]_\ell+((R^{n}D_{t}v_{l},w))+(\p_{t}(R^{n}v_{l}),w)_{\mathcal{H}^{0}}\\
				&\quad+ \left(\int_{0}^{t}\p_{1}(\mathcal{R}_{zz}(\p_{1}\zeta_{0},\p_{1}\eta^{k})\p_{1}\p_{t}\xi_{l}^{k}\p_{1}\p_{t}\eta^{k}),(w\cdot \mathcal{N}^{n})\right)_{L^{2}}+[\int_{0}^{t}\mathcal{R}_{zz}(\p_{1}\zeta_{0},\p_{1}\eta^{k})\p_{1}\p_{t}\xi_{l}^{k}\p_{1}\p_{t}\eta^{k},w\cdot \mathcal{N}^{n}]_{\ell},\label{eq:galerkin}
			\end{aligned}
		\end{equation}
		for each $w\in \mathcal{W}^m(t)$.
		From equation \eqref{eq:xi_d}, we may compute
		\[\begin{aligned}
			&(\xi_{d}^m(t,x),w\cdot\mathcal{N}^{n}(t))_{1,\Sigma_{0}}
			= \left(\int_{0}^{t}\partial_{t}^{2}\eta_{0}(X(-s,x))ds+2\int_{0}^{t}\int_{0}^{s}(\partial_{t}u_{1}^{k}\partial_{1}\partial_{t}\xi_{l}^{k})(y,X(y,X(-s,x))dy\,ds,w\cdot\mathcal{N}^{n}(t) \right)_{1,\Sigma_{0}}\\
			&\quad+(\xi_{d}(0),w\cdot\mathcal{N}^{n}(t) )_{1,\Sigma_{0}}-\left(\int_{0}^{t}\Big[(v_{d}^{m}\cdot \mathcal{N}^{n})(0,X(-s,x)) + (v_{d}^{m}\cdot \mathcal{N}^{n})(s,x) \Big]ds,w\cdot \mathcal{N}^{n}(t) \right)_{1,\Sigma_{0}}\\&\quad-\left(\int_{0}^{t}\int_{0}^{s}\Big[u_{1}^{k}\partial_{1}((v_{d}^{m}\cdot \mathcal{N}^{n})(y,X(y,X(-s,x))) -2 (D_{t}(R^{n}v_{l}^{k})\cdot \mathcal{N}^{n})(y,X(y,X(-s,x)))\Big]dy\,ds, w\cdot\mathcal{N}^{n}(t) \right)_{1,\Sigma_{0}},
		\end{aligned}
		\]
		\noindent and
		\[\begin{aligned}
			&\left (\int_{0}^{t}\mathcal{R}_{z}(\partial_{1}\zeta_{0},\partial_{1}\eta^{k})\partial_{1}\partial_{t}\xi_{d},\partial_{1}(w\cdot\mathcal{N}^{n})\right)_{L^{2}}=\left(\int_{0}^{t}\mathcal{R}_{z}(\partial_{1}\zeta_{0},\partial_{1}\eta^{k})\partial_{t}\partial_{1}\xi_{d}(0)(X(-s,x)),\partial_{1}(w \cdot \mathcal{N}^{n}) \right)_{L^{2}}\notag\\
			&\quad+2(\int_{0}^{t}\mathcal{R}_{z}(\partial_{1}\zeta_{0},\partial_{1}\eta^{k})\int_{0}^{s}\p_{1}(\partial_{t}u_{1}^{k}\partial_{1}\partial_{t}\xi_{l}^{k})(y,X(y,X(-s,x))),\partial_{1}(w\cdot \mathcal{N}^{n}))_{L^{2}}\notag\\
			&\quad+\left(\int_{0}^{t}\mathcal{R}_{z}(\partial_{1}\zeta_{0},\partial_{1}\eta^{k})\big[ -\p_{1}(v_{d}^{m} \cdot \mathcal{N}^{n})(0,X(-s,x)) + \p_{1} (v_{d}^{m} \cdot \mathcal{N}^{n})(s,x)\big] ,\p_{1}(w\cdot \mathcal{N}^{n}) \right)_{L^{2}}\\
			&\quad+ \left(\int_{0}^{t}\int_{0}^{s}\mathcal{R}_{z}(\p_{1}\zeta_{0},\p_{1}\eta^{k})\big[\p_{1}(D_{t}(R^{n}v_{l}^{k})\cdot \mathcal{N}^{n}) - \p_{1}(u_{1}^{k}\partial_{1}(v_{d}^{m} \cdot \mathcal{N}^{n}))(y,X(y,X(-s,x)))\big],\partial_{1}(w\cdot \mathcal{N}^{n})\right)_{L^{2}}.
		\end{aligned}\]
		\noindent Substituting these two equations to \eqref{eq:galerkin}, this equation is equivalent to the equation of $d^m_j$ given by
		\begin{equation}\label{eq:galerkin_2}
			\begin{aligned}
				\dot{d}^m_i(\omega^i,\omega^j)_{\mathcal{H}^0}+d^m_i\tilde{\mathfrak{B}} -\int_{0}^{t}d_{i}^{m}(s) \tilde{\mathfrak{C}}(t,s)\,\mathrm{d}s-\int_{0}^{t}\int_{0}^{s}d_{i}^{m}(y) \tilde{\mathfrak{J}}(t,s,y)dy\,ds= \tilde{\mathfrak{F}}(t),
			\end{aligned}
		\end{equation}
		for $i,j=1, \dots, m$, where
		\[
		\tilde{\mathfrak{B}}(t) = (R^{n}(t)\omega^i,\omega^j)_{\mathcal{H}^0}+((\omega^i,\omega^j))+[\omega^i\cdot\mathcal{N}^{n},\omega^j\cdot\mathcal{N}^{n}]_{\ell}
		\]
		\[
		\begin{aligned}
			\tilde{\mathfrak{C}}(t,s)&=\Big((\omega^{i}\cdot \mathcal{N}^{n})(0,X(-s,x)) +(\omega^{i}\cdot \mathcal{N}^{n})(s,x),\omega^{j}\cdot \mathcal{N}^{n}(t) \Big)_{1,\Sigma_{0}}\\
			&\quad+ \Big( \mathcal{R}_{z}(\p_{1}\zeta_{0},\p_{1}\eta^{k})\omega^{i}\cdot \mathcal{N}^{n}(0,X(-s,x)) + \mathcal{R}_{z}(\p_{1}\zeta_{0},\p_{1}\eta^{k})\omega^{i}\cdot \mathcal{N}^{n}(s,x)ds,\p_{1}(\omega^{j}\cdot \mathcal{N}^{n}) \Big)_{L^{2}},
		\end{aligned} 
		\]
		\[
		\begin{aligned}
			\tilde{\mathfrak{J}}(t,s, y) & =(u_{1}^{k}\partial_{1}((v_{d}^{m})\cdot \mathcal{N}^{n})(y,X(y,X(-s,x))), \omega^{j}\cdot\mathcal{N}^{n}(t))_{1,\Sigma_{0}}\\
			&\quad- \Big(\mathcal{R}_{z}(\p_{1}\zeta_{0},\p_{1}\eta^{k})\p_{1}(u_{1}^{k}\p_{1}(\omega^{i}\cdot \mathcal{N}^{n}(y,X(y,X(-s,x))))),\p_{1}(\omega^{j}\cdot \mathcal{N}^{n}) \Big)_{L^{2}},
		\end{aligned}
		\]
		\[
		\begin{aligned}
			\tilde{\mathfrak{F}}(t) &= \int_{\Om}F^1\cdot \omega^jJ^{n}-\int_{-\ell}^{\ell} F^4\cdot \omega^j-\int_{\Sigma_s}F^5(\omega^j\cdot\tau)J^{n}-(\eta_0, \omega^j\cdot\mathcal{N}^{n}(t))_{1,\Sigma_{k}}-[F^7,\omega^j\cdot\mathcal{N}^{n}]_\ell\\
			&-(\xi_{d}(0),\omega^{j}\cdot \mathcal{N}^{n})_{1,\Sigma_{n}}
			-\Big(\int_{0}^{t}\partial_{t}^{2}\eta_{0}(X(-s,x))\mathrm{d}s+ \int_{0}^{s}(\p_{1}(\partial_{t}u_{1}^{k}\partial_{1}\partial_{t}\xi_{l}^{k})(y,X(y,X(-s,x)))\mathrm{d}y\,\mathrm{d}s,\omega^{i}\cdot \mathcal{N}^{n} \Big)_{1,\Sigma_{k}} \\
			&+\left(\int_{0}^{t}\int_{0}^{s}\mathcal{R}_{z}(\p_{1}\zeta_{0},\p_{1}\eta^{k})\p_{1}(D_{t}(R^{n}D_{t}v_{l}^{k})\cdot \mathcal{N}^{n}))(y,X(y,X(-s,x))) \,\mathrm{d}y\,\mathrm{d}s,\partial_{1}(\omega^{j}\cdot \mathcal{N}^{n})\right)_{L^{2}}+((R^{n}D_{t}v_{l},\omega^{j}))\\
			&+(\p_{t}(R^{n}v_{l}),\omega^{j})+\left(\int_{0}^{t}\p_{1}(\mathcal{R}_{zz}(\p_{1}\zeta_{0},\p_{1}\eta^{k})\p_{1}\p_{t}\xi_{l}^{k}\p_{1}\p_{t}\eta^{k})\mathrm{d}s,(\omega^{j}\cdot \mathcal{N}^{n}) \right)_{L^{2}}\\
			&\quad+\left[\int_{0}^{t}\mathcal{R}_{zz}(\p_{1}\zeta_{0},\p_{1}\eta^{k})\p_{1}\p_{t}\xi_{l}^{k}\p_{1}\p_{t}\eta^{k}\mathrm{d}s,\omega^{j}\cdot \mathcal{N}^{n}\right]_{\ell}.
		\end{aligned}
		\] 
		Since $\{\omega^j(t)\}_{j=1}^\infty$ is a basis of $\mathcal{W}_\sigma(t)$, the $m\times m$ matrix with $i,j$ entry $(\omega^i, \omega^j)_{\mathcal{H}^0}$ is invertible. Thus the initial data $d^m_i(0)$ are determined uniquely via $d^m_i(0)(\omega^i(0), \omega^j(0))_{\mathcal{W}}=(\mathcal{P}_0^mu_0, \omega^j(0))_{\mathcal{W}}$.
		Then we view \eqref{eq:galerkin_2} as an  integro-differentia system of the form
		\begin{equation}\label{eq:integral}
			\dot{d}^m(t)+\mathfrak{B}(t)d^m(t)+\int_0^t\mathfrak{C}(t,s)d^m(s)\,\mathrm{d}s+\int_{0}^{t}\int_{0}^{s}\mathfrak{J}(t,s,y)d^{m}(y)dy ds=\tilde{\mathfrak{F}}(t).
		\end{equation}
		The $m\times m$ matrix $\mathfrak{B}$ belongs to $C^1([0,T])$  and $\mathfrak{C}(t,s)$ belongs to $C^1(D)$, where $D=\{(t,s)|0\le s\le t\le T\}$, and $\mathfrak{I}$ belongs to $L^{2}(Q)$, where $Q=\{(t,s,y)|0\leq y\leq s\leq t\}$ and the forcing term $\tilde{\mathfrak{F}}\in H^1([0,T])$. By the linear theory of Volterra integro-differential equations developed by S. Grossman and R. Miller \cite{GM70}, or \cite{Wa}, directly using Picard iteration, \eqref{eq:integral} has a unique solution $d^m\in C^2([0,T])$ achieving the initial data $d^m(0)$.
		
		\paragraph{\underline{Step 3 -- Energy Estimate for $D_tv_{d}^m$}}
		
		In this step, we fix $n,k$ and given functions $v_{l},\xi_{l}$, and construct a uniform bound for $(v_{d}^{m},\xi_{d}^{m})$ independent of $m$. 
		
		Because the system’s kinematic boundary condition is quasilinear, we must first derive the energy estimate for $(D_{t}v_{d}^{m},\partial_{t}\xi_{d}^{m})$. 
		
		Suppose that $w=b_j^m\omega^j\in\mathcal{W}^m(t)$ for $b_j^m\in C^2([0, T])$. From the definition of $\mathcal{W}_{\sigma}(t)$, $\partial_{t}w$ may not be a function in $\mathcal{W}_{\sigma}^{m}(t)$ since the basis $w^{j}$ depends on time $t$. Therefore, we cannot establish the second-order system by simply taking a time derivative of the first-order equation. 
		
		Fortunately, $\p_tw(t)-R^{n}(t)w(t)\in\mathcal{W}_{\sigma}^m(t)$ provided that $w(t)\in\mathcal{W}_{\sigma}^{m}(t)$. From the definition of $\mathcal{W}_{\sigma}(t)$, we now use this $w$ in \eqref{eq:galerkin}, differentiate the resulting equation with respect to time, and then subtract this from the equation \eqref{eq:galerkin} with the test function $ D_tw=\partial_{t}w-R^{n}w$ where $R^{n}$ is defined in Proposition \ref{prop:solid_boundary}. This eliminates the $\p_tw$  terms and leaves us with the following equation
		\begin{equation}\label{eq:pa_tu_m_1}
			\begin{aligned}
				&\left(\p_tD_{t}v_{d}^m,w\right)_{\mathcal{H}^0}+(\p_{t}(R^{n}v_{d}^{m}),w)_{\mathcal{H}^{0}}+(D_tv_{d}^m+R^{n}v_{d}^{m},R^{n}w)_{\mathcal{H}^0}+((D_{t}v_{d}^m+R^{n}v_{d}^{m},w))+((v_{d}^{m},R^{n}w))\\
				&+(\partial_{t}\xi_{d}^{m},w\cdot\mathcal{N}^{n})_{1,\Sigma_{k}}+(\xi_{d}^m, R^{n}w\cdot\mathcal{N}^{n})_{1,\Sigma_{k}}+(\xi_{d}^m,w\cdot\p_t\mathcal{N}^{n})_{1,\Sigma_{k}}+[\p_tv_{d}^{m}\cdot \mathcal{N}^{n},w\cdot\mathcal{N}^{n}]_\ell+[v_{d}^{m}\cdot \partial_{t}\mathcal{N}^{n},w\cdot \mathcal{N}^{n}]_{\ell}\\
				&+[v_{d}^{m}\cdot \mathcal{N}^{n},R^{n} w\cdot\mathcal{N}^{n}]_\ell+[v_{d}^{m},w\cdot\p_t\mathcal{N}^{n}]_\ell+(\p_tv_{d}^m, w\p_tJ^{n})_{L^2(\Om)}+\beta(D_{t}v_{d}^m\cdot\tau,(w\cdot\tau)\p_tJ^{n})_{L^2(\Sigma_s)}\\
				&+\Big(\int_{0}^{t}\mathcal{R}_{z}(\partial_{1}\zeta_{0},\partial_{1}\eta^{k})\partial_{1}\partial_{t}\xi_{d},\partial_{1}(w\cdot \partial_{t}\mathcal{N}^{n}) \Big)_{L^{2}}+\Big(\int_{0}^{t}\mathcal{R}_{z}(\partial_{1}\zeta_{0},\partial_{1}\eta^{k})\partial_{1}\partial_{t}\xi_{d},\partial_{1}(R^{n}w\cdot \mathcal{N}^{n}) \Big)_{L^{2}}\\
				&+\Big(\p_{1}(\mathcal{R}_{zz}(\p_{1}\zeta_{0},\p_{1}\eta^{k})\p_{1}\p_{t}\xi_{l}^{k}\p_{1}\p_{t}\eta^{k}),(w\cdot \mathcal{N}^{n}) \Big)_{L^{2}}+[\mathcal{R}_{zz}(\p_{1}\zeta_{0},\p_{1}\eta^{k})\p_{1}\p_{t}\xi_{l}^{k}\p_{1}\p_{t}\eta^{k},w\cdot \mathcal{N}^{n}]_{\ell}+(\p_{t}^{2}(R^{n}v_{l}),w)_{\mathcal{H}^{0}}\\
				&+\Big(\p_t(R^{n}v_{l}),R^{n}w)_{\mathcal{H}^{0}}+((R^{n}v_{l}),w\p_{t}J^{n}(J^{n})^{-1} \Big)_{\mathcal{H}^{0}}\\
				&=\p_t\mathcal{F}(w)-\mathcal{F}(\p_tw)+\mathcal{F}(R^{n}w)+((\p_{t}(R^{n}v_{l}),w))+((R^{n}D_{t}v_{l},R^{n}w))\\
				&\quad-\int_\Om\frac{\mu}{2}\mathbb{D}_{\p_t\mathcal{A}^{n}}v_{d}^m:\mathbb{D}_{\mathcal{A}^{n}}wJ^{n}
				-\int_\Om\frac{\mu}{2} J^{n}\Big(\mathbb{D}_{\mathcal{A}^{n}}v_{d}^m:\mathbb{D}_{\p_t\mathcal{A}^{n}}w+\p_tJ^{n}K^{n}\mathbb{D}_{\mathcal{A}^{n}}v_{d}^m:\mathbb{D}_{\mathcal{A}^{n}}w \Big)\\
				&\quad-\int_\Om\frac{\mu}{2}\mathbb{D}_{\p_t\mathcal{A}^{n}}(R^{n}v_{l}):\mathbb{D}_{\mathcal{A}^{n}}wJ^{n}
				-\int_\Om\frac{\mu}{2}J^{n} \Big(\mathbb{D}_{\mathcal{A}^{n}}(R^{n}v_{l}):\mathbb{D}_{\p_t\mathcal{A}^{n}}w+\p_tJ^{n}K^{n}\mathbb{D}_{\mathcal{A}^{n}}(R^{n}v_{l}):\mathbb{D}_{\mathcal{A}^{n}}w \Big),
			\end{aligned}
		\end{equation}
		where for brevity we have written
		\begin{equation}\label{eq:cal_f}
			\begin{aligned}
				\mathcal{F}(w)&=\int_{\Om}F^1\cdot wJ^{n}+\int_{-\ell}^{\ell}F^4\cdot w-\int_{\Sigma_s}F^5(w\cdot\tau)J^{n}-[F^{7},w\cdot \mathcal{N}^{n}]_{\ell}.
			\end{aligned}
		\end{equation}
		
		We first consider the terms on the boundary $\Sigma$ and the contact points. According to \cite[Lemma A.1]{ZhT17}, we know that $(R^{n})^T\mathcal{N}^{n}=-\p_t\mathcal{N}^{n}$. Then it holds that
		\[
		(\xi_{d}^m, R^{n}w\cdot\mathcal{N}^{n})_{1,\Sigma_{k}}+(\xi_{d}^m,w\cdot\p_t\mathcal{N}^{n})_{1,\Sigma_{k}}+[v_{d}^{m}\cdot \mathcal{N}^{n},R^{n} w\cdot\mathcal{N}^{n}]_\ell+[v_{d}^{m}\cdot \mathcal{N}^{n},w\cdot\p_t\mathcal{N}^{n}]_\ell=0,
		\]
		\[
		\left(\int_{0}^{t}\mathcal{R}_{z}(\partial_{1}\zeta_{0},\partial_{1}\eta^{k})\partial_{1}\partial_{t}\xi_{d},\partial_{1}(w\cdot \partial_{t}\mathcal{N}^{n}) \right)_{L^{2}}+ \left(\int_{0}^{t}\mathcal{R}_{z}(\partial_{1}\zeta_{0},\partial_{1}\eta^{k})\partial_{1}\partial_{t}\xi_{d},\partial_{1}(R^{n}w\cdot \mathcal{N}^{n}) \right)_{L^{2}}=0.
		\]
		Let the test function be $w=D_{t}v_{d}^{m}$. We have the following relation.
		\[
		[\partial_{t}v_{d}^{m}\cdot \mathcal{N}^{n},D_{t}v_{d}^{m}\cdot \mathcal{N}^{n}]_{\ell}=[D_{t}v_{d}^{m}\cdot \mathcal{N}^{n}]^{2}_{\ell}+[D_{t}v_{d}^{m}\cdot \mathcal{N}^{n},R^{n}v_{d}^{m}\cdot \mathcal{N}^{n}]_{\ell}.
		\]
		The second term in the equation above vanishes since $R^{n}v_{d}^{m}\cdot \mathcal{N}^{n}(\pm\ell)=-v_{d}^{m}\cdot \partial_{t}\mathcal{N}^{n}(\pm \ell)=0$. Moreover, by the same reason, we have
		\begin{equation}\label{eq:bc_dtum}
			[\partial_{t}v_{d}^{m}\cdot \partial_{t}\mathcal{N}^{n},w \cdot \mathcal{N}^{n}]_{\ell}=[D_{t}v_{d}^{m}\cdot \partial_{t}\mathcal{N}^{n},w \cdot \mathcal{N}^{n}]_{\ell}+[R^{n}v_{d}^{m}\cdot \partial_{t}\mathcal{N}^{n},w \cdot \mathcal{N}^{n}]=0,
		\end{equation}
		where we used the fact that $v_{d}^{m}$ and $D_{t}v_{d}^{m}$ are functions in $\mathcal{W}_{\sigma}(t)$, which implies that $(D_{t}v_{d}^{m}\cdot \partial_{t}\mathcal{N}^{n})(\pm \ell)=0$, and $R^{n}v_{d}^{m}\cdot \partial_{t}\mathcal{N}^{n}(\pm\ell)=v_{d}^{m}\cdot \partial_{t}\mathcal{N}^{n}(\pm \ell)=0$.
		
		For the term in $(\cdot)_{1,\Sigma}$, using the kinematic boundary condition
		\[
		\partial_{t}^{2}\xi_{d}^{m}=D_{t}(v_{d}^{m})\cdot \mathcal{N}^{n}+D_{t}(R^{n}v_{l}^{k})\cdot \mathcal{N}^{n} - \partial_{t}u_{1}^{k} \partial_{t}\partial_{1}\xi_{l}^{k}+u_{1}^{k}\partial_{1}\partial_{t}\xi_{d}^{m},
		\]
		we obtain
		\[
		\begin{aligned}
			(\partial_{t}\xi_{d}^{m}, D_{t}v_{d}^{m}\cdot\mathcal{N}^{n})_{1,\Sigma_{k}}=(\partial_{t}\xi_{d}^{m},\partial_{t}^{2}\xi_{d}^{m})_{1,\Sigma_{k}} +(\partial_{t}\xi_{d}^{m},\partial_{t}u_{1}^{k} \partial_{t}\partial_{1}\xi_{l}^{k})_{1,\Sigma_{k}}-(\partial_{t}\xi_{d}^{m}, u_{1}^{n}\partial_{1}\partial_{t}\xi_{d}^{m})_{1,\Sigma_{k}}\\
			-(\p_{t}\xi_{d}^{m},D_{t}(R^{n}v_{l}^{k})\cdot \mathcal{N}^{n})_{1,\Sigma_{k}}
			=\frac{d}{dt}\frac12\|\partial_{t}\xi_{d}^{m}\|_{1,\Sigma_{k}}^2-\int_{-\ell}^{\ell}\mathcal{R}_{zz}(\partial_{1}\zeta_{0},\partial_{1}\eta^{k})\partial_{1}\partial_{t}\eta^{k}(\partial_{t}\partial_{1}\xi_{d}^{m})^{2}\\
			+(\partial_{t}\xi_{d}^{m},\partial_{t}u_{1}^{k}\partial_{t}\partial_{1}\xi_{l}^{k})_{1,\Sigma_{k}}-(\partial_{t}\xi_{d}^{m}, u_{1}^{k}\partial_{1}\partial_{t}\xi_{d}^{m})_{1,\Sigma_{k}}
			-(\p_{t}\xi_{d}^{m},D_{t}(R^{n}v_{l}^{k})\cdot \mathcal{N}^{n})_{1,\Sigma_{k}},
		\end{aligned}
		\]
		\noindent where we recall the definition of $(\cdot)_{1,\Sigma_{k}}$ inner product as follows
		\[
		(\partial_{t}\xi_{d},\partial_{t}^{2}\xi_{d})_{1,\Sigma_{k}}=\frac{1}{2}\frac{d}{dt}\vert \vert \partial_{t}\xi_{d}\vert \vert_{1,\Sigma_{k}}^{2}-\frac{1}{2}\sigma\int_{-\ell}^{\ell} \mathcal{R}_{zz}(\partial_{1}\zeta_{0},\partial_{1}\eta^{k}) \partial_{1}\partial_{t}\eta^{k} (\partial_{1}\partial_{t}\xi_{d})^{2}.
		\]
		
		Applying all the computations above to equation \eqref{eq:pa_tu_m_1}, setting the test function to $w=D_{t}v_{d}^{m}$, and using integration by parts, we deduce
		\[
		(\partial_{t}D_{t}v_{d}^{m},D_{t}v_{d}^{m})_{\mathcal{H}^{0}}=\frac{1}{2}\frac{d}{dt}\vert \vert D_{t}v_{d}^{m}\vert \vert^{2}_{\mathcal{H}^{0}}-\int_{\Omega}\partial_{t}J^{n}(D_{t}v_{d}^{m})^{2}.
		\]
		\noindent Thus, equation \eqref{eq:pa_tu_m_1} can be rewritten as
		\begin{equation}\label{eq:e_2}
			\begin{aligned}
				&\frac{d}{dt}\left(\frac12\|D_{t}v_{d}^m\|_{\mathcal{H}^0}^2+\frac12\|\partial_{t}\xi_{d}^{m}\|_{1,\Sigma_{k}}^2\right)+\|D_tv_{d}^m\|_{\mathcal{H}^1}^2+[D_{t}v_{d}^{m}\cdot \mathcal{N}^{n}]_\ell^2=I+II+III+IV,
			\end{aligned}
		\end{equation}
		where
		\[
		\begin{aligned}
			I=&-\beta(\partial_{t}v_{d}^m\cdot\tau,(\partial_{t}v_{d}^{m}\cdot\tau)\p_tJ^{n})_{L^2(\Sigma_s)} -(\partial_{t}\xi_{d}^{m},\partial_{t}u_{1}^{k} \partial_{t}\partial_{1}\xi_{l}^{k})_{1,\Sigma_{k}}-(\partial_{t}\xi_{d}^{m}, u_{1}^{k}\partial_{1}\partial_{t}\xi_{d}^{m})_{1,\Sigma_{k}}\\
			&
			+(\p_{t}\xi_{d}^{m},D_{t}(R^{n}v_{l}^{k})\cdot \mathcal{N}^{n})_{1,\Sigma_{k}}+(\p_{1}(\mathcal{R}_{zz}(\p_{1}\zeta_{0},\p_{1}\eta^{k})\p_{1}\p_{t}\xi_{l}^{k}\p_{1}\p_{t}\eta^{k}),(D_{t}v_{d}^{m}\cdot \mathcal{N}^{n}))_{L^{2}}\\
			&+[\mathcal{R}_{zz}(\p_{1}\zeta_{0},\p_{1}\eta^{k})\p_{1}\p_{t}\xi_{l}^{k}\p_{1}\p_{t}\eta^{k},D_{t}v_{d}^{m}\cdot \mathcal{N}^{n}]_{\ell},
		\end{aligned}
		\]
		\[
		\begin{aligned}
			II&=((v_{d}^{m},R^{n}D_{t}v_{d}^{m}))-((R^{n}v_{d}^{m},D_{t}v_{d}^{m}))-\int_\Om\frac{\mu}{2}(\mathbb{D}_{\p_t\mathcal{A}^{n}}v_{d}^m:\mathbb{D}_{\mathcal{A}^{n}}(D_{t}v_{d}^{m})J^{n}\\
			&\quad-\int_\Om\frac{\mu}{2}\Big(\mathbb{D}_{\mathcal{A}^{n}}v_{d}^m:\mathbb{D}_{\p_t\mathcal{A}^{n}}(D_{t}v_{d}^{m})+\p_tJ^{n}K^{n}\mathbb{D}_{\mathcal{A}^{n}}v_{d}^m:\mathbb{D}_{\mathcal{A}^{n}}(D_{t}v_{d}^{m}) \Big)J^{n}\\
			&\quad+((\p_{t}(R^{n}v_{l}),D_{t}v_{d}^{m}))+((R^{n}v_{l},R^{n}D_{t}v_{d}^{m}))-\int_\Om\frac{\mu}{2}\mathbb{D}_{\p_t\mathcal{A}^{n}}(R^{n}v_{l}):\mathbb{D}_{\mathcal{A}^{n}}(D_{t}v_{d}^{m})J^{n}\\
			&
			\quad-\int_\Om\frac{\mu}{2}J^{n}\Big(\mathbb{D}_{\mathcal{A}^{n}}(R^{n}v_{l}):\mathbb{D}_{\p_t\mathcal{A}^{n}}(D_{t}v_{d}^{m})+\p_tJ^{n}K^{n}\mathbb{D}_{\mathcal{A}^{n}}(R^{n}v_{l}):\mathbb{D}_{\mathcal{A}^{n}}(D_{t}v_{d}^{m}) \Big),
		\end{aligned}
		\]
		
		\[
		\begin{aligned}
			III%&=\p_t\mathcal{F}(D_{t}v_{d}^{m})-\mathcal{F}(D_{t}^{2}v_{d}^{m})\\
			&=\int_{\Om}\left[\p_tF^1\cdot(D_{t}v_{d}^{m})+\p_tJ^{n}K^{n}F^1\cdot(D_{t}v_{d}^{m})+F^1\cdot R^{n}(D_{t}v_{d}^{m})\right]J^{n}-\int_{\Sigma}\partial_{t}F^{4}\cdot D_{t}v_{d}^{m}-F^{4}\cdot R^{n}D_{t}v_{d}^{m}\\
			&\quad-\int_{\Sigma_s}\left[\p_tF^5 (D_{t}v_{d}^{m})+\p_tJ^{n}K^{n}F^5 (D_{t}v_{d}^{m})+F^5 R^{n}(D_{t}v_{d}^{m})\right]\cdot\tau J^{n}-[\p_tF^7,D_tv_{d}^m\cdot\mathcal{N}^{n}]_\ell,
		\end{aligned}
		\]
		
		\[
		\begin{aligned}
			IV&=(\partial_{t}v_{d}^{m},R^{n}D_{t}v_{d}^{m})_{\mathcal{H}^{0}}-(\p_{t}(R^{n}v_{d}^{m}),D_{t}v_{d}^{m})_{\mathcal{H}^{0}(\Omega)}+((R^{n}v_{d}^{m},D_{t}v_{d}^{m}))+(\p_{t}^{2}(R^{n}v_{l}),D_{t}v_{d}^{m})_{\mathcal{H}^{0}}\\
			&\quad+(\p_{t}(R^{n}v_{l}),R^{n}D_{t}v_{d}^{m})_{\mathcal{H}^{0}}+(\p_{t}(R^{n}v_{l}),D_{t}v_{d}^{m}\p_{t}J^{n}(J^{n})^{-1})_{\mathcal{H}^{0}}.
		\end{aligned}
		\]
		
		We now estimate each term on the right-hand side of the second line in \eqref{eq:e_2}. We first estimate each term in $I$. From the fact that $\vert \mathcal{R}_{zz}(\partial_{1}\zeta_{0},\partial_{1}\eta)\vert\leq C$ for some constant $C$, we obtain the following inequality
		\begin{align}{\label{est:dt_u}}
			\int_{-\ell}^{\ell}\beta(\partial_{t}v_{d}^{m}\cdot \tau,(\partial_{t}v_{d}^{m} \cdot \tau)\partial_{t}J^{n})_{L^{2}(\Sigma_{s})}\lesssim \vert \vert \partial_{t}J^{n}\vert \vert_{L^{\infty}}\vert \vert \partial_{t}v_{d}^{m}\vert\vert^{2}_{L^{2}(\Sigma_{s})}\lesssim \vert \vert \partial_{t}\eta^{n}\vert \vert_{H^{\frac{3}{2}+\frac{\varepsilon_{-}-\alpha}{2}}}\vert \vert \partial_{t}v_{d}^{m}\vert \vert_{H^{1}}^{2}.
		\end{align}
		For the terms in $(\cdot)_{1,\Sigma}$ norm, we have:
		\begin{align}
			(\partial_{t}\xi_{d}^{m},\partial_{t}u_{1}^{k}\partial_{t}\partial_{1}\xi_{l}^{k})_{1,\Sigma_{k}}\lesssim \vert \vert \partial_{t}\xi_{d}^{m}\vert \vert_{H^{1}}\vert \vert \partial_{t}u_{1}^{k}\vert \vert_{W^{1,+\infty}(\Sigma)}\vert \vert \partial_{t}\xi_{l}^{k}\vert \vert_{H^{2}}\lesssim \vert \vert \partial_{t}\xi_{d}^{m}\vert \vert_{H^{1}}\vert \vert \partial_{t}u^{k}\vert \vert_{H^{\frac{5}{2}}}\vert \vert \partial_{t}\xi_{l}^{k}\vert \vert_{H^{2}}\label{eq:n1},
		\end{align}
		\noindent and
		\begin{align}
			&(\partial_{t}\xi_{d}^{m},u_{1}^{k}\partial_{1}\partial_{t}\xi_{d}^{m})_{1,\Sigma_{k}}\notag\\
			&=\int_{-\ell}^{\ell}g\partial_{t}\xi_{d}^{m}u_{1}^{k}\partial_{1}\partial_{t}\xi_{d}^{m}+\sigma\frac{\p_1\partial_{t}\xi_{d}^{m}u_{1}^{k}\partial_{1}^{2}\partial_{t}\xi_{d}^{m}}{(1+|\p_1\zeta_0|^2)^{3/2}}+\int_{-\ell}^{\ell}\mathcal{R}_{z}(\partial_{1}\zeta_{0},\partial_{1}\eta^{k})u_{1}^{k}\partial_{t}\partial_{1}\xi_{d}^{m}\partial_{1}^{2}\partial_{t}\xi_{d}^{m}\notag\\
			&\quad+\int_{-\ell}^{\ell}g\partial_{t}\xi_{d}^{m}\partial_{1}u_{1}^{k}\partial_{t}\xi_{d}^{m}+\sigma\frac{\p_1\partial_{t}\xi_{d}^{m}\partial_{1}u_{1}^{k}\partial_{1}\partial_{t}\xi_{d}^{m}}{(1+|\p_1\zeta_0|^2)^{3/2}}+\int_{-\ell}^{\ell}\mathcal{R}_{z}(\partial_{1}\zeta_{0},\partial_{1}\eta^{k})\partial_{1}u_{1}^{k}\partial_{t}\partial_{1}\xi_{d}^{m}\partial_{1}\partial_{t}\xi_{d}^{m}\notag\\
			&=-\frac{1}{2}\int_{-\ell}^{\ell}\mathcal{R}_{zz}(\partial_{1}\zeta_{0},\partial_{1}\eta^{k})u_{1}^{k}\partial_{t}\partial_{1}\xi_{d}^{m}\partial_{1}\partial_{t}\xi_{d}^{m}(\partial_{1}^{2}\zeta_{0}+\partial_{1}^{2}\eta^{k})\notag\\
			&\quad+\frac{1}{2}\int_{-\ell}^{\ell}g\partial_{t}\xi_{d}^{m}\partial_{1}u_{1}^{k}\partial_{t}\xi_{d}^{m}+\sigma\frac{\p_1\partial_{t}\xi_{d}^{m}\partial_{1}u_{1}^{k}\partial_{1}\partial_{t}\xi_{d}^{m}}{(1+|\p_1\zeta_0|^2)^{3/2}}+\frac{1}{2}\int_{-\ell}^{\ell}\mathcal{R}_{z}(\partial_{1}\zeta_{0},\partial_{1}\eta^{k})\partial_{1}u_{1}^{k}\partial_{t}\partial_{1}\xi_{d}^{m}\partial_{1}\partial_{t}\xi_{d}^{m}\notag\\
			&\lesssim \vert \vert \partial_{t}\xi_{d}^{m}\vert \vert^{2}_{H^{1}}(\vert \vert \eta^{k}\vert \vert_{H^{3}}+\vert \vert u^{k}\vert\vert_{H^{\frac{5}{2}}})\label{eq:n2},
		\end{align}
		\noindent where we used integration by part, and the boundary terms vanish due to the fact $u_{1}(\pm \ell)=0$. 
		
		Finally, for the term including $v_{l}^{k},\xi_{l}^{k}$, by H\"older's inequality, we have the following estimate
		\begin{align}{\label{eq:n0}}
			\begin{aligned}
				(\p_{t}\xi_{d}^{m},D_{t}(R^{n}v_{l}^{k})\cdot \mathcal{N}^{n})_{1,\Sigma_{k}}
				& \lesssim \|\p_{t}\xi_{d}^{m}\|_{H^{1}}\bigg(\|D_{t}v_{l}^{k}\|_{H^{1}(\Sigma)}\|\p_{t}\eta^{n}\|_{W^{1,+\infty}}+\|v_{l}^{k}\|_{H^{1}(\Sigma)}\|\p_{t}^{2}\eta^{n}\|_{W^{1,+\infty}}\bigg)\\
				&\lesssim \|\p_{t}\xi_{d}^{m}\|_{H^{1}}\bigg(\|D_{t}v_{l}^{k}\|_{H^{\frac{3}{2}}}\|\p_{t}\eta^{n}\|_{W^{3-\frac{1}{q_{-}},q_{-}}}+\|v_{l}^{k}\|_{H^{\frac{3}{2}}}\|\p_{t}^{2}\eta^{n}\|_{H^{\frac{5}{2}}}\bigg),
			\end{aligned}
		\end{align}
		\begin{align}
			\begin{aligned}
				&(\p_{1}(\mathcal{R}_{zz}(\p_{1}\zeta_{0},\p_{1}\eta^{k})\p_{1}\p_{t}\xi_{l}^{k}\p_{1}\p_{t}\eta^{k}),(D_{t}v_{d}^{m}\cdot \mathcal{N}^{n}))_{L^{2}}+[\mathcal{R}_{zz}(\p_{1}\zeta_{0},\p_{1}\eta^{k})\p_{1}\p_{t}\xi_{l}^{k}\p_{1}\p_{t}\eta^{k},D_{t}v_{d}^{m}\cdot \mathcal{N}^{n}]_{\ell}\\
				\lesssim& \|\p_{t}\xi_{l}^{k}\|_{W^{3-\frac{1}{q_{-}},q_{-}}}\|\p_{t}\eta^{k}\|_{H^{\frac{3}{2}+\frac{\varepsilon_{-}-\alpha}{2}}}\|D_{t}v_{d}^{m}\|_{H^{1}}+\|\p_{t}\eta^{k}\|_{W^{3-\frac{1}{q_{-}},q_{-}}}\|\p_{t}\xi_{l}^{k}\|_{H^{\frac{3}{2}+\frac{\varepsilon_{-}-\alpha}{2}}}\|D_{t}v_{d}^{m}\|_{H^{1}}\\
				&+\|\p_{t}\eta^{k}\|_{H^{\frac{3}{2}+\frac{\varepsilon_{-}-\alpha}{2}}}\|\p_{t}\xi_{l}^{k}\|_{H^{\frac{3}{2}+\frac{\varepsilon_{-}-\alpha}{2}}}\|[D_{t}v_{d}^{m}\cdot\mathcal{N}^{n}]_{\ell}\|_{H^{1}}.
			\end{aligned}
		\end{align}
		
		In order to estimate terms in $II$, we use the relation 
		$R^{n}\sim \nabla\p_t\bar{\eta}^{n}+\nabla\p_t\bar{\eta}^{n}\nabla\bar{\eta}^{n}\quad \text{and}\ \|J^{n}\|_{L^\infty}\lesssim 1$.
		Applying Cauchy's inequality, Sobolev's inequality and the trace theory, we begin by estimating the term inside the $\mathcal{W}$ inner product. We have
		\begin{align}
			\begin{aligned}
				((v_{d}^{m},R^{n}D_{t}v_{d}^{m}))\lesssim& \int_{\Omega}\vert \nabla v_{d}^{m}\vert\vert \nabla R^{n}\vert\vert D_{t}v_{d}^{m}\vert+\int_{\Omega}\vert \nabla v_{d}^{m}\vert\vert R^{n}\vert\vert \nabla D_{t}v_{d}^{m}\vert\\
				\lesssim& \vert \vert v_{d}^{m}\vert\vert_{H^{1}}\vert \vert \partial_{t}\bar{\eta}^{n}\vert \vert_{H^{2+\frac{\varepsilon_{-}-\alpha}{2}}}\vert \vert D_{t}v_{d}^{m}\vert \vert_{L^{\frac{4}{\varepsilon_{-}-\alpha}}}+\vert \vert \partial_{t}\bar{\eta}^{n}\vert \vert_{W^{1,+\infty}}\vert \vert v_{d}^{m}\vert \vert_{H^{1}}\vert \vert D_{t}v_{d}^{m}\vert \vert_{H^{1}}\\
				\lesssim& \vert \vert \partial_{t}{\eta}^{n}\vert \vert_{H^{\frac{3}{2}+\frac{\varepsilon_{-}-\alpha}{2}}}\vert \vert v_{d}^{m }\vert \vert_{H^{1}}\vert \vert D_{t}v_{d}^{m}\vert \vert_{H^{1}},
			\end{aligned}
		\end{align}
		and similarly
		\begin{align}
			\begin{aligned}
				((R^{n}v_{d}^{m},D_{t}v_{d}^{m}))\lesssim \vert \vert \partial_{t}{\eta}^{n}\vert \vert_{H^{\frac{3}{2}+\frac{\varepsilon_{-}-\alpha}{2}}}\vert \vert v_{d}^{m }\vert \vert_{H^{1}}\vert \vert D_{t}v_{d}^{m}\vert \vert_{H^{1}}.
			\end{aligned}
		\end{align}
		Then for other terms involved in II, we have the following estimate
		\begin{equation}
			\begin{aligned}
				\int_{\Omega} \frac{\mu}{2}(\mathbb{D}_{\partial_{t}\mathcal{A}^{n}}v_{d}^{m}:\mathbb{D}_{\mathcal{A}^{n}}(D_{t}v_{d}^{m}))J^{n}&\lesssim \vert \vert D_{t}v_{d}^{m}\vert \vert_{H^{1}}\vert\vert v_{d}^{m}\vert \vert_{H^{1}}\vert \vert \partial_{t}\eta^{n}\vert \vert_{H^{\frac{3}{2}+\frac{\varepsilon_{-}-\alpha}{2}}}\\
				&\lesssim \vert \vert D_{t}v_{d}^{m}\vert \vert^{2}_{H^{1}}\vert \vert \partial_{t}\eta^{n}\vert \vert^{2}_{H^{\frac{3}{2}+\frac{\varepsilon_{-}-\alpha}{2}}}+\vert \vert v_{d}^{m}\vert \vert_{H^{1}}^{2}.
			\end{aligned}
		\end{equation}
		Similarly, we have
		\begin{align}
			\int_{\Omega} \frac{\mu}{2}\mathbb{D}_{\mathcal{A}^{n}}v_{d}^{m}:\mathbb{D}_{\partial_{t}\mathcal{A}^{n}}(D_{t}v_{d}^{m})\lesssim\vert \vert D_{t}v_{d}^{m}\vert \vert^{2}_{H^{1}}\vert \vert \partial_{t}\eta^{n}\vert \vert^{2}_{H^{\frac{3}{2}+\frac{\varepsilon_{-}-\alpha}{2}}}+\vert \vert v_{d}^{m}\vert \vert_{H^{1}}^{2},
		\end{align}
		\begin{align}
			\int_{\Omega} \partial_{t}J^{n}\mathbb{D}_{\mathcal{A}^{n}}v_{d}^{m}:\mathbb{D}_{\mathcal{A}}^{n}(D_{t}v_{d}^{m})\lesssim \vert \vert D_{t}v_{d}^{m}\vert \vert^{2}_{H^{1}}\vert \vert \partial_{t}\eta\vert \vert^{2}_{H^{\frac{3}{2}+\frac{\varepsilon_{-}-\alpha}{2}}}+\vert \vert v_{d}^{m}\vert \vert_{H^{1}}^{2},
		\end{align}
		\begin{align}
			\begin{aligned}
				((\p_{t}(R^{n}v_{l}),D_{t}v_{d}^{m}))\lesssim  \|D_{t}v_{d}^{m}\|_{\mathcal{H}^{1}}(\|\p_{t}^{2}\eta^{n}\|_{W^{3-\frac{1}{q_{-}},q_{-}}}\|v_{l}^{k}\|_{H^{1}}+\|\p_{t}\eta^{n}\|_{W^{3-\frac{1}{q_{-}},q_{-}}}\|D_{t}v_{l}^{k}\|_{H^{1}}),
			\end{aligned}
		\end{align}
		\begin{align}
			\begin{aligned}
				&((R^{n}D_{t}v_{l},R^{n}D_{t}v_{d}^{m}))
				\lesssim \|D_{t}v_{d}^{m}\|_{H^{1}}(\|\p_{t}\eta^{n}\|_{W^{3-\frac{1}{q_{-}},q_{-}}}\|D_{t}v_{l}^{k}\|_{H^{1}})\\
				& 
				-\int_\Om\frac{\mu}{2}(J^{n}\mathbb{D}_{\mathcal{A}^{n}}(R^{n}v_{l}):\mathbb{D}_{\p_t\mathcal{A}^{n}}(D_{t}v_{d}^{m})+\p_tJ^{n}K^{n}\mathbb{D}_{\mathcal{A}^{n}}(R^{n}v_{l}):\mathbb{D}_{\mathcal{A}^{n}}(D_{t}v_{d}^{m}))\\
				&-\int_\Om\frac{\mu}{2}(\mathbb{D}_{\p_t\mathcal{A}^{n}}(R^{n}v_{l}):\mathbb{D}_{\mathcal{A}^{n}}(D_{t}v_{d}^{m})J^{n}\lesssim  \vert \vert D_{t}v_{d}^{m}\vert \vert^{2}_{H^{1}}\vert \vert \partial_{t}\eta\vert \vert^{2}_{H^{\frac{3}{2}+\frac{\varepsilon_{-}-\alpha}{2}}}+\|\p_{t}\eta^{n}\|_{W^{3-\frac{1}{q_{-}},q_{-}}}^{2}\vert \vert v_{l}\vert \vert_{H^{1}}^{2}.
			\end{aligned} 
		\end{align}
		
		To estimate terms involved in $III$, we require more refined bounds, which we establish individually. We use the dual space estimate, the standard Sobolev embedding theorem and H\"older's inequality to bound terms involving one time derivative of forces as follows
		\begin{align}
			\begin{aligned}
				\int_{\Om}\p_tF^1\cdot(D_{t}v_{d}^{m})J^{n}-\int_{-\ell}^{\ell}\p_tF^4\cdot D_{t}v_{d}^{m}-\int_{\Sigma_s}\p_tF^5(D_{t}v_{d}^{m})\cdot\tau J^{n}
				\lesssim \|\p_t(F^1-F^4-F^5)\|_{(\mathcal{H}^1)^{\ast}}\|D_{t}v_{d}^m\|_{\mathcal{H}^1}.
			\end{aligned}
		\end{align}
		For terms involving $F^1$, we estimate them by H\"older's inequality with $\f1{q_-}+\f{\varepsilon_-}2=1$, and Sobolev embedding $H^1\hookrightarrow L^{\varepsilon_-/2}$
		\begin{align}
			\begin{aligned}
				\int_{\Om}\left[\p_tJ^{n}K^{n}F^1\cdot(D_{t}v_{d}^{m})+F^1\cdot R^{n}(D_{t}v_{d}^{m})\right]J^{n}
				&\lesssim (\|\p_tJ^{n}K^{n}\|_{L^\infty}+\|R^{n}\|_{L^\infty})\|F^1\|_{L^{q_-}(\Om)}\|D_tv_{d}^m\|_{L^{\varepsilon_-/2}(\Om)}\\
				&\lesssim\|\p_t\eta^{n}\|_{H^{3/2+(\varepsilon_--\alpha)/2}}\|F^1\|_{L^{q_-}(\Om)}(\|D_tv_{d}^m\|_{\mathcal{H}^1}).
			\end{aligned}
		\end{align}
		Next, we estimate the integral involving $F^4$. By H\"older's inequality and the Sobolev embedding
		$
		W^{1-1/q_-, q_-}(\Sigma)\hookrightarrow L^{1/\varepsilon_-}(\Sigma)\quad \text{and} \ H^{1/2}(\Sigma)\hookrightarrow L^{1/\varepsilon_-}(\Sigma)$,
		and trace theory, we have
		\begin{align}
			\begin{aligned}
				\int_{-\ell}^{\ell}F^4\cdot R^{n}(D_{t}v_{d}^{m})&\lesssim \|F^4\|_{L^{1/(1-\varepsilon_-)}(\Sigma)}\|R^{n}\|_{L^\infty(\Sigma)}(\|D_tv_{d}^m\|_{L^{1/\varepsilon_-}(\Sigma)})\\
				&\lesssim\|\p_t\eta^{n}\|_{H^{3/2+(\varepsilon_--\alpha)/2}}\|F^4\|_{W^{1-1/q_-, q_-}}(\|D_tv_{d}^m\|_{\mathcal{H}^1}).
			\end{aligned}
		\end{align}
		The integral involving $F^5$ is bounded similarly as $F^4$ via
		\begin{equation}\label{est:iii}
			\begin{aligned}
				&\int_{\Sigma_s}\left[\p_tJ^{n}K^{n}F^5(D_{t}v_{d}^{m})+F^5 R^{n}(D_{t}v_{d}^{m})\right]\cdot\tau J^{n}\\
				&\lesssim (\|\p_tJ^{n}K^{n}\|_{L^\infty(\Sigma_s)}+\|R^{n}\|_{L^\infty(\Sigma_s)})\|F^5\|_{L^{1/(1-\varepsilon_-)}(\Sigma_s)}(\|D_tv_{d}^m\|_{L^{1/\varepsilon_-}(\Sigma_s)})\\
				&\lesssim\|\p_t\eta^{n}\|_{H^{3/2+(\varepsilon_--\alpha)/2}}\|F^5\|_{W^{1-1/q_-, q_-}}(\|D_tv_{d}^m\|_{\mathcal{H}^1}).
			\end{aligned}
		\end{equation}
		Finally, the estimate for the term involving $F^{7}$ is straightforward to estimate as follows 
		\[ [\p_tF^7,D_tv_{d}^m\cdot\mathcal{N}^{n}]_\ell\lesssim [\p_tF^7]_\ell[D_tv_{d}^m\cdot\mathcal{N}^{n}]_\ell.
		\]

		For the terms involved in $IV$, they can be estimated as follows
		\begin{align}
			(\partial_{t}v_{d}^{m},R^{n}D_{t}v_{d}^{m})_{L^{2}(\Omega)}\lesssim \vert \vert \partial_{t}{\eta}^{n}\vert \vert_{H^{\frac{3}{2}+\frac{\varepsilon_{-}-\alpha}{2}}}\vert \vert D_{t}v_{d}^{m}\vert \vert^{2}_{L^{2}}+\vert \vert \partial_{t}\eta^{n}\vert \vert_{H^{\frac{3}{2}+\frac{\varepsilon_{-}-\alpha}{2}}}^{2}\vert \vert v_{d}^{m}\vert \vert_{L^{2}}\vert \vert D_{t}v_{d}^{m}\vert \vert_{L^{2}},
		\end{align}
		and
		\begin{align}
			(\partial_{t}(R^{n}v_{d}^{m}),D_{t}v_{d}^{m})_{L^{2}(\Omega)}=&\vert \vert D_{t}v_{d}^{m}\vert \vert_{L^{2}}\vert \vert \partial_{t}v_{d}^{m}\vert \vert_{L^{2}}\vert \vert \partial_{t}\eta^{n}\vert \vert_{H^{\frac{3}{2}+\frac{\varepsilon_{-}-\alpha}{2}}}+\vert \vert \partial_{t}^{2}\eta^{n}\vert \vert_{H^{\frac{3}{2}+\frac{\varepsilon_{-}-\alpha}{2}}}\vert \vert v_{d}^{m}\vert \vert_{L^{2}(\Omega)}\vert \vert D_{t}v_{d}^{m}\vert \vert_{L^{2}(\Omega)}\notag\\
			\lesssim& \vert \vert D_{t}v_{d}^{m}\vert \vert^{2}_{L^{2}}\vert \vert \partial_{t}\eta^{n}\vert \vert_{H^{\frac{3}{2}+\frac{\varepsilon_{-}-\alpha}{2}}}+\vert \vert \partial_{t}^{2}\eta^{n}\vert \vert_{H^{\frac{3}{2}+\frac{\varepsilon_{-}-\alpha}{2}}}\vert \vert v_{d}^{m}\vert \vert_{L^{2}(\Omega)}\vert \vert D_{t}v_{d}^{m}\vert \vert_{L^{2}(\Omega)}\notag\\
			&+\vert \vert D_{t}v_{d}^{m}\vert \vert_{L^{2}}\vert \vert v_{d}^{m}\vert \vert_{L^{2}}\vert \vert \partial_{t}\eta^{n}\vert \vert_{H^{\frac{3}{2}+\frac{\varepsilon_{-}-\alpha}{2}}}^{2},
		\end{align}
		\begin{align}
			((R^{n}v_{d}^{m},D_{t}v_{d}^{m}))\lesssim \|D_{t}v_{d}^{m}\|_{H^{1}(\Omega)}\|R^{n}\|_{L^{\infty}}\|v_{d}^{m}\|_{H^{1}(\Omega)}\lesssim \|D_{t}v_{d}^{m}\|_{H^{1}(\Omega)}\|\p_{t}\eta^{n}\|_{W^{3-\frac{1}{q_{-}},q_{-}}}\|v_{d}^{m}\|_{H^{1}(\Omega)},
		\end{align}
		\begin{align}
			\begin{aligned}
				& (\p_{t}^{2}(R^{n}v_{l}),D_{t}v_{d}^{m})_{\mathcal{H}^{0}}\\
				\lesssim& \|D_{t}v_{d}^{m}\|_{L^{2}(\Omega)}\bigg(\|\p_{t}^{3}\bar{\eta}^{n}\|_{H^{1}}\|v_{l}\|_{W^{2,q_{+}}}+\|\p_{t}^{2}\bar{\eta}^{n}\|_{H^{1}(\Omega)}\|\p_{t}v_{l}\|_{L^{\infty}}+\|\p_{t}\bar{\eta}^{n}\|_{H^{2-\alpha }(\Omega)}\|\p_{t}^{2}v_{l}\|_{H^{1}}\bigg)\\
				\lesssim& \|D_{t}v_{d}^{m}\|_{L^{2}(\Omega)}\bigg(\|\p_{t}^{3}\bar{\eta}^{n}\|_{H^{1}}\|v_{l}\|_{W^{2,q_{+}}}+\|\p_{t}^{2}\bar{\eta}^{n}\|_{H^{1}(\Omega)}\|\p_{t}v_{l}\|_{W^{2,q_{-}}}+\|\p_{t}\bar{\eta}^{n}\|_{H^{2-\alpha }(\Omega)}\|\p_{t}^{2}v_{l}\|_{H^{1}}\bigg),
			\end{aligned}
		\end{align}
		and similarly,
		\begin{align}{\label{est:ep5}}
			\begin{aligned}
				(\p_{t}(R^{n}v_{l}),R^{n}v_{d}^{m})_{\mathcal{H}^{0}}\lesssim& \|\p_{t}\eta^{n}\|_{W^{3-\frac{1}{q_{-}},q_{-}}}\|v_{d}^{m}\|_{L^{2}(\Omega)}\|\p_{t}\eta^{n}\|_{W^{3-\frac{1}{q_{-}},q_{-}}}\|\p_{t}v_{l}\|_{L^{2}(\Omega)}\\
				&+\|\p_{t}\eta^{n}\|_{W^{3-\frac{1}{q_{-}},q_{-}}}\|v_{d}^{m}\|_{L^{2}(\Omega)}\|\p_{t}^{2}\eta^{n}\|_{H^{1}}\|v_{l}\|_{W^{2,q_{+}}(\Omega)},\\
				(\p_{t}(R^{n}v_{l}),\p_{t}J^{n}(J^{n})^{-1}v_{d}^{m})_{\mathcal{H}^{0}}\lesssim& \|\p_{t}\eta^{n}\|_{W^{3-\frac{1}{q_{-}},q_{-}}}\|v_{d}^{m}\|_{L^{2}(\Omega)}\|\p_{t}\eta^{n}\|_{W^{3-\frac{1}{q_{-}},q_{-}}}\|\p_{t}v_{l}\|_{L^{2}(\Omega)}\\
				&+\|\p_{t}\eta^{n}\|_{W^{3-\frac{1}{q_{-}},q_{-}}}\|v_{d}^{m}\|_{L^{2}(\Omega)}\|\p_{t}^{2}\eta^{n}\|_{H^{1}}\|v_{l}\|_{W^{2,q_{+}}(\Omega)}.
			\end{aligned}
		\end{align}
		
		Thus, by the Cauchy-Schwarz inequality and by combining \eqref{est:dt_u}--\eqref{est:ep5}, we conclude the energy structure
		\begin{align}
			\begin{aligned}
				&\frac{d}{dt}\frac12\left(\|D_tv_{d}^m\|_{\mathcal{H}^0}^2+\|\partial_{t}\xi_{d}^{m}\|_{H^{1}}^2\right)+\frac14\|D_tv_{d}^m\|_{\mathcal{H}^1}^2+\frac12[D_{t}v_{d}^{m}\cdot \mathcal{N}^{n}]_\ell^2\\
				&\le C(\|\p_t\eta^{n}\|_{H^{3/2+(\varepsilon_--\alpha)/2}}+\vert \vert u^{k}\vert \vert_{H^{\frac{5}{2}}}+\vert \vert \eta^{n}\vert \vert_{H^{3}})\left(\frac12\|D_tv_{d}^m\|_{\mathcal{H}^0}^2+\frac{1}{2}\vert \vert \partial_{t}\xi_{d}^{m}\vert \vert_{H^{1}}^{2}\right)\\
				&\quad +C\left(1+\|\p_t\eta^{n}\|_{H^{3/2+(\varepsilon_--\alpha)/2}}^2\right) ( \|v_{d}^m\|_{\mathcal{H}^1}^2 +\|\p_t(F^1-F^4-F^5)\|_{(\mathcal{H}^1)^{\ast}}^2)+C[\p_tF^7]_\ell^2\\
				&\quad+C\left(1+\|\p_t\eta^{n}\|_{H^{3/2+(\varepsilon_--\alpha)/2}}^2\right)(\|F^1\|_{L^{q_-}}^2+\|F^4\|_{W^{1-1/q_-,q_-}}^2+\|F^5\|_{W^{1-1/q_-,q_-}}^2)+\|\p_{t}u^{k}\|_{H^{\frac{5}{2}}}\|\p_{t}\xi_{l}^{k}\|^{2}_{H^{2}}\\
				&\quad+ \|D_{t}v_{l}^{k}\|^{2}_{H^{2}}\|\p_{t}\eta^{n}\|^{2}_{H^{\frac{5}{2}}}+\|v_{l}^{k}\|^{2}_{H^{2}}\|\p_{t}^{2}\eta^{n}\|^{2}_{H^{\frac{5}{2}}}+\|\p_{t}\eta^{n}\|_{W^{3-\frac{1}{q_{-}},q_{-}}}\|\p_{t}\xi_{d}^{m}\|_{H^{1}}^{2}+\|\p_{t}^{3}\bar{\eta}^{n}\|_{H^{1}}\|v_{l}\|_{W^{2,q_{+}}}^{2}\label{est:pa_tu_m1}.
			\end{aligned}
		\end{align}
		
		\noindent Note that the sequence $\{\|D_tv_{d}^m(0)\|_{H^0}\}$ are uniformly bounded and $D_tv_{d}^m(0) \rightharpoonup D_t^{2}u(0)$ strongly in $H^0$, as proved in Theorem \ref{thm:initial_convergence}. Integrating both sides of equation \eqref{est:pa_tu_m1} from $0$ to $T$, we then employ Gronwall's inequality, H\"older's inequality, the smallness of $\mathfrak{K}(\eta)$ and the initial data $\p_tu(0)$ to derive that

		\begin{equation}\label{eq:pa_tu_m1}
			\begin{aligned}
				&\sup_{0\le t\le T}(\|D_tv_{d}^m\|_{0}^2+\|\p_t\xi_{d}^m\|_1^2)+\|D_tv_{d}^m\|_{L^2H^1}^2+\|D_tv_{d}^m\|_{L^2H^0(\Sigma_s)}^2+\|[D_tv_{d}^m\cdot\mathcal{N}^{n}]_\ell\|_{L^2([0,T])}^2\\
				&\lesssim \exp\big\{(\|\p_t\eta^{n}\|_{L^\infty H^{3/2+\varepsilon_-/2}}+\vert \vert u^{k}\vert \vert_{L^{\infty}H^{\frac{5}{2}}}+\vert \vert \eta^{n}\vert \vert_{L^{\infty}H^{3}})T\big\}\bigg( 
				\sum_{j=0}^{2}\|D_t^{j}u(0)\|_{L^2(\Om)}^2+\|\p_t^{2}\xi(0)\|_{H^{3/2+(\varepsilon_--\alpha)/2}}^2\\
				&\quad +\|(F^1-F^4-F^5)(0)\|_{(\mathcal{H}^1)^\ast}^2
				+(1+\|\p_t\eta^n\|_{L^\infty H^{3/2+\varepsilon_-/2}}^2)\|v_{d}^m\|_{L_{t}^{2}\mathcal{H}^1}^2 +\mathfrak{F}^{n} +\mathscr{S}^{n, k} \bigg),
			\end{aligned}
		\end{equation}
		where $\mathfrak{F}^{n}$ is defined by replacing $\eta$ with $\eta^{n}$ in $\mathfrak{F}$, and
		\[
		\begin{aligned}
			\mathscr{S}^{n, k}=&\|\partial_{t}\eta^{k}\|^{2}_{L_{t}^{2}W^{3-1/q_{-},q_{-}}}
			\|\partial_{t}\xi_{l}^{k}\|^{2}_{L_{t}^{\infty}H^{3/2+(\varepsilon_{-}-\alpha)/2}}
			+\|\partial_{t}u^{k}\|_{L_{t}^{2}H^{5/2}}
			\|\partial_{t}\xi_{l}^{k}\|_{L_{t}^{\infty}H^{2}}^{2}+T^{\frac{1}{2}}\|\p_{t}^{3}\bar{\eta}^{n}\|_{L_{t}^{2}H^{1}}\|v_{l}\|_{L_{t}^{\infty}W^{2,q_{+}}}^{2}
			\\
			&+ \|D_{t}v_{l}^{k}\|^{2}_{L_{t}^{\infty}H^{2}}\|\p_{t}\eta^{n}\|^{2}_{L_{t}^{2}H^{\frac{5}{2}}}+T\|v_{l}^{k}\|^{2}_{L_{t}^{\infty}H^{2}}\|\p_{t}^{2}\eta^{n}\|^{2}_{L_{t}^{\infty}H^{\frac{5}{2}}}
			+\|\partial_{t}\xi_{l}^{k}\|^{2}_{L_{t}^{2}W^{3-1/q_{-},q_{-}}}
			\|\partial_{t}\eta^{k}\|^{2}_{L_{t}^{\infty}H^{3/2+(\varepsilon_{-}-\alpha)/2}}.
		\end{aligned}
		\]
		
		Using the identity
		$ v_{d}^{m}(t)
		= v_{d}^{m}(0)
		+ \int_{0}^{t}\bigl(D_{t}v_{d}^{m}+R^{n}v_{d}^{m}\bigr)\,ds $
		together with Theorem~\ref{thm:initial_convergence}, we choose $T>0$ sufficiently
		small (specifically, $T\lesssim \min(\delta,\frac{1}{2})$) so that 
		$
		C\Bigl(1+\|\partial_{t}\eta^{n}\|_{H^{3/2+(\varepsilon_{-}-\alpha)/2}}^{2}\Bigr)
		\|v_{d}^{m}\|_{\mathcal{H}^{1}}^{2}
		\;\lesssim\; 
		T\,\|\partial_{t}v_{d}^{m}\|_{L_{t}^{2}H^{1}}^{2}$,
		and hence can be absorbed into the left-hand side. Here we have also used the estimate
		for the difference between $D_{t}v_{d}$ and $\partial_{t}v_{d}$, which will be
		established in Lemma~\ref{lem:difference_u_Dt_u}.
		As a consequence, \eqref{eq:pa_tu_m1} can be rewritten in the form 
		\begin{equation}\label{eq:pa_tu_m2}
			\begin{aligned}
				&\sup_{0\le t\le T}
				\bigl(
				\|\partial_{t}v_{d}^{m}\|_{0}^{2}
				+\|\partial_{t}\xi_{d}^{m}\|_{1}^{2}
				\bigr)
				+\|\partial_{t}v_{d}^{m}\|_{L^{2}H^{1}}^{2}
				+\|\partial_{t}v_{d}^{m}\|_{L^{2}H^{0}(\Sigma_{s})}^{2}
				+\|[\partial_{t}v_{d}^{m}\cdot\mathcal{N}^{n}]_{\ell}\|_{L^{2}([0,T])}^{2}
				\\
				&\;\lesssim\;
				\exp\!\Bigl\{
				\bigl(
				\|\partial_{t}\eta^{n}\|_{L^{\infty}H^{3/2+\varepsilon_{-}/2}}
				+\|u^{k}\|_{L^{\infty}H^{5/2}}
				+\|\eta^{n}\|_{L^{\infty}H^{3}}
				\bigr)T
				\Bigr\} \Bigg(\mathcal{E}(0)
				+\|(F^{1}-F^{4}-F^{5})(0)\|_{(\mathcal{H}^1)^{*}}^{2}+\mathfrak{F}^{n}+\mathscr{S}^{n, k}\Bigg).
			\end{aligned}
		\end{equation}

		This estimate provides a uniform energy–dissipation bound for
		$(\partial_{t}v_{d}^{m},\partial_{t}\xi_{d}^{m})$ on the time interval $[0,T]$.
		
		\paragraph{\underline{Step 4 -- Energy-Dissipation bound for $(v_{d}^{m},\xi_{d}^{m})$}}
		
		The regularity of these lower-order terms follows immediately from integrating their temporal derivatives. Using \eqref{eq:pa_tu_m2}, we obtain the following estimate for $(v_{d}^{m},\xi_{d}^{m})$
		\begin{equation}\label{eq:xi_m}
			\begin{aligned}
				&\sup_{0\le t\le T}(\|v_{d}^m\|_{0}^2+\|\xi_{d}^m\|_1^2)+\|v_{d}^m\|_{L^\infty H^1}^2+\|v_{d}^m\|_{L^\infty H^0(\Sigma_s)}^2+\|[v_{d}^m\cdot\mathcal{N}]_\ell\|_{L^\infty([0,T])}^2\\
				&\lesssim \exp\{(\|\p_t\eta^{n}\|_{L^\infty H^{3/2+\varepsilon_-/2}}+\vert \vert u_{1}^{k}\vert \vert_{L^{\infty}H^{\frac{5}{2}}}+\vert \vert \eta^{n}\vert \vert_{L^{\infty}H^{3}})T\}\Bigg(\mathcal{E}(0)
				+\|(F^{1}-F^{4}-F^{5})(0)\|_{(\mathcal{H}^1)^{*}}^{2}+\mathfrak{F}^{n}+\mathscr{S}^{n, k}\Bigg).
			\end{aligned}
		\end{equation}
		
		\paragraph{\underline{Step 5 -- Passing to the Limit}}
		We now use the energy estimates \eqref{eq:xi_m} and \eqref{eq:pa_tu_m2} to pass to the limit as $m \to \infty$, with $n$ and $k$ fixed, and $v_{l},\xi_{l}$ are given.
		
		According to Proposition \ref{prop:isomorphism} and energy estimates, the sequences $\{v_{d}^m\}$ and $\{\p_tv_{d}^m\}$ are uniformly bounded both in $L^2H^1 \cap L^2H^0(\Sigma_s) \cap L^\infty H^0$. $\{\xi_{d}^m\}$ and $\{\p_t\xi_{d}^m\}$ are uniformly bounded in $L^\infty H^1$, $\{\p_t\xi^m\}$, $\{[v_{d}^m\cdot\mathcal{N}]_\ell\}$ and $\{[\p_tv_{d}^m\cdot\mathcal{N}]_\ell\}$ are uniformly bounded in $L^2([0,T])$.  Up to the extraction of a subsequence, it holds that
		\begin{align}\label{converge_1}
			\begin{aligned}
				v_{d}^m\rightharpoonup v_{d}\ \text{weakly-}\ \text{in}\ L^2 H^1\cap L^2H^0(\Sigma_s),\quad \p_tv_{d}^m\rightharpoonup\p_tv_{d}\ \text{weakly in}\ L^2H^1\cap L^2H^0(\Sigma_s),\\
				v_{d}^m\stackrel{\ast}\rightharpoonup v_{d} \text{weakly-}\ast\ \text{in}\ L^\infty H^0, \quad\p_t v_{d}^m\stackrel{\ast}\rightharpoonup \p_tv_{d}\ \text{weakly-}\ast\ \text{in}\ L^\infty H^0,\\
				\xi_{d}^m\stackrel{\ast}\rightharpoonup \xi_{d}\ \text{weakly-}\ast\ \text{in}\ L^\infty H^1, \quad\p_t\xi_{d}^m\stackrel{\ast}\rightharpoonup \p_t\xi_{d}\ \text{weakly-}\ast\ \text{in}\ L^\infty H^1,\\
				[v_{d}^m\cdot\mathcal{N}^{n}]_\ell\rightharpoonup[v_{d}\cdot\mathcal{N}^{n}]_\ell\ \text{weakly-}\ \text{in}\ L^2,\ [\p_tv_{d}^m\cdot\mathcal{N}^{n}]_\ell\rightharpoonup[\p_tv_{d}\cdot\mathcal{N}^{n}]_\ell\ \text{weakly in}\ L^2.
			\end{aligned}
		\end{align}
		By lower semicontinuity, the energy estimates \eqref{eq:pa_tu_m2} and \eqref{eq:xi_m} imply that
		\begin{equation}\label{eq:xi_m=}
			\begin{aligned}
				\|v_{d}\|_{L^\infty H^0}^2+\|\p_tv_{d}\|_{L^\infty H^0}^2+\|v_{d}\|_{L^2 H^1}^2+\|v_{d}\|_{L^2 H^0(\Sigma_s)}^2+\|\p_tv_{d}\|_{L^2H^1}^2+\|\p_tv_{d}\|_{L^2 H^0(\Sigma_s)}^2+\|[v_{d}\cdot\mathcal{N}^{n}]_\ell\|_{L^2}^2\\
				+\|[\p_tv_{d}\cdot\mathcal{N}^{n}]_\ell\|_{L^2}^2+\|\xi_{d}\|_{L^\infty H^1}^2+\|\p_t\xi_{d}\|_{L^\infty H^1}^2
			\end{aligned}
		\end{equation}
		is bounded by the terms in \eqref{eq:xi_m} from second line to the last up to a universal constant.
		
		\paragraph{\underline{Step 6 -- Improved Bound for $\p_t\xi_{d}$}}
		
		In this step, we fix $n,k$ and the prescribed functions $\xi_{l},v_{l}$, and apply the functional calculus developed in \cite[Section 8]{GT2020} to the modified gravity-capillary operator $\mathcal{K}_{k}$ induced by the $(\cdot)_{1,\Sigma_{k}}$ inner product defined by equation \eqref{eq:inner_product——1}.
		
		To be specific, $\mathcal{K}_{k}$ is defined as follows:
		\[
		\mathcal{K}_{k}\varphi:=g\varphi-\sigma \p_{1}\Big(\frac{\p_{1}\varphi}{(1+|\p_{1}\zeta_{0}|^{2})^{\frac{3}{2}}}\Big)-\sigma\p_{1}\big(\mathcal{R}_{z}(\p_{1}\zeta_{0},\p_{1}\eta^{k})\p_{1}\varphi\big).
		\]
		Since $\eta^{k}$ is smooth, we use the same argument as in \cite{GT2020} to develop the same results as in \cite{GT2020} for $\mathcal{K}_{k}$ and $D_{j}^{s}$ induced by $\mathcal{K}_{k}$.
		
		Let's first derive some preliminary results. Let the auxiliary function $\psi$ satisfy
		\begin{equation}\label{def:psi}
			-\Delta\psi=0 \quad \text{in} \ \Om,\quad \p_{\nu}\psi=(D_j^s(\partial_{t}\xi_{d}-a_{0}(t)\zeta_{0}))/|\mathcal{N}_0| \quad \text{on} \ \Sigma,\quad \p_{\nu}\psi=0 \quad \text{on} \ \Sigma_s,
		\end{equation}
		where $s=1-2\tilde{\alpha}\in [0,1)$ for some $\tilde{\alpha}\in (0,\frac{1}{2})$ to be determined, 
		$\zeta_{0}$ is the equilibrium state of the free surface, and $a_{0}(t)$ is a function depending only on $t$ such that:
		\begin{align}{\label{def:a}}
			\int_{-\ell}^{\ell} \p_{t}\xi_{d}- a_{0}(t)\zeta_{0} =0.
		\end{align}
		\noindent Using kinematic boundary condition (the fourth equation in system \eqref{eq:quasi_linear_{s}}), we have
		\begin{align}{\label{eq:mean}}
			\int_{-\ell}^{\ell}\p_{t}\xi_{d}=&\int_{-\ell}^{\ell}v_{d}\cdot \mathcal{N}^{n} +\int_{-\ell}^{\ell}(\int_{0}^{t}\p_{t}u_{1}^{k}\p_{1}\p_{t}^{2}\xi_{l}^{k} )  +\int_{-\ell}^{\ell}(\int_{0}^{t}u_{1}^{k}\p_{1}\p_{t}\xi_{d} )  +\int_{-\ell}^{\ell} I^{n,k}_{2} +\int_{-\ell}^{\ell} (R^{n}v_{l}^{k})\cdot \mathcal{N}^{n}.
		\end{align}
		Inserting \eqref{eq:mean} into \eqref{def:a} yields the explicit expression for $a_{0}(t)$
		\begin{align}{\label{eq:a}}
			\begin{aligned}
				a_{0}(t)=&\frac{\int_{-\ell}^{\ell}v_{d}\cdot \mathcal{N}^{n} +\int_{-\ell}^{\ell}(\int_{0}^{t}\p_{t}u_{1}^{k}\p_{1}\p_{t}^{2}\xi_{l}^{k} )  +\int_{-\ell}^{\ell}(\int_{0}^{t}u_{1}^{k}\p_{1}\p_{t}\xi_{d})  +\int_{-\ell}^{\ell} I^{n,k}_{2} +\int_{-\ell}^{\ell} (R^{n}v_{l}^{k})\cdot \mathcal{N}^{n}}{\int_{-\ell}^{\ell}\zeta_{0}}.
			\end{aligned}
		\end{align}
		\noindent Using \eqref{eq:a}, we have the following estimate using Sobolev embedding, Cauchy's inequality and trace theorem
		\begin{align}{\label{eq:bdd_a_1}}
			|a_{0}(t)|\lesssim&\|v_{d}\|_{H^{1}}+t\|\p_{t}u^{k}\|_{L^{\infty}H^{1}}\|\p_{t}\xi_{l}^{k}\|_{L_{t}^{\infty}H^{1}}+t\|u^{k}\|_{L_{t}^{\infty}H^{1}}\|\p_{t}\xi_{d}\|_{L_{t}^{\infty}H^{1}}+\|u(0)\|_{H^{1}}\|\p_{t}\eta(0)\|_{L_{t}^{\infty}H^{1}}\\
			&+\|\p_{t}\eta^{n}\|_{H^{1}}\|v_{l}^{k}\|_{H^{1}}. \no
		\end{align}
		\noindent Taking the temporal derivative of both sides of \eqref{eq:a}, we obtain the expression for $a_{0}^{\prime}(t)$:
		\begin{align}
			\begin{aligned}
				a_{0}^{\prime}(t)=&\frac{\int_{-\ell}^{\ell}\p_{t}v_{d}\cdot \mathcal{N}^{n}+\int_{-\ell}^{\ell}v_{d}\cdot \p_{t}\mathcal{N}^{n}+\int_{-\ell}^{\ell}\p_{t}u_{1}^{k}\p_{1}\p_{t}\xi_{l}^{k}  +\int_{-\ell}^{\ell}u_{1}^{k}\p_{1}\p_{t}\xi_{d} }{\int_{-\ell}^{\ell}\zeta_{0}}+\frac{\int_{-\ell}^{\ell} D_{t}(R^{n}v_{l}^{k})\cdot \mathcal{N}^{n}}{\int_{-\ell}^{\ell}\zeta_{0}},
			\end{aligned}
		\end{align}
		\noindent which leads to the following estimate:
		\begin{align}{\label{eq:bdd_a_2}}
			\begin{aligned}
				|a^{\prime}(t)|\lesssim &\|\p_{t}v_{d}\|_{H^{1}}+\|\p_{t}\eta^{n}\|_{H^{1}}\|v_{d}\|_{H^{1}}+\|\p_{t}u^{k}\|_{H^{1}}\|\p_{t}\xi_{l}^{k}\|_{H^{1}}+\|u^{k}\|_{H^{1}}\|\p_{t}\xi_{d}\|_{H^{1}}+\|\p_{t}^{2}\eta^{n}\|_{H^{1}}\|v_{d}^{k}\|_{H^{1}}\\
				& +\|\p_{t}\eta^{n}\|_{H^{1}}\|D_{t}v_{l}^{k}\|_{H^{1}}.
			\end{aligned}
		\end{align}
		
		For the test function $\psi$ defined in \eqref{def:psi}, \cite[Proposition 9.1]{GT2020} provides the estimates
		\begin{equation}\label{est:psi_t_0}
			\|\psi\|_{H^1}\lesssim\|(\partial_{t}\xi_{d}-a_{0}(t)\zeta_{0})\|_{H^{s-1/2}},\ \|\psi\|_{H^2}\lesssim\|D_j^s(\partial_{t}\xi_{d}-a_{0}(t)\zeta_{0})\|_{H^{1/2}},\ \|\p_t\psi\|_{H^1}\lesssim\|\p_t(\partial_{t}\xi_{d}-a_{0}(t)\zeta_{0})\|_{H^{s-1/2}}.
		\end{equation}
		Inserting \eqref{eq:bdd_a_1} and \eqref{eq:bdd_a_2} into \eqref{est:psi_t_0}, we have
		\begin{align}\label{est:psi_t}
			\begin{aligned}
				\|\psi\|_{H^1}\lesssim&\|(\partial_{t}\xi_{d})\|_{H^{s-1/2}}+\|v_{d}\|_{H^{1}}+t\|\p_{t}u^{k}\|_{L^{\infty}H^{1}}\|\p_{t}\xi_{l}^{k}\|_{L_{t}^{\infty}H^{1}}+t\|u^{k}\|_{L_{t}^{\infty}H^{1}}\|\p_{t}\xi_{d}\|_{L_{t}^{\infty}H^{1}}\\
				&\quad+\|u(0)\|_{H^{1}}\|\p_{t}\eta(0)\|_{L_{t}^{\infty}H^{1}}+\|\p_{t}\eta^{n}\|_{H^{1}}\|v_{l}^{k}\|_{H^{1}},\\
				\|\psi\|_{H^2}\lesssim&\|D_j^s(\partial_{t}\xi_{d})\|_{H^{1/2}}\|v_{d}\|_{H^{1}}+t\|\p_{t}u^{k}\|_{L^{\infty}H^{1}}\|\p_{t}\xi_{l}^{k}\|_{L_{t}^{\infty}H^{1}}+t\|u^{k}\|_{L_{t}^{\infty}H^{1}}\|\p_{t}\xi_{d}\|_{L_{t}^{\infty}H^{1}}\\
				&\quad+\|u(0)\|_{H^{1}}\|\p_{t}\eta(0)\|_{L_{t}^{\infty}H^{1}}+\|\p_{t}\eta^{n}\|_{H^{1}}\|v_{l}^{k}\|_{H^{1}},\\ \|\p_t\psi\|_{H^1}\lesssim&\|\p_t(\partial_{t}\xi_{d})\|_{H^{s-1/2}}+\|\p_{t}v_{d}\|_{H^{1}}+\|\p_{t}\eta^{n}\|_{H^{1}}\|v_{d}\|_{H^{1}}+\|\p_{t}u^{k}\|_{H^{1}}\|\p_{t}\xi_{l}^{k}\|_{H^{1}}+\|u^{k}\|_{H^{1}}\|\p_{t}\xi_{d}\|_{H^{1}}\\
				& +\|\p_{t}^{2}\eta^{n}\|_{H^{1}}\|v_{d}^{k}\|_{H^{1}}+\|\p_{t}\eta^{n}\|_{H^{1}}\|D_{t}v_{l}^{k}\|_{H^{1}}.
			\end{aligned}
		\end{align}
		
		We now work to enhance the regularity of $\partial_{t}\xi_{d}$. Integrating both sides of \eqref{eq:pa_tu_m_1} from $0$ to $T$ and using the convergence results from Step 5, we obtain 
		\begin{equation}\label{eq:4th}
			\begin{aligned}
				&\left<\p_t^2v_{d}, w\right>_{\ast}+\int_0^T\Big[((\p_tv_{d},w))+(\partial_{t}v_{d},R^{n}w)+((v_{d},R^{n}w))+(\p_t\xi_{d},w\cdot\mathcal{N}^{n})_{1,\Sigma_{k}}+[D_tv_{d}\cdot\mathcal{N}^{n},w\cdot\mathcal{N}^{n}]_\ell\Big]\\
				&+\int_{0}^{T}[\p_tF^7,w\cdot\mathcal{N}^{n}]_\ell+\int_{0}^{T}(\partial_{t}v_{d},w \partial_{t}J^{n})_{L^2(\Omega)}+\int_{0}^{T}\beta(\partial_{t}v_{d}\cdot \tau,(w \cdot \tau)\partial_{t}J^{n})\\
				&=\int_0^T\int_{\Om}\left[\p_tF^1\cdot wJ^{n}+\p_tJ^{n}KF^1\cdot w\right]+\int_{0}^{T}(\p_{t}(R^{n}v_{l}),w)_{\mathcal{H}^{0}}+\int_{0}^{T}((\p_{t}(R^{n}v_{l}),w))+(((R^{n}v_{l}),R^{n}w))\\
				&\quad +\int_{0}^{T}((R^{n}v_{l}),R^{n}w)_{\mathcal{H}^{0}}+\int_{0}^{T}((R^{n}v_{l}),w\p_{t}J^{n}(J^{n})^{-1})_{\mathcal{H}^{0}} -\int_0^T\int_{\Sigma_s}\beta(v_{d}\cdot\tau)(w\cdot\tau)\p_tJ\\
				&\quad-\int_0^T\int_{-\ell}^{\ell}\p_tF^4\cdot w-\int_0^T\int_{\Sigma_s}\left[\p_tF^5 w+\p_tJ^{n}K^{n}F^5w\right]\cdot\tau J^{n}\\
				&\quad-\int_0^T\int_\Om\frac{\mu}{2}(\mathbb{D}_{\p_t\mathcal{A}^{n}}v_{d}:\mathbb{D}_{\mathcal{A}^{n}}w+\mathbb{D}_{\mathcal{A}^{n}}v_{d}:\mathbb{D}_{\p_t\mathcal{A}^{n}}w+\p_tJK\mathbb{D}_{\mathcal{A}^{n}}v_{d}:\mathbb{D}_{\mathcal{A}^{n}}w)J^{n}-\int_{0}^{T}(\p_tJ^{n}K^{n}\p_tv_{d}, w)_{\mathcal{H}^0_T}\\
				&\quad+\big(\p_{1}(\mathcal{R}_{zz}(\p_{1}\zeta_{0},\p_{1}\eta^{k})\p_{1}\p_{t}\xi_{l}^{k}\p_{1}\p_{t}\eta^{k}),(w\cdot \mathcal{N}^{n})\big)_{L^{2}}+[\mathcal{R}_{zz}(\p_{1}\zeta_{0},\p_{1}\eta^{k})\p_{1}\p_{t}\xi_{l}^{k}\p_{1}\p_{t}\eta^{k},w\cdot \mathcal{N}^{n}]_{\ell},
			\end{aligned}
		\end{equation}
		where
		$\langle\partial_{t}^{2}v_{d},w\rangle_{*}=(\partial_{t}v_{d},w)_{\mathcal{H}^{0}}(T)-(\partial_{t}v_{d},w)_{\mathcal{H}^{0}}(0)-\int_{0}^{T}(\partial_{t}v_{d},\p_{t} w)_{\mathcal{H}^{0}}-\int_{0}^{T}(\p_{t}v_{d},\p_{t}J^{n}w)_{L^{2}(\Omega)}
		$ for any $w\in \mathcal{W}_{\sigma}$ such that $\p_{t}w\in \mathcal{H}^{0}$.
		
		Set the test function to be $w=M^{n}\nabla\psi\in \mathcal{W}_\sigma$ in \eqref{eq:4th} where $\psi$ is defined by \eqref{def:psi}. Using the fact that $\operatorname{div}_{\mathcal{A}^{n}}M^{n}\psi=0$, we obtain that
		\begin{equation}\label{eq:test_2}
			\begin{aligned}
				\int_{0}^{T}(\partial_{t}\xi_{d},D_{j}^{s}\partial_{t}\xi_{d})_{1,\Sigma_{k}}=I+II+III+IV+V,
			\end{aligned}
		\end{equation}
		where $I=\left<\partial_{t}^{2}v_{d},M^{n}\nabla\psi \right>_{*}$, and
		\[\begin{aligned}  
			II=&\int_{0}^{T} \left[((\p_tv_{d},M^{n}\nabla \psi))+(\partial_{t}v_{d},R^{n}M^{n}\nabla \psi)+((v_{d},R^{n}M^{n}\nabla\psi))+[\p_tv_{d}\cdot\mathcal{N}^{n},M^{n}\nabla\psi\cdot \mathcal{N}^{n}]_\ell\right]\notag\\&+\int_{0}^{T}\left[[\p_tF^7,M^{n}\nabla\psi\cdot\mathcal{N}^{n}]_\ell+(\partial_{t}v_{d},M^{n}\nabla\psi \partial_{t}J^{n})_{L^2(\Omega)}+\beta(\partial_{t}v_{d}\cdot \tau,(M^{n}\nabla\psi \cdot \tau)\partial_{t}J^{n})\right],
		\end{aligned}\]
		\[\begin{aligned}
			III=&\int_0^T\int_{\Om}\left[\p_tF^1\cdot M^{n}\nabla\psi +\p_tJ^{n}K^{n}F^1\cdot M^{n}\nabla\psi\right]J^{n}
			-\int_0^T\int_{-\ell}^{\ell} \p_tF^4\cdot M^{n}\nabla\psi \notag\\
			&-\int_0^T\int_{\Sigma_s}\left[\p_tF^5 M^{n}\nabla\psi+\p_tJ^{n}K^{n}F^5M^{n}\nabla\psi\right]\cdot\tau J^{n},\\
			IV=&-\int_0^T\int_\Om\frac{\mu}{2}(\mathbb{D}_{\p_t\mathcal{A}^{n}}v_{d}:\mathbb{D}_{\mathcal{A}^{n}}M^{n}\nabla\psi+\mathbb{D}_{\mathcal{A}^{n}}v_{d}:\mathbb{D}_{\p_t\mathcal{A}^{n}}M^{n}\nabla\psi+\p_tJ^{n}K^{n}\mathbb{D}_{\mathcal{A}^{n}}v_{d}:\mathbb{D}_{\mathcal{A}^{n}}M^{n}\nabla\psi)J^{n}\\
			&\quad-\int_{0}^{T}(\p_{t}v_{d},R^{n}M^{n}\nabla\psi)_{\mathcal{H}^{0}}-\int_0^T\int_{\Sigma_s}\beta(v_{d}\cdot\tau)(M^{n}\nabla\psi\cdot\tau)\p_tJ^{n}-\int_{0}^{T}(\p_tJ^{n}K^{n}\p_tv_{d}, M^{n}\nabla\psi)_{\mathcal{H}^0_T}\notag\\
			&\quad+\int_{0}^{T}(\p_{t}\xi_{d},a_{0}(t)D_{j}^{s}\zeta_{0})_{1,\Sigma_{k}},
		\end{aligned}\]
		and
		\[\begin{aligned}
			\begin{aligned}
				V=&\int_{0}^{T}(\p_{t}(R^{n}v_{l}),M^{n}\nabla \psi)_{\mathcal{H}^{0}}+\int_{0}^{T}((R^{n}v_{l}),R^{n}M^{n}\nabla \psi)_{\mathcal{H}^{0}}+\int_{0}^{T}(R^{n}v_{l},M^{n}\nabla\psi\p_{t}J^{n}(J^{n})^{-1})_{\mathcal{H}^{0}}\\
				&\quad+\int_{0}^{T}((\p_{t}(R^{n}v_{l}),M^{n}\nabla \psi))+(
				((R^{n}v_{l}),R^{n}M^{n}\nabla\psi))\\
				&\quad+\int_{0}^{T}(\p_{1}(\mathcal{R}_{zz}(\p_{1}\zeta_{0},\p_{1}\eta^{k})\p_{1}\p_{t}\xi_{l}^{k}\p_{1}\p_{t}\eta^{k}),M^{n}\nabla \psi)_{L^{2}}+\int_{0}^{T}[\mathcal{R}_{zz}(\p_{1}\zeta_{0},\p_{1}\eta^{k})\p_{1}\p_{t}\xi_{l}^{k}\p_{1}\p_{t}\eta^{k},M^{n}\nabla \psi]_{\ell}.
			\end{aligned}
		\end{aligned}\]
		The left hand side of \eqref{eq:test_2} can be rewritten as follows
		$
		(\partial_{t}\xi_{d},D_j^s\p_t\xi_{d})_{1,\Sigma}
		=\|D_j^s\p_t\xi_{d}\|_{\mathcal{H}^{1-\f{s}2}_\mathcal{K}}^2$.
		Hence, it suffices to estimate each term from $I$ to $V$.
		
		To estimate $I$, we rewrite $I$ by definition and estimate:
		\begin{align}\label{est:dtu_m1}
			\left<\partial_{t}^{2}v_{d},M^{n}\nabla\psi\right>_{*}&=(\partial_{t}v_{d},M^{n}\nabla \psi)_{L^{2}(\Omega)}(T)-(\partial_{t}v_{d},M^{n}\nabla\psi)_{L^{2}(\Omega)}(0)-\int_{0}^{T}(\partial_{t}v_{d},\p_{t} (M^{n}\nabla \psi))_{L^{2}(\Omega)}\notag\\
			&\lesssim \vert \vert \partial_{t}v_{d}\vert \vert_{L^{\infty}L^{2}}\vert \vert \psi\vert \vert_{L^{\infty}H^{1}}+T^{\frac{1}{2}}\vert \vert \partial_{t}v_{d}\vert \vert_{L^{\infty}L^{2}}\vert \vert \partial_{t}\psi\vert \vert_{L^{2}H^{1}}\notag\\
			&\lesssim \vert \vert \partial_{t}v_{d}\vert \vert_{L^{\infty}L^{2}}\Big(\vert \vert (\partial_{t}\xi_{d})\vert \vert_{L_{t}^{\infty}H^{s-\frac{1}{2}}}+T^{\frac{1}{2}}\vert \vert \partial_{t}^{2}\xi_{d}\vert \vert_{L^{2}H^{s-\frac{1}{2}}}\Big)+\|\p_{t}v_{d}\|_{L^{\infty}L^{2}}\Big(\|\p_t(\partial_{t}\xi_{d})\|_{H^{s-1/2}}\notag\\
			&\quad+\|\p_{t}v_{d}\|_{H^{1}}+\|\p_{t}\eta^{n}\|_{H^{1}}\|v_{d}\|_{H^{1}}+\|\p_{t}u^{k}\|_{H^{1}}\|\p_{t}\xi_{l}^{k}\|_{H^{1}}+\|u^{k}\|_{H^{1}}\|\p_{t}\xi_{d}\|_{H^{1}}\notag\\
			& \quad+\|\p_{t}^{2}\eta^{n}\|_{H^{1}}\|v_{d}^{k}\|_{H^{1}}+\|\p_{t}\eta^{n}\|_{H^{1}}\|D_{t}v_{l}^{k}\|_{H^{1}}\Big).
		\end{align}
		Using the kinematic boundary condition and trace theorem, we have:
		\begin{align}
			\vert \vert\partial_{t}^{2}\xi_{d}\vert \vert_{H^{s-\frac{1}{2}}}\lesssim& \vert \vert (\p_{t}v_{d})\cdot \mathcal{N}^{n}\vert \vert_{H^{s-\frac{1}{2}}}+\|(v_{d})\cdot \p_{t}\mathcal{N}^{n}\|_{H^{s-\frac{1}{2}}}+\vert \vert \partial_{t}u_{1}^{k}\partial_{t}\partial_{1}\xi_{l}^{k}\vert \vert_{H^{s-\frac{1}{2}}}+\vert \vert u_{1}^{k}\partial_{t}\partial_{1}\xi_{d}\vert \vert_{H^{s-\frac{1}{2}}}\notag\\
			&+\|D_{t}(R^{n}v_{l}^{k})\cdot \mathcal{N}^{n}\|_{H^{s-\frac{1}{2}}}\notag\\
			\lesssim &\vert \vert \p_{t}v_{d}\vert \vert_{H^{1}}+\|v_{d}\|_{H^{1}}\|\p_{t}\eta^{n}\|_{W^{3-\frac{1}{q_{-}},q_{-}}}+\vert \vert \partial_{t}u^{k}\vert \vert_{W^{2,q_{-}}}\vert \vert \partial_{t}\xi_{l} ^{k}\vert \vert_{H^{\frac{3}{2}+\frac{\varepsilon_{-}-\alpha}{2}}}\notag\\
			&+\vert \vert u_{1}^{k}\vert \vert_{W^{2,q_{-}}}\vert \vert \partial_{t}\xi_{d}\vert \vert_{H^{s+\frac{1}{2}}}+\|\p_{t}^{2}\eta^{n}\|_{H^{\frac{3}{2}-\alpha}}\|v_{l}^{k}\|_{W^{2,q_{-}}}.
		\end{align}
		Plugging this inequality into \eqref{est:dtu_m1} and using the fact that $s=1-2\tilde\alpha<1$, we have
		\begin{equation}
			\begin{aligned}
				\left<\partial_{t}^{2}v_{d},M^{n}\nabla \psi\right>_{*}&\lesssim \vert \vert \partial_{t}v_{d}\vert \vert_{L^{\infty}L^{2}}\vert \vert \partial_{t}\xi_{d}\vert \vert_{L^{\infty}H^{1}}+\vert \vert \p_{t}v_{d}\vert \vert_{H^{1}}+\vert \vert \partial_{t}u^{k}\vert \vert_{L^{\infty}W^{2,q_{-}}}\vert \vert \partial_{t}\xi^{k} \vert \vert_{L^{2}H^{\frac{3}{2}+\frac{\varepsilon_{-}-\alpha}{2}}}\\
				&\quad+\vert \vert u_{1}^{k}\vert \vert_{L^{\infty}W^{2,q_{-}}}\vert \vert \partial_{t}\xi_{d}\vert \vert_{L^{2}H^{s+\frac{1}{2}}}+\|v_{d}\|_{H^{1}}\|\p_{t}\eta^{n}\|_{W^{3-\frac{1}{q_{-}},q_{-}}}\\
				&\quad+\|\p_{t}v_{d}\|_{L^{\infty}L^{2}}\Big(\|v_{d}\|_{L_{t}^{\infty}H^{1}}+\|\p_{t}v_{d}\|_{L_{t}^{2}H^{1}}+\|\p_{t}u^{k}\|_{L^{\infty}H^{1}}\|\p_{t}\xi_{l}^{n}\|_{L_{t}^{\infty}H^{1}}\Big)\\
				&\quad+\|\p_{t}v_{d}\|_{L^{\infty}L^{2}}\Big(\|u\|_{L_{t}^{\infty}H^{1}}\|\p_{t}\xi_{d}\|_{L_{t}^{\infty}H^{1}}+\|u(0)\|_{H^{1}}\|\p_{t}^{2}\eta(0)\|_{H^{1}}\Big)+\|\p_{t}v_{d}\|_{L_{t}^{2}H^{1}}\\
				&\quad+\|\p_{t}\eta\|_{L_{t}^{\infty}H^{1}}\|v_{d}\|_{L_{t}^{\infty}H^{1}}+\|\p_{t}^{2}u\|_{L_{t}^{2}H^{1}}\|\p_{t}^{2}\eta\|_{L_{t}^{\infty}H^{1}}+\|u\|_{L_{t}^{\infty}H^{1}}\|\p_{t}\xi_{d}\|_{L_{t}^{\infty}H^{1}}\\
				&\quad+\|\p_{t}v_{d}\|_{L_{t}^{\infty}L^{2}}\Big(\|\p_{t}^{2}\eta^{n}\|_{L_{t}^{\infty}H^{1}}\|v_{l}^{k}\|_{L_{t}^{\infty}H^{1}}+\|\p_{t}\eta^{n}\|_{L_{t}^{\infty}H^{1}}\|D_{t}v_{l}^{k}\|_{L_{t}^{2}H^{1}}\Big).
			\end{aligned}
		\end{equation}
		
		We now turn to II and estimate each term directly:
		\begin{align}
			&\int_{0}^{T} ((\partial_{t}v_{d},M^{n}\nabla \psi))\lesssim\vert \vert \partial_{t}v_{d}\vert \vert_{L^{2}H^{1}}\vert \vert \psi\vert \vert_{L^{2}H^{2}}\\
			&\lesssim\ \vert \vert \partial_{t}v_{d}\vert \vert_{L^{2}H^{1}}\vert \vert D_{j}^{s}(\partial_{t}\xi_{d})\vert \vert_{L^{2}H^{\frac{1}{2}}}+\|\p_{t}v_{d}\|_{L^{2}H^{1}}\Big(\|v_{d}\|_{L_{t}^{\infty}H^{1}}+\|\p_{t}u^{k}\|_{L^{\infty}H^{1}}\|\p_{t}\xi_{l}^{k}\|_{L_{t}^{\infty}H^{1}}\notag\\
			&\qquad+\|u^{k}\|_{L_{t}^{\infty}H^{1}}\|\p_{t}\xi_{d}\|_{L_{t}^{\infty}H^{1}}+\|u(0)\|_{H^{1}}\|\p_{t}\eta(0)\|_{H^{1}}+\|\p_{t}\eta^{n}\|_{L_{t}^{\infty}H^{1}}\|v_{l}^{k}\|_{L_{t}^{\infty}H^{1}}\Big),\notag\\
			&\int_{0}^{T}(\partial_{t}v_{d},R^{n}M^{n}\nabla \psi)\lesssim T\vert \vert \partial_{t}v_{d}\vert \vert_{L^{\infty}L^{2}}\vert \vert \psi\vert \vert_{L_{t}^{\infty}H^{1}}\\
			&\lesssim\ T\vert \vert \partial_{t}v_{d}\vert \vert_{L^{\infty}L^{2}}\vert \vert \partial_{t}\xi_{d}\vert \vert_{L^{\infty}H^{1}}+\|\p_{t}v_{d}\|_{L^{\infty}L^{2}}\Big(\|v_{d}\|_{L_{t}^{\infty}H^{1}}+\|\p_{t}u^{k}\|_{L^{\infty}H^{1}}\|\p_{t}\xi_{l}^{k}\|_{L_{t}^{\infty}H^{1}}\notag\\
			&\qquad+\|u^{k}\|_{L_{t}^{\infty}H^{1}}\|\p_{t}\xi_{d}\|_{L_{t}^{\infty}H^{1}}+\|u(0)\|_{H^{1}}\|\p_{t}\eta(0)\|_{H^{1}}+\|\p_{t}\eta^{n}\|_{L_{t}^{\infty}H^{1}}\|v_{l}^{k}\|_{L_{t}^{\infty}H^{1}}\Big),\notag
		\end{align}
		and
		\begin{align}
			\begin{aligned}
				&\int_{0}^{T}((v_{d},R^{n}M^{n}\nabla \psi)) \lesssim\ T^{\frac{1}{2}}\vert \vert v_{d}\vert \vert_{L^{\infty}H^{1}}\vert \vert D_{j}^{s}(\partial_{t}\xi_{d})\vert \vert_{L^{2}H^{\frac{1}{2}}}+\|v_{d}\|_{L^{\infty}H^{1}}\big(\|v_{d}\|_{L_{t}^{\infty}H^{1}}+\|\p_{t}u^{k}\|_{L^{\infty}H^{1}}\|\p_{t}\xi_{l}^{k}\|_{L_{t}^{\infty}H^{1}}\\
				&\qquad+\|u^{k}\|_{L_{t}^{\infty}H^{1}}\|\p_{t}\xi_{d}\|_{L_{t}^{\infty}H^{1}}+\|u(0)\|_{H^{1}}\|\p_{t}\eta(0)\|_{H^{1}}+\|\p_{t}\eta^{n}\|_{L_{t}^{\infty}H^{1}}\|v_{l}^{k}\|_{L_{t}^{\infty}H^{1}}\big),\\
				&\int_{0}^{T}[\partial_{t}v_{d}\cdot \mathcal{N}^{n},M^{n}\nabla \psi\cdot \mathcal{N}^{n}]_{\ell}\lesssim\ \vert \vert [\partial_{t}v_{d}\cdot \mathcal{N}^{n}]_{\ell}\vert \vert_{L_{t}^{2}}\vert \vert D_{j}^{s}\partial_{t}\xi_{d}\vert \vert_{L^{2}H^{\frac{1}{2}+}}+\|[\p_{t}v_{d}\cdot \mathcal{N}^{n}]_{\ell}\|_{L^{2}H^{1}}\big(\|v_{d}\|_{L_{t}^{\infty}H^{1}}\\
				&\qquad +\|\p_{t}u^{k}\|_{L^{\infty}H^{1}}\|\p_{t}\xi_{l}^{k}\|_{L_{t}^{\infty}H^{1}}+\|u^{k}\|_{L_{t}^{\infty}H^{1}}\|\p_{t}\xi_{d}\|_{L_{t}^{\infty}H^{1}}+\|u(0)\|_{H^{1}}\|\p_{t}\eta(0)\|_{H^{1}}+\|\p_{t}\eta^{n}\|_{L_{t}^{\infty}H^{1}}\|v_{l}^{k}\|_{L_{t}^{\infty}H^{1}}\big),\\
				&\int_{0}^{T}[\partial_{t}F^{7},M^{n}\nabla \psi\cdot \mathcal{N}^{n}]_{\ell}\lesssim\  \vert \vert [\partial_{t}F^{7}]\vert \vert_{\ell}\vert \vert D_{j}^{s}\partial_{t}\xi_{d}\vert \vert_{L^{2}H^{\frac{1}{2}+}}+\|[\p_{t}F^{7}]_{\ell}\|_{L^{2}}\big(\|v_{d}\|_{L_{t}^{\infty}H^{1}}+\|\p_{t}u^{k}\|_{L^{\infty}H^{1}}\|\p_{t}\xi_{l}^{k}\|_{L_{t}^{\infty}H^{1}}\\
				&\qquad+\|u^{k}\|_{L_{t}^{\infty}H^{1}}\|\p_{t}\xi_{d}\|_{L_{t}^{\infty}H^{1}}+\|u(0)\|_{H^{1}}\|\p_{t}\eta(0)\|_{H^{1}}+\|\p_{t}\eta^{n}\|_{L_{t}^{\infty}H^{1}}\|v_{l}^{k}\|_{L_{t}^{\infty}H^{1}}\big),\\
				&\int_{0}^{T}(\partial_{t}v_{d},M^{n}\nabla \psi\partial_{t}J^{n})_{\mathcal{H}^{0}(\Sigma)}\lesssim\ T^{\frac{1}{2}}\vert \vert \partial_{t}v_{d}\vert \vert_{L^{\infty}L^{2}}\vert \vert \p_{t}\xi_{d}\vert \vert_{L^{2}H^{1}}+\|\p_{t}v_{d}\|_{L^{2}H^{1}}\big(\|v_{d}\|_{L_{t}^{\infty}H^{1}}\\
				&\qquad+\|\p_{t}u^{k}\|_{L^{\infty}H^{1}}\|\p_{t}\xi_{l}^{k}\|_{L_{t}^{\infty}H^{1}}+\|u^{k}\|_{L_{t}^{\infty}H^{1}}\|\p_{t}\xi_{d}\|_{L_{t}^{\infty}H^{1}}+\|u(0)\|_{H^{1}}\|\p_{t}\eta(0)\|_{H^{1}}+\|\p_{t}\eta^{n}\|_{L_{t}^{\infty}H^{1}}\|v_{l}^{k}\|_{L_{t}^{\infty}H^{1}}\big),
			\end{aligned}
		\end{align}
		as well as
		\begin{align}
			\int_{0}^{T}\beta(\partial_{t}v_{d}\cdot \tau, (M\nabla \psi\cdot \tau)\partial_{t}J^{n})_{\mathcal{H}^{0}(\Sigma_{s})}=&\ 0,
		\end{align}
		where we used the fact that $\p_{\nu}\psi=0$ on $\Sigma_{s}$ and \eqref{eq:bdd_a_2}.
		
		For III, we note that it can be rewritten as:
		\begin{align}
			III=\p_t\mathcal{F}(M^{n}\nabla\psi)-\mathcal{F}(\p_t(M^{n}\nabla\psi))+\mathcal{F}(R^{n}M^{n}\nabla\psi).
		\end{align}
		According to \eqref{eq:cal_f}, for forcing terms $F^j$ with $j=1, 3, 4, 5$, we may apply the similar estimates as in the energy bound for $\p_tv_{d}^m$ in Step 3. In this way, we obtain
		\begin{equation}
			\begin{aligned}
				&\p_t\mathcal{F}(M^{n}\nabla\psi)-\mathcal{F}(\p_t(M^{n}\nabla\psi))+\mathcal{F}(R^{n}M^{n}\nabla\psi)\\
				&\lesssim\big[\|\p_t(F^1-F^4-F^5)\|_{(\mathcal{H}^1)^{\ast}}+\|\p_t\eta^{n}\|_{H^{3/2+(\varepsilon_--\alpha)/2}}(\|F^1\|_{L^{q_-}}+\|F^4\|_{W^{1-1/q_-,q_-}}
				+\|F^5\|_{W^{1-1/q_-,q_-}})\big]\|\psi\|_{H^2}\\
				&\lesssim\big[\|\p_t(F^1-F^4-F^5)\|_{(\mathcal{H}^1)^{\ast}}+\|\p_t\eta^{n}\|_{H^{3/2+(\varepsilon_--\alpha)/2}}(\|F^1\|_{L^{q_-}}+\|F^4\|_{W^{1-1/q_-,q_-}}
				+\|F^5\|_{W^{1-1/q_-,q_-}})\big]\|D_j^s\p_t\xi_{d}\|_{\mathcal{H}^{1/2}_\mathcal{K}}\\
				&\quad+\big[\|\p_t(F^1-F^4-F^5)\|_{(\mathcal{H}^1)^{\ast}}+\|\p_t\eta^{n}\|_{H^{3/2+(\varepsilon_--\alpha)/2}}(\|F^1\|_{L^{q_-}}+\|F^4\|_{W^{1-1/q_-,q_-}}
				+\|F^5\|_{W^{1-1/q_-,q_-}})\big]\\
				&\quad \times \big(\|v_{d}\|_{L_{t}^{\infty}H^{1}}+\|\p_{t}u^{k}\|_{L^{\infty}H^{1}}\|\p_{t}\xi_{l}^{k}\|_{L_{t}^{\infty}H^{1}}+\|u^{k}\|_{L_{t}^{\infty}H^{1}}\|\p_{t}\xi_{d}\|_{L_{t}^{\infty}H^{1}}+\|u(0)\|_{H^{1}}\|\p_{t}\eta(0)\|_{H^{1}}+\|\p_{t}\eta^{n}\|_{H^{1}}\|v_{l}^{k}\|_{H^{1}}\big).
			\end{aligned}
		\end{equation}
		
		We then estimate each term in IV. For terms in the first line of IV, we bound them by
		\begin{equation}
			\begin{aligned}
				&-\int_\Om\frac{\mu}{2}(\mathbb{D}_{\p_t\mathcal{A}^{n}}v_{d}:\mathbb{D}_{\mathcal{A}^{n}}(M^{n}\nabla\psi)+\mathbb{D}_{\mathcal{A}^{n}}v_{d}:\mathbb{D}_{\p_t\mathcal{A}^{n}}(M^{n}\nabla\psi)+\p_tJ^{n}K^{n}\mathbb{D}_{\mathcal{A}^{n}}v_{d}:\mathbb{D}_{\mathcal{A}^{n}}(M^{n}\nabla\psi))J^{n}\\
				&\lesssim \|\p_t\bar{\eta}^{n}\|_{W^{1,\infty}}\|v_{d}\|_{\mathcal{H}^1}\|\psi\|_{H^2}\\
				&\lesssim\|\p_t\eta^{n}\|_{H^{3/2+\varepsilon_-/2}}\|v_{d}\|_{\mathcal{H}^1}\|D_j^s\p_t\xi_{d}\|_{\mathcal{H}^{1/2}_\mathcal{K}}
				+\|\p_t\eta^{n}\|_{H^{3/2+\varepsilon_-/2}}\|v_{d}\|_{\mathcal{H}^1}\big(\|v_{d}\|_{L_{t}^{\infty}H^{1}}\\
				&\quad+\|\p_{t}u^{k}\|_{L^{\infty}H^{1}}\|\p_{t}\xi_{l}^{k}\|_{L_{t}^{\infty}H^{1}}+\|u^{k}\|_{L_{t}^{\infty}H^{1}}\|\p_{t}\xi_{d}\|_{L_{t}^{\infty}H^{1}}+\|u(0)\|_{H^{1}}\|\p_{t}\eta(0)\|_{H^{1}}+\|\p_{t}\eta^{n}\|_{H^{1}}\|v_{l}^{k}\|_{H^{1}}\big).
			\end{aligned}
		\end{equation}
		For the remaining terms in IV, we bound them as
		\begin{align}
			\begin{aligned}
				&-(\p_tv_{d},R^{n}M^{n}\nabla\psi)_{\mathcal{H}^0}+\int_\Om\p_tv_{d}\cdot M^{n}\nabla\psi\p_tJ^{n}+\int_{\Sigma_s}\beta(v_{d}\cdot\tau)(M^{n}\nabla\psi\cdot\tau)\p_tJ^{n}+(\p_{t}\xi_{d},a_{0}(t)D_{j}^{s}\zeta_{0})_{1,\Sigma_{n}}\\
				&\lesssim\|\p_tv_{d}\|_{\mathcal{H}^0}(\|R^{n}\|_{L^\infty(\Om)}+\|\p_tJ^{n}\|_{L^\infty(\Om)})\|\psi\|_{H^1}+\|v_{d}\|_{L^2(\Sigma_s)}\|\p_tJ^{n}\|_{L^\infty(\Sigma_s)}\|\psi\|_{H^2}\\
				&\lesssim \|\p_t\eta^{n}\|_{H^{3/2+\varepsilon_-/2}}\|\p_tv_{d}\|_{\mathcal{H}^0}\|\p_t\xi_{d}\|_{\mathcal{H}^{s-1/2}}+\|\p_t\eta^{n}\|_{H^{3/2+\varepsilon_-/2}}\|v_{d}\|_{L^2(\Sigma_s)}\|D_j^s\p_t\xi_{d}\|_{\mathcal{H}^{1/2}_\mathcal{K}}\\
				&\quad+\big(\|\p_{t}\eta\|_{H^{\frac{3}{2}+\frac{\varepsilon_{-}-\alpha}{2}}}\|\p_{t}v_{d}\|_{\mathcal{H}^{0}}+\|\p_{t}\eta\|_{H^{\frac{3}{2}+\frac{\varepsilon_{-}-\alpha}{2}}}\|v_{d}\|_{H^{1}}\big)\Big(\|v_{d}\|_{L_{t}^{\infty}H^{1}}+\|\p_{t}u^{k}\|_{L^{\infty}H^{1}}\|\p_{t}\xi_{l}^{k}\|_{L_{t}^{\infty}H^{1}}\\
				&\quad\quad+\|u^{k}\|_{L_{t}^{\infty}H^{1}}\|\p_{t}\xi_{d}\|_{L_{t}^{\infty}H^{1}}+\|u(0)\|_{H^{1}}\|\p_{t}\eta(0)\|_{H^{1}}+\|\p_{t}\eta^{n}\|_{H^{1}}\|v_{l}^{k}\|_{H^{1}}\Big).
			\end{aligned}
		\end{align}
		
		Finally, for the term $V$, using the similar computation as in Step 4, we obtain
		\begin{align}\label{est:dtu_m2}
			\begin{aligned}
				V\lesssim& \|\psi\|_{L_{t}^{2}H^{2}}\Big(\|\p_{t}^{2}{\eta}^{n}\|_{L_{t}^{2}H^{\frac{3}{2}}}\|v_{i}\|_{L_{t}^{\infty}W^{2,q_{-}}}+\|\p_{t}{\eta}^{n}\|_{L_{t}^{2}W^{3-\frac{1}{q_{-}},q_{-}}}\|D_{t}v_{l}\|_{L_{t}^{\infty}H^{1+\frac{\varepsilon_{-}}{2}}}\\
				&\quad+\|\p_{t}\eta^{n}\|_{L_{t}^{\infty}H^{\frac{3}{2}+\frac{\varepsilon_{-}-\alpha}{2}}}\|D_{t}v_{l}\|_{L_{t}^{\infty}W^{2,q_{-}}}+\|\p_{t}{\eta}^{n}\|_{L_{t}^{2}W^{3-\frac{1}{q_{-}},q_{-}}}\|v_{l}\|_{L_{t}^{\infty}W^{2,q_{-}}}\Big)\\
				&\quad+\|\p_{t}\xi_{l}^{k}\|_{L_{t}^{\infty}H^{\frac{3}{2}+\frac{\varepsilon_{-}-\alpha}{2}}}\|\p_{t}\eta^{k}\|_{L_{t}^{2}W^{3-\frac{1}{q_{-}},q_{-}}}\|\psi\|_{L_{t}^{2}H^{1}}+\|\p_{t}\eta^{k}\|_{L_{t}^{\infty}H^{\frac{3}{2}+\frac{\varepsilon_{-}-\alpha}{2}}}\|\p_{t}\xi^{k}\|_{L_{t}^{2}W^{3-\frac{1}{q_{-}},q_{-}}}\|\psi\|_{L_{t}^{2}H^{1}}\\
				\lesssim&\mathscr{S}^{n,k}\cdot \Big(\|v_{d}\|_{L_{t}^{\infty}H^{1}}+\|\p_{t}u^{k}\|_{L^{\infty}H^{1}}\|\p_{t}\xi_{l}^{k}\|_{L_{t}^{\infty}H^{1}}+\|u^{k}\|_{L_{t}^{\infty}H^{1}}\|\p_{t}\xi_{d}\|_{L_{t}^{\infty}H^{1}}\\
				&\quad+\|u(0)\|_{H^{1}}\|\p_{t}\eta(0)\|_{H^{1}}+\|\p_{t}\eta^{n}\|_{H^{1}}\|v_{l}^{k}\|_{H^{1}}\Big)+\big(\|v_{d}\|_{L_{t}^{\infty}H^{1}}+\|\p_{t}u^{k}\|_{L^{\infty}H^{1}}\|\p_{t}\xi_{l}^{k}\|_{L_{t}^{\infty}H^{1}}\\
				&\quad+\|u^{k}\|_{L_{t}^{\infty}H^{1}}\|\p_{t}\xi_{d}\|_{L_{t}^{\infty}H^{1}}+\|u(0)\|_{H^{1}}\|\p_{t}\eta(0)\|_{H^{1}}+\|\p_{t}\eta^{n}\|_{H^{1}}\|v_{l}^{k}\|_{H^{1}}\big).
			\end{aligned}
		\end{align}

		Finally, using \eqref{est:dtu_m1}--\eqref{est:dtu_m2}, together with Cauchy's inequality, and letting $j\rightarrow +\infty$, we obtain the following enhanced estimate for $\p_t\xi_{d}$: 
		\begin{equation}\label{est:enhance_2'}
			\begin{aligned}
				\|\p_t\xi_{d}\|_{L^2H^{3/2-\tilde{\alpha}}}^2
				&\lesssim\exp\big\{(\|\p_t\eta^{n}\|_{L^\infty H^{3/2+\varepsilon_-/2}}+\vert \vert u^{k}\vert \vert_{L^{\infty}H^{\frac{5}{2}}}+\vert \vert \eta^{n}\vert \vert_{L^{\infty}H^{3}})T\big\}\bigg(\mathcal{E}(0)\\
				&\quad+\|(F^1-F^4-F^5)(0)\|_{(\mathcal{H}^1)^\ast}^2+\mathfrak{F}^{n}+\|u_{1}^{k}\|_{L_{t}^{\infty}W^{2,q_{-}}}\|\p_{t}\xi_{d}\|^{2}_{H^{\frac{3}{2}-2\tilde{\alpha}}}+\mathscr{S}^{n,k}\bigg).
			\end{aligned}
		\end{equation}
		Notice that the right-hand side of \eqref{est:enhance_2'} contains the term $\vert \vert \partial_{t}\xi_{d}\vert \vert_{H^{\frac{3}{2}-2\tilde{\alpha}}}$, which cannot be controlled by the existing bound $\partial_{t}\xi_{d}\in L^{2}H^{1}$ from the energy estimate \eqref{eq:xi_m=}. Thus, we utilize a bootstrapping argument to improve the $\dt\xi_d$ estimates.
		
		Setting $\tilde{\alpha}=\frac{1}{4}$ which implies $\frac{3}{2}-2\tilde\alpha=1$, we have $\vert \vert \partial_{t}\xi_{d}\vert \vert_{H^{\frac{3}{2}-2\tilde{\alpha}}}=\vert \vert \partial_{t}\xi_{d}\vert \vert_{H^{1}}$ is controllable. Based on \eqref{est:enhance_2'}, we obtain the improved bound for $\partial_{t}\xi_{d}\in L^{2}H^{\frac{3}{2}-\tilde\alpha}=\partial_{t}\xi_{d}\in L^{2}H^{\frac{5}{4}}$. 
		Resetting $\tilde\alpha=\frac{1}{8}$, the term $\vert \vert \partial_{t}\xi_{d}\vert \vert_{H^{\frac{3}{2}-2\tilde{\alpha}}}=\vert \vert \partial_{t}\xi_{d}\vert \vert_{H^{\frac{5}{4}}}$ is bounded by the result of the previous step. Again based on \eqref{est:enhance_2'}, we obtain the improved bound for $\partial_{t}\xi_{d}\in L^{2}H^{\frac{3}{2}-\tilde\alpha}=\partial_{t}\xi_{d}\in L^{2}H^{\frac{11}{8}}$.
		Iterating this procedure allows us to bootstrap the regularity of $\p_{t}\xi_{d}$. Let $\tilde{\alpha}_{i}$ denote the value of $\tilde{\alpha}$ chosen at the $i$-th step. {We have $\tilde{\alpha}_{i+1}=\frac{1}{2}\tilde{\alpha}_{i}$}. It is then straightforward to verify that $\tilde{\alpha}_{i}\rightarrow 0$. Therefore, there exists an integer $N$ such that $\tilde{\alpha}_{i}\leq \alpha$ for $i>N$. At that stage, the term including $\|\p_{t}\xi_{d}\|_{L_{t}^{2}H^{\frac{3}{2}-2\tilde{\alpha}}}$ can be absorbed by the LHS, and we obtain the following estimate:
		\begin{equation}\label{est:enhance_2}
			\begin{aligned}
				\|\p_t\xi_{d}\|_{L^2H^{3/2-{\alpha}}}^2&\lesssim\exp\big\{(\|\p_t\eta^{n}\|_{L^\infty H^{3/2+\varepsilon_-/2}}+\vert \vert u^{k}\vert \vert_{L^{\infty}H^{\frac{5}{2}}}+\vert \vert \eta^{n}\vert \vert_{L^{\infty}H^{3}})T\big\}\bigg(\mathcal{E}(0)\\
				&\quad +\|(F^1-F^4-F^5)(0)\|_{(\mathcal{H}^1)^\ast}^2 +\mathfrak{F}^{n}+\mathscr{S}^{n,k}\bigg)
			\end{aligned}
		\end{equation}
		
		\paragraph{\underline{Step 7 -- Pressure Estimates and Strong Solution}}
		In this step, we fix $n,k$ and the prescribed functions $\xi_{l},v_{l}$. 
		Using the convergence in \eqref{converge_1} and the estimate \eqref{est:enhance_2}, we pass to the limit in \eqref{eq:galerkin} for almost every $t\in [0, T]$ so that
		\begin{equation}\label{eq:weak_limit}
			\begin{aligned}
				&(\p_tv_{d},w)_{\mathcal{H}^0}+((v_{d},w))+(\xi_{d},w\cdot\mathcal{N}^{n})_{1,\Sigma_{0}}+ \left(\int_{0}^{t}\mathcal{R}_{z}(\partial_{1}\zeta_{0},\partial_{1}\eta^{k})\partial_{1}\partial_{t}\xi_{d},\partial_{1}(w\cdot\mathcal{N}^{n}) \right)_{L^{2}}+[v_{d}\cdot \mathcal{N}^{n},w\cdot\mathcal{N}^{n}]_{\ell}\\&
				=\int_{\Om}F^1\cdot wJ^{n}+\int_{-\ell}^{\ell}F^4\cdot w-\int_{\Sigma_s}F^5(w\cdot\tau)J^{n}-[F^7,w\cdot\mathcal{N}^{n}]_\ell+((R^{n}D_{t}v_{l}^{k}),w)_{\mathcal{H}^{0}}\\
				&\quad+((R^{n}v_{l}^{k},w))-(I^{n,k}_{1},\p_{1}(w\cdot \mathcal{N}^{n}))_{L^{2}(-\ell.\ell)}-\left(\int_{0}^{t}\mathcal{R}_{z}(\partial_{1}\zeta_{0},\partial_{1}\eta^{k})\partial_{1}\p_{1}\p_{t}\eta^{k}\p_{1}\p_{t}\xi_{l}^{k},\partial_{1}(w\cdot\mathcal{N}^{n}) \right)_{L^{2}},
			\end{aligned}
		\end{equation}
		coupled with the kinematic boundary condition \eqref{eq:kinematic_0}
		for any $w\in \mathcal{W}(t)$.
		Then applying Theorem \ref{thm:pressure}, we recover the pressure $q_{d}$.
		Moreover, $(v_{d},q_{d},\xi_{d})$ is the strong solution to equation \eqref{eq:quasi_linear_{s}} satisfies the following bound
		\begin{equation}\label{est:diss_4}
			\begin{aligned}
				&\| v_{d}\|_{L^2W^{2,q_-}}^2 + \|q_{d}\|_{L^2W^{1,q_-}}^2 + \|\xi_{d}\|_{L^2W^{3-1/q_-,q_-}}^2 \\
				&\lesssim  \|\p_tv_{d}\|_{L^2H^0}^2 +\|\xi_{d} \|_{L^2H^1}^2 + \|v_{d}\|_{L^2H^1}^2 + \|\xi_{d}\|_{L^2H^{3/2-\alpha}}^2 + \|[v_{d}\cdot \mathcal{N}]_\ell\|_{L^2_t}^2\\
				&\quad + \|F^1\|_{L^2L^{q_-}}^2 + \|F^4\|_{L^2W^{1-1/q_-,q_-}}^2 + \|F^5\|_{L^2W^{1-1/q_-,q_-}}^2 + \|[F^7]_\ell\|_{L_t^{2}}^2\\
				&\quad+\|\p_{t}\eta^{n}\|_{L_{t}^{\infty}W^{3-\frac{1}{q_{-}},q_{-}}}^{2}\|v_{l}^{k}\|^{2}_{L_{t}^{2}W^{2,q_{+}}}+T\|\p_{t}u^{k}\|^{2}_{L_{t}^{2}W^{2,q_{-}}}\|\p_{t}\xi_{l}^{k}\|^{2}_{L_{t}^{2}W^{3-\frac{1}{q_{-}},q_{-}}}.
			\end{aligned}
		\end{equation}

		Applying the Sobolev embedding $H^{3/2-\alpha}\hookrightarrow W^{2-1/q_-,q_-}$, we substitute \eqref{est:enhance_2} into \eqref{est:diss_4} to obtain
		\begin{equation}\label{est:ellip_1}
			\begin{aligned}
				&\|v_{d}\|_{L^2W^{2,q_-}}^2+\|q_{d}\|_{L^2W^{1,q_-}}^2+\|\xi_{d}\|_{L^2W^{3-1/q_-,q_-}}^2\\
				& \lesssim \exp\big\{(\|\p_t\eta^{n}\|_{L^\infty H^{3/2+\varepsilon_-/2}}+\vert \vert u^{k}\vert \vert_{L^{\infty}H^{\frac{5}{2}}}+\vert \vert \eta^{n}\vert \vert_{L^{\infty}H^{3}})T\big\}\bigg(\mathcal{E}(0)+\|(F^1-F^4-F^5)(0)\|_{(\mathcal{H}^1)^\ast}^2
				+\mathfrak{F}+\mathscr{S}^{n,k}\bigg)
			\end{aligned}
		\end{equation}
		Therefore, the bound \eqref{est:ellip_1} implies that $(v_{d}, q_{d}, \xi_{d})$ enjoys higher regularity than that required for weak solutions. This allows us to derive a uniform bound for the solution to \eqref{eq:quasi_linear_{s}}  that is independent of $k$ which we establish in the next step.
		
		\paragraph{\underline{Step 8 -- Uniform $k$ Bound}}
		
		In this step, we fix $n$ and the prescribed functions $\xi_{l},v_{l}$.
		
		Suppose that the strong solution of \eqref{eq:quasi_linear_{s}} is denoted by $(v_{d}^{k},q_{d}^{k},\xi_{d}^{k})$. We now aim to show that this solution converges to the solution of the following system
		\begin{equation}{\label{eq:quasi_linear_{s2}}}
			\begin{cases}
				\partial_{t}v_{d}+\operatorname{div}_{\mathcal{A}^{n}}S_{\mathcal{A}^{n}}(v_{d},q_{d})+\dive_{\mathcal{A}^{n}}\nabla_{\mathcal{A}}(R^{n}v_{l})+\p_{t}(R^{n}v_{l})=F^{1}(u,p,\eta)~~~&\operatorname{in}~~\Omega,\\
				\operatorname{div}_{\mathcal{A}^{n}}v_{d}=0~~~&\operatorname{in}~~\Omega,\\
				S_{\mathcal{A}^{n}}(q_{d},v_{d})\mathcal{N}^{n}+\nabla_{\mathcal{A}^{n}}(R^{n}v_{l})\mathcal{N}^{n}=g\xi_{d}\mathcal{N}^{n}-\sigma\partial_{1}(\frac{1}{\sqrt{1+\vert \partial_{1}\zeta_{0}\vert^{2}}}\partial_{1}\xi_{d})\mathcal{N}^{n}+\p_{1}I_{1}^{n}\mathcal{N}^{n}\\
				\quad\quad\quad\quad\quad\quad\quad+\sigma\partial_{1}(\int_{0}^{t}\mathcal{R}_{z}(\partial_{1}\zeta_{0},\partial_{1}\eta)(\partial_{1}\p_{t}\xi_{l})\p_{1}\p_{t}\eta)\mathcal{N}^{n}+\sigma\partial_{1}(\int_{0}^{t}\mathcal{R}_{z}(\partial_{1}\zeta_{0},\partial_{1}\eta)\partial_{1}\partial_{t}\xi_{d})\mathcal{N}^{n}\\
				\quad\quad\quad\quad\quad\quad\quad+F^{4}(u,p,\eta)~~~&\operatorname{on}~~\Sigma,\\
				(S_{\mathcal{A}^{n}}(q_{d},v_{d})\nu+\nabla_{\mathcal{A}^{n}}(R^{n}v_{l})\nu-\beta v_{d})\cdot \tau=F^{5}(u,\eta,p)~~~&\operatorname{on}~~\Sigma_{s},\\
				v_{d}\cdot \nu=0~~~&\operatorname{on}~~\Sigma_{s},\\
				\partial_{t}\xi_{d}=v_{d}\cdot \mathcal{N}^{n}+(R^{n}v_{l})\cdot \mathcal{N}^{n}+\int_{0}^{t}(\partial_{t}u_{1}\cdot\partial_{t}\p_{1}\xi_{l}+u_{1}\partial_{1}\partial_{t}\xi_{d})+I_{2}^{n}~~~&\operatorname{on}~~\Sigma,\\
				\sigma(\mp \frac{\partial_{1}\xi_{d}}{(1+\vert \partial_{1}\zeta_{0}\vert^{2})^{\frac{3}{2}}}\pm \int_{0}^{t}\mathcal{R}_{z}(\partial_{1}\zeta_{0},\partial_{1}\eta)\partial_{t}\partial_{1}\xi_{d})(\pm \ell)\pm\int_{0}^{t}\mathcal{R}_{z}(\p_{1}\zeta_{0},\p_{1}\eta)(\p_{t}\p_{1}\xi_{l})\p_{t}\p_{1}\eta\pm I_{1}^{n})\\
				\qquad\qquad\qquad\qquad\qquad\qquad\qquad=\kappa (v_{d}\cdot \mathcal{N}^{n})(\pm \ell)-{F}^{7},
			\end{cases}
		\end{equation}
		as $k\to \infty$, where
		$I_{1}^{n}=\p_{1}(\mathcal{R}_{z}(\p_{1}\zeta_{0},\p_{1}\eta(0))\p_{1}\p_{t}\xi_{0}^{n})~~~~\operatorname{and }~~~~~~~I_{2}^{n}=u(0)\p_{1}\p_{t}\xi_{0}^{n}$.
		To prove this convergence result, we need to first establish a uniform bound for $(v_{d}^{k},q_{d}^{k},\xi_{d}^{k})$. Observing that the bounds for $(v_{d}^{k},\xi_{d}^{k},q_{d}^{k})$ depend on $(\vert \vert u^{k}\vert \vert_{L^{\infty}H^{\frac{5}{2}}}$, $\vert \vert \eta^{k}\vert \vert_{L^{\infty}H^{3}})$, which may blow up as $k\rightarrow \infty$, we modify the estimates for several terms in Steps 3, 4 and 6 to remove the dependency.
		
		For the energy-dissipation estimate, we only modify the bounds for \eqref{eq:n1}, \eqref{eq:n2} and \eqref{eq:n0}, while keeping the estimates for all other terms unchanged. For \eqref{eq:n1}, we estimate 
		\begin{align}
			(\partial_{t}\xi_{d}^{m},\partial_{t}u_{1}^{k}\partial_{t}\partial_{1}\xi_{l}^{k})_{1,\Sigma_{k}}\lesssim& \vert \vert \partial_{t}\xi_{d}^{m}\vert \vert_{W^{1,\frac{1}{ \varepsilon_{-}}}}\vert \vert \partial_{t}u_{1}^{k}\vert \vert_{W^{1,\frac{1}{1-\varepsilon_{-}}}(\Sigma)}\vert \vert \partial_{t}\xi_{l}^{k}\vert \vert_{W^{1,+\infty}}
			+\vert \vert \partial_{t}\xi_{d}^{m}\vert \vert_{W^{1,\frac{1}{\varepsilon_{-}}}} \vert \vert  \p_{t}u_{1}^{k}\vert\vert_{L^{\infty}(\Sigma)}\vert \vert \partial_{t}\xi_{l}^{k}\vert \vert_{W^{2,\frac{1}{1-\varepsilon_{-}}}}\notag\\\lesssim& \vert \vert \partial_{t}\xi_{d}^{m}\vert \vert_{H^{\frac{3}{2}-\alpha}}\vert \vert \partial_{t}u^{k}\vert \vert_{W^{2,q_{-}}}\vert \vert \partial_{t}\xi_{l}^{k}\vert \vert_{H^{\frac{3}{2}+\frac{\varepsilon_{-}-\alpha}{2}}}+ \vert \vert \partial_{t}\xi_{d}^{m}\vert \vert_{H^{\frac{3}{2}-\alpha}}\vert \vert \partial_{t}u^{k}\vert \vert_{H^{1+\frac{\varepsilon_{-}}{2}}}\vert \vert \partial_{t}\xi_{l}^{k}\vert \vert_{W^{3-\frac{1}{q_{-}},q_{-}}}\label{eq:n3}.
		\end{align}
		Similarly for \eqref{eq:n2}, we have
		\begin{align}
			(\partial_{t}\xi_{d}^{k},u_{1}^{k}\partial_{1}\partial_{t}\xi_{d}^{k})_{1,\Sigma_{k}}\notag
			&=\int_{-\ell}^{\ell}g\partial_{t}\xi_{d}^{k}u_{1}^{k}\partial_{1}\partial_{t}\xi_{d}^{k}+\sigma\frac{\p_1\partial_{t}\xi_{d}^{k}u_{1}^{k}\partial_{1}^{2}\partial_{t}\xi_{d}^{k}}{(1+|\p_1\zeta_0|^2)^{3/2}}+\int_{-\ell}^{\ell}\mathcal{R}_{z}(\partial_{1}\zeta_{0},\partial_{1}\eta^{k})u_{1}^{k}\partial_{t}\partial_{1}\xi_{d}^{k}\partial_{1}^{2}\partial_{t}\xi_{d}^{k}\notag\\
			&\quad+\int_{-\ell}^{\ell}g\partial_{t}\xi_{d}^{k}\partial_{1}u_{1}^{k}\partial_{t}\xi_{d}^{k}+\sigma\frac{\p_1\partial_{t}\xi_{d}^{k}\partial_{1}u_{1}^{k}\partial_{1}\partial_{t}\xi_{d}^{k}}{(1+|\p_1\zeta_0|^2)^{3/2}}+\int_{-\ell}^{\ell}\mathcal{R}_{z}(\partial_{1}\zeta_{0},\partial_{1}\eta^{k})\partial_{1}u_{1}^{k}\partial_{t}\partial_{1}\xi_{d}^{k}\partial_{1}\partial_{t}\xi_{d}^{k}\notag\\
			&=-\frac{1}{2}\int_{-\ell}^{\ell}\mathcal{R}_{zz}(\partial_{1}\zeta_{0},\partial_{1}\eta^{k})u_{1}^{k}\partial_{t}\partial_{1}\xi_{d}^{k}\partial_{1}\partial_{t}\xi_{d}^{k}(\partial_{1}^{2}\zeta_{0}+\partial_{1}^{2}\eta^{k})\notag\\
			&\quad+\frac{1}{2}\int_{-\ell}^{\ell}g\partial_{t}\xi_{d}^{k}\partial_{1}u_{1}^{k}\partial_{t}\xi_{d}^{k}+\sigma\frac{\p_1\partial_{t}\xi_{d}^{k}\partial_{1}u_{1}^{k}\partial_{1}\partial_{t}\xi_{d}^{k}}{(1+|\p_1\zeta_0|^2)^{3/2}}+\frac{1}{2}\int_{-\ell}^{\ell}\mathcal{R}_{z}(\partial_{1}\zeta_{0},\partial_{1}\eta^{k})\partial_{1}u_{1}^{k}\partial_{t}\partial_{1}\xi_{d}^{k}\partial_{1}\partial_{t}\xi_{d}^{k}\notag\\
			&\lesssim \vert \vert \partial_{t}\xi_{d}^{k}\vert \vert^{2}_{H^{\frac{3}{2}-\alpha}}(\vert \vert \eta^{k}\vert \vert_{W^{3-\frac{1}{q_{-}},q_{-}}}+\vert \vert u^{k}\vert\vert_{W^{1,\frac{1}{1-\varepsilon_{-}}}(\Sigma)}) \notag\\
			&\lesssim \vert \vert \partial_{t}\xi_{d}^{k}\vert \vert^{2}_{H^{\frac{3}{2}-\alpha}}(\vert \vert \eta^{k}\vert \vert_{W^{3-\frac{1}{q_{-}},q_{-}}}+\vert \vert u^{k}\vert\vert_{W^{2,q_{-}}}), \label{eq:n4}
		\end{align}
		and similarly
		\[
		\begin{aligned}
			(\partial_{t}\xi_{d}^{k},D_{t}(R^{n}v_{l}^{k})\cdot \mathcal{N}^{n})_{1,\Sigma_{k}}\lesssim &\|\p_{t}\xi_{d}^{k}\|_{H^{\frac{3}{2}-\alpha}}(\|\p_{t}\eta^{n}\|_{H^{\frac{3}{2}+\frac{\varepsilon_{-}-\alpha}{2}}}\|D_{t}v_{l}^{k}\|_{W^{2,q_{-}}}+\|\p_{t}\eta^{n}\|_{W^{3-\frac{1}{q_{-}},q_{-}}}\|D_{t}v_{l}^{k}\|_{1+\frac{\varepsilon_{-}}{2}}\\
			&\quad +\|\p_{t}^{2}\eta^{n}\|_{H^{2}}\|v_{l}^{k}\|_{W^{2,q_{-}}}+\|\p_{t}^{2}\eta^{n}\|_{H^{\frac{3}{2}-\alpha}}\|v_{l}^{k}\|_{W^{2,q_{-}}}).
		\end{aligned}
		\]
		
		For the uniform $\frac{3}{2}-\alpha$ estimate of free surface, we apply the functional calculus for gravity-capillary operator $\mathcal{K}$(Not the modified operator). Most of the computations are identical to those in Step 6 except for the estimate of the term including $\mathcal{R}_{z}(\p_{1}\zeta_{0},\p_{1}\eta^{k})$. We estimate this term as follows
		\begin{align}{\label{eq:n5}}
			\begin{aligned}
				&\int_{-\ell}^{\ell}\p_{1}D_{j}^{s}(\p_{t}\xi_{d}-a_{0}(t)\rho_{0})\mathcal{R}_{z}(\p_{1}\zeta_{0},\p_{1}\eta^{k})\p_{1}\p_{t}\xi_{d}^{k}\lesssim \|\p_{1}D_{j}^{s}(\p_{t}\xi_{d}-a_{0}(t)\rho_{0})\|_{H^{-\frac{s}{2}}}\|\mathcal{R}_{z}(\p_{1}\zeta_{0},\p_{1}\eta^{k})\|_{H^{\frac{1}{2}+}}\|\p_{1}\p_{t}\xi_{d}^{k}\|_{H^{\frac{s}{2}}}\\
				\lesssim& \big(\|\p_{t}\xi_{d}\|_{H^{\frac{3}{2}-\alpha}}+\|v_{d}\|_{H^{1}}+\|\p_{t}u^{k}\|_{L_{t}^{\infty}H^{1}}\|\p_{t}\xi_{l}^{k}\|_{L_{t}^{\infty}H^{1}}
				+\|u^{k}\|_{L_{t}^{\infty}H^{1}}\|\p_{t}\xi_{d}^{k}\|_{L_{t}^{\infty}H^{1}}+\|\p_{t}\eta^{n}\|_{W^{3-\frac{1}{q_{+}},q_{+}}}\|v_{l}^{k}\|_{W^{2,q_{-}}}\big)\\
				&
				\times\|\eta^{k}\|_{W^{3-\frac{1}{q_{-}},q_{-}}}\|\p_{t}\xi_{d}^{k}\|_{H^{\frac{3}{2}-\alpha}}.
			\end{aligned}
		\end{align}
		
		We now substitute estimates \eqref{eq:n3} and \eqref{eq:n4} into the bounds obtained in Step 3, and apply \eqref{eq:n5} in the 
		$\frac{3}{2}-\alpha$ estimate. 
		Combing these estimates, we obtain the following uniform bound for $(v_{d}^{k},\xi_{d}^{k},q_{d}^{k})$:
		\begin{equation}\label{est:uniform_n1}
			\begin{aligned}
				&\sum_{i=0}^{1}\bigg(\sup_{0\le t\le T}
				\bigl(
				\|\partial_{t}^{i}v_{d}^{m}\|_{0}^{2}
				+\|\partial_{t}^{i}\xi_{d}^{m}\|_{1}^{2}
				\bigr)
				+\|\partial_{t}^{i}v_{d}^{m}\|_{L^{2}H^{1}}^{2}+\|[\p_{t}^{i}v_{d}^{m}\cdot \mathcal{N}^{n}]_{\ell}\|_{L_t^2}^{2}+\|\p_{t}^{i}\xi_{d}\|_{L_{t}^{2}H^{\frac{3}{2}-\alpha}}\bigg)\\
				&\quad+\|v_{d}\|_{L^2W^{2,q_-}}^2+\|q_{d}\|_{L^2W^{1,q_-}}^2+\|\xi_{d}\|_{L^2W^{3-1/q_-,q_-}}^2\\
				&\lesssim\exp\{(\|\p_t\eta^{n}\|_{L^\infty H^{3/2+(\varepsilon_--\alpha)/2}}+\vert \vert \eta^{n}\vert \vert_{L^{\infty}H^{3}})T\}\bigg(\mathcal{E}(0)+\|(F^1-F^4-F^5) +\mathfrak{F}^{n}(0)\|_{(\mathcal{H}^1)^\ast}^2\\
				&\quad+\mathfrak{K}(u^{k},p^{k},\eta^{k})\mathfrak{K}_{-}(v_{l}^{k},\xi_{l}^{k})+(T^{\frac{1}{2}}\|\p_{t}^{2}\eta^{n}\|^{2}_{L_{t}^{\infty}H^{\frac{5}{2}}}+T^{\frac{1}{2}}\|\p_{t}^{3}\bar{\eta}^{n}\|_{L_{t}^{\infty}H^{1}}^{2})(\mathfrak{K}(u^{k},p^{k},\eta^{k})+\mathfrak{K}_{-}(v_{l}^{k},\xi_{l}^{k}))\bigg),
			\end{aligned}
		\end{equation}
		where $\mathfrak{K}_{-}$ is defined by \eqref{def:DEK_0-} in the next step.
		\noindent Using the convergence result established in \eqref{eq:convergence_e}, we obtain the following bound from \eqref{est:uniform_n1}.
		\begin{equation}\label{est:uniform_n2}
			\begin{aligned}
				&\sum_{i=0}^{1}\bigg(\sup_{0\le t\le T}
				\bigl(
				\|\partial_{t}^{i}v_{d}^{m}\|_{0}^{2}
				+\|\partial_{t}^{i}\xi_{d}^{m}\|_{1}^{2}
				\bigr)
				+\|\partial_{t}^{i}v_{d}^{m}\|_{L^{2}H^{1}}^{2}+\|[\p_{t}^{i}v_{d}^{m}\cdot \mathcal{N}^{n}]_{\ell}\|^{2}+\|\p_{t}^{i}\xi_{d}\|_{L_{t}^{2}H^{\frac{3}{2}-\alpha}}\bigg)\\
				&\quad+\|v_{d}\|_{L^2W^{2,q_-}}^2+\|q_{d}\|_{L^2W^{1,q_-}}^2+\|\xi_{d}\|_{L^2W^{3-1/q_-,q_-}}^2\\
				&\lesssim\exp\{(\|\p_t\eta^{n}\|_{L^\infty H^{3/2+(\varepsilon_--\alpha)/2}}+\vert \vert \eta^{n}\vert \vert_{L^{\infty}H^{3}})T\}\bigg(\mathcal{E}(0)+\|(F^1-F^4-F^5)(0)\|_{(\mathcal{H}^1)^\ast}^2 +\mathfrak{F}^{n}\\
				&\quad+\mathfrak{K}(u,p,\eta)\mathfrak{K}_{-}(v_{l},\xi_{l})+(T^{\frac{1}{2}}\|\p_{t}^{2}\eta^{n}\|^{2}_{L_{t}^{\infty}H^{\frac{5}{2}}}+T^{\frac{1}{2}}\|\p_{t}^{3}\bar{\eta}^{n}\|_{L_{t}^{\infty}H^{1}}^{2})(\mathfrak{K}(u,p,\eta)+\mathfrak{K}_{-}(v_{l},\xi_{l}))\bigg)
				.
			\end{aligned}
		\end{equation} 
		
		We now have a uniform bound for $(v^{k}_{d},q^{k}_{d},\xi^{k}_{d})$ and $(\partial_{t}v_{d}^{k},\partial_{t}\xi_{d}^{k})$ independent of $k$. This uniform boundedness implies that, up to the extraction of a subsequence,
		\[
		(v_{d}^{k},\xi_{d}^{k},q_{d}^{k})\rightharpoonup (v_{d},\xi_{d},q_{d})~~\operatorname{in}~~(L^{2}W^{2,q_{-}}\times L^{2}W^{3-\frac{1}{q_{-}},q_{-}}\times L^{2}W^{1,q_{-}}),
		\]
		\[ (v_{d}^{k},\xi_{d}^{k},q_{d}^{k})\stackrel{\ast}\rightharpoonup (v_{d},\xi_{d},q_{d})~~\operatorname{in}~~(L^{\infty}H^{1}\times L^{\infty}H^{1}\times L^{\infty}L^{2}),
		\]
		\[
		(\partial_{t} v_{d}^{k},\partial_{t}\xi_{d}^{k})\stackrel{\ast}\rightharpoonup (\partial_{t}v_{d},\partial_{t}\xi_{d}) \ \text{in}\ L^{\infty}L^{2}\times L^{\infty}H^{1}),\  (\partial_{t} v_{d}^{k},\partial_{t}\xi_{d}^{k})\rightharpoonup (\partial_{t}v_{d},\partial_{t}\xi_{d})~~\operatorname{in}~~( (L_{t}^{2}H^{1}\times L_{t}^{2}H^{\frac{3}{2}-\alpha}).
		\]
		
		\noindent From the weak convergence above, together with the strong convergence of $\eta^{k}$ to $\eta$, we conclude that $(v_d,\xi_{d})$ is the weak solution to the following equation
		\begin{equation}\label{eq:weak_limit1}
			\begin{aligned}
				& (\p_tv_{d}, w)_{\mathcal{H}^0}+((v_{d},w))+ (\xi_{d}, w\cdot\mathcal{N}^{n})_{1,\Sigma_{0}}+(\int_{0}^{t}\mathcal{R}_{z}(\partial_{1}\zeta_{0},\partial_{1}\eta)\partial_{1}\p_{t}\xi_{d},\partial_{1}(w\cdot\mathcal{N}^{n}))_{L^{2}} + [v_{d}\cdot\mathcal{N}^{n},w\cdot\mathcal{N}^{n}]_\ell\\ &
				=\int_\Om F^1\cdot wJ^{n}-\int_{-\ell}^{\ell} F^4\cdot w-\int_{\Sigma_s}F^5(w\cdot\tau)J^{n} -[F^7,w\cdot\mathcal{N}^{n}]_\ell +(\p_{t}(R^{n}v_{l}),w)_{\mathcal{H}^{0}}+((R^{n}v_{l},w))\\
				&\quad-(I_{1}^{n},\p_{1}(w\cdot \mathcal{N}^{n}))_{L^{2}(-\ell.\ell)} +(\p_{1}\int_{0}^{t}\mathcal{R}_{z}(\partial_{1}\zeta_{0},\partial_{1}\eta)\partial_{1}\p_{t}\xi_{l}\p_{t}\p_{1}\eta,(w\cdot\mathcal{N}^{n}))_{L^{2}}\\
				&\quad+\left[\int_{0}^{t}\mathcal{R}_{z}(\partial_{1}\zeta_{0},\partial_{1}\eta)\partial_{1}\p_{t}\xi_{l}\p_{t}\p_{1}\eta,w\cdot \mathcal{N}^{n} \right]_\ell.
			\end{aligned}
		\end{equation}
		For the kinematic boundary condition, we note that the kinematic boundary condition for smooth-value problem is expressed as
		\begin{align}{\label{eq:kinematic_s}}
			\partial_{t}\xi_{d}^{k}=v_{d}^{k}\cdot \mathcal{N}^{n}+R^{n}v_{l}^{k}\cdot \mathcal{N}^{n}+\int_{0}^{t}\partial_{t}u_{1}^{k}\cdot \partial_{t}\partial_{1}\xi_{l}^{k}+\int_{0}^{t}u_1^{k}\partial_{1}\partial_{t}\xi_{d}+I^{n,k}_{2}.
		\end{align}
		
		\noindent For the right hand side of the equation above, we have the following estimate
		\[\begin{aligned}
			\vert \vert v_{d}^{k}\cdot \mathcal{N}+R^{n}v_{l}^{k}+\int_{0}^{t}\partial_{t}u_{1}^{k}\cdot \partial_{t}\partial_{1}\xi_{l}^{k}+\int_{0}^{t}u_1^{k}\partial_{1}\partial_{t}\xi_{d}^{k}+I_{2}^{n,k}\vert \vert_{L_{t}^{\infty}H^{\frac{1}{2}-\alpha}}
			\lesssim \vert \vert \partial_{t}u\vert \vert_{L^{\infty}W^{2,q_{-}}}\vert \vert \partial_{t}\xi_{l}^{k}\vert \vert_{L^{\infty}W^{3-\frac{1}{q_{-}},q_{-}}}\\
			+\vert \vert v_{d}^{n}\vert \vert_{L_{t}^{\infty}H^{1}}+\vert \vert u^{k}\vert \vert_{L^{\infty}W^{2,q_{-}}}\vert \vert \partial_{t}\xi_{d}^{k}\vert\vert_{L_{t}^{2}H^{\frac{3}{2}-\alpha}}+\vert \vert I_{2}^{n,k}\vert \vert_{H^{\frac{1}{2}-\alpha}}
			+\|\p_{t}\eta^{n}\|_{L_{t}^{2}W^{3-\frac{1}{q_{-}},q_{-}}}\|v_{l}^{k}\|_{L_{t}^{\infty}W^{2,q_{-}}}.
		\end{aligned}\]
		
		\noindent Therefore, both the left- and right-hand sides of the kinematic boundary condition for the smooth, given-data problem are uniformly bounded in $L^{\infty}H^{\frac{1}{2}-\alpha}$. Using this uniform boundedness and the compact embedding theorem, we infer that equation \eqref{eq:kinematic_s} converges to:
		\[
		\partial_{t}\xi_{d}=v_{d}\cdot \mathcal{N}^{n}+R^{n}v_{l}\cdot \mathcal{N}^{n}+\int_{0}^{t}\partial_{t}u_{1}\cdot \partial_{t}\partial_{1}\xi_{l}+\int_{0}^{t}u_1\partial_{1}\partial_{t}\xi_{d}+I^{n}_{2}
		\]
		\noindent strongly in $L^{\infty}L^{2}$ as $k\rightarrow \infty$. Therefore,  $(\xi_{d},v_{d})$ is the weak solution to \eqref{eq:quasi_linear_{s2}}. Moreover, since $(v_{d},\xi_{d},q_{d})\in (L^{2}W^{2,q_{-}}\times L^{2}W^{3-\frac{1}{q_{-}},q_{-}}\times L^{2}W^{1,q_{-}})$, this weak solution is a strong solution to \eqref{eq:quasi_linear_{s2}} and is bounded by the right-hand side of \eqref{est:uniform_n2}.
		
		{\paragraph{\underline{Step 9 -- Contraction Map}}
			
			In this step, we fix $n$ and prove that system \eqref{eq:quasi_linear_{s2}} induces a linear map from $(v_{l},\xi_{l})$ to $(v,\xi)$.
			
			Setting $v_{d}=D_{t}v, q_{d}=\p_{t}q, \xi_{d}=\p_{t}\xi$, and combining this with the initial data constructed in Appendix \ref{sec:initial_l}, we can recover $(v, q, \xi)$.
			
			To construct the contraction map, we introduce the following definitions:
			\begin{equation}\label{def:dissipation-}
				\begin{aligned}
					\mathscr{D}_{-}(u,\eta)&:=(\|u\|_{L^2W^{2,q_-}}^2+\|\eta\|_{L^2W^{3-1/q_-, q_-}}^2)+\sum_{j=0}^2\Big(\|\p_t^ju\|_{L^{2}H^{1}}^2+\|\p_t^ju\|_{L^{2}L^2(\Sigma_s)}^2\Big) +\|\p_t^3\eta\|_{L^{2}H^{1/2-\alpha}}^2\\&\quad
					+\sum_{j=0}^2\Big(\|\p_t^j\eta\|_{L^{2}H^{3/2-\alpha}}^2+\|[\p_t^{j}u\cdot \mathcal{N}^{n}]_\ell^2\|_{L_{t}^{2}}\Big)+\Big(\|\p_tu\|_{L^{2}W^{2,q_{-}}}^2+\|\p_{t}\eta\|_{L^{2}W^{3-1/q_-, q_-}}^2\Big),
				\end{aligned}
			\end{equation}
			\begin{equation}\label{def:energy-}
				\begin{aligned}
					\mathscr{E}_{-}(u,\eta)&: =\|u\|_{L^{\infty}W^{2,q_-}}^2+\sum_{i=0}^{1}\|\p_t^{i}u\|_{L_{t}^{\infty}H^{1+\varepsilon_-/2}}^2+\sum_{k=0}^2\|\p_t^ku\|_{L_{t}^{\infty}H^0}^2\\&\quad+\|\eta\|_{L_{t}^{\infty}W^{3-1/q_-, q_-}}^2+\sum_{i=0}^{1}\|\p_t^{i}\eta\|_{L_{t}^{\infty}H^{3/2+(\varepsilon_--\alpha)/2}}^2+\sum_{j=0}^2(\|\p_t^j\eta\|_{L_{t}^{\infty}H^{1}}^2+\|[\p_{t}^{j}u\cdot \mathcal{N}^{n}]_{\ell}\|_{L_{t}^{\infty}}^{2}),
				\end{aligned}
			\end{equation}
			\begin{equation}\label{def:DEK_0-}
				\mathfrak{K}_{-}(u,\eta ):=\mathscr{E}_{-}(u,\eta)+\mathscr{D}_{-}(u,\eta).
			\end{equation}
			\noindent together with the functional space
			\begin{align}
				\mathscr{H}:=\{(v,\xi)\in L_{t}^{\infty}\mathcal{W}_{\sigma}\times L_{t}^{\infty}H^{1}|\mathfrak{K}_{-}(v,\xi)<+\infty,~v(0,x)=v_{0}^{n},~\p_{t}v(0,x)=\p_{t}v_{0}^{n},~\xi(0,x)=\xi_{0}^{n},\p_{t}\xi(0,x)=\p_{t}\xi_{0}^{n}\}
			\end{align}
			\noindent where $v_{0}^{n},\p_{t}v_{0}^{n},\xi_{0}^{n},\p_{t}\xi_{0}^{n}$ are defined as in Appendix \ref{sec:initial_l}. 
			
			{We first prove that the space $\mathscr{H}$ is complete. Since $\mathfrak{K}_{-}(u,\eta)$ induces a complete Sobolev space,it suffices to show that the limit satisfies the fixed boundary condition. 
				
				When $(v,\xi)\in \mathscr{H}$, it holds that
				\[
				\|\p_{t}v\|_{L_{t}^{2}W^{2,q_{-}}}+\|\p_{t}\xi\|_{L_{t}^{2}W^{3-\frac{1}{q_{-}},q_{-}}}+\|\p_{t}^{2}v\|_{L_{t}^{2}H^{1}}+\|\p_{t}^{2}\xi\|_{L_{t}^{2}H^{\frac{3}{2}-\alpha}}<\mathfrak{K}_{-}(v,\xi)<\infty
				\]
				which implies that
				$(v,\xi,\p_{t}v,\p_{t}\xi)\in H_{t}^{1}W^{2,q_{-}} \times H_{t}^{1}W^{3-\frac{1}{q_{-}},q_{-}}\times H_{t}^{1}H^{1}\times H_{t}^{1}H^{\frac{3}{2}-\alpha}$.
				Therefore, any Cauchy sequence $\{(v^{r},\xi^{r})\}$ in $\mathscr{H}$ implies that $\{(v^{r},\xi^{r},\p_{t}v^{r},\p_{t}\xi^{r})\}$ is a Cauchy sequence in 
				$H_{t}^{1}W^{2,q_{-}} \times H_{t}^{1}W^{3-\frac{1}{q_{-}},q_{-}}\times H_{t}^{1}H^{1}\times H_{t}^{1}H^{\frac{3}{2}-\alpha}$.
				This implies that the limit of the Cauchy sequence satisfies the initial condition required in the definition of the space $\mathscr{H}$. Therefore, $\mathscr{H}$ is a complete space.
			}
			
			We denote the metric induced by norm $\mathfrak{K}_{-}(u,\eta)$ as $d(u,\eta)$. Then the estimate \eqref{est:uniform_n2} implies that
			\[
			\begin{aligned}
				\mathfrak{K}_{-}(v,\xi )\lesssim& \exp\{(\|\p_{t}\eta^{n}\|_{L^{\infty}H^{\frac{3}{2}+\frac{\varepsilon_{-}}{2}}}+\|\eta^{n}\|_{L^{\infty}H^{3}})T\}\big(\mathscr{E}(u(0),p(0),\eta(0))\\
				&+ \mathfrak{F}^n +(\delta+(T^{\frac{1}{2}}\|\p_{t}^{2}\eta^{n}\|^{2}_{L_{t}^{\infty}H^{\frac{5}{2}}}+T^{\frac{1}{2}}\|\p_{t}^{3}\bar{\eta}^{n}\|_{L_{t}^{\infty}H^{1}}^{2})) \mathfrak{K}_{-}(v_{l},\xi_{l})\big),
			\end{aligned}
			\]
			\noindent where for simplicity, $\mathfrak{F}$ denotes the norms of the forcing terms.
			
			Moreover, given two different tuples $(v_{l}^{1},\xi_{l}^{1})$ and $(v_{l}^{2},\xi_{l}^{2})$. We subtract the respective evaluations of equation \eqref{eq:quasi_linear_{s2}} to obtain a new system for the difference functions $(v^{1}-v^{2},\xi^{1}-\xi^{2},q^{1}-q^{2})$:
			\[
			\begin{cases}
				\partial_{t}(v_{d}^{1}-v_{d}^{2})+\operatorname{div}_{\mathcal{A}^{n}}S_{\mathcal{A}^{n}}(v_{d}^{1}-v_{d}^{2},q_{d}^{1}-q_{d}^{2})+\dive_{\mathcal{A}^{n}}\nabla_{\mathcal{A}}(R^{n}(v_{l}^{1}-v_{l}^{2}))+\p_{t}(R^{n}(v_{l}^{1}-v_{l}^{2}))=0~~~&\operatorname{in}~~\Omega,\\
				\operatorname{div}_{\mathcal{A}^{n}}(v_{d}^{1}-v_{d}^{2})=0~~~&\operatorname{in}~~\Omega,\\
				S_{\mathcal{A}^{n}}(q_{d}^{1}-q_{d}^{2},v_{d}^{1}-v_{d}^{2})\mathcal{N}^{n}+\nabla_{\mathcal{A}^{n}}(R^{n}(v_{l}^{1}-v_{l}^{2}))\mathcal{N}^{n}=g(\xi_{d}^{1}-\xi_{d}^{2})\mathcal{N}^{n}-\sigma\partial_{1}(\frac{1}{\sqrt{1+\vert \partial_{1}\zeta_{0}\vert^{2}}}\partial_{1}(\xi_{d}^{1}-\xi_{d}^{2}))\mathcal{N}^{n}\\
				\quad\quad\quad\quad\quad\quad\quad+\sigma\partial_{1}(\int_{0}^{t}\mathcal{R}_{z}(\partial_{1}\zeta_{0},\partial_{1}\eta)(\partial_{1}\p_{t}(\xi_{l}^{1}-\xi_{l}^{2}))\p_{1}\p_{t}\eta)\mathcal{N}^{n}\\
				\quad\quad\quad\quad\quad\quad\quad+\sigma\partial_{1}(\int_{0}^{t}\mathcal{R}_{z}(\partial_{1}\zeta_{0},\partial_{1}\eta)\partial_{1}\partial_{t}(\xi_{d}^{1}-\xi_{d}^{2}))\mathcal{N}^{n}~~~&\operatorname{on}~~\Sigma,\\
				(S_{\mathcal{A}^{n}}(q_{d}^{1}-q_{d}^{2},v_{d}^{1}-v_{d}^{2})\nu+\nabla_{\mathcal{A}^{n}}(R^{n}(v_{l}^{1}-v_{l}^{2}))\nu-\beta (v_{d}^{1}-v_{d}^{2})\cdot \tau=0~~~&\operatorname{on}~~\Sigma_{s},\\
				(v_{d}^{1}-v_{d}^{2})\cdot \nu=0~~~&\operatorname{on}~~\Sigma_{s},\\
				\partial_{t}(\xi_{d}^{1}-\xi_{d}^{2})=(v_{d}^{1}-v_{d}^{2})\cdot \mathcal{N}^{n}+(R^{n}(v_{l}^{1}-v_{l}^{2}))\cdot \mathcal{N}^{n}+\int_{0}^{t}(\partial_{t}u_{1}\cdot\partial_{t}\p_{1}(\xi_{l}^{1}-\xi_{l}^{2})+u_{1}\partial_{1}\partial_{t}(\xi_{d}^{1}-\xi_{d}^{2}))~~~&\operatorname{on}~~\Sigma,\\
				\sigma(\mp \frac{\partial_{1}(\xi_{d}^{1}-\xi_{d}^{2})}{(1+\vert \partial_{1}\zeta_{0}\vert^{2})^{\frac{3}{2}}}\pm \int_{0}^{t}\mathcal{R}_{z}(\partial_{1}\zeta_{0},\partial_{1}\eta)\partial_{t}\partial_{1}(\xi_{d}^{1}-\xi_{d}^{2}))(\pm \ell)ds\pm\int_{0}^{t}\mathcal{R}_{z}(\p_{1}\zeta_{0},\p_{1}\eta)(\p_{t}\p_{1}(\xi_{l}^{1}-\xi_{l}^{2}))\p_{t}\p_{1}\eta)\\
				\quad\quad\quad\quad\quad\quad\quad=\kappa ((v_{d}^{1}-v_{d}^{2})\cdot \mathcal{N}^{n})(\pm \ell).
			\end{cases}
			\]
			Then, applying energy and elliptic estimates similar to those in the previous steps, we obtain
			\[
			\begin{aligned}
				\mathfrak{K}_{-}\bigl(
				v^{1}-v^{2},
				\xi^{1}-\xi^{2}
				\bigr)
				\;\lesssim&\;
				\delta\exp\{(\|\p_{t}\eta^{n}\|_{L^{\infty}H^{\frac{3}{2}+\frac{\varepsilon_{-}}{2}}}+\|\eta^{n}\|_{L^{\infty}H^{3}})T\}\,
				\bigg(\delta+(T^{\frac{1}{2}}\|\p_{t}^{2}\eta^{n}\|^{2}_{L_{t}^{\infty}H^{\frac{5}{2}}}\\
				& +T^{\frac{1}{2}}\|\p_{t}^{3}\bar{\eta}^{n}\|_{L_{t}^{\infty}H^{1}}^{2}) \cdot \mathfrak{K}_{-}\bigl(
				v_{l}^{1}-v_{l}^{2},
				\xi_{l}^{1}-\xi_{l}^{2}
				\bigr)\bigg),
			\end{aligned}
			\]
			for the same $\delta\in(0,1)$.
			
			Together with the estimates obtained above, this inequality implies that the
			Picard iteration induced by \eqref{eq:quasi_linear_{s}} defines a contraction
			mapping in the functional space $\mathscr{H}$. Consequently, for each fixed $n$, the iteration admits a unique fixed
			point $(v^{n},\xi^{n},q^{n})$, which solves the system \eqref{eq:quasi_linear_n}
			
			Moreover, the solution satisfies the estimate
			\[
			\begin{aligned}
				&\mathfrak{K}_{-}(v^{n},\xi^{n})+\|q^{n}\|_{L^{\infty}W^{1,q_{-}}}+\|\p_{t}q^{n}\|_{L^{2}W^{1,q_{-}}}+\|\p_{t}q^{n}\|_{L_{t}^{\infty}H^{0}}
				\\
				\;\lesssim\;&
				\exp\!\Bigl\{
				\bigl(
				\|\partial_{t}\eta^{n}\|_{L^{\infty}H^{3/2+(\varepsilon_{-}-\alpha)/2}}
				+\|\eta^{n}\|_{L^{\infty}H^{3}}
				\bigr)T
				\Bigr\}\bigg\{ \mathcal{E}(u_{0},p_{0},\xi_{0}) \\
				&+\|(F^1-F^4-F^5)(0)\|^{2}_{(\mathcal{H}^1)^\ast}+\mathfrak{F}^{n}+T^{\frac{1}{2}}\|\p_{t}^{2}\eta^{n}\|^{2}_{L_{t}^{\infty}H^{\frac{5}{2}}}+T^{\frac{1}{2}}\|\p_{t}^{3}\bar{\eta}^{n}\|_{L_{t}^{\infty}H^{1}}^{2})(\mathfrak{K}(u,p,\eta))\bigg\}
				.
			\end{aligned}
			\]
			
			\paragraph{\underline{Step 10 -- Uniform $n$ Bound}}
			In this step, we derive a uniform bound for the solution to \eqref{eq:quasi_linear_{s1}}.
			
			First, for any prescribed $n$, we apply the Theorem \ref{thm:pressure_+} to equation \eqref{eq:quasi_linear_n} to show that
			\[
			\begin{aligned}
				&\| v^{n}\|_{L^\infty W^{2,q_+}}^2 + \|q^{n}\|_{L^\infty W^{1,q_+}}^2 + \|\xi^{n}\|_{L^\infty W^{3-1/q_+,q_+}}^2 \\
				&\lesssim  \mathcal{Z}+\exp\!\Bigl\{
				\bigl(
				\|\partial_{t}\eta^{n}\|_{L^{\infty}H^{3/2+(\varepsilon_{-}-\alpha)/2}}
				+\|\eta^{n}\|_{L^{\infty}H^{3}}
				\bigr)T
				\Bigr\}\bigg\{ \mathcal{E}(u_{0},p_{0},\xi_{0}) \\
				&\quad\quad\quad\quad+\|(F^1-F^4-F^5)(0)\|^{2}_{(\mathcal{H}^1)^\ast}+\mathfrak{F}^{n}+T^{\frac{1}{2}}\|\p_{t}^{2}\eta^{n}\|^{2}_{L_{t}^{\infty}H^{\frac{5}{2}}}+T^{\frac{1}{2}}\|\p_{t}^{3}\bar{\eta}^{n}\|_{L_{t}^{\infty}H^{1}}^{2})(\mathfrak{K}(u,p,\eta))\bigg\}.
			\end{aligned}
			\]
			With this $q_{+}$ estimate, we define a function $w$ as follows:
			\[
			\operatorname{div}_{\mathcal{A}}w=J^{n}(-2\operatorname{div}_{\p_{t}\mathcal{A}^{n}}\p_{t}v^{n}-\operatorname{div}_{\mathcal{A}^{n}}\p_{t}^{2}v^{n})-\langle J^{n}(-2\operatorname{div}_{\p_{t}\mathcal{A}^{n}}\p_{t}v^{n}-\operatorname{div}_{\mathcal{A}^{n}}\p_{t}^{2}v^{n})\rangle,
			\]
			where $\langle f\rangle:=\frac{1}{|\Omega|}\int_{\Omega} f$.
			With the $q_{+}$ estimate, the function $w$ is well-defined and enjoys good regularity. The details are given in \cite[Proposition 10.1]{GT2020}.
			
			After differentiating \eqref{eq:quasi_linear_n} with respect to time, we take $\partial_{t}^{2}v^{n} - w$ as the test function. Combining the results of \cite[Section 10.2]{GT2020}, the $\frac{3}{2}-\alpha$ estimate, and the elliptic estimates derived in Step 8, we obtain the following bound:
			\[
			\begin{aligned}
				\mathfrak{K}(v^{n},\xi^{n},q^{n})\lesssim&(1+T)\exp\big\{T(\|\p_t\eta\|_{L^\infty H^{3/2+(\varepsilon_--\alpha)/2}}+\|\eta\|_{L^{\infty}W^{3-\frac{1}{q_{-}},q_{-}}})\big\}\\
				&\quad\quad \times\bigg\{ \mathcal{E}(u_{0},p_{0},\xi_{0}) +\|(F^1-F^4-F^5)(0)\|^{2}_{(\mathcal{H}^1)^\ast}+\mathfrak{F}^{n}+\mathcal{Z}+\mathfrak{K}(u,p,\eta)\bigg\}.
			\end{aligned}
			\]
			\noindent This uniform bound implies the existence of a time \(T>0\), independent of \(n\), such that \((v^{n},\xi^{n},q^{n})\) exists on the interval \((0,T)\) for every \(n\). Moreover, combining this estimate with the convergence result \eqref{convergence_n}, we obtain the following uniform bound
			\begin{align}
				\begin{aligned}
					\mathfrak{K}(v^{n},\xi^{n},q^{n})\lesssim&(1+T)\exp\big\{T(\|\p_t\eta\|_{L^\infty H^{3/2+(\varepsilon_--\alpha)/2}}+\|\eta\|_{L^{\infty}W^{3-\frac{1}{q_{-}},q_{-}}})\big\}\\
					&\quad\quad \times \bigg\{ \mathcal{E}(u_{0},p_{0},\xi_{0}) +\|(F^1-F^4-F^5)(0)\|^{2}_{(\mathcal{H}^1)^\ast}+\mathfrak{F}+\mathcal{Z}+\mathfrak{K}(u,p,\eta)\bigg\}.
				\end{aligned}
			\end{align}
			\noindent With this uniform bound, which is independent of $n$, we may let $n\to\infty$.
			Using the same argument as Step~9 to establish convergence,
			the sequence $(v^{n},\xi^{n},q^{n})$ converges to a solution of
			\eqref{eq:quasi_linear}. This completes the proof.}
	\end{proof}
	
	Now, it remains to prove Lemma \ref{lem:difference_u_Dt_u} showing the difference between $D_{t}v_{d}$ and $\p_{t}v_{d}$.
	
	\begin{lemma}\label{lem:difference_u_Dt_u}
		
		$D_{t}=\p_{t}-R$ where $R$ is defined in Proposition \ref{prop:solid_boundary}. It holds that
		\[
		\begin{aligned}
			\|\p_tv_{d}-D_tv_{d}\|_{L^2W^{2,q_-}}^2\lesssim \vert \vert \p_{t}\eta\vert \vert_{L_{t}^{\infty}W^{3-\frac{1}{q_{-}},q_{-}}}^{2}\vert \vert v_{d}\vert \vert_{L^{2}W^{2,q_{-}}}^{2},\\
			\|\p_t v_{d}-D_tv_{d}\|_{L^2H^1}^2+\|\p_tv_{d}-D_tv_{d}\|_{L^2H^0(\Sigma_s)}^2\lesssim \vert \vert \p_{t}\eta\vert \vert^{2}_{L_{t}^{\infty}H^{\frac{3}{2}+\frac{\varepsilon_{-}-\alpha}{2}}}\|v_{d}\|_{L^2H^1}^2,\\
			\|\operatorname{div}_{\mathcal{A}}\partial_{t}v_{d}-\operatorname{div}_{\mathcal{A}}D_{t}v_{d}\|^{2}_{L_{t}^{2}L^{2}}\lesssim T\vert \vert \p_{t}\eta\vert \vert^{2}_{L_{t}^{\infty}H^{\frac{3}{2}+\frac{\varepsilon_{-}-\alpha}{2}}}\|v_{d}\|_{L^\infty H^1}^2.
		\end{aligned}
		\]
	\end{lemma}
	
	\begin{proof}
		
		For the first inequality, we have the following computation using the fact that $R\sim \partial_{t}\nabla\bar{\eta}$:
		\[\begin{aligned}
			\vert \vert Rv_{d}\vert \vert^{2}_{W^{2,q_{-}}}\lesssim& \vert \vert \nabla^{2}R v_{d}\vert \vert^{2}_{L^{q_{-}}}+\vert \vert \nabla R\nabla v_{d}\vert \vert_{L^{q_{-}}}^{2}+\vert\vert R\nabla^{2}v_{d}\vert\vert_{L^{q_{-}}}^{2}\notag\\
			\lesssim&\vert \vert \nabla^{2}R\vert \vert^{2}_{L^{q_{-}}}\vert \vert v_{d}\vert \vert^{2}_{L^{\infty}}+\vert \vert \nabla R\vert \vert^{2}_{L^{\frac{1}{1-{\varepsilon_{-}}}}}\vert \vert \nabla v_{d}\vert \vert_{L^{\frac{1}{1-\varepsilon_{-}}}}^{2}+\vert \vert R\vert \vert^{2}_{L^{\infty}}\vert \vert \nabla^{2}v_{d}\vert \vert_{L^{q_{-}}}\notag\\
			\lesssim& \vert \vert \partial_{t}\bar{\eta}\vert \vert^{2}_{W^{3,q_{-}}}\vert \vert v_{d}\vert \vert_{H^{1+\frac{\varepsilon_{-}}{2}}}^{2}+\vert \vert \partial_{t}\bar{\eta}\vert \vert_{W^{3,q_{-}}}^{2}\vert \vert v_{d}\vert \vert^{2}_{W^{2,q_{-}}}+\vert \vert \partial_{t}\bar{\eta}\vert \vert_{H^{2+\frac{\varepsilon_{-}-\alpha}{2}}}^{2}\vert \vert v_{d}\vert \vert_{W^{2,q_{-}}}.
		\end{aligned}\]
		
		\noindent Using the relation between $\bar{\eta}$ and $\eta$, we have:
		\[
		\vert \vert Rv_{d}\vert \vert^{2}_{L^{2}W^{2,q_{-}}}\lesssim \vert \vert \eta\vert \vert_{L_{t}^{\infty}W^{3-\frac{1}{q_{-}},q_{-}}}^{2}\vert \vert v_{d}\vert \vert_{L^{2}W^{2,q_{-}}}^{2}.
		\]
		Similarly, we derive the second inequality as follows
		\[\begin{aligned}
			\|\p_t v_{d}-D_tv_{d}\|_{L^2H^1}^2+\|\p_tv_{d}-D_tv_{d}\|_{L^2H^0(\Sigma_s)}^2\lesssim & \vert \vert Rv_{d}\vert \vert_{L^{2}H^{1}}^{2}\lesssim \vert \vert \nabla Rv_{d}\vert \vert^{2}_{L^{2}L^{2}}+\vert \vert R\nabla v_{d}\vert \vert_{L^{2}L^{2}}^{2}\notag\\
			\lesssim& \vert \vert \partial_{t}\bar{\eta}\vert \vert_{L^{\infty}W^{2,\frac{4}{2-(\varepsilon_{-}-\alpha)}}}^{2}\vert\vert v_{d}\vert \vert_{L^{2}L^{\frac{4}{\varepsilon_{-}-\alpha}}}^{2}+\vert \vert \partial_{t}\bar{\eta}\vert \vert^{2}_{L^{\infty}W^{1,+\infty}}\vert \vert v_{d}\vert \vert^{2}_{L^{2}H^{1}}\notag\\
			\lesssim& \|\p_{t}\eta\|^{2}_{L^{\infty}H^{2+\frac{\varepsilon_{-}-\alpha}{2}}}\|v_{d}\|_{L_{t}^{2}H^{1}}^{2}.
		\end{aligned}\]
		
		Finally, the third inequality can be derived by
		\[
		\|\operatorname{div}_{\mathcal{A}}\partial_{t}v_{d}-\operatorname{div}_{\mathcal{A}}D_{t}v_{d}\|^{2}_{L_{t}^{2}L^{2}}\lesssim T\vert \vert Rv_{d}\vert \vert_{L_{t}^{\infty}H^{1}}^{2}\lesssim T\vert \vert \eta\vert \vert^{2}_{L_{t}^{\infty}H^{\frac{3}{2}+\frac{\varepsilon_{-}-\alpha}{2}}}\|v_{d}\|_{L^\infty H^1}^2.
		\]
		Hence, we have proved the lemma.
	\end{proof}
	\begin{remark}
		The result of Lemma 4.6 continues to hold when $R$ is replaced by $R^{n}$.
	\end{remark}
	
	%%%%%%%%%%%%%%%%%%%%%%%%%%%%%%%%%%%%%%%%%%%%%%
	\section{Local Well-Posedness of the Nonlinear System}
	%%%%%%%%%%%%%%%%%%%%%%%%%%%%%%%%%%%%%%%%%%%%%%

	In this section, we investigate the local well-posedness of the nonlinear \eqref{eq:geometric} or equivalently \eqref{eq:quasi_linear}. In contrast to Section \ref{sec:strong}, we now assume that the forcing terms $F^i$ indeed depend on the unknowns and thus an iteration/contraction-mapping argument is necessary. We will heavily rely on the estimates \eqref{est:bound_linear} obtained in Section \ref{sec:strong}.

	%%%%%%%%%%%%%%%%%%%%%%%%%%%%%%%%%%%%%%%%%%%%%%
	\subsection{Forcing Term Estimates}\label{sec:nonlinear}
	%%%%%%%%%%%%%%%%%%%%%%%%%%%%%%%%%%%%%%%%%%%%%%

	In this subsection, we derive estimates for the forcing terms $F^i$ in \eqref{eq:quasi_linear}. We first recall the definitions of $\mathscr{D}$ and $\mathscr{E}$ from \eqref{def:dissipation} and \eqref{def:energy}, respectively, and then introduce two additional spaces.
	
	\begin{equation}\label{def:xy}
		\begin{aligned}
			\mathcal{X}=\Big\{(u,p,\eta)|\mathscr{E}(u,p,\eta)<\infty\Big\},\quad \|(u,p,\eta)\|_{\mathcal{X}}=\left[\mathscr{E}(u,p,\eta)\right]^{1/2},\\ \mathcal{Y}=\Big\{(u,p,\eta)|\mathscr{D}(u,p,\eta)<\infty\Big\}, \quad\|(u,p,\eta)\|_{\mathcal{Y}}=\left[\mathscr{D}(u,p,\eta)\right]^{1/2}.
		\end{aligned}
	\end{equation}
	
	From the formulation of \eqref{eq:geometric}, the nonlinear interaction terms acting as forcing functions are given in Appendix \ref{sec:dive_forcing}.
	To close the energy estimates, we must control the forcing terms in the sense of Theorem \ref{thm:linear_low}. We first estimate the forcing terms in the bulk.
	\begin{proposition}\label{prop:force_bulk}
		Let $F^{1}, F^{4}, F^{5}$ be the  forcing terms defined in Appendix \ref{sec:dive_forcing}. Then,
		\[
		\begin{aligned}
			\|F^1\|_{L^2 L^{q_-}}+\|F^{4}\|_{L^2W^{1-\frac{1}{q_{-}},q_{-}}}+\|F^{5}\|_{L^{2}W^{1-\frac{1}{q_{-}},q_{-}}}\lesssim \mathscr{E}^{1/2}\mathscr{D}^{1/2}.
		\end{aligned}
		\]
	\end{proposition}
	\begin{proof}
		\ 
		\paragraph{\underline{Step 1 -- Estimate of $F^{1}$}} 
		
		Using H\"older's inequality, Sobolev embedding, the trace theory, and the bound on $\mathcal{A}$, we estimate terms including  $\operatorname{div}_{\mathcal{A}}$ in $F^{1}$ as follows
		\begin{align}{\label{eq:f_1}}
			\begin{aligned}
				\vert \vert \operatorname{div}_{\partial_{t}\mathcal{A}}S_{\mathcal{A}}(p,u)\vert \vert_{L^{2}L^{q_{-}}}\lesssim& \vert \vert \partial_{t}\bar{\eta}\vert \vert_{L^{\infty}W^{1,+\infty}}\vert \vert \bar{\eta}\vert \vert_{L^{\infty}W^{2,\frac{2}{1-\varepsilon_{-}}}}(\vert \vert p\vert \vert_{L^{\infty}W^{0,\frac{2}{1-\varepsilon_{-}}}}+\vert \vert u\vert \vert_{L^{\infty}W^{1,\frac{2}{1-\varepsilon_{-}}}})\notag\\
				&+\vert \vert \partial_{t}\bar{\eta}\vert \vert_{L^{\infty}W^{1,\infty}}\vert \vert \bar{\eta}\vert \vert_{L^{\infty}W^{1,+\infty}}(\vert \vert p\vert \vert_{L^{\infty}W^{1,q_{-}}}+\vert \vert u\vert \vert_{L^{\infty}W^{2,q_{-}}})\notag\\
				\lesssim& \vert \vert \partial_{t}\eta\vert \vert_{L^{\infty}H^{\frac{3}{2}+\frac{\varepsilon_{-}-\alpha}{2}}}\vert \vert \eta\vert \vert_{L^{\infty}W^{3-\frac{1}{q_{-}},q_{-}}}(\vert \vert p\vert \vert_{L^{\infty}W^{1,q_{-}}}+\vert \vert u\vert \vert_{L^{\infty}W^{2,q_{-}}}).
			\end{aligned}
		\end{align}
		
		\noindent Similarly, we have
		\begin{align}
			\begin{aligned}
				\vert \vert \operatorname{div}_{\mathcal{A}}\mathbb{D}_{\partial_{t}\mathcal{A}}u\vert \vert_{L^{2}L^{q_{-}}}\lesssim& \vert \vert \partial_{t}\eta\vert \vert_{L^{\infty}W^{3-\frac{1}{q_{-}},q_{-}}}\vert \vert u\vert \vert_{L^{\infty}W^{2,q_{-}}}.
			\end{aligned}
		\end{align}
		
		\noindent  We now estimate terms of the form $\partial_{t}^{i}u\nabla_{\partial_{t}^{j}\mathcal{A}}\partial_{t}^{k}u$, we use the fact that $|\partial_{t}^{i}\nabla^{j}\mathcal{A}|\lesssim  \sum_{k=0}^{j}|\partial_{t}^{i}\partial_{1}^{k}\bar{\eta}|$:
		\begin{align}{\label{eq:f_2}}
			\begin{aligned}
				\| {\partial_{t}}u\cdot \nabla_{\mathcal{A}}u\|_{L^{2}L^{q_{-}}}\lesssim& \vert \vert \partial_{t}u\vert \vert_{L^{\infty}L^{\frac{4}{2-\varepsilon_-}}}\| u\|_{L^{\infty}W^{1,\frac{2}{1-\varepsilon_{-}}}}\lesssim \|\partial_{t}u\|_{L^{\infty}H^{1+\frac{\varepsilon_{-}}{2}}}\|u\|_{L^{\infty}W^{2,q_{-}}},\\
				\| u\cdot \nabla_{\partial_{t}\mathcal{A}}u\|_{L^{2}L^{q_{-}}}\lesssim& \vert \vert u\vert \vert_{L^{\infty}L^{\frac{4}{2-\varepsilon_-}}}\| u\|_{L^{\infty}W^{1,\frac{2}{1-\varepsilon_{-}}}}\vert \vert \partial_{t}\eta\vert \vert_{L^{\infty}W^{1,+\infty}}\\
				\lesssim& \|u\|_{L^{\infty}H^{1+\frac{\varepsilon_{-}}{2}}}\|u\|_{L^{\infty}W^{2,q_{-}}}\vert \vert \partial_{t}{\eta}\vert \vert_{L^{\infty}H^{\frac{3}{2}+\frac{\varepsilon_{-}-\alpha}{2}}},\\
				\| u\cdot \nabla_{\mathcal{A}}\p_{t}u\|_{L^{2}L^{q_{-}}}\lesssim& \vert \vert u \vert \vert_{L^{\infty}L^{\infty}}\| \partial_{t}u\|_{L^{2}W^{1,\frac{2}{1-\varepsilon_{-}}}}\lesssim \|\partial_{t}u\|_{L^{2}W^{1,q_{-}}}\|u\|_{L^{\infty}W^{2,q_{-}}}.
			\end{aligned}
		\end{align}
		For the remaining terms involved in $F^{1}$, we have the following computation:
		\begin{align}
			\partial_{t}(\p_{t}\bar{\eta}WK\partial_{2}u)=&\partial_{t}^{2}\bar{\eta}WK\partial_{2}u+\partial_{t}\bar{\eta}\p_{t}WK\partial_{2}u+\partial_{t}\bar{\eta}W\p_{t}K\partial_{2}u+\partial_{t}\bar{\eta}WK\partial_{2}\p_{t}u.
		\end{align}
		
		\noindent Using the fact that $|\partial_{t}^{i}\nabla^{j}K|\lesssim \sum_{k=0}^{j}|\partial_{t}^{i}\partial_{1}^{k}\bar{\eta}|$, and that $|\p_{t}^{i}\nabla^{j}W|\lesssim \sum_{k=0}^{j}|\partial_{t}^{i}\partial_{1}^{k}\bar{\eta}|$, we obtain the following estimates 
		\begin{align}{\label{eq:f_3}}
			\begin{aligned}
				\|\partial_{t}^{2}\bar{\eta}WK\partial_{2}u\|_{L^{2}L^{q_{-}}}\lesssim &\|\p_{t}^{2}\bar{\eta}\|_{L^{2}L^{\frac{2}{1-\varepsilon_{-}}}}\|u\|_{L^{\infty}W^{1,\frac{2}{1-\varepsilon_{-}}}}\\
				\lesssim& \|\p_{t}^{2}\bar{\eta}\|_{L^{2}H^{1-\alpha}}\|u\|_{L^{\infty}W^{1,\frac{2}{1-\varepsilon_{-}}}}\lesssim \|\p_{t}^{2}{\eta}\|_{L^{2}H^{\frac{1}{2}-\alpha}}\|u\|_{L^{\infty}W^{2,q_{-}}},\\
				\|\partial_{t}\bar{\eta}\p_{t}WK\partial_{2}u\|_{L^{2}L^{q_{-}}}\lesssim &\|\p_{t}\bar{\eta}\|_{L^{\infty}L^{+\infty}}\|u\|_{L^{\infty}W^{1,\frac{2}{1-\varepsilon_{-}}}}\|\partial_{t}\bar{\eta}\|_{L^{2}W^{1,+\infty}}\\\lesssim& \|\p_{t}{\eta}\|_{L^{\infty}W^{3-\frac{1}{q_{-}},q_{-}}}\|u\|_{L^{\infty}W^{2,q_{-}}}\|\partial_{t}{\eta}\|_{L^{2}W^{3-\frac{1}{q_{-}},q_{-}}},\\
				\|\partial_{t}\bar{\eta}W\p_{t}K\partial_{2}u\|_{L^{2}L^{q_{-}}}\lesssim &\|\p_{t}\bar{\eta}\|_{L^{\infty}L^{+\infty}}\|u\|_{L^{\infty}W^{1,\frac{2}{1-\varepsilon_{-}}}}\|\partial_{t}\bar{\eta}\|_{L^{2}W^{1,+\infty}}\\\lesssim& \|\p_{t}{\eta}\|_{L^{\infty}W^{3-\frac{1}{q_{-}},q_{-}}}\|u\|_{L^{\infty}W^{2,q_{-}}}\|\partial_{t}{\eta}\|_{L^{2}W^{3-\frac{1}{q_{-}},q_{-}}},\\
				\|\partial_{t}\bar{\eta}WK\partial_{2}\p_{t}u\|_{L^{2}L^{q_{-}}}\lesssim &\|\p_{t}\bar{\eta}\|_{L^{\infty}L^{+\infty}}\|\p_{t}u\|_{L^{2}W^{1,\frac{2}{1-\varepsilon_{-}}}}\lesssim \|\p_{t}{\eta}\|_{L^{\infty}W^{3-\frac{1}{q_{-}},q_{-}}}\|\p_{t}u\|_{L^{2}W^{2,q_{-}}}.
			\end{aligned}
		\end{align}
		\noindent Combining estimates \eqref{eq:f_1} to \eqref{eq:f_3}, we obtain the estimate for $F^{1}$.
		%\begin{align}
		%  \|F^{1}\|_{L^{2}L^{q_{-}}}\lesssim \mathscr{E}^{\frac{1}{2}}\mathscr{D}^{\frac{1}{2}}.
		% \end{align}
	
	\paragraph{\underline{Step 2 -- Estimate of $F^{4}$}}
	
	We now bound each term involved in $F^{4}$. First, we bound the terms including the velocity field $u$ as follows
	\[
	\begin{aligned} \|\mathbb{D}_{\p_{t}\mathcal{A}}u\mathcal{N}\|_{L^{2}W^{1-\frac{1}{q_{-}},q_{-}}}&\lesssim\|\partial_{t}\partial_{1}\eta\|_{L^{2}W^{1,q_{-}}}\|\nabla u\|_{L^{\infty}W^{1-\frac{1}{q_{-}},q_{-}}(\Sigma)}\lesssim \|\partial_{t}\eta\|_{L^{2}W^{3-\frac{1}{q_{-}},q_{-}}}\|u\|_{L^{\infty}W^{2,q_{-}}},\\
		\|S_{\mathcal{A}}(u,p)\p_{t}\mathcal{N}\|_{L^{2}W^{1-\frac{1}{q_{-}},q_{-}}} &\lesssim \|\partial_{t}\partial_{1}\eta\|_{L^{2}W^{1,q_{-}}}(\|\nabla u\|_{L^{\infty}W^{1-\frac{1}{q_{-}},q_{-}}(\Sigma)}+\| p\|_{L_{t}^{\infty}W^{1-\frac{1}{q_{-}},q_{-}}})\\
		&\lesssim \|\partial_{t}\eta\|_{L^{2}W^{3-\frac{1}{q_{-}},q_{-}}}(\|u\|_{L^{\infty}W^{2,q_{-}}}+\|p\|_{L_{t}^{\infty}W^{1,q_-{}}}).
	\end{aligned}
	\]
	\noindent Then we estimate each term involving the mean curvature. We have
	\[
	\begin{aligned}
		&\|2\left[g(\eta)-\sigma\p_1\left(\frac{\p_1\eta}{(1+|\p_1\zeta_0|)^{3/2}}\right)-\sigma\partial_{1}(\mathcal{R}(\partial_{1}\zeta_{0},\partial_{1}\eta))\right]\p_t\mathcal{N}\|_{L^{2}W^{1-\frac{1}{q_{-}},q_{-}}}\\
		\quad &\lesssim \|2\left[g(\eta)-\sigma\p_1\left(\frac{\p_1\eta}{(1+|\p_1\zeta_0|)^{3/2}}\right)-\sigma\partial_{1}(\mathcal{R}(\partial_{1}\zeta_{0},\partial_{1}\eta))\right]\|_{L^{2}W^{1-\frac{1}{q_{-}},q_{-}}}\|\p_{t}\mathcal{N}\|
		_{L^{\infty}W^{1,q_{-}}}   \\
		\quad &\lesssim  (1+\|\eta\|_{L^{\infty}W^{3-\frac{1}{q_{-}},q_{-}}})\|\eta\|_{L^{\infty}W^{3-\frac{1}{q_{-}},q_{-}}}\|\p_{t}\eta\|_{L_{t}^{2}W^{3-\frac{1}{q_{-}},q_{-}}},
	\end{aligned}
	\]
	where we used the trace theorem. 
	
	Therefore, combining all the computations in this step, we obtain the estimate for $\|F^{4}\|_{L^{2}W^{1-\frac{1}{q_{-}},q_{-}}}$.
	% \begin{align}
		%  \begin{aligned}
			%      \|F^{4}\|_{L^{2}W^{1-\frac{1}{q_{-}},q_{-}}}\lesssim \mathscr{E}^{\frac{1}{2}}\mathscr{D}^{\frac{1}{2}}.
			%  \end{aligned}
		%\end{align}
		
		\paragraph{\underline{Step 3 -- Estimate for $F^{5}$}}
		
		By the definition of $F^{5}$, we have:
		\[
		\begin{aligned}
			\| \mathbb{D}_{\p_{t}\mathcal{A}}u\nu\cdot \tau\|_{L^{2}W^{1-\frac{1}{q_{-}},q_{-}}}\lesssim& \|\p_{t}\eta\|_{L^{2}W^{2,q_{-}}}\|u\|_{L^{\infty}W^{2-\frac{1}{q_{-}},q_{-}}(\Sigma)}\lesssim \|\p_{t}\eta\|_{L^{2}W^{3-\frac{1}{q_{-}},q_{-}}}\|u\|_{L^{\infty}W^{2,q_{-}}}
		\end{aligned}
		\]
		\noindent which implies the estimate for $\|F^{5}\|_{L^{2}W^{1-\frac{1}{q_{-}},q_{-}}}$.
		% \begin{align}
			%    \|F^{5}\|_{L^{2}W^{1-\frac{1}{q_{-}},q_{-}}}\lesssim \mathscr{E}^{\frac{1}{2}}\mathscr{D}^{\frac{1}{2}}.
			%\end{align}
			
			Combining the estimates for $F^{1}$, $F^{4}$ and $F^{5}$ completes the proof.
		\end{proof}
		
		We now estimate the time derivatives of $F^{1},F^{4}$ and $F^{5}$ in $(\mathcal{H}^1)^{*}$.
		\begin{theorem}\label{prop:force_adjoint}
			It holds that
			\[
			\|\p_{t}(F^{1}-F^{4}-F^{5})\|_{L^{2}(\mathcal{H}^1)^{*}}\lesssim \mathscr{E}^{1/2}\mathscr{D}^{1/2}.
			\]
		\end{theorem}
		\begin{proof}
			We will mainly use duality as defined in \eqref{duality}.
			
			\paragraph{\underline{Step 1 -- Estimate for $\partial_{t}F^{1}$ in $(\mathcal{H}^1)^{*}$}} 
			
			By definition of $F^{1}$, we first estimate the terms of the form $\operatorname{div}_{\p_{t}^{i}\mathcal{A}}\nabla_{\p_{t}^{j}\mathcal{A}}\p_{t}^{k}u$ included in $F^{1}$. It holds that
			\[
			\begin{aligned}
				\left| \int_{0}^{T}\int_{\Omega}\p_{t}(\operatorname{div}_{\p_{t}\mathcal{A}}S_{\mathcal{A}}(p,u))w J \right|
				&= \left| \int_{0}^{T}\int_{\Omega}(\operatorname{div}_{\p_{t}^{2}\mathcal{A}}S_{\mathcal{A}}(p,u)+\operatorname{div}_{\p_{t}\mathcal{A}}S_{\p_{t}\mathcal{A}}(p,u)+\operatorname{div}_{\p_{t}\mathcal{A}}S_{\mathcal{A}}(\p_{t}p,\p_{t}u))w \right| \\
				&\lesssim \|\p_{t}^{2}\bar{\eta}\|_{L^{2}W^{1,\frac{2}{\alpha}}}(1+\|\bar\eta\|_{L^{\infty}W^{2,\frac{2}{1-\varepsilon_{-}}}})(\|p\|_{L^{\infty}W^{1,q_{-}}}+\|u\|_{L^{\infty}W^{2,q_{-}}})\|w\|_{L^{2}L^{\frac{2}{\varepsilon_{-}-\alpha}}}\\
				&\quad+\|\p_{t}\bar{\eta}\|_{L^{2}W^{1,+\infty}}(\|\p_{t}\bar\eta\|_{L^{\infty}W^{2,\frac{2}{1-\varepsilon_{-}}}})(\|p\|_{L^{\infty}W^{1,q_{-}}}+\|u\|_{L^{\infty}W^{2,q_{-}}})\|w\|_{L^{2}L^{\frac{2}{\varepsilon_{-}}}}\\
				&\quad+\|\p_{t}\bar{\eta}\|_{L^{2}W^{1,+\infty}}(1+\|\bar\eta\|_{L^{\infty}W^{2,\frac{2}{1-\varepsilon_{-}}}})(\|\p_{t}p\|_{L^{\infty}W^{1,q_{-}}}+\|\p_{t}u\|_{L^{\infty}W^{2,q_{-}}})\|w\|_{L^{2}L^{\frac{2}{\varepsilon_{-}}}}\\
				&\lesssim \mathscr{E}^{\frac{1}{2}}\mathscr{D}^{\frac{1}{2}}\|w\|_{L^{2}H^{1}}.
			\end{aligned}
			\]
			
			\noindent Similarly, we have
			\[
			\begin{aligned}
				\left|\int_{0}^{T}\int_{\Omega}\p_{t}(\operatorname{div}_{\mathcal{A}}\mathbb{D}_{\partial_{t}\mathcal{A}}(u)w J \right| & = \left| \int_{0}^{T}\int_{\Omega}(\operatorname{div}_{\p_{t}\mathcal{A}}\mathbb{D}_{\p_{t}\mathcal{A}}(u)+\operatorname{div}_{\mathcal{A}}\mathbb{D}_{\p_{t}^{2}\mathcal{A}}(u)+\operatorname{div}_{\mathcal{A}}\mathbb{D}_{\p_{t}\mathcal{A}}(\p_{t}u))w J\right|\\
				&\lesssim \|\p_{t}\bar{\eta}\|_{L^{\infty}W^{1,\infty}}(\|\p_{t}\bar{\eta}\|_{L^{2}W^{2,\frac{2}{1-\varepsilon_{-}}}})(\|u\|_{L^{+\infty}W^{2,q_{-}}})\|w\|_{L^{2}L^{\frac{2}{\varepsilon_{-}}}}\\
				&\quad+(1+\|\bar{\eta}\|_{L^{\infty}W^{1,+\infty}})(\|\p_{t}^{2}\bar{\eta}\|_{L^{2}W^{2,\frac{2}{1+\alpha}}})(\|u\|_{L^{\infty}W^{2,q_{-}}})\|w\|_{L^{2}L^{\frac{2}{\varepsilon_{-}-\alpha}}}\\
				&\quad+\|\p_{t}\bar{\eta}\|_{L^{+\infty}W^{1,+\infty}}(1+\|\p_{t}\bar{\eta}\|_{L^{2}W^{2,\frac{2}{1-\varepsilon_{-}}}})(\|\p_{t}u\|_{L^{\infty}W^{2,q_{-}}})\|w\|_{L^{2}L^{\frac{2}{\varepsilon_{-}}}}\\
				&\lesssim \mathscr{E}^{\frac{1}{2}}\mathscr{D}^{\frac{1}{2}}\|w\|_{L^{2}H^{1}}.
			\end{aligned}
			\]
			We now estimate the nonlinear terms of the form $\partial_{t}^{i}u\cdot \nabla_{\partial_{t}^{j}\mathcal{A}}\partial_{t}^{k}u$ included in $\p_{t}F^{1}$. We have:
			\[\begin{aligned}
				& \int_{0}^{T}\int_{\Omega}\p_{t}(u\cdot \nabla_{\p_{t}\mathcal{A}}u)w= \int_{0}^{T}\int_{\Omega}(\p_{t}u\cdot \nabla_{\p_{t}\mathcal{A}}u)w+\int_{0}^{T}\int_{\Omega}(u\cdot \nabla_{\p_{t}^{2}\mathcal{A}}u)w +\int_{0}^{T}\int_{\Omega}(u\cdot \nabla_{\p_{t}\mathcal{A}}\p_{t}u)w\notag\\
				\lesssim& \|\p_{t}u\|_{L^{\infty}L^{\infty}}\|\p_{t}\bar{\eta}\|_{L^{2}W^{1,\infty}}\|u\|_{L^{\infty}H^{1}}\|w\|_{L^{2}L^{2}}\notag +\|u\|_{L^{\infty}L^{\infty}}\|\p_{t}^{2}\bar{\eta}\|_{L^{2}W^{1,\frac{2}{\alpha}}}\|u\|_{L^{\infty}H^{1}}\|w\|_{L^{2}L^{\frac{2}{1-\alpha}}}\notag\\
				&+\|u\|_{L^{\infty}L^{\infty}}\|\p_{t}\bar{\eta}\|_{L^{2}W^{1,\infty}}\|\p_{t}u\|_{L^{\infty}H^{1}}\|w\|_{L^{2}L^{2}}\notag\\
				\lesssim&\|\p_{t}u\|_{L^{\infty}W^{2,q_{-}}}\|\p_{t}{\eta}\|_{L^{2}W^{3-\frac{1}{q_{-}},q_{-}}}\|u\|_{L^{\infty}H^{1}}\|w\|_{L^{2}H^{1}}+\|u\|_{L^{\infty}W^{2,q_{-}}}\|\p_{t}^{2}{\eta}\|_{L^{2}H^{\frac{3}{2}-\alpha}}\|u\|_{L^{\infty}H^{1}}\|w\|_{L^{2}H^{1}}\notag\\
				\lesssim&  \mathscr{E}^{\frac{1}{2}}\mathscr{D}^{\frac{1}{2}}\|w\|_{L^{2}H^{1}}.
			\end{aligned}\]
			\noindent Similarly, we have
			\[\begin{aligned}
				\int_{0}^{T}\int_{\Omega}\p_{t}(\p_{t}u\cdot \nabla_\mathcal{A}u)w
				&\lesssim\|\p_{t}^{2}u\|_{L^{2}H^{1}}\|{\eta}\|_{L^{\infty}W^{3-\frac{1}{q_{-}},q_{-}}}\|u\|_{L^{\infty}W^{2,q_{-}}}\|w\|_{L^{2}H^{1}}+\|\p_{t}u\|_{L^{2}W^{2,q_{-}}}\|\p_{t}{\eta}\|_{L^{\infty}W^{3-\frac{1}{q_{-}},q_{-}}}\\
				&\qquad\qquad\times \|u\|_{L^{\infty}H^{1}}\|w\|_{L^{2}H^{1}}
				+\|\p_{t}u\|_{L^{2}W^{2,q_{-}}}\|{\eta}\|_{L^{\infty}W^{3-\frac{1}{q_{-}},q_{-}}}\|\p_{t}u\|_{L^{\infty}H^{1}}\|w\|_{L^{2}H^{1}}\\
				&\lesssim  \mathscr{E}^{\frac{1}{2}}\mathscr{D}^{\frac{1}{2}}\|w\|_{L^{2}H^{1}},
			\end{aligned}\]
			\[\begin{aligned}
				\int_{0}^{T}\int_{\Omega}\p_{t}(u\cdot \nabla_\mathcal{A}\p_{t}u)w
				& \lesssim\|\p_{t}^{2}u\|_{L^{2}H^{1}}\|{\eta}\|_{L^{\infty}W^{3-\frac{1}{q_{-}},q_{-}}}\|u\|_{L^{\infty}W^{2,q_{-}}}\|w\|_{L^{2}H^{1}}+\|\p_{t}u\|_{L^{2}W^{2,q_{-}}}\|\p_{t}{\eta}\|_{L^{\infty}W^{3-\frac{1}{q_{-}},q_{-}}}\\
				&\qquad\qquad\times\|u\|_{L^{\infty}H^{1}}\|w\|_{L^{2}H^{1}}+\|\p_{t}u\|_{L^{2}W^{2,q_{-}}}\|{\eta}\|_{L^{\infty}W^{3-\frac{1}{q_{-}},q_{-}}}\|\p_{t}u\|_{L^{\infty}H^{1}}\|w\|_{L^{2}H^{1}} \\
				&\lesssim  \mathscr{E}^{\frac{1}{2}}\mathscr{D}^{\frac{1}{2}}\|w\|_{L^{2}H^{1}}.
			\end{aligned}\]
			Finally, we estimate the term $\partial_{t}^{2}(\p_{t}\bar{\eta}WK\p_{2}u)$ involved in $F^{1}$. After decomposing $\partial_{t}^{2}(\p_{t}\bar{\eta}WK\p_{2}u)$, we first have the following computation
			\begin{align}{\label{eq:f_5}}
				\begin{aligned}
					\int_{0}^{T}\int_{\Omega}(\p_{t}^{3}\bar{\eta}WK\p_{2}u)w\lesssim&\|\p_{t}^3\bar{\eta}\|_{L^{2}L^{\frac{2}{\alpha}}}\|\p_{2}u\|_{L^{\infty}L^{\frac{2}{1-\varepsilon_{-}}}}\|\omega\|_{L^{2}L^{\frac{2}{1+\varepsilon_{-}-\alpha}}}\\
					\lesssim &\|\p_{t}^{3}\bar{\eta}\|_{L^{2}H^{\frac{1}{2}-\alpha}}\|u\|_{L^{\infty}W^{2,q_{-}}}\|w\|_{L^{2}H^{1}}\lesssim \mathscr{E}^{\frac{1}{2}}\mathscr{D}^{\frac{1}{2}}\|w\|_{L^{2}H^{1}},\\
					\int_{0}^{T}\int_{\Omega}(\p_{t}\bar{\eta}\p_{t}^{2}WK\p_{2}u)w\lesssim&\|\p_{t}\bar{\eta}\|_{L^{\infty}L^{\infty}}\|\p_{2}u\|_{L^{\infty}L^{\frac{2}{1-\varepsilon_{-}}}}\|\p_{t}^{2}\bar{\eta}\|_{L^{2}W^{1,\frac{2}{\alpha}}}\|w\|_{L^{2}L^{\frac{2}{1+\varepsilon_{-}-\alpha}}}\\
					\lesssim &\|\p_{t}\bar{\eta}\|_{L^{\infty}W^{3-\frac{1}{q_{-}},q_{-}}}\|\p_{t}^{2}\bar{\eta}\|_{L^{2}H^{\frac{3}{2}-\alpha}}\|u\|_{L^{\infty}W^{2,q_{-}}}\|w\|_{L^{2}H^{1}}\lesssim\mathscr{E}^{\frac{1}{2}}\mathscr{D}^{\frac{1}{2}}\|w\|_{L^{2}H^{1}},\\
					\int_{0}^{T}\int_{\Omega}(\p_{t}\bar{\eta}W\p_{t}^{2}K\p_{2}u)w\lesssim&\|\p_{t}\bar{\eta}\|_{L^{\infty}L^{\infty}}\|\p_{2}u\|_{L^{\infty}L^{\frac{2}{1-\varepsilon_{-}}}}\|\p_{t}^{2}\bar{\eta}\|_{L^{2}W^{1,\frac{2}{\alpha}}}\|w\|_{L^{2}L^{\frac{2}{1+\varepsilon_{-}-\alpha}}}\\
					\lesssim &\|\p_{t}\bar{\eta}\|_{L^{\infty}W^{3-\frac{1}{q_{-}},q_{-}}}\|\p_{t}^{2}\bar{\eta}\|_{L^{2}H^{\frac{3}{2}-\alpha}}\|u\|_{L^{\infty}W^{2,q_{-}}}\|w\|_{L^{2}H^{1}}\lesssim \mathscr{E}^{\frac{1}{2}}\mathscr{D}^{\frac{1}{2}}\|w\|_{L^{2}H^{1}},\\
					\int_{0}^{T}\int_{\Omega}(\p_{t}\bar{\eta}WK\p_{2}\p_{t}^{2}u)w\lesssim&\|\p_{t}\bar{\eta}\|_{L^{\infty}L^{\infty}}\|\p_{2}\p_{t}^{2}u\|_{L^{\infty}H^{1}}\|w\|_{L^{2}L^{2}}\lesssim \mathscr{E}^{\frac{1}{2}}\mathscr{D}^{\frac{1}{2}}\|w\|_{L^{2}H^{1}}.
				\end{aligned}
			\end{align}
			
			\noindent For other terms obtained from the decomposition of $\p_{t}^{2}(\p_{t}\bar{\eta}WK\p_{2}u)$, we notice that they do not contain the highest-order temporal derivative terms. Therefore, they can be estimated via similar computation as in \eqref{eq:f_5}. We omit the details and only write the final result as follows
			\[
			\int_{0}^{T}\int_{\Omega}\p_{t}^{2}(\p_{t}\bar{\eta}WK\p_{2}u)w\lesssim \mathscr{E}^{\frac{1}{2}}\mathscr{D}^{\frac{1}{2}}\|w\|_{L^{2}H^{1}}.
			\]
			From the discussion and computation above, we obtain the estimate for $F^{1}$ in $(\mathcal{H}^1)^{*}$
			\[
			\|\p_{t}F^{1}\|_{L^{2}(\mathcal{H}^1)^{*}}\lesssim \mathscr{E}^{\frac{1}{2}}\mathscr{D}^{\frac{1}{2}}.
			\]
			
			\paragraph{\underline{Step 2 -- Estimate for $\p_{t}F^{4}$ in $(\mathcal{H}^1)^{*}$}} 
			
			We first estimate the surface pressure terms included in $\p_{t}F^{4}$.
			\[
			\begin{aligned}
				\int_{0}^{T}\int_{-\ell}^{\ell}\p_{t}(\mu \mathbb{D}_{\p_{t}\mathcal{A}}u\mathcal{N})\cdot w=& \int_{0}^{T}\int_{-\ell}^{\ell}(\mu \mathbb{D}_{\p_{t}^{2}\mathcal{A}}u\mathcal{N})\cdot w+ \int_{0}^{T}\int_{-\ell}^{\ell}(\mu \mathbb{D}_{\p_{t}\mathcal{A}}\p_{t}u\mathcal{N})\cdot w+ \int_{0}^{T}\int_{-\ell}^{\ell}(\mu \mathbb{D}_{\p_{t}\mathcal{A}}u\p_{t}\mathcal{N})\cdot w\\
				\lesssim& \|\p_{t}^{2}\eta\|_{L^{2}W^{1,\frac{1}{\alpha}}}\|u\|_{L^{\infty}W^{1,\frac{1}{1-\varepsilon_{-}}}(\Sigma)}\| w\|_{L^{2}L^{\frac{1}{\varepsilon_{-}-\alpha}}}+\|\p_{t}\eta\|_{L^{2}W^{1,\infty}}\|\p_{t}u\|_{L^{\infty}W^{1,\frac{1}{1-\varepsilon_{-}}}(\Sigma)}\\
				&\times\| w\|_{L^{2}L^{\frac{1}{\varepsilon_{-}}}}+\|\p_{t}\eta\|_{L^{2}W^{1,\infty}}\|u\|_{L^{\infty}W^{1,\frac{1}{1-\varepsilon_{-}}}(\Sigma)}\|\p_{t}\eta\|_{L^{\infty}W^{1,+\infty}}\| w\|_{L^{2}L^{\frac{1}{\varepsilon_{-}}}}\\
				\lesssim& \|\p_{t}^{2}\eta\|_{L^{2}H^{\frac{3}{2}-\alpha}}\|u\|_{L^{\infty}W^{2,q_{-}}}\| w\|_{L^{2}H^{1}}+\|\p_{t}\eta\|_{L^{2}W^{3-\frac{1}{q_{-}},q_{-}}}\|\p_{t}u\|_{L^{\infty}W^{2,q_{-}}}\| w\|_{L^{2}H^{1}}\\
				&+\|\p_{t}\eta\|_{L^{2}W^{3-\frac{1}{q_{-}},q_{-}}}\|u\|_{L^{\infty}W^{2,q_{-}}}\|\p_{t}\eta\|_{L^{\infty}W^{3-\frac{1}{q_{-}},q_{-}}}\| w\|_{L^{2}H^{1}}\\
				\lesssim&\mathscr{E}^{\frac{1}{2}}\mathscr{D}^{\frac{1}{2}}\|w\|_{L^{2}H^{1}}.
			\end{aligned}
			\]
			
			\noindent Similarly, we have
			\[
			\begin{aligned}
				\int_{0}^{T}\int_{-\ell}^{\ell}\p_{t}(\mu \mathbb{D}_{\mathcal{A}}u\p_{t}\mathcal{N})\cdot w
				\lesssim\mathscr{E}^{\frac{1}{2}}\mathscr{D}^{\frac{1}{2}}\|w\|_{L^{2}H^{1}}.
			\end{aligned}
			\]
			We then estimate the terms including the capillary operator after decomposing $\p_{t}F^{4}$. We have:
			\[
			\begin{aligned}
				&\quad \int_{0}^{T}\int_{-\ell}^{\ell} \p_{t}(2\left[g(\eta)-\sigma\p_1\left(\frac{\p_1\eta}{(1+|\p_1\zeta_0|)^{3/2}}\right)-\sigma\partial_{1}\mathcal{R}(\partial_{1}\zeta_{0},\partial_{1}\eta)\right]\p_t\mathcal{N})\cdot w\\
				&\lesssim \|\p_{t}\eta\|_{L^{2}W^{2,\frac{1}{1-\varepsilon_{-}}}}\|\p_{t}\eta\|_{L^{\infty}W^{1,+\infty}}\|w\|_{L^{2}L^{\frac{1}{\varepsilon_{-}}}}+\|\eta\|_{L^{\infty}W^{2,\frac{1}{1-\varepsilon_{-}}}}\|\p_{t}^{2}\eta\|_{L^{2}W^{1,\frac{1}{\alpha}}}\|w\|_{L^{2}L^{\frac{1}{\varepsilon_{-}}}}\\
				&\lesssim \|\p_{t}\eta\|_{L^{2}W^{3-\frac{1}{q_{-}},q_{-}}}\|\p_{t}\eta\|_{L^{\infty}W^{3-\frac{1}{q_{-}},q_{-}}}\|w\|_{L^{2}L^{\frac{1}{\varepsilon_{-}}}}+\|\eta\|_{L^{\infty}W^{3-\frac{1}{q_{-}},q_{-}}}\|\p_{t}^{2}\eta\|_{L^{2}H^{\frac{3}{2}-\alpha}}\|w\|_{L^{2}L^{\frac{1}{\varepsilon_{-}}}}\\
				&\lesssim \mathscr{E}^{\frac{1}{2}}\mathscr{D}^{\frac{1}{2}}\|w\|_{L^{2}H^{1}}.
			\end{aligned}
			\]
			\noindent Therefore, combining all the computations for $\p_{t}F^{4}$ above, we obtain
			\[
			\|\partial_{t}F^{4}\|_{L^{2}(\mathcal{H}^1)^{*}}\lesssim \mathscr{E}^{\frac{1}{2}}\mathscr{D}^{\frac{1}{2}}.
			\]
			
			\paragraph{\underline{Step 3 -- Estimate for $\partial_{t}F^{5}$ in $(\mathcal{H}^1)^{*}$}} 
			
			We have the following estimate:
			\[
			\begin{aligned}
				&\int_{0}^{T}\int_{-\ell}^{\ell} \p_{t}(\mathbb{D}_{\p_{t}\mathcal{A}}u\nu\cdot \tau)w
				=\int_{0}^{T}\int_{-\ell}^{\ell} (\mathbb{D}_{\p_{t}^{2}\mathcal{A}}u\nu\cdot \tau)w+\int_{0}^{T}\int_{-\ell}^{\ell} (\mathbb{D}_{\p_{t}\mathcal{A}}\p_{t}u\nu\cdot \tau)w\\
				\lesssim &\|\p_{t}^{2}\eta\|_{L^{2}W^{1,\frac{1}{\alpha}}}\|u\|_{L^{\infty}W^{1,\frac{1}{1-\varepsilon_{-}}}(\Sigma)}\|w\|_{L^{2}L^{\frac{1}{\varepsilon_{-}-\alpha}}}+\|\p_{t}\eta\|_{L^{2}W^{1,+\infty}}\|u\|_{L^{\infty}W^{1,\frac{1}{1-\varepsilon_{-}}}(\Sigma)}\|w \|_{L^{2}L^{\frac{1}{\varepsilon_{-}}}}\\
				\lesssim& \|\p_{t}^{2}\eta\|_{L^{2}H^{\frac{3}{2}-\alpha}}\|u\|_{L^{\infty}W^{2,q_{-}}}\|w \|_{L^{2}H^{1}}+\|\p_{t}\eta\|_{L^{2}W^{3-\frac{1}{q_{-}},q_{-}}}\|u\|_{L^{\infty}W^{2,q_{-}}}\|w\|_{L^{2}H^{1}}\\
				\lesssim& \mathscr{E}^{\frac{1}{2}}\mathscr{D}^{\frac{1}{2}}\|w\|_{L^{2}H^{1}}.
			\end{aligned}
			\]
			Therefore, we obtain
			\[
			\|\p_{t}F^{5}\|_{L^{2}(\mathcal{H}^1)^{*}}\lesssim \mathscr{E}^{\frac{1}{2}}\mathscr{D}^{\frac{1}{2}}.
			\]
			Finally, combining Step 1 to Step 3, we obtain the following result.
			\[
			\|\p_{t}(F^{1}-F^{4}-F^{5})\|_{L^{2}(\mathcal{H}^1)^{*}}\lesssim\mathscr{E}^{\frac{1}{2}}\mathscr{D}^{\frac{1}{2}},
			\]
			\noindent which finishes the proof.
		\end{proof}
		
		We now estimate $F^{7}$ via the following theorem.
		\begin{theorem}{\label{prop:force_contact}}
			$F^{7}$ is defined in Appendix \ref{sec:dive_forcing}. It holds that
			\[
			\|[\p_{t}F^{7}]_{\ell}\|_{L_{t}^{2}}\lesssim(1+T) \mathscr{E}^{\frac{1}{2}}\mathscr{D}^{\frac{1}{2}}.
			\]
		\end{theorem}
		
		\begin{proof}
			We have
			\[
			\|[\kappa\p_{t}(\hat{\mathscr{W}}'(u\cdot \mathcal{N})\p_{t}u\cdot \mathcal{N})]\|_{L_{t}^{2}}\lesssim \|[\p_{t}u\cdot \mathcal{N}]_{\ell}\|_{L_{t}^{\infty}}\|[\p_{t}u\cdot \mathcal{N}]_{\ell}\|_{L_{t}^{2}}+\|[u\cdot \mathcal{N}]_{\ell}\|_{L_{t}^{\infty}}\|[\p_{t}^{2}u\cdot \mathcal{N}]\|_{L_{t}^{2}}\lesssim \mathscr{E}^{\frac{1}{2}}\mathscr{D}^{\frac{1}{2}}.
			\]
		\end{proof}
		
		Next, we establish the $q_{+}$ boundedness for the nonlinear forcing terms by the following theorem.
		\begin{theorem}{\label{thm:forcing_+}}
			$F^{1},F^{4},F^{5},F^{7}$ are defined in Appendix \ref{sec:dive_forcing}. They satisfy the following estimate
			\[
			\|\int_{0}^{t}F^1\|_{L^\infty 
				L^{q_+}} + \|\int_{0}^{t}F^4\|_{L^\infty W^{1-1/q_+,q_+}} + \|\int_{0}^{t}F^5\|_{L^\infty W^{1-1/q_+,q_+}} + \|\int_{0}^{t}[F^7]_\ell\|_{L_t^{\infty}}\lesssim \mathscr{E}^{\frac{1}{2}}\mathscr{D}^{\frac{1}{2}}.
			\]
		\end{theorem}
		\begin{proof}
			\ 
			\paragraph{\underline{Step 1 -- Estimate for $F^{1}$}} 
			By the definition of $F^{1}$ and the temporal integration by parts, we have
			\[
			\begin{aligned}
				\|\int_{0}^{t}\dive_{\p_{t}\mathcal{A}}S_{\mathcal{A}}(p,u)\|_{L_{t}^{\infty}L^{q_{+}}}\lesssim &\|\p_{t}\bar{\eta}\|_{L_{t}^{\infty}W^{1,\infty}}(1+\|\eta\|_{L_{t}^{\infty}W^{3-\frac{1}{q_{-}},q_{-}}})(\|p\|_{L_{t}^{\infty}W^{1,q_{+}}}+\|u\|_{L_{t}^{\infty}W^{2,q_{+}}})\\
				\lesssim&\|\p_{t}{\eta}\|_{L_{t}^{\infty}H^{\frac{3}{2}+\frac{\varepsilon_{-}-\alpha}{2}}}(1+\|\eta\|_{L_{t}^{\infty}W^{3-\frac{1}{q_{-}},q_{-}}})(\|p\|_{L_{t}^{\infty}W^{1,q_{+}}}+\|u\|_{L_{t}^{\infty}W^{2,q_{+}}})\\
				\lesssim&\mathscr{E}^{\frac{1}{2}}\mathscr{D}^{\frac{1}{2}},\\
				\|\int_{0}^{t}\dive_{\mathcal{A}}S_{\p_{t}\mathcal{A}}(p,u)\|_{L_{t}^{\infty}L^{q_{+}}}\lesssim &T^{\frac{1}{2}}(\|\nabla^{2}\bar{\eta}\|_{L_{t}^{\infty}L^{\frac{2}{1-\varepsilon_{+}}}})(\|\p_{t}p\|_{L_{t}^{2}L^{\frac{2}{1-\varepsilon_{-}}}}+\|\p_{t}u\|_{L_{t}^{2}W^{1,\frac{2}{1-\varepsilon_{-}}}})\\
				&\quad+T(\|\p_{t}\nabla\bar{\eta}\|_{L_{t}^{\infty}L^{\infty}})(\|p\|_{L_{t}^{\infty}W^{1,q_{+}}}+\|u\|_{L_{t}^{\infty}W^{2,q_{+}}})\\
				\lesssim&T\|\p_{t}{\eta}\|_{L_{t}^{\infty}H^{\frac{3}{2}+\frac{\varepsilon_{-}-\alpha}{2}}}(\|p\|_{L_{t}^{\infty}W^{1,q_{+}}}+\|u\|_{L_{t}^{\infty}W^{2,q_{+}}})\\
				&+T^{\frac{1}{2}}\|\eta\|_{L_{t}^{\infty}W^{3-\frac{1}{q_{+}},q_{+}}}(\|\p_{t}p\|_{L_{t}^{2}W^{1,q_{-}}}+\|\p_{t}u\|_{L_{t}^{2}W^{2,q_{-}}})
				\\
				\lesssim&(1+T)\mathscr{E}^{\frac{1}{2}}\mathscr{D}^{\frac{1}{2}}.
			\end{aligned}
			\]
			For the remaining terms involved in $\int_{0}^{t}F^{1}$, they can be bounded as follows
			\[
			\begin{aligned}
				\|u\cdot \nabla_{\mathcal{A}}u\|_{L_{t}^{\infty}L^{q_{+}}}\lesssim&\|u\|_{L_{t}^{\infty}L^{\infty}}\|u\|_{L_{t}^{\infty}W^{1,q_{+}}}\lesssim\|u\|_{L_{t}^{\infty}W^{2,q_{+}}}^{2}\lesssim\mathscr{E}^{\frac{1}{2}}\mathscr{D}^{\frac{1}{2}},\\
				\|\p_{t}\bar{\eta}WK\p_{2}u\|_{L_{t}^{\infty}L^{q_{+}}}\lesssim&\|\p_{t}\eta\|_{L_{t}^{\infty}H^{1}}\|u\|_{L_{t}^{\infty}W^{2,q_{+}}}\lesssim \mathscr{E}^{\frac{1}{2}}\mathscr{D}^{\frac{1}{2}}.
			\end{aligned}
			\]
			
			\paragraph{\underline{Step 2 -- Estimates for $F^{4}$ and $F^5$}} 
			Noticing that $\int_{0}^{t}F^{5}$ has a similar form compared to $\int_{0}^{t}F^{4}$, we only give the detailed estimate of $\int_{0}^{t}F^{4}$ here. We have
			\[
			\begin{aligned}
				&\|\int_{0}^{t}S_{\mathcal{A}}(p,u)\p_{t}\mathcal{N}\|_{L_{t}^{\infty}W^{1-\frac{1}{q_{+}},q_{+}}}
				\lesssim T^{\frac{1}{2}}\|\p_{t}{\eta}\|_{L_{t}^{2}W^{3-\frac{1}{q_{-}},q_{-}}}(1+\|\eta\|_{L_{t}^{\infty}W^{3-\frac{1}{q_{-}},q_{-}}})(\|p\|_{L_{t}^{\infty}W^{1,q_{+}}}+\|u\|_{L_{t}^{\infty}W^{2,q_{+}}})\\
				& \qquad \qquad \qquad \qquad\qquad\qquad\qquad\lesssim T\mathscr{E}^{\frac{1}{2}}\mathscr{D}^{\frac{1}{2}},\\
				&\|\int_{0}^{t}S_{\p_{t}\mathcal{A}}(p,u)\mathcal{N}\|_{L_{t}^{\infty}W^{1-\frac{1}{q_{+}},q_{+}}}\lesssim T^{\frac{1}{2}}\|\p_{t}{\eta}\|_{L_{t}^{2}W^{3-\frac{1}{q_{-}},q_{-}}}(1+\|\eta\|_{L_{t}^{\infty}W^{3-\frac{1}{q_{-}},q_{-}}})(\|p\|_{L_{t}^{\infty}W^{1,q_{+}}}+\|u\|_{L_{t}^{\infty}W^{2,q_{+}}})\\
				& \qquad \qquad \qquad \qquad\qquad\qquad\qquad\lesssim T\mathscr{E}^{\frac{1}{2}}\mathscr{D}^{\frac{1}{2}},\\
				&\|\int_{0}^{t}(\mathcal{K}(\eta)-\p_{1}\mathcal{R}(\p_{1}\zeta_{0},\p_{1}\eta))\p_{t}\mathcal{N}\|_{L_{t}^{\infty}W^{1-\frac{1}{q_{+}},q_{+}}}\lesssim T^{\frac{1}{2}}\|\p_{t}{\eta}\|_{L_{t}^{2}W^{3-\frac{1}{q_{-}},q_{-}}}(\|\eta\|_{L_{t}^{\infty}W^{3-\frac{1}{q_{-}},q_{-}}})
				\lesssim T\mathscr{E}^{\frac{1}{2}}\mathscr{D}^{\frac{1}{2}}.\\
			\end{aligned}
			\]
			
			\paragraph{\underline{Step 3 -- Estimate for $F^{7}$}} 
			By the definition of $F^{7}$, we have
			\[
			\|\int_{0}^{t}F^{7}\|_{L_{t}^{\infty}}=\|\hat{\mathcal{W}}(u\cdot\mathcal{N})(\pm\ell)\|_{L_{t}^{\infty}}\lesssim \|[u\cdot \mathcal{N}]_{\pi}^{2}\|_{L_{t}^{\infty}}\lesssim\mathscr{E}^{\frac{1}{2}}\mathscr{D}^{\frac{1}{2}}.
			\]
			Combining Step 1 -- Step 3, the theorem is proved.
		\end{proof}

		%%%%%%%%%%%%%%%%%%%%%%%%%%%%%%%%%%%%%%%%%%%%%%
		\subsection{Initial Data for the Nonlinear System}\label{sec:initial_nonlinear1}
		%%%%%%%%%%%%%%%%%%%%%%%%%%%%%%%%%%%%%%%%%%%%%%

		We assume that the initial data
		\begin{align}\label{initial_ep_nonlinear}
			(u_0, p_0, \eta_0, \p_tu(0), \p_tp(0), \p_t\eta(0), \p_t^2u(0),\p_t^2\eta(0))
		\end{align}
		for \eqref{eq:geometric} are in the space $X$ defined via
		\[
		\begin{aligned}
			X:=W^{2,q_+}(\Om)\times W^{1,q_+}(\Om)\times W^{3-1/q_+,q_+}(\Sigma)\times H^{1+\frac{\varepsilon_{-}}{2}}(\Om)\times H^{\frac{\varepsilon_{-}}{2}}(\Om)\times H^{\frac{3}{2}+\frac{\varepsilon_{-}-\alpha}{2}}(\Sigma)
			\times L^{2}(\Om)\times H^{1}(\Sigma),
		\end{aligned}
		\]
		and satisfy the compatibility conditions \eqref{compat_C2}
		as well as zero average conditions \eqref{cond:zero},
		where $X$ is a Banach space, with the square norm
		\[
		\begin{aligned}
			&\|(u_0, p_0, \eta_0, \p_tu(0), \p_tp(0), \p_t\eta(0), \p_t^2u(0), \p_{t}^{2}p(0),\p_t^2\eta(0), \p_{t}^{3}u(0),\p_t^3\eta(0))\|_{X^\varepsilon}^2
			=\mathcal{E}(0),
		\end{aligned}
		\]
		where $\mathcal{E}(0)$ is defined via \eqref{energy}.
		From Theorem \ref{thm:initial}, we have known that our initial data set is non-empty, provided that $\|\p_t^2u(0)\|_{H^0}^2+\|\p_t^2\eta(0)\|_{H^1}^2+\|\p_t^2\eta(0)\|_{W^{2-1/q_-,q_-}}^2$ is sufficiently small.

		%%%%%%%%%%%%%%%%%%%%%%%%%%%%%%%%%%%%%%%%%%%%%%
		\subsection{Existence and Uniqueness of Solution}
		%%%%%%%%%%%%%%%%%%%%%%%%%%%%%%%%%%%%%%%%%%%%%%

		We now consider the local well-posedness of the nonlinear problem \eqref{eq:geometric}. Our strategy is to work in a complete metric space that imposes high-regularity bounds while being equipped with a metric involving only low-regularity norms, and then to show that a fixed point in this space yields a solution to \eqref{eq:geometric}.
		
		We now define the desired metric space.
		\begin{definition}\label{def:S}
			Suppose that $T>0$. For $\delta\in(0,\infty)$ we define the space
			\[
			\begin{aligned}
				S(T, \delta)
				&=\Big\{(u,p,\eta):\Om\to\mathbb{R}^2\times\mathbb{R}\times\mathbb{R}\Big|(u,p,\eta)\in\mathcal{X}\cap\mathcal{Y},\ \text{with}\ \mathfrak{K}(u,p,\eta)^{1/2}\le\delta\\
				&\quad\text{and}\ (u,p,\eta)\ \text{achieve the initial data in Appendix \ref{sec:initial}}\Big\}.
			\end{aligned}
			\]
			We endow this space with the metric
			\begin{equation}\label{def:metric}
				\begin{aligned}
					d((u,p,\eta),(v,q,\xi))&=\|u-v\|_{L^\infty H^1}+\|u-v\|_{L^{2}W^{2,q_{-}}}+\|\p_{t}u-\p_{t}v\|_{L^2 H^1}\\&\quad+\sum_{i=0}^{1}(\|\p_{t}^{i}u-\p_{t}^{i}v\|_{L^{\infty}H^{0}})+\|p-q\|_{L^{2}W^{1,q_{-}}}+\|p-q\|_{L_{t}^{\infty}H^{0}}\\&\quad+\sum_{i=0}^{1}\|\p_{t}^{i}\eta-\p_{t}^{i}\xi\|_{L^\infty H^1}
					+\|\eta-\xi\|_{L^\infty H^{3/2-\alpha}}+\|\eta-\xi\|_{L^2 W^{3-1/q_-.q_-}}
					\\
					&\quad+\|\p_t\eta-\p_t\xi\|_{L^2 H^{3/2-\alpha}}+\sum_{i=0}^{1}\|[\p_t^{i}\eta-\p_t^{i}\xi]_\ell\|_{L^\infty([0,T])}+\|[\p_t^2\eta-\p_t^2\xi]_\ell\|_{L^2([0,T])},
				\end{aligned}
			\end{equation}
			for any $(u,p,\eta),(v,q,\xi)\in S(T, \delta)$, where the temporal norm is evaluated on the set $[0,T]$.
		\end{definition}
		
		\begin{remark}
			In the metric defined above, we only require the $q_{-}$ elliptic regularity. This is enough for the energy estimate. We do not require $q_{+}$ since the integral structure prevents us from deriving a $q_{+}$ estimate for the difference functions.
		\end{remark}
		
		For $\mathcal{X}$ and $\mathcal{Y}$ defined in \eqref{def:xy}, the subset satisfying $\mathfrak{K}(u,p,\eta)^{\frac{1}{2}}\leq \delta$ is weak* compact in $\mathcal{X}$ and weak compact in $\mathcal{Y}$. Thus it is easy to see that the space $S(T, \delta)$ is complete for each fixed $\delta$. 
		
		Now we employ the metric space $S(T,\delta)$ and a contraction mapping argument to construct the solutions $(u, p, \eta)$. The procedure is completed by several steps: 
		\begin{itemize}
			\item 
			Firstly, we choose an appropriate $\delta$ to determine the working space $S(T, \delta)$.
			\item 
			Secondly, we define the iterative operator from $S(T, \delta)$ to $S(T, \delta)$. In order to show that the operator is a contraction map in the following steps, we give the PDEs of differences. 
			\item 
			Thirdly, we use the energy method to control first time derivative of velocity and surface function in the energy space. 
			\item 
			In the fourth step, we improve the regularity of first time derivative of surface function so that the first time derivative of velocity and surface function are controlled under the metric of $S(T, \delta)$. 
			\item 
			In the fifth step, we give elliptic estimate for velocity and surface function in low regularity under the metric of $S(T, \delta)$. 
			\item 
			Finally, we combine the results of above steps to give the contraction.
		\end{itemize}
		
		\begin{theorem}\label{thm: fixed point}
			There exists a universal constant $\delta>0$ which is sufficiently small such that if the initial data for the nonlinear system are given as in Section \ref{sec:initial_nonlinear1} satisfying
			$\|(u_0, p_0, \eta_0, \p_tu(0), \p_tp(0), \p_t\eta(0), \p_t^2u(0), \p_t^2\eta(0)\|_{X}^2\le \delta$,
			then there exists a unique solution $(u, p, \eta)$ to \eqref{eq:geometric}, belonging to the metric space $S(T, \delta)$, where $T>0$ is sufficiently small.  In particular $(u, p, \eta)\in \mathcal{X}\cap\mathcal{Y}$, where $\mathcal{X}$ and $\mathcal{Y}$ are defined in \eqref{def:xy}.
		\end{theorem}
		
		\begin{proof}
			\
			\paragraph{\underline{Step 1 -- Solving the Linear Problem}}
			Suppose that $\mathfrak{K}(u,p,\eta)\le\delta^2$ is sufficiently small. For every $(u,p,\eta)\in S(T, \delta)$ given, let $(\tilde{u},\tilde{p},\tilde{\eta})$ instead of $(u, p, \xi)$ be the unique solution to the linear problem \eqref{eq:quasi_linear_0}.
			Then we apply the Theorem \ref{thm:linear_low} to the system \eqref{eq:quasi_linear_0}, with the Propositions \ref{prop:force_bulk}--\ref{thm:forcing_+} under the assumption on $\mathfrak{K}(u,p,\eta)$ to obtain that
			\[
			\mathfrak{K}(\tilde{u}, \tilde{p}, \tilde{\eta})\le C_0\Big[1+T\Big]e^{T\delta}\left(C_1\mathcal{E}(0)+C_2\Big(1+{T}\Big)\delta^4\right),
			\]
			for some universal constants $C_0$, $C_1$ and $C_2$. 
			
			We first choose $T$ such that $T\le 1$. Then we choose $\delta\in (0, 1)$ such that $4C_0C_2e^{\delta}\delta^2\le\frac12$. Finally, we restrict the initial data $\mathcal{E}(0)$ so that $4C_0C_1e\mathcal{E}(0)\le\frac12\delta^2$. Consequently, we obtain $\mathfrak{K}(\tilde{u}, \tilde{p}, \tilde{\eta})\le\delta^2$, which implies that $(\tilde{u}, \tilde{p}, \tilde{\eta})\in S(T, \delta)$.
			
			\paragraph{\underline{Step 2 -- PDEs for the Differences}}
			Define the operator $A:S(T, \delta)\rightarrow S(T, \delta)$ via
			\begin{equation}\label{def:map_a}
				A(u,p,\eta)=(\tilde{u},\tilde{p}, \tilde{\eta}) \ \text{for each}\ (u,p,\eta)\in S(T, \delta).
			\end{equation}
			
			In order to show that the operator $A$ is contract on $S(T, \delta)$, we choose two elements $(u^i,p^i,\eta^i)\in S(T, \delta)$, 
			and define $A(u^i,p^i,\eta^i)=(\tilde{u}^i,\tilde{p}^i, \tilde{\eta}^i)$ as above, $i=1,2$. For simplicity, we will abuse notation and denote $u=u^1-u^2$, $p=p^1-p^2$, $\eta=\eta^1-\eta^2$ and the same for $\tilde{u},\tilde{p}, \tilde{\eta}$. From the difference of equation for $(\tilde{u}^i,\tilde{p}^i, \tilde{\eta}^i)$, $i=1,2$, we know that
			the solutions are regular enough to be differentiated in time to result in the following equation
			\begin{equation}\label{linear_fix2}
				\left\{
				\begin{aligned}
					&\p_t(\p_t\tilde{u})+\dive_{\mathcal{A}^1}S_{\mathcal{A}^1}(\p_t\tilde{p},\p_t\tilde{u})=-\mu \dive_{\mathcal{A}^1}(S_{\p_t\mathcal{A}^1-\p_t\mathcal{A}^2}({u}^2,{p}^{2}))+R^{1,1} \  &\text{in}&\ \Om,\\
					&\dive_{\mathcal{A}^1}\p_t\tilde{u}=R^{2,1} \  &\text{in}&\ \Om,\\
					&S_{\mathcal{A}^1}(\p_t\tilde{p},\p_t\tilde{u})\mathcal{N}^1=\mu\mathbb{D}_{\p_t\mathcal{A}^1-\p_t\mathcal{A}^2}\tilde{u}^2\mathcal{N}^1+g(\p_t\tilde{\eta})\mathcal{N}^1-\sigma\p_1\left(\frac{\p_1\p_t\tilde{\eta}}{(1+|\p_1\zeta_0|^2)^{3/2}}\right)\mathcal{N}^1\\&\quad\quad\quad\quad\quad\quad\quad-\sigma\p_{1}(\mathcal{R}_{z}(\p_{1}\zeta_{0},\p_{1}\eta^{1})\p_{1}\p_{t}\tilde{\eta})\mathcal{N}^{1}(t)\\&\quad\quad\quad\quad\quad\quad\quad-\sigma \p_{1}((\mathcal{R}_{z}(\p_{1}\zeta_{0},\p_{1}\eta^{2})-\mathcal{R}_{z}(\p_{1}\zeta_{0},\p_{1}\eta^{1}))\p_{1}\p_{t}\tilde{\eta}^{2})\mathcal{N}^{1}(t)+R^{4,1} \ &\text{on}&\ \Sigma,\\
					&(S_{\mathcal{A}^1}(\p_t\tilde{p},\p_t\tilde{u})\nu-\beta \p_t\tilde{u})\cdot\tau=\mu\mathbb{D}_{\p_t\mathcal{A}^1-\p_t\mathcal{A}^2}{u}^2\nu\cdot\tau+R^{5,1} \  &\text{on}&\ \Sigma_s,\\
					&\p_t\tilde{u}\cdot\nu=0 \  &\text{on}&\ \Sigma_s,\\
					&\p_t(\p_t\tilde{\eta})=\p_t\tilde{u}\cdot\mathcal{N}^1+R^{6,1}+u_{1}^{1}\p_{1}\p_{t}\tilde{\eta}+u_{1}\p_{1}\p_{t}\tilde{\eta}^{2} \  &\text{on}&\ \Sigma,\\
					&\kappa(\p_{t}\tilde{u}^{1}\cdot \mathcal{N}^{1}-\p_{t}\tilde{u}^{2}\cdot \mathcal{N}^{2})(\pm\ell,t)=\mp\sigma\frac{\p_1\p_t\tilde{\eta}}{(1+|\p_1 \zeta_0|^2)^{3/2}}(\pm\ell,t)\mp\sigma\p_{1}(\mathcal{R}_{z}(\p_{1}\zeta_{0},\p_{1}\eta^{1})\p_{1}\p_{t}\tilde{\eta}\\&\quad\quad\quad\quad\quad\quad\quad\mp\sigma \p_{1}((\mathcal{R}_{z}(\p_{1}\zeta_{0},\p_{1}\eta^{2})-\mathcal{R}_{z}(\p_{1}\zeta_{0},\p_{1}\eta^{1}))\p_{1}\p_{t}\tilde{\eta}^{2})-R^{7,1}
				\end{aligned}
				\right.
			\end{equation}
			with zero initial data, where $R^{1},R^{4},R^{7}$,$R^{1,1}$, $R^{2,1}$, $R^{4,1}$,$R^{5,1}$, $R^{6,1}$, $R^{7,1}$ are defined by 
			\begin{equation}\label{remainder2'}
				\begin{aligned}
					R^{1,1}&=\p_tR^1+\mu\dive_{\p_t\mathcal{A}^1}(S_{\mathcal{A}^1}({u},{p}))+\mu\dive_{\p_t\mathcal{A}^1}(S_{(\mathcal{A}^1-\mathcal{A}^2)}({u}^2,{p}^{2}))\\
					&\quad+\mu\dive_{\mathcal{A}^1}(S_{(\p_{t}\mathcal{A}^1)}({u},{p}))+\mu \dive_{\mathcal{A}^{1}-\mathcal{A}^{2}}S_{\p_{t}\mathcal{A}^{2}}S(u^{2},p^{2})\\
					&\quad+\mu \operatorname{div}_{\p_{t}\mathcal{A}^{1}-\p_{t}\mathcal{A}^{2}}S_{\mathcal{A}^{2}}(p,u)+\mu\operatorname{div}_{\mathcal{A}^{1}}S_{\mathcal{A}^{1}-\mathcal{A}^{2}}(\p_{t}\tilde{p}^{2},\p_{t}\tilde{u}^{2})+\mu\operatorname{div}_{\mathcal{A}^{1}-\mathcal{A}^{2}}S_{\mathcal{A}^{1}}(\p_{t}\tilde{p}^{2},\p_{t}\tilde{u}^{2}),\\
					R^{2,1}&=-(\operatorname{div}_{\p_{t}\mathcal{A}^{1}-\p_{t}\mathcal{A}^{2}}{u}^{2})-(\operatorname{div}_{\mathcal{A}^{1}-\mathcal{A}^{2}}\p_{t}\tilde{u}^{2})-\operatorname{div}_{\p_{t}\mathcal{A}^{1}}u,\\
					R^{4,1}&=\p_tR^4 +\mu S_{\mathcal{A}^1}({u},{p})\p_t\mathcal{N}^1 +\mu S_{\p_{t}\mathcal{A}^{1}}({u},{p})\mathcal{N}^{1}+g({\eta})\p_t\mathcal{N}^1-\sigma\p_1\left(\frac{\p_1{\eta}}{(1+|\p_1\zeta_0|^2)^{3/2}}\right)\p_t\mathcal{N}^1\\
					&\quad-\sigma \p_{1}(\mathcal{R}_{z}(\p_{1}\zeta_{0},\p_{1}\eta^{1})\p_{1}\p_{t}\tilde{\eta}^{1})(\mathcal{N}^{1}-\mathcal{N}^{2})-\p_{1}(\mathcal{R}(\p_{1}\zeta_{0},\p_{1}\eta^{1}))\p_{t}(\mathcal{N}^{1}-\mathcal{N}^{2})\\
					&\quad-(\p_{1}\mathcal{R}(\p_{1}\zeta_{0},\p_{1}\eta^{1}
					)-\p_{1}\mathcal{R}(\p_{1}\zeta_{0},\p_{1}\eta^{2}))\p_{t}\mathcal{N}^{2}+\mu S_{\mathcal{A}^{1}-\mathcal{A}^{2}}(\p_{t}\tilde{u}^{2},\p_{t}\tilde{p}^{2})\mathcal{N}^{2}+\mu S_{\mathcal{A}^{1}-\mathcal{A}^{2}}({u}^{2},{p}^{2})\p_{t}\mathcal{N}^{2}\\
					&\quad+g({\eta}^{2})(\p_t\mathcal{N}^1-\p_{t}\mathcal{N}^{2})-\sigma\p_1\left(\frac{\p_1{\eta}^{2}}{(1+|\p_1\zeta_0|^2)^{3/2}}\right)(\p_t\mathcal{N}^1-\p_{t}\mathcal{N}^{2})\\
					&\quad+g(\p_{t}\tilde{\eta}^{2})(\mathcal{N}^1-\mathcal{N}^{2})-\sigma\p_1\left(\frac{\p_1{\p_{t}\tilde{\eta}}^{2}}{(1+|\p_1\zeta_0|^2)^{3/2}}\right)(\mathcal{N}^1-\mathcal{N}^{2}),\\
					R^{5,1}&=-S_{\p_{t}\mathcal{A}^{1}}({p},{u})
					+\mu S_{\mathcal{A}^{1}-\mathcal{A}^{2}}(\p_{t}\tilde{u}^{2},\p_{t}\tilde{p}^{2}),\\
					\quad R^{6,1}&= \p_{t}\tilde{u}^{2}\cdot(\mathcal{N}^{1}-\mathcal{N}^{2}),\\
					\quad R^{7,1}&=\p_tR^7,
				\end{aligned}
			\end{equation}
			where
			\begin{equation}{\label{remainder0'}}
				\begin{aligned}
					R^{1}=&u\cdot \nabla_{\mathcal{A}^{1}}u^{1}+u^{2}\cdot \nabla_{\mathcal{A}^{1}-\mathcal{A}^{2}}u^{1}+u^{2}\cdot \nabla_{\mathcal{A}^{2}}u+\p_{t}\bar{\eta}W^{1}K^{1}\p_{2}u^{1}+\p_{t}\bar{\eta}^{2}(W^{1}-W^{2})K^{1}\p_{2}u^{1}\\
					&\quad+\p_{t}\bar{\eta}^{2}W^{2}(K^{1}-K^{2})\p_{2}u^{1}+\p_{t}\bar{\eta}^{2}W^{2}K^{2}\p_{2}u^1,\\
					R^{4}=&S_{\mathcal{A}^{2}}(\tilde{p}^{2},\tilde{u}^{2})(\mathcal{N}^{1}-\mathcal{N}^{2}),&\\
					R^{7}=&\hat{\mathcal{W}}(\p_{t}\eta^{1})-\hat{\mathcal{W}}(\p_{t}\eta^{2}).&
				\end{aligned}
			\end{equation}
			
			\paragraph{\underline{Step 3 -- Introducing the Test Function}}
			In this step, we introduce the test function. 
			We begin by introducing the following notations about the metrics.
			\[
			d(u,p,\eta):=d((u^{1},p^{1},\eta^{1}),(u^{2},p^{2},\eta^{2})), \quad
			d(\tilde{u},\tilde{p},\tilde{\eta}):=d((\tilde{u}^{1},\tilde{p}^{1},\tilde{\eta}^{1}),(\tilde{u}^{2},\tilde{p}^{2},\tilde{\eta}^{2})).
			\]
			\noindent We then introduce $w\in H_{0}^{1}(\Omega)$ that solves the following equation
			\begin{align}{\label{eq:Bogo}}
				\begin{aligned}
					J^{1}\dive_{\mathcal{A}^{1}}w=J^{1}R^{2,1}-\left<J^{1}R^{2,1}\right>_{\Omega},~~~~\operatorname{where}~ \left<J^{1}R^{2,1} \right>_{\Omega}=\frac{1}{|\Omega|}\int_{\Omega}J^{1}R^{2,1}.
				\end{aligned}
			\end{align}
			We refer to \cite[Section 10]{GT2020} for the details of solvability of \eqref{eq:Bogo} and the corresponding estimate for $w$.
			
			\paragraph{\underline{Step 4 -- Energy Estimates for $(\p_t\tilde{u}, \p_t\tilde{\eta})$}}
			In this step, we aim to establish the energy estimates for $(\p_t\tilde{u}, \p_t\tilde{\eta})$. Multiplying the first equation of \eqref{linear_fix2} by {$J^{1}(\p_{t}\tilde{u}-w)$}, integrating over $\Om$ and using integration by parts, we obtain the following equation
			\begin{equation}\label{energy_fix}
				\begin{aligned}
					&\frac{d}{dt}\left(\int_\Om\frac{|\p_t\tilde{u}|^2}2J^1+\int_{-\ell}^\ell\frac g2|\p_t\tilde{\eta}|^2+\frac\sigma2\frac{|\p_1\p_t\tilde{\eta}|^2}{(1+|\p_1\zeta_0|^2)^{3/2}}-\sigma\int_{-\ell}^{\ell}\mathcal{R}_{z}(\p_{1}\zeta_{0},\p_{1}\eta^{1})(\p_{1}\p_{t}\tilde{\eta})^{2}-\int_{\Omega}J^{1}\p_{t}\tilde{u}w\right)\\  &\quad+\frac\mu2\int_\Om|\mathbb{D}_{\mathcal{A}^1}\p_t\tilde{u}|^2J^1+\beta\int_{\Sigma_s}|\p_t\tilde{u}|^2J^1+[\p_t{u}^{1}\cdot\mathcal{N}^1-\p_{t}u^{2}\cdot \mathcal{N}^{2}]_\ell^2 \\
					&=\int_\Om\frac{|\p_t\tilde{u}|^2}2\p_tJ^1-\frac\mu2\int_\Om\left( \mathbb{D}_{\p_t\mathcal{A}^1-\p_t\mathcal{A}^2}{u}^2\right):\mathbb{D}_{\mathcal{A}^1}(\p_t\tilde{u}-w)J^1+\int_{\Omega}\p_{t}\tilde{u}\p_{t}(J^{1}w)+\int_{\Omega}\p_{t}\tilde{p}\langle R^{2,1}J^{1}\rangle\\
					&\quad-\int_{-\ell}^\ell R^{4,1}\cdot\p_t\tilde{u}+\p_{1}R^{6,1}(\sigma\frac{\p_{1}\p_{t}\tilde{\eta}}{(1+\vert \p_{1}\zeta_{0}\vert^{\frac{3}{2}})^{\frac{3}{2}}}-\sigma(\mathcal{R}_{z}(\p_{1}\zeta_{0},\p_{1}\eta^{1})\p_{1}\p_{t}\tilde{\eta}))\\
					&\quad+\int_{-\ell}^{\ell}(R^{6,1})(g\p_{t}^{2}\tilde{\eta})+\int_{-\ell}^{\ell} (\p_{1}R^{6,1})\sigma (\mathcal{R}_{z}(\p_{1}\zeta_{0},\p_{1}\eta^{2})-\mathcal{R}_{z}(\p_{1}\zeta_{0},\p_{1}\eta^{1})\p_{1}\p_{t}\tilde{\eta}^{2})\\
					&\quad+\int_{-\ell}^{\ell} \p_{1}(u_{1}^{1}\p_{1}\p_{t}\tilde{\eta}+u_{1}\p_{1}\p_{t}\tilde{\eta}^{2})((g\p_{t}\tilde{\eta}+\sigma\frac{\p_{1}\p_{t}\tilde{\eta}}{(1+\vert \p_{1}\zeta_{0}\vert^{\frac{3}{2}})^{\frac{3}{2}}}-\sigma(\mathcal{R}_{z}(\p_{1}\zeta_{0},\p_{1}\eta^{1})\p_{1}\p_{t}\tilde{\eta})))\\
					&\quad+ \int_{-\ell}^{\ell} \p_{1}(u^{1}\p_{1}\p_{t}\tilde{\eta}+u\p_{1}\p_{t}\tilde{\eta}^{2})(\sigma ((\mathcal{R}_{z}(\p_{1}\zeta_{0},\p_{1}\eta^{2})-\mathcal{R}_{z}(\p_{1}\zeta_{0},\p_{1}\eta^{1}))\p_{1}\p_{t}\tilde{\eta}^{2}))\\
					&\quad+\int_{-\ell}^{\ell}\p_{1}\p_{t}\tilde{\eta}\sigma ((\mathcal{R}_{z}(\p_{1}\zeta_{0},\p_{1}\eta^{2})-\mathcal{R}_{z}(\p_{1}\zeta_{0},\p_{1}\eta^{1}))\p_{1}\p_{t}\tilde{\eta}^{2})+\int_\Om R^{1,1}\cdot(\p_t\tilde{u}J^1-wJ^{1})\\
					&\quad+\int_{-\ell}^{\ell} R^{5,1}\cdot \p_{t}\tilde{u}-[\p_t\tilde{u}\cdot\mathcal{N}^1,  R^{7,1}]_\ell-[\p_{t}^{2}\tilde{u}^{1}\cdot \mathcal{N}^{1}-\p_{t}\tilde{u}^{2}\cdot \mathcal{N}^{2},\p_{t}u^{2}\cdot (\mathcal{N}^{2}-\mathcal{N}^{1})]_{\ell}.
				\end{aligned}
			\end{equation}
			
			We now estimate each term on the right-hand side of \eqref{energy_fix}. For the first line of the RHS of \eqref{energy_fix}, using H\"older's inequality, we have
			\[
			\begin{aligned}
				\int_{0}^{T}\int_{\Omega} \frac{|\p_{t}\tilde{u}|^{2}}{2}\p_{t}J^{1}&\lesssim \|\p_{t}u\|_{L_{t}^{2}L^{2}}^{2}\|\p_{t}\bar{\eta}^{1}\|_{L_{t}^{\infty
					}W^{1,\infty}}\lesssim \|\p_{t}\tilde{u}\|_{L_{t}^{2}L^{2}}^{2}\|\p_{t}{\eta}^{1}\|_{L_{t}^{\infty
					}H^{\frac{3}{2}+\frac{\varepsilon_{-}-\alpha}{2}}}\lesssim\delta^{\frac{1}{2}}d(\tilde{u},\tilde{\eta},\tilde{p})d(u,\eta,p),
			\end{aligned}
			\]
			\[
			\begin{aligned}
				\int_{0}^{T}\int_{\Omega}\frac{\mu}{2}(\mathbb{D}_{\p_{t}\mathcal{A}^{1}-\p_{t}\mathcal{A}^{2}}u^{2}):\mathbb{D}_{\mathcal{A}^{1}}(\p_{t}\tilde{u})J^{1}&\lesssim \|\p_{t}\bar{\eta}\|_{L_{t}^{\infty}H^{1}}\|u^{2}\|_{L_{t}^{\infty}H^{1}}\|\p_{t}^{2}\tilde{u}\|_{L_{t}^{2}H^{1}}&\lesssim \delta^{\frac{1}{2}}d(\tilde{u},\tilde{\eta},\tilde{p})d(u,\eta,p).
			\end{aligned}
			\]
			
			\subparagraph{\underline{Terms related to $R^1$}}
			For terms in the bulk, by definition in \eqref{remainder2'}, we first handle terms involving $\p_tR^1$ as follows
			\begin{equation}\label{est:remain_bulk11}
				\begin{aligned}
					&\int_\Om \p_tR^1\cdot\p_t\tilde{u}J^1\\
					&\lesssim
					\int_\Om ((|\p_t\bar{\eta}^1|+|\p_t\bar{\eta}^2|)|\nabla\p_{t}\bar{\eta}|+|\p_t^2\bar{\eta}^2||\nabla\bar{\eta}|)|\p_2u^1||\p_t\tilde{u}|+\int_\Om ((|\p_t\bar{\eta}^1|+|\p_t\bar{\eta}^2|)|\nabla\bar{\eta}|)|\p_2\p_{t}u^1||\p_t\tilde{u}|\\
					&\quad+\int_\Om ((|\p_t\bar{\eta}^1|+|\p_t\bar{\eta}^2|)|\nabla\p_t\bar{\eta}^{1}|+(|\p_t^{2}\bar{\eta}^1|+|\p_t^{2}\bar{\eta}^2|)|\nabla\bar{\eta}^{1}|)|\p_2u||\p_t\tilde{u}|+\int_\Om ((|\p_t\bar{\eta}^1|+|\p_t\bar{\eta}^2|)|\nabla\bar{\eta}^{1}|)|\p_2\p_{t}u||\p_t\tilde{u}|\\
					&\quad+\int_\Om ((|\p_t^{2}\bar{\eta}|)|\nabla\bar{\eta}^{2}|)|\p_2u^{2}||\p_t\tilde{u}|+\int_\Om ((|\p_t\bar{\eta}|)|\nabla\bar{\eta}^{2}|)|\p_2\p_{t}u^{2}||\p_t\tilde{u}|+((|\p_t\bar{\eta}|)|\nabla\p_t\bar{\eta}^{2}|)|\p_2u^{2}||\p_t\tilde{u}|\\
					&\quad+\int_\Om(|\p_tu|+|u||\nabla\p_t\bar{\eta}^1|)|\nabla u^1||\p_t\tilde{u}|+\int_{\Omega}(|\p_{t}u|)|\nabla \p_tu^1||\p_t\tilde{u}|+\int_{\Omega}(|u^{1}||\nabla \bar{\eta}|)|\nabla \p_{t}u^{1}||\p_{t}\tilde{u}|\\
					&\quad+\int_\Om(|\p_tu^2|+|u^{2}|\nabla \p_{t}\bar{\eta}^{2}||)|\nabla u||\p_t\tilde{u}|+(u^2)|\nabla\p_tu||\p_t\tilde{u}|+\int_{\Omega}(|\p_{t}u^{1}||\nabla\bar{\eta}|+|u^{1}||\p_{t}\bar{\eta}|)|\nabla u||\p_{t}\tilde{u}|,
				\end{aligned}
			\end{equation}
			where we have used $\|\mathcal{A}^i\|_{L^\infty}+\|J^i\|_{L^\infty}+\|K^i\|_{L^\infty}\lesssim 1$, $i=1, 2$.
			Then by H\"older's inequality, Sobolev embedding and trace theory, after integrating over $[0, T]$, the second line in \eqref{est:remain_bulk11} is bounded as follows
			\begin{equation}\label{est:remain_bulk12}
				\begin{aligned}
					&\int_0^T\|\nabla\p_t\bar{\eta}\|_{L^{2/\alpha}}((\|\p_{t}\bar{\eta}^2\|_{L^{\infty}}+\|\p_{t}\bar{\eta}^{1}\|_{L^{\infty}})\|\nabla\tilde{u}^2\|_{L^2}
					)\|\p_t\tilde{u}\|_{L^{2/(1-\alpha)}}\\
					&\quad+\|\nabla\bar{\eta}\|_{L^\infty}(\|\p_t\bar{\eta}^2\|_{L^{\infty}}+\|\p_{t}\bar{\eta}^{1}\|_{L^{\infty}})\|\p_{t}\nabla\tilde{u}^2\|_{L^{q_-}}\|\p_t\tilde{u}\|_{L^{2/\varepsilon_-}}\\&\quad+\|\nabla \bar{\eta}\|_{L^{\frac{2}{\alpha}}}(\|\p_t^{2}\bar{\eta}^2\|_{L^{\infty}}+\|\p_{t}^{2}\bar{\eta}^{1}\|_{L^{\infty}})\|\nabla\tilde{u}^2\|_{L^{q_-}}\|\p_{t}\tilde{u}\|_{L^{\frac{2}{\varepsilon_{-}-\alpha}}}\\
					&\lesssim \int_0^T\|\p_t\eta\|_{H^{3/2-\alpha}}(\|\p_{t}\eta^{1}\|_{H^{\frac{3}{2}+\frac{\varepsilon_{-}-\alpha}{2}}}+\|\p_{t}\eta^{2}\|_{H^{\frac{3}{2}+\frac{\varepsilon_{-}-\alpha}{2}}})\|u^{2}\|_{W^{2,q_{-}}}\|\p_t\tilde{u}\|_1\\
					&\quad+\int_0^T\|\eta\|_{W^{3-\frac{1}{q_{-}},q_{-}}}(\|\p_{t}\eta^{1}\|_{H^{\frac{3}{2}+\frac{\varepsilon_{-}-\alpha}{2}}}+\|\p_{t}\eta^{2}\|_{H^{\frac{3}{2}+\frac{\varepsilon_{-}-\alpha}{2}}})\|\p_{t}u^{2}\|_{W^{2,q_{-}}}\|\p_t\tilde{u}\|_1\\
					&\quad+\int_0^T\|\eta\|_{W^{3-\frac{1}{q_{-}},q_{-}}}(\|\p_{t}^{2}\eta^{1}\|_{H^{1}}+\|\p_{t}^{2}\eta^{2}\|_{H^{1}})\|u^{2}\|_{W^{2,q_{-}}}\|\p_t\tilde{u}\|_1\\
					&\lesssim \delta^{1/2}(\|\p_t\eta\|_{L^2H^{3/2-\alpha}}+\|\eta\|_{L^{\infty}W^{3-\frac{1}{q_{-}},q_{-}}})\|\p_t\tilde{u}\|_{L^2H^1}.
				\end{aligned}
			\end{equation}
			The third line in \eqref{est:remain_bulk11}, after an integration over $[0, T]$, can be estimated similarly by
			\begin{equation}\label{est:remain_bulk13}
				\begin{aligned}
					&\int_0^T\|\nabla\p_t\bar{\eta}^{1}\|_{L^{2/\alpha}}((\|\p_{t}\bar{\eta}^2\|_{L^{\infty}}+\|\p_{t}\bar{\eta}^{1}\|_{L^{\infty}})\|\nabla{u}\|_{L^2}
					)\|\p_t\tilde{u}\|_{L^{2/(1-\alpha)}}\\
					&\quad+\|\nabla\bar{\eta}^{1}\|_{L^{4/(\alpha-\varepsilon_-)}}(\|\p_t\bar{\eta}^2\|_{L^{\infty}}+\|\p_{t}\bar{\eta}^{1}\|_{L^{\infty}})\|\nabla\p_{t}{u}\|_{L^{2}}\|\p_t\tilde{u}\|_{L^{2/(1+\varepsilon_{-}-\alpha)}}\\
					&\quad+\|\nabla\bar{\eta}^{1}\|_{L^{\frac{2}{\alpha}}}(\|\p_t\bar{\eta}^2\|_{L^{\infty}}+\|\p_{t}\bar{\eta}^{1}\|_{L^{\infty}})\|\p_{t}\nabla{u}\|_{L^{q_-}}\|\p_{t}\tilde{u}\|_{L^{\frac{2}{\varepsilon_{-}-\alpha}}}\\
					&\lesssim \int_0^T\|\p_t\eta^{1}\|_{H^{3/2-\alpha}}(\|\p_{t}\eta^{1}\|_{H^{\frac{3}{2}+\frac{\varepsilon_{-}-\alpha}{2}}}+\|\p_{t}\eta^{2}\|_{H^{\frac{3}{2}+\frac{\varepsilon_{-}-\alpha}{2}}})\|u\|_{W^{2,q_{-}}}\|\p_t\tilde{u}\|_1\\
					&\quad+\int_0^T\|\eta^{1}\|_{W^{3-\frac{1}{q_{-}},q_{-}}}(\|\p_{t}\eta^{1}\|_{H^{\frac{3}{2}+\frac{\varepsilon_{-}-\alpha}{2}}}+\|\p_{t}\eta^{2}\|_{H^{\frac{3}{2}+\frac{\varepsilon_{-}-\alpha}{2}}})\|u\|_{W^{2,q_{-}}}\|\p_t\tilde{u}\|_1\\
					&\quad+\int_0^T\|\eta^{1}\|_{W^{3-\frac{1}{q_{-}},q_{-}}}(\|\p_{t}\eta^{1}\|_{H^{\frac{3}{2}+\frac{\varepsilon_{-}-\alpha}{2}}}+\|\p_{t}\eta^{2}\|_{H^{\frac{3}{2}+\frac{\varepsilon_{-}-\alpha}{2}}})\|\p_{t}u\|_{H^{1}}\|\p_t\tilde{u}\|_1\\
					&\lesssim \delta^{1/2}(\|u\|_{L^{2}W^{2,q_{+}}}+\|\p_{t}u\|_{L^{2}H^{1}})\|\p_t\tilde{u}\|_{L^2H^1}.
				\end{aligned}
			\end{equation}
			\noindent For the fourth line in \eqref{est:remain_bulk11}, we can use similar computation to show that it is bounded by:
			\begin{align}
				\begin{aligned}
					&\int_0^T\|\nabla\bar{\eta}^{2}\|_{L^{\infty}}((\|\p_{t}^{2}\bar{\eta}\|_{L^{\frac{2}{\alpha}}})\|\nabla{u}^{2}\|_{L^2}
					)\|\p_t\tilde{u}\|_{L^{2/(1-\alpha)}}+\|\nabla\p_{t}\bar{\eta}^{2}\|_{L^{\infty}}(\|\p_{t}\bar{\eta}\|_{L^{\infty}})\|\nabla{u}^{2}\|_{L^{2}}\|\p_t\tilde{u}\|_{L^{2}}\\
					&\quad+\|\nabla\bar{\eta}^{1}\|_{L^\infty}(\|\p_{t}\bar{\eta}\|_{L^{\infty}})\|\p_{t}\nabla{u}^{2}\|_{L^{2}}\|\p_t\tilde{u}\|_{L^{2}}\\
					&\lesssim \int_0^T\|\p_t^{2}\eta\|_{1/2-\alpha}(\|\eta^{2}\|_{H^{\frac{3}{2}+\frac{\varepsilon_{-}-\alpha}{2}}})\|u^{2}\|_{H^{1}}\|\p_t\tilde{u}\|_1+\int_0^T\|\p_t\eta^{2}\|_{H^{\frac{3}{2}+\frac{\varepsilon_{-}-\alpha}{2}}}(\|\p_{t}\eta\|_{H^{\frac{3}{2}-\alpha}})\|u^{2}\|_{W^{2,q_{-}}}\|\p_{t}\tilde{u}\|_1\\
					&\quad+\int_0^T\|\eta^{1}\|_{W^{3-\frac{1}{q_{-}},q_{-}}}(\|\p_{t}\eta\|_{H^{\frac{3}{2}-\alpha}})\|\p_{t}u^{2}\|_{H^{1}}\|\p_t\tilde{u}\|_1\\
					&\lesssim \delta^{1/2}(\|\p_{t}\eta\|_{L^{2}H^{\frac{3}{2}-\alpha}}+\|\p_{t}\eta\|_{L^{2}H^{\frac{1}{2}-\alpha}})\|\p_t\tilde{u}\|_{L^2H^1},
				\end{aligned}
			\end{align}
			where we have used the Sobolev embedding $H^{\varepsilon_-/2}\hookrightarrow L^{4/(2-\varepsilon_-)}$ and $H^1\hookrightarrow L^{4/(2\varepsilon_++\varepsilon_--2\alpha)}$ for bounded domain $\Om$.
			The fifth line in \eqref{est:remain_bulk11}, after an integration over $[0, T]$ (without loss of generality, we may assume $T<1$), is estimated by
			\begin{equation}\label{est:remain_bulk14}
				\begin{aligned}
					&\int_{0}^{T}(\|\p_{t}u\|_{L^{\frac{4}{\varepsilon_{-}}}}+\|u\|_{L^{\infty}}\|\p_{t}\bar\eta^{1}\|_{W^{1,\infty}})\|u^{1}\|_{W^{1,\frac{4}{2-\varepsilon_{-}}}}\|\p_{t}\tilde{u}\|_{H^{1}}+\int_{0}^{T}\|u\|_{L^{\infty}}\|\p_{t}u^{1}\|_{H^{1}}\|\p_{t}\tilde{u}\|_{L^{2}}\\
					&+\int_{0}^{T}(\|u^{1}\|_{L^{\infty}}\|\bar{\eta}\|_{W^{1,\infty}})\|\p_{t}u^{1}\|_{H^{1}}\|\p_{t}\tilde{u}\|_{H^{1}}\\
					\lesssim & \delta^{\frac{1}{2}} (\|\p_{t}u\|_{L^{2}H^{1}}+\|u\|_{L^{2}W^{2,q_{+}}}+\|\eta\|_{L^{\infty}H^{\frac{3}{2}+\frac{\varepsilon_{-}-\alpha}{2}}})\|\p_{t}\tilde{u}\|_{L^{2}H^{1}}.
				\end{aligned}
			\end{equation}
			After an integration over $[0, T]$, the last line in \eqref{est:remain_bulk11} can be similarly bounded by
			\begin{equation}\label{est:remain_bulk15_1}
				\begin{aligned}
					\delta^{\frac{1}{2}}(\|u\|_{L^{\infty}W^{2,q_{-}}}+\|\p_{t}u\|_{L^2W^{2,q_{-}}}+\|\p_{t}^{2}u\|_{L^{2}H^{1}})\|\p_{t}^{2}\tilde{u}\|_{L^{2}H^{1}}.
				\end{aligned}
			\end{equation}
			\noindent We omit the details here.
			\noindent Combining \eqref{est:remain_bulk11}--\eqref{est:remain_bulk15_1}, we obtain the following estimate
			\begin{equation}\label{est:remain_bulk15}
				\begin{aligned}
					\int_{0}^{T}\int_{\Omega}\p_{t}R^{1}\cdot \p_{t}\tilde{u}J^{1}\lesssim \delta^{1/2}d(u,p,\eta)d(\tilde{u},\tilde{p},\tilde{\eta}).
				\end{aligned}
			\end{equation}
			
			\subparagraph{\underline{Other terms in $R^{1,1}$}}
			We now estimate other terms involved in expression of $R^{1,1}$ such that
			\begin{equation}\label{est:remain_bulk21}
				\begin{aligned}
					&\int_0^T\int_\Om \dive_{\p_t\mathcal{A}^1}(\mathbb{D}_{(\mathcal{A}^1-\mathcal{A}^2)}{u}^2)\cdot\p_t\tilde{u}J^1\lesssim\int_0^T\int_\Om |\nabla\p_t\bar{\eta}^1|\left(|\nabla^2\bar{\eta}||\nabla{u}^2|+|\nabla\bar{\eta}||\nabla^2{u}^2|\right)|\p_t\tilde{u}|\\
					&\lesssim\int_0^T\|\nabla\p_t\bar{\eta}^1\|_{L^\infty}(\|\nabla^2\bar{\eta}\|_{L^{2/(1-\varepsilon_-)}}\|\nabla{u}^2\|_{L^{2/(1-\varepsilon_-)}}\|\p_t\tilde{u}\|_{L^{1/\varepsilon_-}}+\|\nabla\bar{\eta}\|_{L^\infty}\|\nabla^2{u}^2\|_{L^{q_-}}\|\p_t\tilde{u}\|_{L^{2/\varepsilon_-}})\\
					&\lesssim\int_0^T\|\p_t\eta^1\|_{H^{3/2+(\varepsilon_--\alpha)/2}}\|\eta\|_{W^{3-1/q_-,q_-}}\|{u}^2\|_{W^{2,q_-}}\|\p_t\tilde{u}\|_{1}\\
					&\lesssim  \|\p_t\eta^1\|_{L^{\infty}H^{3/2+(\varepsilon_--\alpha)/2}}\|\eta\|_{L_{t}^{2}W^{3-1/q_-,q_-}}\|{u}^2\|_{L_{t}^{\infty}W^{2,q_-}}\|\p_t\tilde{u}\|_{L_{t}^{2}H^{1}}\lesssim\delta\|\p_{t}\eta\|_{L_{t}^2W^{3-1/q_-,q_-}}\|\p_t\tilde{u}\|_{L_{t}^2H^1}.
				\end{aligned}
			\end{equation}
			Then using the similar estimate, we obtain the following estimate
			\begin{equation}
				\begin{aligned}
					&\int_0^T\int_\Om \dive_{\mathcal{A}^1}(\mathbb{D}_{(\mathcal{A}^1-\mathcal{A}^2)}\p_{t}{u}^2)\cdot\p_t\tilde{u}J^1\lesssim\int_0^T\int_\Om \left(|\nabla^2\bar{\eta}||\p_{t}\nabla{u}^2|+|\nabla\bar{\eta}||\p_{t}\nabla^2{u}^2|\right)|\p_t\tilde{u}|\\
					&\lesssim\int_0^T(\|\nabla^2\bar{\eta}\|_{L^{2/(1-\varepsilon_-)}}\|\p_{t}\nabla{u}^2\|_{L^{4/(1-\varepsilon_-)}}\|\p_t\tilde{u}\|_{L^{1/\varepsilon_-}}+\|\nabla\bar{\eta}\|_{L^\infty}\|\p_{t}\nabla^2{u}^2\|_{L^{q_-}}\|\p_t\tilde{u}\|_{L^{2/\varepsilon_-}})\\
					&\lesssim\int_0^T\|\eta\|_{H^{\frac{3}{2}+\frac{\varepsilon_{-}-\alpha}{2}}}\|\p_{t}{u}^2\|_{W^{2,q_-}}\|\p_t\tilde{u}\|_{1}+\int_{0}^{T}\|\eta\|_{W^{3-\frac{1}{q_{-}},q_{-}}}\|\p_{t}u^{2}\|_{H^{1+\frac{\varepsilon_{-}}{2}}}\|\p_{t}\tilde{u}\|_{H^{1}}\\
					&\lesssim  \|\eta\|_{L_{t}^{2}W^{3-1/q_-,q_-}}\|\p_{t}{u}^2\|_{L_{t}^{\infty}H^{1+\frac{\varepsilon_{-}}{2}}}\|\p_t^{2}\tilde{u}\|_{L_{t}^{2}H^{1}}+ \|\eta\|_{L_{t}^{\infty}H^{\frac{3}{2}+\frac{\varepsilon_{-}-\alpha}{2}}}\|\p_{t}{u}^2\|_{L_{t}^{\infty}W^{2,q_{-}}}\|\p_{t}\tilde{u}\|_{L_{t}^{2}H^{1}}\\
					&\lesssim\delta(\|\eta\|_{L^2W^{3-1/q_-,q_-}}+\|\eta\|_{L_{t}^{\infty}H^{1+\frac{\varepsilon_{-}}{2}}})\|\p_t\tilde{u}\|_{L^2H^1}.
				\end{aligned}
			\end{equation}
			Next, we have the following estimate
			\begin{equation}\label{est:remain_bulk22}
				\begin{aligned}
					&\int_0^T\int_\Om (\mu\dive_{\mathcal{A}^1}(\mathbb{D}_{\p_{t}\mathcal{A}^1}{u})\cdot\p_t\tilde{u}J^1\lesssim\int_0^T\int_\Om \left(|\p_{t}\nabla^2\bar{\eta}^1||\nabla{u}|+|\p_{t}\nabla \bar{\eta}^{1}||\nabla^2{u}|\right)|\p_t\tilde{u}|\\
					&\lesssim\int_0^T\big[\|\p_{t}\nabla^2\bar{\eta}^1\|_{L^{2/(1-\varepsilon_-)}}\|\nabla {u}\|_{L^{4/(2-\varepsilon_-)}}\|\p_t\tilde{u}\|_{L^{4/(3\varepsilon_-)}}+\|\p_{t}\bar{\eta}^{1}\|_{W^{1,+\infty}}\|\nabla^2{u}\|_{L^{q_-}}\|\p_t\tilde{u}\|_{L^{2/\varepsilon_-}}\big]\\
					&\lesssim\int_0^T\|\p_t\eta^1\|_{H^{3/2+(\varepsilon_--\alpha)/2}}\|{u}\|_{W^{2,q_-}}\|\p_t\tilde{u}\|_{1}+\int_{0}^{T}\|\p_{t}\eta^{1}\|_{W^{3-\frac{1}{q_{-}},q_-}}\|u\|_{H^{1+\frac{\varepsilon_{-}}{2}}}\|\p_{t}\tilde{u}\|_{1}\\
					&\lesssim\delta^{1/2}(\|{u}\|_{L_{t}^2W^{2,q_-}}+\|u\|_{L_{t}^{\infty}H^{1+\frac{\varepsilon_{-}}{2}}})\|\p_t\tilde{u}\|_{L_{t}^2H^1}.
				\end{aligned}
			\end{equation}
			Similarly, it holds that
			\begin{align}
				\begin{aligned}
					&\int_0^T\int_\Om (\mu\dive_{\p_{t}\mathcal{A}^1}(\mathbb{D}_{\mathcal{A}^1}{u})\cdot\p_t\tilde{u}J^1\lesssim\int_0^T\int_\Om |\p_{t}\nabla\bar{\eta}^{1}|\left(|\nabla^2\bar{\eta}^1||\nabla{u}|+|\nabla^2{u}|\right)|\p_t\tilde{u}|\\
					&\lesssim\int_0^T\|\p_{t}\nabla \bar{\eta}^{1}\|_{L^{\infty}}\big[\|\nabla^2\bar{\eta}^1\|_{L^{2/(1-\varepsilon_-)}}\|\nabla {u}\|_{L^{2/(1-\varepsilon_-)}}\|\p_t\tilde{u}\|_{L^{1/\varepsilon_-}}+\|\nabla^2{u}\|_{L^{q_-}}\|\p_t\tilde{u}\|_{L^{2/\varepsilon_-}}\big]\\
					&\lesssim\int_0^T\|\p_t\eta^1\|_{H^{3/2+(\varepsilon_--\alpha)/2}}\big[(1+\|\eta^1\|_{W^{3-1/q_-,q_-}})\|{u}\|_{W^{2,q_-}}\big]\|\p_t\tilde{u}\|_{1}\\
					&\lesssim\delta^{1/2}(\|{u}\|_{L_{t}^2W^{2,q_-}})\|\p_t\tilde{u}\|_{L_{t}^2H^1}.
				\end{aligned}
			\end{align}
			Finally, we have the following estimates
			\begin{equation}\label{est:remain_bulk23}
				\begin{aligned}
					&\int_0^T\int_\Om (\mu\dive_{\p_t\mathcal{A}^1-\p_{t}\mathcal{A}^{2}}(\mathbb{D}_{\mathcal{A}^2}{u}^{2})-\nabla_{\p_t\mathcal{A}^1-\p_{t}\mathcal{A}^{2}}{p}^{2})\cdot\p_t\tilde{u}J^1\lesssim\int_0^T\int_\Om |\nabla\p_t\bar{\eta}|\left(|\nabla^2\bar{\eta}^2||\nabla{u}^{2}|+|\nabla^2{u}^{2}|+|\nabla{p}^{2}|\right)|\p_t\tilde{u}|\\
					&\lesssim\int_0^T\|\nabla\p_t\bar{\eta}\|_{L^{\frac{2}{\alpha}}}\big[\|\nabla^2\bar{\eta}^2\|_{L^{2/(1-\varepsilon_-)}}\|\nabla{u}^{2}\|_{L^{2/(1-\varepsilon_-)}}\|\p_t\tilde{u}\|_{L^{1/\varepsilon_--\frac{\alpha}{2}}}+(\|\nabla^2{u}^{2}\|_{L^{q_-}}+\|\nabla{p}^{2}\|_{L^{q_-}})\|\p_t\tilde{u}\|_{L^{2/\varepsilon_-}}\big]\\
					&\lesssim\int_0^T\|\p_t\eta\|_{H^{\frac{3}{2}-\alpha}}\big[(1+\|\eta^2\|_{W^{3-1/q_-,q_-}})\|{u}^{2}\|_{W^{2,q_-}}+\|{p}^{2}\|_{W^{1,q_-}}\big]\|\p_t\tilde{u}\|_{1}\\
					&\lesssim \|\p_t\eta\|_{L_{t}^{2}H^{\frac{3}{2}-\alpha}}\big[(1+\|\eta^2\|_{L_{t}^{\infty}W^{3-1/q_-,q_-}})\|{u}^{2}\|_{L_{t}^{\infty}W^{2,q_-}}+\|{p}^{2}\|_{L_{t}^{\infty}W^{1,q_-}}\big]\|\p_t\tilde{u}\|_{L_{t}^{2}H^{1}}\\
					&\lesssim\delta^{1/2}\|\p_t\eta\|_{L_{t}^{2}H^{\frac{3}{2}-\alpha}}\|\p_t\tilde{u}\|_{L^2H^1},
				\end{aligned}
			\end{equation}
			and 
			\begin{equation}
				\begin{aligned}
					&\int_0^T\int_\Om (\mu\dive_{\mathcal{A}^1-\mathcal{A}^{2}}(\mathbb{D}_{\p_{t}\mathcal{A}^2}{u}^{2}))\cdot\p_t\tilde{u}J^1\lesssim\int_0^T\int_\Om |\nabla\bar{\eta}|\left(|\p_{t}\nabla^2\bar{\eta}^2||\nabla{u}^{2}|+|\p_{t}\nabla\bar{\eta}^{2}||\nabla^2{u}^{2}|\right)|\p_t\tilde{u}|\\
					&\lesssim\int_0^T\|\nabla\bar{\eta}\|_{L^{+\infty}}\big[\|\p_{t}\nabla^2\bar{\eta}^2\|_{L^{2/(1-\varepsilon_-)}}\|\nabla{u}^{2}\|_{L^{2/(1-\varepsilon_-)}}\|\p_t\tilde{u}\|_{L^{1/\varepsilon_-}}+\|\p_{t}\nabla\bar{\eta}\|_{L^{\infty}}(\|\nabla^2{u}^{2}\|_{L^{q_-}})\|\p_t\tilde{u}\|_{L^{2/\varepsilon_-}}\big]\\
					&\lesssim\int_0^T\|\eta\|_{H^{\frac{3}{2}+\frac{\varepsilon_{-}-\alpha}{2}}}\big[(\|\p_{t}\eta^2\|_{W^{3-1/q_-,q_-}})\|{u}^{2}\|_{W^{2,q_-}}\big]\|\p_t\tilde{u}\|_{1}\\
					&\lesssim \|\eta\|_{L_{t}^{\infty}H^{\frac{3}{2}+\frac{\varepsilon_{-}-\alpha}{2}}}\big[(\|\p_{t}\eta^2\|_{L_{t}^{2}W^{3-1/q_-,q_-}})\|{u}^{2}\|_{L_{t}^{\infty}W^{2,q_-}}\big]\|\p_t\tilde{u}\|_{L_{t}^{2}H^{1}}\\
					&\lesssim\delta^{1/2}\|\eta\|_{L_{t}^{\infty}H^{\frac{3}{2}+\frac{\varepsilon_{-}-\alpha}{2}}}\|\p_t\tilde{u}\|_{L^2H^1},
				\end{aligned}
			\end{equation}
			and
			\begin{equation}\label{est:remain_bulk231}
				\begin{aligned}
					&\int_0^T\int_\Om (\mu\dive_{\mathcal{A}^1-\mathcal{A}^{2}}(\mathbb{D}_{\mathcal{A}^2}\p_{t}\tilde{u}^{2})+\nabla_{\mathcal{A}^{1}-\mathcal{A}^{2}}\p_{t}\tilde{p}^{2})\cdot\p_t\tilde{u}J^1\\
					&\lesssim\int_0^T\int_\Om |\nabla\bar{\eta}|\left(|\nabla^2\bar{\eta}^2||\p_{t}\nabla\tilde{u}^{2}|+|\nabla\bar{\eta}^{2}||\p_{t}\nabla^2\tilde{u}^{2}|+|\nabla\p_{t}\tilde{p}^{2}|\right)|\p_t\tilde{u}|\\
					&\lesssim\int_0^T\|\nabla\bar{\eta}\|_{L^{\infty}}\big[\|\nabla^2\bar{\eta}^2\|_{L^{2/(1-\varepsilon_-)}}\|\p_{t}^{2}\nabla\tilde{u}^{2}\|_{L^{2/(1-\varepsilon_-)}}\|\p_t\tilde{u}\|_{L^{1/\varepsilon_-}}+(1+\|\nabla\bar{\eta}\|_{L^{\infty}})(\|\p_{t}\nabla^2\tilde{u}^{2}\|_{L^{q_-}}\\
					&\qquad+\|\p_{t}\nabla\tilde{p}^{2}\|_{L^{q_{-}}})\|\p_t\tilde{u}\|_{L^{2/\varepsilon_-}}\big]\\
					&\lesssim\int_0^T\|\eta\|_{H^{\frac{3}{2}+\frac{\varepsilon_{-}-\alpha}{2}}}\big[(1+\|\eta^2\|_{W^{3-1/q_-,q_-}})(\|\p_{t}\tilde{u}^{2}\|_{W^{2,q_-}}+\|\p_{t}\tilde{p}^{2}\|_{W^{1,q_{-}}})\big]\|\p_t^2\tilde{u}\|_{1}\\
					&\lesssim \|\eta\|_{L_{t}^{\infty}H^{\frac{3}{2}+\frac{\varepsilon_{-}-\alpha}{2}}}\big[(1+\|\eta^2\|_{L_{t}^{\infty}W^{3-1/q_-,q_-}})(\|\p_{t}\tilde{u}^{2}\|_{L_{t}^{2}W^{2,q_-}}+\|\p_{t}\tilde{p}^{2}\|_{L_{t}^{2}W^{1,q_{-}}})\big]\|\p_t\tilde{u}\|_{L_{t}^{2}H^{1}}\\
					&\lesssim\delta^{1/2}\|\eta\|_{L_{t}^{\infty}H^{\frac{3}{2}+\frac{\varepsilon_{-}-\alpha}{2}}}\|\p_t\tilde{u}\|_{L_{t}^2H^1}.
				\end{aligned}
			\end{equation}
			Hence, combining the estimate \eqref{est:remain_bulk11} and estimates \eqref{est:remain_bulk21} to  \eqref{est:remain_bulk231}, we obtain the following result
			\begin{equation}\label{est:remain_bulk2}
				\begin{aligned}
					\int_\Om R^{1,1}\cdot\p_t^{2}\tilde{u}J^1
					&\lesssim \delta^{1/2}d(u,p,\eta)d(\tilde{u},\tilde{p},\tilde{\eta}).
				\end{aligned}
			\end{equation}
			
			\subparagraph{\underline{Terms related to $R^{4,1}$ and $R^{5,1}$}}
			We now estimate integrals involving $R^{4,1}$. We first estimate the terms included in $\p_tR^4$ (In the following estimate,  we use $\tilde{u}$ to denote the trace of $\tilde{u}$). By definition, we have
			\begin{equation}\label{est:remain_b21}
				\begin{aligned}
					&|\int_0^T\int_{-\ell}^\ell \p_tR^4\cdot\p_t\tilde{u}|\lesssim\int_0^T\int_{-\ell}^\ell\left[(|\p_t{p}^2|+|\nabla\p_t{u}^2|+|\nabla \p_{t}{\eta}^{2}||\nabla {u}^{2}|)|\p_1\eta|\right]|\p_t\tilde{u}|
					+\int_{0}^{T}\int_{-\ell}^{\ell}(|{p}^2|+|\nabla {u}^2|)|\p_1\p_t\eta||\p_t\tilde{u}|.
				\end{aligned}
			\end{equation}
			The first and second lines on the right-hand side of \eqref{est:remain_b21}, using the H\"older and Sobolev inequalities, are estimated by
			\begin{equation}\label{est:remain_b22}
				\begin{aligned}
					&\int_0^T\bigg[(\|\p_t{p}^2\|_{L^{1/(1-\varepsilon_-)}(\Sigma)}+\|\nabla\p_t{\eta}^2\|_{L^\infty}\|\nabla {u}^2\|_{L^{1/(1-\varepsilon_-)}(\Sigma)}+\|\nabla\p_t{u}^2\|_{L^{1/(1-\varepsilon_-)}(\Sigma)})\bigg]\|\p_{t}\tilde{u}\|_{L^{1/(\varepsilon_--\alpha)}(\Sigma)}
					\times\|\p_1\eta\|_{L^{1/\alpha}}\\&\quad\quad+(\|{p}^2\|_{L^{1/(1-\varepsilon_-)}(\Sigma)}+\|\nabla {u}^2\|_{L^{1/(1-\varepsilon_-)}(\Sigma)})\|\p_1\p_t\eta\|_{L^{1/\alpha}}\|\p_t\tilde{u}\|_{L^{1/(\varepsilon_--\alpha)}(\Sigma)}\\
					&\lesssim\int_0^T\bigg[(\|\p_t{p}^2\|_{W^{1,q_-}}+\|\p_t\eta^2\|_{H^{3/2+(\varepsilon_--\alpha)/2}}\| {u}^2\|_{W^{2,q_-}}+\|\p_t{u}^2\|_{W^{2,q_-}})\bigg]\|\p_t\tilde{u}\|_{1}\|\eta\|_{H^{3/2-\alpha}}\\
					&\quad\quad+(\|{p}^2\|_{W^{1,q_-}}+\| {u}^2\|_{W^{2,q_-}})\|\p_t\eta\|_{H^{3/2-\alpha}}\|\p_t\tilde{u}\|_{1}\\
					&\lesssim\delta^{1/2}(\|\eta\|_{L^{2}W^{3-\frac{1}{q_{-}},q_{-}}}+\|\p_t\eta\|_{L^2H^{3/2-\alpha}})\|\p_t\tilde{u}\|_{L^2H^1}.
				\end{aligned}
			\end{equation}
			
			\noindent Therefore, we obtain
			\begin{equation}\label{est:remain_b2}
				\begin{aligned}
					|\int_0^T\int_{-\ell}^\ell \p_tR^4\cdot\p_t\tilde{u}|\lesssim \delta^{1/2}d(u,p,\eta)d(\tilde{u},\tilde{p},\tilde{\eta}).
				\end{aligned}
			\end{equation}
			Similarly, we estimate other terms involved in $R^{4,1}$, we have
			\begin{align}
				\begin{aligned}
					\int_{0}^{T}\int_{-\ell}^{\ell}S_{\mathcal{A}^{1}}(u,p)\p_{t}\mathcal{N}^{1}\cdot \p_{t}\tilde{u}\lesssim& \int_{0}^{T}(\|u\|_{W^{1,\frac{1}{1-\varepsilon_{-}}}(\Sigma)}+\|p\|_{L^{\frac{1}{1-\varepsilon_{-}}}(\Sigma)})\|\p_{t}\eta^{1}\|_{W^{1,+\infty}(\Sigma)}\|\p_{t}\tilde{u}\|_{L^{\frac{1}{\varepsilon_{-}}}(\Sigma)}\\
					\lesssim& \delta^{\frac{1}{2}}(\|u\|_{L^{2}W^{2,q_{-}}}+\|p\|_{L^{2}W^{1,q_{-}}})\|\p_{t}\tilde{u}\|_{L^{2}H^{1}},\\
					\int_{0}^{T}\int_{-\ell}^{\ell}S_{\p_{t}\mathcal{A}^{1}}(u,p)\mathcal{N}^{1}\cdot \p_{t}\tilde{u}\lesssim& \int_{0}^{T}\|\p_{t}\p_{1}\eta^{1}\|_{L^{\infty}}(\|u\|_{W^{1,\frac{1}{1-\varepsilon_{-}}}(\Sigma)}+\|p\|_{L^{\frac{1}{1-\varepsilon_{-}}}(\Sigma)})\|\eta^{1}\|_{W^{1,+\infty}(\Sigma)}\|\p_{t}\tilde{u}\|_{L^{\frac{1}{\varepsilon_{-}}}(\Sigma)}\\
					\lesssim& \delta^{\frac{1}{2}}(\|u\|_{L^{\infty}W^{2,q_{-}}}+\|p\|_{L^{\infty}W^{1,q_{-}}})\|\p_{t}\tilde{u}\|_{L^{2}H^{1}},
				\end{aligned}
			\end{align}
			\noindent and similarly we have
			\begin{align}
				\begin{aligned}
					\int_{0}^{T} \int_{-\ell}^{\ell} \mathbb{D}_{\p_{t}\mathcal{A}^{1}-\p_{t}\mathcal{A}^{2}}{u}^{2}\mathcal{N}^{2}\cdot \p_{t}\tilde{u}\lesssim&\int_{0}^{T}\|\p_{t}\eta\|_{W^{1,\frac{1}{\alpha}}}\|{u}^{2}\|_{W^{1,\frac{1}{1-\varepsilon_{-}}}(\Sigma)}\|\p_{t}\tilde{u}\|_{L^{\frac{1}{\varepsilon_{-}}}}
					\lesssim\delta^{\frac{1}{2}}\|\p_{t}\eta\|_{L^{2}H^{\frac{3}{2}-\alpha}}\|\p_{t}\tilde{u}\|_{L^{2}H^{1}},\\
					\int_{0}^{T}\int_{-\ell}^{\ell}  \mathbb{D}_{\mathcal{A}^{1}-\mathcal{A}^{2}}\p_{t}\tilde{u}^{2}\mathcal{N}^{2}\cdot \p_{t}\tilde{u}\lesssim&\int_{0}^{T}\|\eta\|_{W^{1,\frac{1}{\alpha}}}\|\p_{t}\tilde{u}^{2}\|_{W^{1,\frac{1}{1-\varepsilon_{-}}}(\Sigma)}\|\p_{t}\tilde{u}\|_{L^{\frac{1}{\varepsilon_{-}-\alpha}}}\\
					\lesssim&\|\p_{t}\tilde{u}^{2}\|_{L^{\infty}W^{2,q_{-}}}\|\eta\|_{L^{\infty}H^{\frac{3}{2}-\alpha}}\|\p_{t}\tilde{u}\|_{L^{2}H^{1}}
					\lesssim\delta^{\frac{1}{2}}\|\eta\|_{L^{\infty}H^{\frac{3}{2}-\alpha}}\|\p_{t}\tilde{u}\|_{L^{2}H^{1}},\\
					\int_{0}^{T}\int_{-\ell}^{\ell}  \mathbb{D}_{\mathcal{A}^{1}-\mathcal{A}^{2}}{u}^{2}\p_{t}\mathcal{N}^{2}\cdot \p_{t}\tilde{u}\lesssim&\int_{0}^{T}\|\p_{t}\eta^{2}\|_{W^{1,+\infty}}\|\eta\|_{W^{1,\infty}}\|{u}^{2}\|_{W^{1,\frac{1}{1-\varepsilon_{-}}}(\Sigma)}\|\p_{t}\tilde{u}\|_{L^{\frac{1}{\varepsilon_{-}}}}\\
					\lesssim&\|\p_{t}\eta^{2}\|_{L^{2}W^{3-\frac{1}{q_{-}},q_-{}}}\|{u}^{2}\|_{L^{\infty}W^{2,q_{-}}}\|\eta\|_{L^{\infty}H^{\frac{3}{2}+\frac{\varepsilon_{-}-\alpha}{2}}}\|\p_{t}\tilde{u}\|_{L^{2}H^{1}}\\
					\lesssim&\delta^{\frac{1}{2}}\|\eta\|_{L^{\infty}H^{\frac{3}{2}+\frac{\varepsilon_{-}-\alpha}{2}}}\|\p_{t}\tilde{u}\|_{L^{2}H^{1}},
				\end{aligned}
			\end{align}
			\noindent and
			\begin{equation}
				\begin{aligned}
					&|\int_0^T\int_{-\ell}^\ell \bigg[ g({\eta})\p_t\mathcal{N}^1-\sigma\p_1\left(\frac{\p_1{\eta}}{(1+|\p_1\zeta_0|)^{3/2}}\right)\p_t\mathcal{N}^1-\p_{1}(\mathcal{R}_{z}(\p_{1}\zeta_{0},\p_{1}\eta))\p_{t}\mathcal{N}^{1}\bigg]\cdot\p_t\tilde{u}|\\
					&\lesssim \int_0^T\int_{-\ell}^\ell(|{\eta}|+|\p_1{\eta}|+|\p_1^2{\eta}|)|\p_t\p_1\eta^1||\p_t\tilde{u}|\lesssim \int_0^T(\|{\eta}\|_{W^{2,1/(1-\varepsilon_-)}})\|\p_t\p_1\eta^1\|_{L^{1/\alpha}}\|\p_t\tilde{u}\|_{L^{1/(\varepsilon_--\alpha)}}\\
					&\lesssim \|\eta\|_{L_{t}^{2}W^{3-\frac{1}{q_{-}},q_{-}}}\|\p_{t}\eta^{1}\|_{L_{t}^{\infty}H^{\frac{3}{2}-\alpha}}\|\p_{t}\tilde{u}\|_{L_{t}^{2}H^{1}}.
				\end{aligned}
			\end{equation}
			Moreover, using the fact that $\|\p_1\eta^i\|_{L^\infty}\lesssim1$ and the definition of $\mathcal{R}_{z}$, we obtain the following estimates
			\begin{equation}\label{est:remain_b32}
				\begin{aligned}
					&|\int_0^T\int_{-\ell}^\ell \p_{1}\bigg[-\sigma\p_1(\mathcal{R}_{z}(\p_{1}\zeta_{0},\p_{1}\eta^{1})-\mathcal{R}_{z}(\p_{1}\zeta_{0},\p_{1}\eta^{2}))\p_t\mathcal{N}^2\bigg]\cdot\p_t\tilde{u}|
					\\
					&\lesssim \int_0^T\int_{-\ell}^\ell(|\p_1^2\eta|)|\p_t\p_1\eta^2||\p_t\tilde{u}|\lesssim \int_0^T(\|\p_1^{2}\eta\|_{W^{1,1/(1-\varepsilon_-)}}\|\p_{t}\eta^{2}\|_{W^{1,\frac{1}{1-\varepsilon_{-}}}})
					\|\p_t\p_1\eta^2\|_{L^{1/\alpha}}\|\p_t\tilde{u}\|_{L^{1/(\varepsilon_--\alpha)}}\\
					&\lesssim\int_0^T (\|\eta\|_{W^{3-1/q_-,q_-}})\|\p_t\eta^1\|_{3/2+}\|\p_t\tilde{u}\|_1\lesssim \delta^{1/2}(\|\eta\|_{L_{t}^2W^{3-1/q_-,q_-}})\|\p_t\tilde{u}\|_{L_{t}^2H^1},
				\end{aligned}
			\end{equation}
			
			\noindent and
			
			\begin{equation}\label{est:remain_b33}
				\begin{aligned}
					&|\int_0^T\int_{-\ell}^\ell \bigg[\mathcal{K}(\eta^{1})-\p_{1}(\mathcal{R}(\p_{1}\zeta_{0},\p_{1}\eta^{1}))\bigg]\p_{t}(\mathcal{N}^{1}-\mathcal{N}^{2})\cdot\p_t\tilde{u}|
					\\
					&\lesssim \int_0^T\int_{-\ell}^\ell(|\p_{1}^{2}\eta^{1}|)|\p_t\p_1\eta||\p_t\tilde{u}|\lesssim \int_0^T(\|\eta^{1}\|_{W^{2,\frac{1}{1-\varepsilon_{-}}}})
					\|\p_t\p_1\eta^2\|_{L^{1/\alpha}}\|\p_t\tilde{u}\|_{L^{1/(\varepsilon_--\alpha)}}\\
					&\lesssim\int_0^T (\|\eta^{1}\|_{W^{3-1/q_-,q_-}})\|\p_t\eta\|_{H^{3/2-\alpha}}\|\p_t\tilde{u}\|_1\lesssim \delta^{1/2}(\|\p_{t}\eta\|_{L^2H^{\frac{3}{2}-\alpha}})\|\p_t\tilde{u}\|_{L^2H^1},
				\end{aligned}
			\end{equation}
			and
			\begin{align}{\label{est:remain_b35}}
				\begin{aligned}
					&|\int_0^T\int_{-\ell}^\ell \bigg[\mathcal{K}(\p_{t}\tilde{\eta}^{1})-\p_{1}(\mathcal{R}_{z}(\p_{1}\zeta_{0},\p_{1}\eta^{1})\p_{1}\p_{t}\tilde{\eta}^{1})\bigg](\mathcal{N}^{1}-\mathcal{N}^{2})\cdot\p_t\tilde{u}|
					\\
					&\lesssim \int_0^T\int_{-\ell}^\ell(|\p_{1}^{2}\p_{t}\tilde{\eta}^{1}|+|\p_{1}\p_{t}\tilde{\eta}^{1}|)|\p_1\eta||\p_t\tilde{u}|\lesssim \int_0^T(\|\tilde{\eta}^{1}\|_{W^{2,\frac{1}{1-\varepsilon_{-}}}}+\|\tilde{\eta}^{1}\|_{W^{1,+\infty}})
					\|\p_1\eta\|_{L^{1/\alpha}}\|\p_t\tilde{u}\|_{L^{1/(\varepsilon_--\alpha)}}\\
					&\lesssim\int_0^T (\|\tilde{\eta}^{1}\|_{W^{3-1/q_-,q_-}})\|\eta\|_{H^{3/2-\alpha}}\|\p_t\tilde{u}\|_1\lesssim \delta^{1/2}(\|\eta\|_{L_{t}^{\infty}H^{\frac{3}{2}-\alpha}})\|\p_t\tilde{u}\|_{L_{t}^2H^1},
				\end{aligned}
			\end{align}
			Therefore, by estimate \eqref{est:remain_b2}, and estimates \eqref{est:remain_b32} to \eqref{est:remain_b35}, we derive the following estimate for $R^{4,1}$
			\begin{align}{\label{eq:remain_b3}}
				\begin{aligned}
					\|R^{4,1}\|_{L_{t}^{2}(\mathcal{H}^1)^{\ast}}\lesssim \delta^{\frac{1}{2}}d(u,\eta,p),
				\end{aligned}
			\end{align}
			and
			\begin{equation}\label{est:remain_b3}
				\begin{aligned}
					\int_0^T\int_{-\ell}^\ell R^{4,1}\cdot\p_t\tilde{u}\lesssim\delta^{1/2}d(u,p,\eta)d(\tilde{u},\tilde{p},\tilde{\eta}).
				\end{aligned}
			\end{equation}
			
			We use similar methods to deal with the terms involving $R^{5,1}$ since the terms in $R^{5,1}$ are all included in $R^{4,1}$. We have
			\begin{align}
				\begin{aligned}
					|\int_0^T\int_{-\ell}^\ell R^{5,1}\cdot\p_t\tilde{u}\lesssim\delta^{1/2}d(u,p,\eta)d(\tilde{u},\tilde{p},\tilde{\eta})
				\end{aligned}
			\end{align}
			
			\subparagraph{\underline{Terms related to $R^{6,1}$}}
			We now estimate the integrals involving $R^{6,1}$. First, using spatial integration by parts, we have
			\begin{equation}\label{est:remain_b4}
				\begin{aligned}
					&|\int_0^T\int_{-\ell}^\ell g(\p_{t}\tilde{\eta})R^{6,1}-\sigma\left(\frac{\p_1\p_{t}(\tilde{\eta})}{(1+|\p_1\zeta_0|^2)^{3/2}}\right)\p_{1}R^{6,1}|\\
					&\lesssim\int_0^T\int_{-\ell}^\ell(|\p_{t}\tilde{\eta}|+|\p_1\p_{t}(\tilde{\eta})|)(|\p_{t}\p_{1}\tilde{u}_1^{2}||\p_1\eta|+|\p_{t}\tilde{u}_{1}^{2}||\p_{1}^{2}\eta|+|\p_{1}u_{1}||\p_{1}\p_{t}\tilde{\eta}^{2}|+|u_{1}||\p_{1}^{2}\p_{t}\tilde{\eta}|^{2}+|\p_{1}u_{1}^{2}||\p_{1}\p_{t}\tilde{\eta}|) \\
					&\lesssim(\|\p_{t}\tilde{\eta}\|_{L^{2}W^{1,\frac{1}{\alpha}}})(\|{u}\|_{L^{\infty}L^{1/(\varepsilon_--\alpha)}(\Sigma)}\|\p_t\tilde{\eta}^2\|_{L^{2}W^{3-\frac{1}{q_{-}},q_{-}}}+\|{u}\|_{L^{2}W^{1,\frac{1}{1-\varepsilon_{-}}}(\Sigma)}\|\p_t\tilde{\eta}^2\|_{L^{\infty}H^{\frac{3}{2}-\alpha}}\\
					&\quad\quad+\|\p_t\tilde{u}^2\|_{L^{\infty}L^\infty(\Sigma)}\|\eta\|_{L_{t}^{2}W^{2,1/(1-\varepsilon_-)}}+\|\p_{t}\tilde{u}^{2}\|_{L_{t}^{2}W^{1,\frac{1}{1-\varepsilon_{-}}}}\|\eta\|_{L_{t}^{\infty}W^{1,+\infty}}+\|\p_{1}{u}^{2}\|_{L_{t}^{2}W^{1,\frac{1}{1-\varepsilon_{-}}}(\Sigma)}\|\p_{t}\tilde{\eta}\|_{L^{\infty}W^{1,\frac{1}{\alpha}}})\\
					&\lesssim\delta^{1/2}(\|\p_{t}\tilde{\eta}\|_{L^2H^{\frac{3}{2}-\alpha}})(\|{u}\|_{L^\infty H^{1}}+\|u\|_{L^{2}W^{2,q_{-}}}+\|\eta\|_{L^2W^{3-1/q_-,q_-}}
					+\|\eta\|_{L_{t}^{\infty}H^{\frac{3}{2}+\frac{\varepsilon_{-}-\alpha}{2}}}+\|\p_t\tilde{\eta}\|_{L^2H^{3/2-\alpha}})\\
					&\lesssim \delta^{\frac{1}{2}}d(\tilde{u},\tilde{\eta},\tilde{p})(d(\tilde{u},\tilde{\eta},\tilde{p})+d({u},{\eta},{p})).
				\end{aligned}
			\end{equation}
			\noindent
			Moreover, the computation of \eqref{est:remain_b4} implies that:
			\[
			\|R^{6,1}\|_{L^{2}L^{\frac{1}{1-\alpha}}}\lesssim \delta^{\frac{1}{2}} (d(\tilde{u},\tilde{\eta},\tilde{p})+d({u},{\eta},{p})). 
			\]
			\noindent  For the following term including $R^{6,1}$, since
			\[\begin{aligned}
				&\|\p_{1}((\mathcal{R}_{z}(\p_{1}\zeta_{0},\p_{1}\eta^{2})-\mathcal{R}_{z}(\p_{1}\zeta_{0},\p_{1}\eta^{1}))\p_{1}\p_{t}\tilde{\eta}^{2})\|_{L_{t}^{2}L^{\frac{1}{\alpha}}}\\&\lesssim (\|\p_{t}\tilde{\eta}^{2}\|_{L_{t}^{2}W^{3-\frac{1}{q_{-}},q_{-}}}\|\eta\|_{L_{t}^{\infty}H^{\frac{3}{2}-\alpha}}+\|\p_{t}\tilde{\eta}^{2}\|_{L_{t}^{\infty}H^{\frac{3}{2}+\frac{\varepsilon_{-}-\alpha}{2}}}\|\eta\|_{L_{t}^{2}W^{3-\frac{1}{q_{-}},q_{-}}})\lesssim\delta^{\frac{1}{2}}d(\tilde{u},\tilde{\eta},\tilde{p}),\no
			\end{aligned}\]
			it holds that
			\begin{align}
				&\int_{0}^{T}\int_{-\ell}^{\ell} \p_{1}((\mathcal{R}_{z}(\p_{1}\zeta_{0},\p_{1}\eta^{2})-\mathcal{R}_{z}(\p_{1}\zeta_{0},\p_{1}\eta^{1}))\p_{1}\p_{t}^{3}\tilde{\eta}^{2})ds R^{6,1} dxdt\\
				&\lesssim \|\p_{1}((\mathcal{R}_{z}(\p_{1}\zeta_{0},\p_{1}\eta^{2})-\mathcal{R}_{z}(\p_{1}\zeta_{0},\p_{1}\eta^{1}))\p_{1}\p_{t}\tilde{\eta}^{2})\mathcal{N}^{1}(t)\|_{L_{t}^{2}L^{\frac{1}{\alpha}}}\|R^{6,1}\|_{L_{t}^{2}L^{\frac{1}{1-\alpha}}}\lesssim \delta^{\frac{1}{2}}(d(\tilde{u},\tilde{\eta},\tilde{p})+d(u,\eta,p))^{2}.\no
			\end{align}
			\noindent Then we obtain the following bound:
			\begin{equation}
				\begin{aligned}
					\text{Terms with}~~R^{6,1}\lesssim \delta^{\frac{1}{2}}d(\tilde{u},\tilde{\eta},\tilde{p})(d(\tilde{u},\tilde{\eta},\tilde{p})+d({u},{\eta},{p})).
				\end{aligned}
			\end{equation}
			
			\subparagraph{\underline{Other Boundary Terms}}
			We now estimate all of the remaining boundary terms. From the computation of terms including $R^{6,1}$, we have the following result
			\[\begin{aligned}
				&\|(g\p_{t}\tilde{\eta})-\sigma(\frac{\p_{1}\p_{t}\tilde{\eta}}{(1+|\p_{1}\zeta_{0}|^{2})^{\frac{3}{2}}})\|_{L^{\frac{1}{\alpha}}}+\|\mathcal{R}_{z}(\p_{1}\zeta_{0},\p_{1}\eta^{1})\p_{1}\p_{t}\tilde{\eta}\|_{L^{\frac{1}{\alpha}}}\\
				&+\|(\mathcal{R}_{z}(\p_{1}\zeta_{0},\p_{1}\eta^{2})-\mathcal{R}_{z}(\p_{1}\zeta_{0},\p_{1}\eta^{1}))\p_{1}\p_{t}\tilde{\eta}^{2}\|_{L^{\frac{1}{\alpha}}}\lesssim \delta^{\frac{1}{2}}(d(u,p,\eta)+d(\tilde u,\tilde p,\tilde \eta)).
			\end{aligned}\]
			
			\noindent Therefore, to bound the following integral
			\begin{align}{\label{eq:complex_1}}
				\begin{aligned}
					&\int_{0}^{T}\int_{-\ell}^{\ell}\p_{1}(u_{1}\p_{1}\p_{t}\tilde{\eta}^{2})\left (g\p_{t}\tilde{\eta}+\sigma\frac{\p_{1}\p_{t}\tilde{\eta}}{(1+\vert \p_{1}\zeta_{0}\vert^2)^{\frac{3}{2}}}-\sigma(\mathcal{R}_{z}(\p_{1}\zeta_{0},\p_{1}\eta^{1})\p_{1}\p_{t}\tilde{\eta})\right)\mathcal{N}^{1}(t)\\
					&\quad+\int_{0}^{T}\int_{-\ell}^{\ell}\p_{1}(u_{1}\p_{1}\p_{t}\tilde{\eta}^{2})\sigma\left(\mathcal{R}_{z}(\p_{1}\zeta_{0},\p_{1}\eta^{2})-\mathcal{R}_{z}(\p_{1}\zeta_{0},\p_{1}\eta^{1}\right)\p_{1}\p_{t}\tilde{\eta}^{2}\mathcal{N}^{1}(t),
				\end{aligned}
			\end{align}
			it suffices to show the following estimate:
			
			\[
			\begin{aligned}
				\|\p_{1}(u_{1}\p_{1}\p_{t}\tilde{\eta}^{2})\|_{L_{t}^{2}L^{\frac{1}{1-\alpha}}}\lesssim&\|\p_{1}u_{1}\p_{1}\p_{t}\tilde{\eta}^{2}\|_{L_{t}^{2}L^{\frac{1}{1-\alpha}}}+\|u_{1}\p_{1}^{2}\p_{t}\tilde{\eta}^{2}\|_{L_{t}^{2}L^{\frac{1}{1-\alpha}}}\\
				\lesssim&\|u\|_{L^{\infty}W^{2,q_{-}}}\|\p_{t}\tilde{\eta}^{2}\|_{L^{2}H^{\frac{3}{2}-\alpha}}+\|u_{1}
				\p_{1}^{2}\p_{t}\tilde{\eta}^{2}\|_{L_{t}^{2}L^{\frac{1}{1-\alpha}}}\\
				\lesssim& \|u\|_{L^{\infty}W^{2,q_{-}}}\|\p_{t}\tilde{\eta}^{2}\|_{L^{2}H^{\frac{3}{2}-\alpha}}+\|u\|_{L_{t}^{\infty}H^{1+\frac{\varepsilon_{-}}{2}}}\|
				\p_{t}\tilde{\eta}^{2}\|_{W^{3-\frac{1}{q_{-}},q_{-}}},
			\end{aligned}
			\]
			
			\noindent which implies that
			
			\begin{align}\label{est:bdd_1}
				\begin{aligned}
					\|\p_{1}u_{1}\p_{1}\p_{t}\tilde{\eta}^{2}\|_{L^{2}L^{\frac{1}{1-\alpha}}}\lesssim \delta^{\frac{1}{2}}(\|u\|_{L^{\infty}W^{2,q_{-}}}+\|\p_{t}u\|_{L^{2}W^{2,q_{-}}})\lesssim\delta^{\frac{1}{2}}d(u,\eta,p).
				\end{aligned}
			\end{align}
			\noindent Therefore , we have
			\[
			\eqref{eq:complex_1}\lesssim \delta^{\frac{1}{2}}d^{2}(u,\eta,p).
			\]
			\noindent Then we estimate the following integral
			\begin{align}{\label{eq:bdd_1}}
				\int_{-\ell}^{\ell}(u_{1}^{1}\p_{1}^{2}\p_{t}\tilde{\eta})\left(g\p_{t}\tilde{\eta}+\sigma\frac{\p_{1}\p_{t}\tilde{\eta}}{(1+\vert \p_{1}\zeta_{0}\vert^2)^{\frac{3}{2}}}-\sigma\mathcal{R}_{z}(\p_{1}\zeta_{0},\p_{1}\eta^{1})\p_{1}\p_{t}\tilde{\eta}\right).
			\end{align}
			
			\noindent For simplicity, we only show the details of estimating the most difficult part
			\[
			\left| \int_{-\ell}^{\ell}(u_{1}^{1}\p_{1}^{2}\p_{t}\tilde{\eta})\sigma\mathcal{R}_{z}(\p_{1}\zeta_{0},\p_{1}\eta^{1})\p_{1}\p_{t}\tilde{\eta} \right|.
			\]
			in \eqref{eq:bdd_1}.
			Using spatial integration by parts (noting that the boundary terms vanish since $u_{1}^{1}=0$), we have
			\[
			\begin{aligned}
				&| \int_{-\ell}^{\ell}(u_{1}^{1}\p_{1}^{2}\p_{t}\tilde{\eta})((\sigma(\mathcal{R}_{z}(\p_{1}\zeta_{0},\p_{1}\eta^{1})\p_{1}\p_{t}\tilde{\eta}))|\\
				\lesssim&\frac{1}{2}|\int_{-\ell}^{\ell}\p_{1}u_{1}^{1}\sigma\mathcal{R}_{z}(\p_{1}\zeta_{0},\p_{1}\eta^{1})(\p_{1}\p_{t}\tilde{\eta})^{2}|+\frac{1}{2}|\int_{-\ell}^{\ell}u_{1}^{1}\sigma\p_{1}\mathcal{R}_{z}(\p_{1}\zeta_{0},\p_{1}\eta^{1})(\p_{1}\p_{t}\tilde{\eta})^{2}|\\
				\lesssim&\|u^{1}\|_{W^{1,\frac{1}{1-\varepsilon_{-}]}}(\Sigma)}\|\eta^{1}\|_{W^{1,+\infty}}\|\p_{t}\tilde{\eta}\|_{H^{\frac{3}{2}-\alpha}}^{2}+\|u^{1}\|_{L^{\infty}(\Sigma)}\|\eta^{1}\|_{W^{2,\frac{1}{1-\varepsilon_{-}}}}\|\p_{t}\tilde{\eta}\|_{H^{\frac{3}{2}-\alpha}}^{2}\\
				\lesssim& \|u^{1}\|_{W^{2,q_{-}}}\|\eta^{1}\|_{H^{\frac{3}{2}+\frac{\varepsilon_{-}-\alpha}{2}}}\|\p_{t}\tilde{\eta}\|_{H^{\frac{3}{2}-\alpha}}^{2}+\|u^{1}\|_{W^{2,q_{-}}}\|\eta^{1}\|_{W^{3-\frac{1}{q_{-}},q_{-}}}\|\p_{t}\tilde{\eta}\|_{H^{\frac{3}{2}-\alpha}}^{2}.
			\end{aligned}
			\]
			\noindent Therefore, we have
			\[
			\int_{0}^{T} | \int_{-\ell}^{\ell}(u_{1}^{1}\p_{1}^{2}\p_{t}\tilde{\eta})((\sigma(\mathcal{R}_{z}(\p_{1}\zeta_{0},\p_{1}\eta^{1})\p_{1}\p_{t}\tilde{\eta}))|\lesssim \delta d^{2}(\tilde{u},\tilde{\eta},\tilde{p}).
			\]
			
			Next, we estimate the following term:
			\begin{align}{\label{eq:bdd_2}}
				\int_{-\ell}^{\ell}u^{1}\p_{1}^{2}\p_{t}\tilde{\eta}(\sigma ((\mathcal{R}_{z}(\p_{1}\zeta_{0},\p_{1}\eta^{2})-\mathcal{R}_{z}(\p_{1}\zeta_{0},\p_{1}\eta^{1}))\p_{1}\p_{t}\tilde{\eta}^{2})).
			\end{align}
			\noindent By spatial integration by parts, we obtain
			\begin{align}{\label{est:bdd_3}}
				&\int_{0}^{T}\int_{-\ell}^{\ell}u^{1}\p_{1}^{2}\p_{t}\tilde{\eta}\sigma (\mathcal{R}_{z}(\p_{1}\zeta_{0},\p_{1}\eta^{2})-\mathcal{R}_{z}(\p_{1}\zeta_{0},\p_{1}\eta^{1}))\p_{1}\p_{t}\tilde{\eta}^{2} \notag\\
				\lesssim& \|u^{1}\|_{L^{\infty}W^{2,q_{-}}}\|\p_{t}\tilde{\eta}\|_{L^{2}H^{\frac{3}{2}-\alpha}}\|\p_{t}\tilde{\eta}^{2}\|_{L_{t}^{2}W^{}}\|\eta\|_{L_{t}^{2}H^{\frac{3}{2}-\alpha}}\notag\\
				&+\|u^{1}\|_{L^{\infty}W^{2,q_{-}}}\|{\eta}\|_{L_{t}^{2}W^{3-\frac{1}{q_{-}},q_{-}}}\|\p_{t}\tilde{\eta}^{2}\|_{L_{t}^{\infty}H^{\frac{3}{2}+\frac{\varepsilon_{-}-\alpha}{2}}}\|\p_{t}\tilde{\eta}\|_{L_{t}^{2}H^{\frac{3}{2}-\alpha}}\notag\\
				\lesssim& \delta d(u,\eta,p)d(\tilde{u},\tilde{\eta},\tilde{p}).
			\end{align}
			
			\subparagraph{\underline{Terms related to pressure}}
			We now estimate the terms involving pressure. Using temporal integration by parts, it holds that
			\[
			\begin{aligned}
				\int_{0}^{t}\int_{\Omega}\p_{t}\tilde{p}\langle R^{2,1}J^{1}\rangle=\tilde{p}\langle R^{2,1}J^{1}\rangle(t)-\int_{0}^{t}\int_{\Omega}\tilde{p}\p_{t}\langle R^{2,1}J^{1}\rangle.
			\end{aligned}
			\]
			\noindent This can be controlled by
			\[
			\begin{aligned}
				|\int_{0}^{T}\int_{\Omega}\p_{t}\tilde{p}\langle R^{2,1}J^{1}\rangle|\lesssim& \int_{0}^{T}(\int_{\Omega}\tilde{p})\int_{\Omega} (|\p_{t}^{2}\nabla\bar{\eta}||\nabla u^{1}|+|\p_{t}\nabla \bar{\eta}||\p_{t}\nabla u^{1}|+|\p_{t}\nabla \bar{\eta}||\p_{t}\nabla\tilde{u}^{2}|\\
				&+|\nabla\bar{\eta}||\p_{t}^{2}\nabla\tilde{u}^{2}|+|\p_{t}^{2}\nabla \bar{\eta}^{1}||\nabla u|+|\p_{t}\nabla \bar{\eta}^{1}||\p_{t}\nabla u|)\\
				&+\sup_{0\leq t\leq T}\int_{\Omega}\tilde{p}\int_{\Omega}|(\p_{t}\nabla \bar{\eta}||\nabla u^{2}|+|\nabla \bar{\eta}||\p_{t}\nabla\tilde{u}^{2}|+|\p_{t}\nabla\bar{\eta}^{1}||\nabla u|)\\
				\lesssim& (\|\tilde{p}\|_{L_{t}^{2}L^{2}}+\|\tilde{p}\|_{L_{t}^{\infty}L^{2}})\bigg( \|\p_{t}^{2}\eta\|_{L_{t}^{2}H^{\frac{1}{2}-\alpha}}\|u^{1}\|_{L_{t}^{\infty}W^{2,q_{+}}}+\|\p_{t}\eta\|_{L_{t}^{2}H^{\frac{3}{2}-\alpha}}\|\p_{t}u^{1}\|_{L_{t}^{\infty}H^{1}}\\
				&+\|\eta\|_{L_{t}^{\infty}H^{\frac{3}{2}+\frac{\varepsilon_{-}-\alpha}{2}}}\|\p_{t}^{2}\tilde{u}^{2}\|_{L_{t}^{2}H^{1}}+\|\p_{t}^{2}\eta^{1}\|_{L_{t}^{\infty}H^{1}}\|u\|_{L_{t}^{\infty}H^{1}}+\|\p_{t}\eta^{1}\|_{L_{t}^{\infty}H^{1}}\|\p_{t}u\|_{L_{t}^{2}H^{1}}\bigg)\\
				&+\|\tilde{p}\|_{L_{t}^{\infty}L^{2}}(\|\p_{t}\eta\|_{L_{t}^{\infty}H^{1}}\|u^{2}\|_{L_{t}^{\infty}H^{1}}+\|\eta\|_{L_{t}^{\infty}H^{1}}\|\p_{t}\tilde{u}^{2}\|_{L_{t}^{\infty}H^{1}}+\|\p_{t}\eta^{1}\|_{L_{t}^{\infty}H^{1}}\|u\|_{L_{t}^{\infty}H^{1}})\\
				\lesssim& \delta^{\frac{1}{2}}(d(\tilde{u},\tilde{\eta},\tilde{p})+d({u},{\eta},{p}))^{2}.
			\end{aligned}
			\]
			
			\subparagraph{\underline{Terms related to $w$}}
			We next establish the estimates for terms involving $w$. From the result in \cite[Section 10]{GT2020}, we have
			\[
			\begin{aligned}
				\|w\|_{L_{t}^{\infty}W_{0}^{1,\frac{4}{3-2\varepsilon_{-}}}}&\lesssim \|R^{2,1}J^{1}\|_{L_{t}^{\infty}L^{\frac{4}{3-2\varepsilon_{+}}}}\\
				&\lesssim \|\nabla \p_{t}\bar{\eta}^{1}\|_{L^{\infty}}\|\nabla\tilde{u}\|_{L^{\infty}L^{\frac{4}{2-\varepsilon_{-}}}}+\|\nabla \p_{t}\bar{\eta}\|_{L_{t}^{\infty}L^{4}(\Omega)}\|\nabla\tilde{u}^{2}\|_{L_{t}^{\infty}L^{\frac{2}{1-\varepsilon_{+}}}}+\|\nabla \bar{\eta}\|_{L_{t}^{\infty}L^{\frac{2}{\alpha}}}\|\nabla \p_{t}\tilde{u}^{2}\|_{L_{t}^{\infty}L^{\frac{4}{2-\varepsilon_{-}}}}\\
				&\lesssim \|\p_{t}\eta^{1}\|_{L^{\infty}H^{\frac{3}{2}+}}\|\tilde{u}\|_{L_{t}^{\infty}H^{1+\frac{\varepsilon_{-}}{2}}}+\|\p_{t}\eta\|_{L^{\infty}H^{1}}\|\tilde{u}^{2}\|_{L^{\infty}W^{2,q_{+}}}+\|\eta\|_{L_{t}^{\infty}H^{\frac{3}{2}-\alpha}}\|\p_{t}\tilde{u}^{2}\|_{L^{\infty}H^{1+\frac{\varepsilon_{-}}{2}}}\\
				&\lesssim \delta^{\frac{1}{2}}(d(u,\eta,p)+d(\tilde{u},\tilde{\eta},\tilde{p})).
			\end{aligned}
			\]
			Using the fact that $1-\varepsilon_{+}\geq \varepsilon_{+}-\varepsilon_{-}>\varepsilon_{+}-\varepsilon_{-}-2\alpha$, the estimate above yields
			\[
			\begin{aligned}
				|\sup_{0\leq t\leq T}\int_{\Omega}\p_{t}\tilde{u}wJ^{1}|&\lesssim \|\p_{t}\tilde{u}\|_{L_{t}^{\infty}H^{0}}\|w\|_{L_{t}^{\infty}H^{0}}\lesssim \|\p_{t}\tilde{u}\|_{L_{t}^{\infty}H^{0}}\|w\|_{L^{\infty}W_{0}^{1,\frac{4}{3-2\varepsilon_{+}}}}\\
				&\lesssim \delta^{\frac{1}{2}}d(\tilde{u},\tilde{\eta},\tilde{p})(d(\tilde{u},\tilde{\eta},\tilde{p})+d({u},{\eta},{p})).
			\end{aligned}
			\]
			\noindent Similarly, we have the following estimate
			\[
			\begin{aligned}
				\|w\|_{L^{2}W_{0}^{1,\frac{2}{1-\varepsilon_{-}}}}\lesssim&\|R^{2,1}J^{1}\|_{L^{2}L^{\frac{2}{1-\varepsilon_{-}}}}\\
				\lesssim&\|\nabla \p_{t}\bar{\eta}^{1}\|_{L^{\infty}L^{\infty}}\|\nabla \tilde{u}\|_{L^{2}L^{\frac{2}{1-\varepsilon_{-}}}}+\|\nabla \p_{t}\bar{\eta}\|_{L^{2}L^{\frac{2}{\alpha}}}\|\nabla \tilde{u}^{2}\|_{L^{\infty}L^{\frac{2}{1-\varepsilon_{+}}}}+\|\nabla \bar{\eta}\|_{L^{2}L^{\frac{2}{\alpha}}}\|\nabla \p_{t}\tilde{u}^{2}\|_{L_{t}^{\infty}L^{\frac{4}{2-\varepsilon_{-}}}}\\
				\lesssim&\|\p_{t}\eta^{1}\|_{L^{\infty}H^{\frac{3}{2}+}}\|\tilde{u}\|_{L^{2}W^{2,q_{-}}}+\|\p_{t}\eta\|_{L^{2}H^{\frac{3}{2}-\alpha}}\|\tilde{u}^{2}\|_{L^{\infty}W^{2,q_{+}}}+\|\eta\|_{L_{t}^{2}H^{\frac{3}{2}-\alpha}}\|\p_{t}\tilde{u}^{2}\|_{L_{t}^{\infty}H^{1+\frac{\varepsilon_{-}}{2}}}\\
				\lesssim&\delta^{\frac{1}{2}}(d(u,\eta,p)+d(\tilde{u},\tilde{\eta},\tilde{p})).
			\end{aligned}
			\]
			\noindent This estimate, together with Cauchy's inequality, implies that
			\[
			\begin{aligned}
				|\int_{0}^{T}\int_{\Omega}R^{1,1}\cdot wJ^{1}|\lesssim& \delta \bigg(\|\p_{t}\eta\|_{L_{t}^{2}H^{\frac{3}{2}-\alpha}}^{2}+\|\eta\|_{L_{t}^{2}W^{3-\frac{1}{q_{-}},q_{-}}}^{2}+\|\eta\|_{L_{t}^{\infty}H^{\frac{3}{2}-\alpha}}^{2}+\|u\|_{L_{t}^{2}W^{2,q_{+}}}^{2}+\|\p_{t}u\|_{L_{t}^{2}H^{1}}^{2}\\
				&+\|\p_{t}\eta\|_{L_{t}^{\infty}H^{1}}^{2}+\|\eta\|_{L_{t}^{2}W^{3-\frac{1}{q_{-}},q_{-}}}^{2}+\|\p_{t}u\|_{L^{2}H^{1}}^{2}+\|\tilde{u}\|_{L^{2}W^{2,q_{-}}}^{2}+\|\tilde{p}\|_{L_{t}^{2}W^{1,q_{-}}}^{2}\bigg)\\
				&\lesssim\delta^{\frac{1}{2}}(d(u,\eta,p)+d(\tilde{u},\tilde{\eta},\tilde{p}))^{2},
			\end{aligned}
			\]
			and similarly
			\[
			\begin{aligned}
				|\int_{0}^{T}\int_{\Omega}\mathbb{D}_{\p_{t}\mathcal{A}^{1}-\p_{t}\mathcal{A}^{2}}\tilde{u}^{2}:\mathbb{D}_{\mathcal{A}^{1}}w|\lesssim \|\p_{t}\eta\|_{L_{t}^{2}H^{\frac{3}{2}-\alpha}}\|w\|_{L_{t}^{2}W_{0}^{1,\frac{2}{1-\varepsilon_{-}}}}\|\tilde{u}^{2}\|_{L_{t}^{\infty}W^{2,q_{+}}}\lesssim \delta^{\frac{1}{2}}(d(u,\eta,p)+d(\tilde{u},\tilde{\eta},\tilde{p}))^{2}.
			\end{aligned}
			\]
			
			Moreover, the estimate for $\p_{t}w$ is given by
			\[
			\begin{aligned}
				\|\p_{t}w\|_{L_{t}^{2}L^{\frac{2}{1-\varepsilon_{-}}}}&\lesssim \|\p_{t}( J^{1}R^{2,1})\|_{L_{t}^{2}L^{q_{-}}}\lesssim (\|\p_{t}\eta^{1}\|_{L_{t}^{\infty}H^{\frac{3}{2}+\frac{\varepsilon_{-}-\alpha}{2}}})\|R^{2,1}\|_{L_{t}^{2}L^{\frac{2}{1-\varepsilon_{-}}}}+\|\p_{t}R^{2,1}\|_{L^{q_{-}}}\\
				&\lesssim \delta d(u,\eta,p)+\|\p_{t}^{2}\bar{\eta}\|_{L_{t}^{2}W^{1,\frac{2}{1+\alpha}}}\|\nabla u^{2}\|_{L_{t}^{\infty}W^{\frac{2}{1-\varepsilon_{+}}}}+\|\p_{t}\bar{\eta}\|_{L_{t}^{\infty}H^{1}}\|\nabla \p_{t}u^{2}\|_{L_{t}^{2}W^{1,\frac{2}{1-\varepsilon_{-}}}}\\
				&\quad+\|\p_{t}\bar{\eta}\|_{L_{t}^{\infty}H^{1}}\|\nabla \p_{t}\tilde{u}^{2}\|_{L_{t}^{2}W^{1,\frac{2}{1-\varepsilon_{-}}}}+\|\bar{\eta}\|_{L_{t}^{\infty}W^{1,\frac{2}{\alpha}}}\|\nabla \tilde{u}^{2}\|_{L_{t}^{2}L^{2}}+\|\p_{t}\bar{\eta}^{1}\|_{L_{t}^{\infty}H^{2+\frac{\varepsilon_{-}-\alpha}{2}}}\|\nabla \p_{t}u\|_{L_{t}^{2}L^{2}}\\
				&\quad+\|\p_{t}^{2}\bar{\eta}^{1}\|_{L_{t}^{2}W^{1\frac{2}{\alpha}}}\|\nabla u\|_{L_{t}^{\infty}L^{2}}\\
				&\lesssim \delta d(u,\eta,p)+\|\p_{t}^{2}{\eta}\|_{L_{t}^{2}H^{\frac{1}{2}-\alpha}}\| u^{2}\|_{L_{t}^{\infty}W^{2,q_{+}}}+\|\p_{t}\bar{\eta}\|_{L_{t}^{\infty}H^{1}}\|\nabla \p_{t}u^{2}\|_{L_{t}^{2}W^{2,q_{-}}}\\
				&\quad+\|\p_{t}\bar{\eta}\|_{L_{t}^{\infty}H^{1}}\|\nabla \p_{t}\tilde{u}^{2}\|_{L_{t}^{2}W^{2,q_{-}}}+\|\bar{\eta}\|_{L_{t}^{\infty}H^{\frac{3}{2}-\alpha}}\|\nabla {u}^{2}\|_{L_{t}^{2}H^{1}}+\|\p_{t}{\eta}^{1}\|_{L_{t}^{\infty}H^{\frac{3}{2}+\frac{\varepsilon_{-}-\alpha}{2}}}\| \p_{t}u\|_{L_{t}^{2}H^{1}}\\
				&\quad+\|\p_{t}^{2}\bar{\eta}^{1}\|_{L_{t}^{2}H^{\frac{3}{2}-\alpha}}\| u\|_{L_{t}^{\infty}H^{1}},
			\end{aligned}
			\]
			which yields that
			
			\[
			\begin{aligned}
				\int_{0}^{T}\int_{\Omega} \p_{t}\tilde{u}\p_{t}(J^{1}w)&=\int_{0}^{T}\int_{\Omega}\p_{t}\tilde{u}\p_{t}J^{1}w+\int_{0}^{T}\int_{\Omega}\p_{t}\tilde{u} J^{1}\p_{t}w\lesssim\|\p_{t}\tilde{u}\|_{L_{t}^{\infty}L^{2}}\|\p_{t}\eta^{1}\|_{L_{t}^{\infty}H^{\frac{3}{2}+\frac{\varepsilon_{-}-\alpha}{2}}}\|w\|_{L_{t}^{\infty}L^{2}}\\
				&\lesssim \delta d(\tilde{u},\tilde{\eta},\tilde{p})(d(\tilde{u},\tilde{\eta},\tilde{p})+d(u,\eta,p)).
			\end{aligned}
			\]
			
			\subparagraph{\underline{Terms related to contact points}}
			We first notice that $\p_tu^1_1=\p_tu^2_1=0$ at the contact points, so that $
			\p_t^2\tilde{\eta}(\pm\ell)=\p_t\tilde{u}\cdot\mathcal{N}(\pm\ell)$.
			We then note that
			\[
			\begin{aligned}
				|\p_{t}R^{7}|&
				\lesssim |\hat{\mathscr{W}}^\prime(u^{1}\cdot \mathcal{N}^{1})||\p_tu^{1}\cdot \mathcal{N}^{1}-\p_{t}u^{2}\cdot \mathcal{N}^{2}|+|\hat{\mathscr{W}}^\prime(u^{1}\cdot \mathcal{N}^{1})-\hat{\mathscr{W}}^{\prime}(u^{2}\cdot \mathcal{N}^{2})||\p_tu^{2}\cdot \mathcal{N}^{2}|\\
				&\lesssim (|u^{1}\cdot \mathcal{N}^{1}|)(|\p_{t}u^{1}\cdot \mathcal{N}^{1}-\p_{t}u^{2}\cdot \mathcal{N}^{2}|)+(|\p_{t}u^{2}\cdot \mathcal{N}^{2}|)(\|u\|_{W^{2,q_{-}}}+\|u^{2}\|_{W^{2,q_{-}}}\|\eta\|_{W^{3-\frac{1}{q_{-}},q_{}-}}),
			\end{aligned}
			\]
			\noindent which implies that:
			\[
			\|\p_{t}R^{7}\|_{L_{t}^{2}} \lesssim \delta^{\frac{1}{2}}(\|\p_t^2u^{1}\cdot \mathcal{N}^{1}-\p_{t}^{2}u^{2}\cdot \mathcal{N}^{2}\|_{L_{t}^{2}}+\|\eta\|_{L^{2}W^{3-\frac{1}{q_{-}},q_{-}}}+\|u\|_{L^{2}W^{2,q_{-}}})\lesssim \delta^{\frac{1}{2}}(d(u,\eta,p)+d(\tilde{u},\tilde{\eta},\tilde{p})).
			\]
			Hence, using trace theory, we have the following estimate
			\begin{equation}\label{est:remain_points}
				\begin{aligned}
					\int_0^T[\p_{t}\tilde{u}\cdot\mathcal{N}^1,  R^{7,1}]_\ell=& \int_{0}^{T}[\p_{t}\tilde{u}^{1}\cdot\mathcal{N}^1-\p_{t}\tilde{u}^{2}\cdot \mathcal{N}^{2},  R^{7,1}]_\ell+\int_{0}^{T}[\p_{t}\tilde{u}^{2}\cdot\mathcal{N}^1-\p_{t}\tilde{u}^{2}\cdot \mathcal{N}^{2},  R^{7,1}]_\ell\\
					\lesssim& \delta^{\frac{1}{2}}d(u,p,\eta)d(\tilde{u},\tilde{p},\tilde{\eta})+\|\p_{t}\tilde{u}^{2}\|_{L^{2}W^{2,q_{-}}}\|\eta\|_{L^{\infty}W^{3-\frac{1}{q_{-}},q_{-}}}\delta^{\frac{1}{2}}d(u,p,\eta)\\
					\lesssim&\delta^{\frac{1}{2}}d(u,p,\eta)d(\tilde{u},\tilde{p},\tilde{\eta}).
				\end{aligned}
			\end{equation}
			Finally, for the last term, we bound it by:
			\[
			\int_{0}^{T}[\p_{t}u^{1}\cdot \mathcal{N}^{1}-\p_{t}u^{2}\cdot \mathcal{N},\p_{t}u^{2}\cdot (\mathcal{N}^{2}-\mathcal{N}^{1})]_{\ell}\lesssim d(u,p,\eta)\|\p_{t}\tilde{u}^{2}\|_{L^{2}W^{2,q_{-}}}\|\eta\|_{L^{\infty}W^{3-\frac{1}{q_{-}},q_{-}}}\lesssim \delta^{\frac{1}{2}}d(u,p,\eta)^{2}.
			\]
			After integrating \eqref{energy_fix} over $[0,T]$, we use the Cauchy's inequality and Propositions \eqref{prop:isomorphism}, Gronwall's Lemma along with \eqref{est:remain_bulk2}, \eqref{est:remain_b2}, \eqref{est:remain_b3}, \eqref{est:remain_b4}, \eqref{est:remain_points} to deduce the bound
			\begin{equation}\label{est:contract_1}
				\begin{aligned}
					d(\tilde{u},\tilde{\eta},\tilde{p})&\le e^{C_1T\delta^{1/2}}C_2\delta^{1/2}(d(u,p,\eta)+d(\tilde{u},\tilde{p},\tilde{\eta}))^{2},
				\end{aligned}
			\end{equation}
			where the universal constants $C_1$ and $C_2$ are determined by the estimates of  \eqref{est:remain_bulk2}, \eqref{est:remain_b2}, \eqref{est:remain_b3}, \eqref{est:remain_b4}, \eqref{est:remain_points}.
			
			The lower-order energy estimates for $(\tilde{u},\tilde{p},\tilde{\eta})$ follow directly from integrating the higher-order terms $\partial_t \tilde{u}$ and $\partial_t \tilde{\eta}$, so we omit the details here.

			\paragraph{\underline{Step 5 -- Enhanced Estimate for $\p_t\tilde{\eta}$}}
			In order to close the estimates for $(\p_t\tilde{u}, \p_t\tilde{\eta})$, we have to estimate $\|\p_t\tilde{\eta}\|_{L^2H^{3/2-\alpha}}$ contained in \eqref{est:contract_1}. This part is similar to \cite[Proof of Theorem 9.3]{GT2020}, but the details are omitted there.  So we give the full proof here. The weak formulation for \eqref{linear_fix2} might be given as
			\begin{equation}\label{enhance_1}
				\begin{aligned}
					&\left<\p_t^2\tilde{u},\psi\right>_\ast+((\p_t\tilde{u},\psi))+(\p_t\tilde{\eta},\psi\cdot\mathcal{N}^1)_{1,\Sigma_{0}}+[\p_t{u}^{1}\cdot\mathcal{N}^1-\p_{t}\tilde{u}^{2}\cdot \mathcal{N}^{2},\psi\cdot\mathcal{N}^1]_{\ell}\\
					&=-\frac\mu2\int_\Om\left( \mathbb{D}_{\p_t\mathcal{A}^1-\p_t\mathcal{A}^2}{u}^2\right):\mathbb{D}_{\mathcal{A}}\psi J^1+\int_{\Om}R^{1,1}\cdot \psi J^1\\
					&\quad-\int_{-\ell}^{\ell}R^{4,1}\cdot \psi-[\psi\cdot\mathcal{N}^1,R^{7,1}]_\ell-\int_{-\ell}^{\ell} R^{5,1}\cdot \psi-\int_{-\ell}^{\ell}\mathcal{R}_{z}(\p_{1}\zeta_{0},\p_{1}\eta)\p_{1}\p_{t}\tilde{\eta} \p_{1}(\psi\cdot \mathcal{N}^{1})
				\end{aligned}
			\end{equation}
			for any $\psi\in W_\sigma(t)$. We now choose $\psi=M^1\nabla\phi$, where $\phi$ solves the equations
			\begin{align*}
				-\Delta\phi=0 \ \text{in} \ \Om,\quad
				\p_\nu\phi=(D_j^s\p_t\tilde{\eta}-\tilde{a}_{0}(t)\zeta_{0})/|\mathcal{N}_0| \ \text{on} \ \Sigma,\quad
				\p_\nu\phi=0 \ \text{on} \ \Sigma_s ,
			\end{align*}
			where the definition of the operator $ D_j^s$ can  be found in \cite[Section 8]{GT2020} and $\tilde{a}_{0}(t)$ is defined to be a function of $t$ such that
			\begin{align}
				\int_{-\ell}^{\ell}(\p_{t}\tilde{\eta}-\tilde{a}_{0}(t)\zeta_{0})dx=0.
			\end{align}
			Using the sixth equation of \eqref{linear_fix2}, we have:
			\begin{align}
				\p_{t}\tilde{\eta}(t)=\int_{0}^{t}\p_{t}\tilde{u}\cdot \mathcal{N}^{1}ds+\int_{0}^{t}R^{6,1}ds+\int_{0}^{t}u_{1}^{1}\p_{1}\p_{t}\tilde{\eta}+\int_{0}^{t}u_{1}\p_{1}\p_{t}\tilde{\eta}^{2},
			\end{align}
			\noindent which implies that
			\begin{align}
				\tilde{a}_{0}(t)=\frac{\int_{-\ell}^{\ell}\big( \int_{0}^{t}\p_{t}\tilde{u}\cdot \mathcal{N}^{1}ds+\int_{0}^{t}R^{6,1}ds+\int_{0}^{t}u_{1}^{1}\p_{1}\p_{t}\tilde{\eta}+\int_{0}^{t}u_{1}\p_{1}\p_{t}\tilde{\eta}^{2}\big)}{\int_{-\ell}^{\ell}\zeta_{0}dx},
			\end{align}
			\noindent and
			\begin{align}
				\tilde{a}_{0}^{\prime}(t)=\frac{\int_{-\ell}^{\ell}\big( \p_{t}\tilde{u}\cdot \mathcal{N}^{1}ds+R^{6,1}+u_{1}^{1}\p_{1}\p_{t}\tilde{\eta}+u_{1}\p_{1}\p_{t}\tilde{\eta}^{2}\big)}{\int_{-\ell}^{\ell}\zeta_{0}dx}.
			\end{align}
			\noindent Using the definition of $R^{6,1}$ and integration by part, we have the following estimate
			\begin{align}
				|\tilde{a}_{0}(t)|\lesssim &\sqrt{T}\|\p_{t}\tilde{u}\|_{L_{t}^{2}H^{1}}+\sqrt{T}\|\p_{t}\tilde{u}^{2}\|_{L_{t}^{2}H^{1}}\|\eta\|_{L_{t}^{\infty}H^{1}}+T\|u^{2}\|_{L_{t}^{\infty}H^{1}}\|\p_{t}\tilde{\eta}\|_{L_{t}^{\infty}H^{1}}+T\|\p_{t}\tilde{\eta}^{2}\|_{L_{t}^{\infty}H^{1}}\|u\|_{L_{t}^{\infty}H^{1}}\notag\\
				\lesssim&\sqrt{T}d(\tilde{u},\tilde{p},\tilde{\eta})+\delta^{\frac{1}{2}}d(\tilde{u},\tilde{p},\tilde{\eta})+\delta^{\frac{1}{2}} d({u},{p},{\eta}),
			\end{align}
			\noindent and similarly
			\begin{align}
				|\tilde{a}^{\prime}_{0}(t)|\lesssim &\|\p_{t}\tilde{u}(t)\|_{H^{1}}+\|\p_{t}\tilde{u}^{2}\|_{L_{t}^{\infty}H^{1}}\|\eta\|_{L_{t}^{\infty}H^{1}}+T\|u\|_{L_{t}^{\infty}H^{1}}\|\p_{t}\eta^{1}\|_{L_{t}^{\infty}H^{1}}+T\|u^{2}\|_{L_{t}^{\infty}H^{1}}\|\p_{t}\tilde{\eta}\|_{L_{t}^{\infty}H^{1}}\notag\\
				\lesssim&\|\p_{t}\tilde{u}(t)\|_{H^{1}}+\delta^{\frac{1}{2}}d(\tilde{u},\tilde{p},\tilde{\eta})+\delta^{\frac{1}{2}}d({u},{p},{\eta}).
			\end{align}
			
			We take $s=1-2\alpha$ so that $\frac12+s+\alpha=\frac32-\alpha$. Note that $M^1=K^1\nabla\Phi^1$ for $\Phi^1$ the straightening map in terms of $\eta^1$, so for the first term in \eqref{enhance_1}, we have
			\begin{equation}\label{enhance_2}
				\left<\p_t^2\tilde{u},\psi\right>_\ast=\frac{d}{dt}\int_\Om \p_t \tilde{u}\cdot\nabla\Phi^1\nabla\phi-\int_\Om \p_t\tilde{u}\cdot\nabla\p_t\Phi^1\nabla\phi-\int_\Om \p_t\tilde{u}\cdot\nabla\Phi^1\nabla\p_t\phi.
			\end{equation}
			Under the assumptions for $S(T, \delta)$, we have the bounds
			\begin{equation}
				\begin{aligned}
					\sup_{0\le t\le T}\int_\Om \p_t \tilde{u}\cdot\nabla\Phi^1\nabla\phi&\lesssim \|\p_t \tilde{u}\|_{L^\infty H^0}(\|\p_t\tilde{\eta}\|_{L^\infty H^1}+|\tilde{a}_{0}(t)|_{L_{t}^{\infty}})\notag\\
					&\lesssim\|\p_t\tilde{u}\|_{L^\infty H^0}(\|\p_t\tilde{\eta}\|_{L^\infty H^1}+\sqrt{T}d(\tilde{u},\tilde{p},\tilde{\eta})+\delta^{\frac{1}{2}} d(\tilde{u},\tilde{p},\tilde{\eta})+\delta^{\frac{1}{2}} d({u},{p},{\eta})),
				\end{aligned}
			\end{equation}
			and
			\begin{equation}
				\begin{aligned}
					&\left|\int_0^T-\int_\Om \p_t\tilde{u}\cdot\nabla\p_t\Phi^1\nabla\phi-\int_\Om \p_t\tilde{u}\cdot\nabla\Phi^1\nabla\p_t\phi\right|\\\lesssim& \sqrt{T}\|\p_t\tilde{u}\|_{L^\infty H^0}(\|\phi\|_{L^2H^1}+\|\p_t\phi\|_{L^2H^1})\\
					\lesssim&\sqrt{T}\|\p_t\tilde{u}\|_{L^\infty H^0}\big(\|\p_t\tilde{\eta}\|_{L^2H^1}+\|\p_t^2\tilde{\eta}\|_{L^2H^{\frac{1}{2}-\alpha}}+\|a(t)\|_{L_{t}^{2}}+\|a(t)^{\prime}\|_{L_{t}^{2}}\big)\notag\\
					\lesssim& \sqrt{T}\|\p_t\tilde{u}\|_{L^\infty H^0}\big(\|\p_t\tilde{\eta}\|_{L^2H^1}+\|\p_t^2\tilde{\eta}\|_{L^2H^{\frac{1}{2}-\alpha}}\big)\\&+\sqrt{T}\|\p_{t}\tilde{u}\|_{L_{t}^{\infty}H^{0}}\big(\|\p_{t}\tilde{u}\|_{L_{t}^{2}H^{1}}+\sqrt{T}d(\tilde{u},\tilde{p},\tilde{\eta})+\delta^{\frac{1}{2}}d(\tilde{u},\tilde{p},\tilde{\eta})+\delta^{\frac{1}{2}} d({u},{p},{\eta})\big).&
				\end{aligned}
			\end{equation}
			For the second term in \eqref{enhance_1}, we estimate
			\begin{equation}
				\begin{aligned}
					\left|\int_0^T\int_\Om \mathbb{D}_{\mathcal{A}^1}\p_t\tilde{u}:\mathbb{D}_{\mathcal{A}^1}\psi\right| \lesssim &\|\p_t\tilde{u}\|_{L^2 H^1}\|\phi\|_{L^2H^2} \\\lesssim& \|\p_t\tilde{u}\|_{L^2 H^1}(\|D_j^s\p_t\tilde{\eta}\|_{L^2\mathcal{H}_{\mathcal{K}}^{1/2}}+|\tilde{a}_{0}(t)|_{L_{t}^{2}})\\
					\lesssim&\|\p_t\tilde{u}\|_{L^2 H^1}(\|D_j^s\p_t\tilde{\eta}\|_{L^2\mathcal{H}_{\mathcal{K}}^{1/2}}+\sqrt{T}d(\tilde{u},\tilde{p},\tilde{\eta})+\delta^{\frac{1}{2}} d(\tilde{u},\tilde{p},\tilde{\eta})+\delta^{\frac{1}{2}} d({u},{p},{\eta})),
				\end{aligned}
			\end{equation}
			where the definition of the space $\mathcal{H}_{\mathcal{K}}^{1/2}$ can also be found in \cite[Section 8]{GT2020}.
			
			The third term in \eqref{enhance_1}, after integration over $[0,T]$ in time, can be rewritten as
			\begin{equation}
				\begin{aligned}
					\int_0^T(\p_t\tilde{\eta},\psi\cdot\mathcal{N}^1)_{1,\Sigma_{0}}&=\int_0^T(\p_t\tilde{\eta}, D_j^s\p_t\tilde{\eta})_{1,\Sigma_{0}}-\int_{0}^{T}(\p_{t}\tilde{\eta},\tilde{a}_{0}(t)D_{j}^{s}\rho_{0})_{1,\Sigma_{0}}\\
					&=\int_0^T(D_j^{s/2}\p_t\tilde{\eta}, D_j^{s/2}\p_t\tilde{\eta})_{1,\Sigma}-\int_{0}^{T}(\p_{t}\tilde{\eta},\tilde{a}_{0}(t)D_{j}^{s}\rho_{0})_{1,\Sigma_{0}}\\
					&\gtrsim\|D_j^{s/2}\p_t\tilde{\eta}\|_{L^2\mathcal{H}_{\mathcal{K}}^1}^2-\|a_{0}(t)\|_{L_{t}^{2}}\|\p_{t}\tilde{\eta}\|_{L_{t}^{2}H^{1}}\\
					&\gtrsim\|D_j^s\p_t\tilde{\eta}\|_{L^2\mathcal{H}_{\mathcal{K}}^{1/2+\alpha}}^2-\|\p_{t}\tilde{\eta}\|_{L_{t}^{2}H^{1}}(\sqrt{T}d(\tilde{u},\tilde{p},\tilde{\eta})+\delta^{\frac{1}{2}}d(\tilde{u},\tilde{p},\tilde{\eta})+\delta^{\frac{1}{2}} d({u},{p},{\eta})),
				\end{aligned}
			\end{equation}
			since $M^1\nabla\phi\cdot\mathcal{N}^1=\nabla\phi\cdot\mathcal{N}_0=D_t^j\p_t\tilde{\eta}-\tilde{a}_{0}(t)D_{j}^{s}\rho_{0}$.
			The fourth term in \eqref{enhance_1}, after integration over $[0,T]$ in time, can be estimated by
			\begin{equation}
				\begin{aligned}
					\int_0^T|[\p_t\tilde{u}^{1}\cdot\mathcal{N}^1-\p_{t}\tilde{u}^{2}\cdot \mathcal{N}^{2},\psi\cdot\mathcal{N}^1]_{\ell}|\lesssim&\int_0^T[\p_t\tilde{u}^{1}\cdot\mathcal{N}^1-\p_{t}\tilde{u}^{2}\cdot \mathcal{N}^{2}]_\ell(\|D_j^s\p_t\tilde{\eta}\|_{\mathcal{H}_{\mathcal{K}}^{1/2+\alpha}}+|\tilde{a}_{0}(t)|)\\
					\lesssim&\|[\p_t\tilde{u}^{1}\cdot\mathcal{N}^1-\p_{t}\tilde{u}^{2}\cdot \mathcal{N}^{2}]_{\ell}\|_{L_{t}^2}\|D_j^s\p_t\tilde{\eta}\|_{L^2\mathcal{H}_{\mathcal{K}}^{1/2+\alpha}}\\
					&+\|[\p_t\tilde{u}^{1}\cdot\mathcal{N}^1-\p_{t}\tilde{u}^{2}\cdot \mathcal{N}^{2}]_{\ell}\|_{L_t^2}(\sqrt{T}d(\tilde{u},\tilde{p},\tilde{\eta})+\delta^{\frac{1}{2}} d(\tilde{u},\tilde{p},\tilde{\eta})+\delta^{\frac{1}{2}} d({u},{p},{\eta})).
				\end{aligned}
			\end{equation}
			Then the symmetric gradient form will result in the estimate
			\begin{equation}
				\begin{aligned}
					&|\int_0^T\int_\Om\left( \mathbb{D}_{\p_t\mathcal{A}^1-\p_t\mathcal{A}^2}{u}^2\right):\mathbb{D}_{\mathcal{A}}\psi J^1|\lesssim \delta^{1/2}\|\phi\|_{L^2H^2}(\|\p_t\eta\|_{L^2H^{3/2-\alpha}}+|\tilde{a}_{0}(t)|)\\
					&\lesssim\delta^{1/2}\|D_j^s\p_t\tilde{\eta}\|_{L^2\mathcal{H}_{\mathcal{K}}^{1/2}}(d(u,\eta,p)+\sqrt{T}d(\tilde{u},\tilde{p},\tilde{\eta})+\delta^{\frac{1}{2}} d(\tilde{u},\tilde{p},\tilde{\eta})+\delta^{\frac{1}{2}} d({u},{p},{\eta})).
				\end{aligned}
			\end{equation}
			
			The terms involving $R^{1, j}$, for $j=1, 4,5$ in \eqref{enhance_1} are estimated, similar to \eqref{est:remain_bulk2} and \eqref{est:remain_b3}.  From the computation in Step 4, we have
			\begin{align}
				\|R^{1,1}-R^{4,1}-R^{5,1}\|_{L^{2}(\mathcal{H}^1)^{*}}\lesssim \delta^{\frac{1}{2}}(d(u,\eta,p)+d(\tilde u,\tilde\eta,\tilde p)).
			\end{align}
			\noindent Therefore:
			\begin{equation}\label{est:r11_r14}
				\begin{aligned}
					&|\int_0^T(\int_{\Om}R^{1,1}\cdot \psi J^1-\int_{-\ell}^{\ell}R^{4,1}\cdot \psi-\int_{-\ell}^{\ell}R^{5,1}\cdot \psi)|\\
					&\lesssim \delta^{\frac{1}{2}}(d(u,\eta,p)+d(\tilde u,\tilde\eta,\tilde p))\|\psi\|_{H^{1}}\lesssim \delta^{\frac{1}{2}}(d(u,\eta,p)+d(\tilde u,\tilde\eta,\tilde p))(\|D_j^s\p_t\tilde{\eta}\|_{L^2\mathcal{H}_{\mathcal{K}}^{1/2}}+|\tilde{a}_{0}(t)|)\\
					&\lesssim \delta^{\frac{1}{2}}(d(u,\eta,p)+d(\tilde u,\tilde\eta,\tilde p))(\|D_j^s\p_t\tilde{\eta}\|_{L^2\mathcal{H}_{\mathcal{K}}^{1/2}}+\sqrt{T}d(\tilde{u},\tilde{p},\tilde{\eta})+\delta^{\frac{1}{2}} d(\tilde{u},\tilde{p},\tilde{\eta})+\delta^{\frac{1}{2}} d({u},{p},{\eta})).
				\end{aligned}
			\end{equation}
			
			Similarly, we still use $\psi\cdot\mathcal{N}^1=M^1\nabla\phi\cdot\mathcal{N}^1=D_j^s\p_t\tilde{\eta}-\tilde{a}_{0}(t)D_{j}^{s}\rho_{0}$ on $\pm\ell$ and the computation for $R^{7,1}$ to bound the following term
			\begin{equation}\label{enhance_3}
				\begin{aligned}
					|\int_0^T[\psi\cdot\mathcal{N}^1, R^{7,1}]_\ell|&\lesssim (\|[D_j^s\p_t\tilde{\eta}]_\ell+|\tilde{a}_{0}(t)|)\|_{L_{t}^{2}}\delta^{\frac{1}{2}}(d(u,\eta,p)+d(\tilde u,\tilde\eta,\tilde p))\\
					&\lesssim  \|D_j^s\p_t\tilde{\eta}\|_{L_{t}^{2}H^{1-s/2}}\delta^{\frac{1}{2}}(d(u,\eta,p)+d(\tilde u,\tilde\eta,\tilde p))\\
					&\quad+\delta^{\frac{1}{2}}(d(u,\eta,p)+d(\tilde u,\tilde\eta,\tilde p))(\sqrt{T}d(\tilde{u},\tilde{p},\tilde{\eta})+\delta^{\frac{1}{2}} d(\tilde{u},\tilde{p},\tilde{\eta})+\delta^{\frac{1}{2}} d({u},{p},{\eta})).
				\end{aligned}
			\end{equation}
			
			Finally, we estimate the last term in \eqref{enhance_1}. We have
			\begin{align}\label{enhance_5}
				\begin{aligned}
					&\int_{0}^{T}(\sigma(\mathcal{R}_{z}(\p_{1}\zeta_{0},\p_{1}\eta^{1})\p_{1}\p_{t}\tilde{\eta}),\p_{1}(\psi\cdot\mathcal{N}^{1}(t)))_{L^{2}}\\
					\lesssim& (\|\p_{1}D_{j}^{s}\p_{t}\tilde{\eta}\|_{L^{2}H^{-\frac{s}{2}}}+\|\tilde{a}_{0}(t)\|_{L_{t}^{2}})\|\int_{0}^{t}(\mathcal{R}_{z}(\p_{1}\zeta_{0},\p_{1}\eta^{1})\p_{1}\p_{t}\tilde{\eta})\|_{L^{2}H^{\frac{s}{2}}}\\
					\lesssim& \|\p_{t}\tilde{\eta}\|^{2}_{L^{2}H^{1+\frac{s}{2}}}\|\p_{t}\eta^{1}\|_{L^{\infty}H^{\frac{3}{2}+}}+\|\p_{t}\tilde{\eta}\|_{L^{2}H^{1+\frac{s}{2}}}(\sqrt{T}d(\tilde{u},\tilde{p},\tilde{\eta})+\delta^{\frac{1}{2}}d(\tilde{u},\tilde{p},\tilde{\eta})+\delta^{\frac{1}{2}} d({u},{p},{\eta}))\\\lesssim&\|\p_{t}\tilde{\eta}\|^{2}_{L^{2}H^{1+\frac{s}{2}}}\|\p_{t}\eta^{1}\|_{L^{\infty}W^{3-\frac{1}{q_{-}},q_{-}}}+\|\p_{t}\tilde{\eta}\|_{L^{2}H^{1+\frac{s}{2}}}(\sqrt{T}d(\tilde{u},\tilde{p},\tilde{\eta})+\delta^{\frac{1}{2}} d(\tilde{u},\tilde{p},\tilde{\eta})+\delta^{\frac{1}{2}} d({u},{p},{\eta}))\\
					\lesssim&\delta^{\frac{1}{2}}\|\p_{t}\tilde{\eta}\|_{L^{2}H^{1+\frac{s}{2}}}^{2}+\sqrt{T}d^{2}(\tilde{u},\tilde{p},\tilde{\eta})+\delta^{\frac{1}{2}}d(\tilde{u},\tilde{p},\tilde{\eta})d({u},{p},{\eta})+\delta^{\frac{1}{2}}d(\tilde{u},\tilde{p},\tilde{\eta})^{2}.
				\end{aligned}
			\end{align}
			
			Setting $T\lesssim \delta$, the combination of \eqref{enhance_1}--\eqref{enhance_5} along with the Cauchy's inequality, after $j\to\infty$, implies that
			\begin{equation}\label{est:contract_3}
				\begin{aligned}
					\|\p_t\tilde{\eta}\|_{L^2H^{3/2-\alpha}}^2
					\lesssim& \delta^{\frac{1}{2}}(d(u,p,\eta)+d(\tilde{u},\tilde{p},\tilde{\eta}))^{2}.
				\end{aligned}
			\end{equation}
			
			Now it remains to use the kinematic boundary condition to show the estimate for $\p_{t}^{2}\tilde{\eta}$. We have:
			
			\begin{align}\label{eq:highest_d}
				\|(\p_t^{2}\tilde{\eta})\|_{L^{2}H^{\frac{1}{2}-\alpha}}=&\|\p_t\tilde{u}\cdot\mathcal{N}^1\|_{L^{2}H^{\frac{1}{2}-\alpha}}+\|R^{6,1}\|_{L^{2}H^{\frac{1}{2}-\alpha}}+\|u^{1}\p_{1}\p_{t}\tilde{\eta}\|_{L^{2}H^{\frac{1}{2}-\alpha}}+\|u\p_{1}\p_{t}\tilde{\eta}^{2}\|_{L^{2}H^{\frac{1}{2}-\alpha}}\\
				=&J_1+J_2+J_3+J_4.\no
			\end{align}
			\noindent For the term $J_1$, by trace theorem, it is bounded by
			\[
			J_{1}\lesssim \|\p_{t}\tilde{u}\|_{L^{2}H^{1}}.
			\]
			\noindent For $J_{2}$, by definition of $R^{6,1}$, we have
			\[
			J_2\lesssim \|\p_{t}\tilde{u}^{2}\|_{L^{2}W^{2,q_{-}}}\|\eta\|_{L^{\infty}H^{\frac{3}{2}-\alpha}}.
			\]
			\noindent For $J_3$, we have
			\[
			J_{3}\lesssim \|u^{1}\p_{1}\p_{t}\tilde{\eta}\|_{L^{2}H^{\frac{1}{2}-\alpha}}
			\lesssim\|u^{1}\|_{L^{\infty}W^{2,q_{-}}}\|\p_{t}\tilde{\eta}\|_{L^{2}H^{\frac{3}{2}-\alpha}}.
			\]
			\noindent Similarly, for the term $J_{4}$, we have:
			\[
			J_{4}\lesssim \|u\|_{L^{\infty}H^{\frac{1}{2}+}(\Sigma)}\|\p_{t}\tilde{\eta}^{2}\|_{L^{2}H^{\frac{3}{2}-\alpha}}\lesssim \|u\|_{L^{\infty}W^{2,q_{+}}}\|\p_{t}\tilde{\eta}^{2}\|_{L^{2}H^{\frac{3}{2}-\alpha}}.
			\]
			\noindent Therefore, combining estimates for $J_{1}$ to $J_{4}$, we obtain
			\[
			\|\p_t^{2}\tilde{\eta}\|_{L^{2}H^{\frac{1}{2}-\alpha}}\lesssim \delta^{\frac{1}{2}}(d(u,p,\eta)+d(\tilde{u},\tilde{p},\tilde{\eta})).
			\]
			
			By integrating $\p_{t}\tilde{\eta}$, we obtain the following estimate for $\tilde{\eta}$
			\begin{align}
				\|\tilde{\eta}\|_{L_{t}^{\infty}H^{\frac{3}{2}-\alpha}}\lesssim \delta^{\frac{1}{2}}(d(u,p,\eta)+d(\tilde{u},\tilde{p},\tilde{\eta}))
			\end{align}

			\paragraph{\underline{Step 6 -- Pressure Estimates and Elliptic Estimates for $(\tilde{u}, \tilde{p}, \tilde{\eta})$}}
			
			We first derive the system for the zero-order terms. Integrating the system \eqref{linear_fix2} from $0$ to $t$ for any $t\in (0,T)$, we obtain
			\begin{equation}\label{linear_fix1}
				\left\{
				\begin{aligned}
					&(\p_t\tilde{u})+\int_{0}^{t}\dive_{\mathcal{A}^1}S_{\mathcal{A}^1}(\p_t\tilde{p},\p_t\tilde{u})=\tilde{R}^{1,1} \  &\text{in}&\ \Om,\\
					&\int_{0}^{t}\dive_{\mathcal{A}^1}\p_t\tilde{u}=\tilde{R}^{2,1} \  &\text{in}&\ \Om,\\
					&\int_{0}^{t}S_{\mathcal{A}^1}(\p_t\tilde{p},\p_t\tilde{u})\mathcal{N}^1=\int_{0}^{t}g(\p_t\tilde{\eta})\mathcal{N}^1-\int_{0}^{t}\sigma\p_1\left(\frac{\p_1\p_t\tilde{\eta}}{(1+|\p_1\zeta_0|)^{3/2}}\right)\mathcal{N}^1\\&\quad\quad\quad\quad\quad\quad\quad-\sigma\int_{0}^{t}\p_{1}(\mathcal{R}_{z}(\p_{1}\zeta_{0},\p_{1}\eta^{1})\p_{1}\p_{t}\tilde{\eta})\mathcal{N}^{1}(t)\\&\quad\quad\quad\quad\quad\quad\quad-\sigma \int_{0}^{t}\p_{1}((\mathcal{R}_{z}(\p_{1}\zeta_{0},\p_{1}\eta^{2})-\mathcal{R}_{z}(\p_{1}\zeta_{0},\p_{1}\eta^{1}))\p_{1}\p_{t}\tilde{\eta}^{2})\mathcal{N}^{1}(t)+\tilde{R}^{4,1} \ &\text{on}&\ \Sigma,\\
					&\int_{0}^{t}(S_{\mathcal{A}^1}(\p_t\tilde{p},\p_t\tilde{u})\nu-\beta \p_t\tilde{u})\cdot\tau=\tilde{R}^{5,1} \  &\text{on}&\ \Sigma_s,\\
					&\p_t\tilde{u}\cdot\nu=0 \  &\text{on}&\ \Sigma_s,\\
					&(\p_t\tilde{\eta})=\int_{0}^{t}\p_t\tilde{u}\cdot\mathcal{N}^1+\tilde R^{6,1} \  &\text{on}&\ \Sigma,\\
					&\kappa(\p_{t}\tilde{\eta})(\pm\ell,t)=\mp\sigma\frac{\p_1\tilde{\eta}}{(1+|\p_1\zeta_0|^2)^{3/2}}(\pm\ell,t)\mp\sigma\int_{0}^{t}\p_{1}(\mathcal{R}_{z}(\p_{1}\zeta_{0},\p_{1}\eta^{1})\p_{1}\p_{t}\tilde{\eta}(\pm \ell,t)\\&\quad\quad\quad\quad\quad\quad\quad\mp\sigma \int_{0}^{t}\p_{1}((\mathcal{R}_{z}(\p_{1}\zeta_{0},\p_{1}\eta^{2})-\mathcal{R}_{z}(\p_{1}\zeta_{0},\p_{1}\eta^{1}))\p_{1}\p_{t}\tilde{\eta}^{2})(\pm \ell,t)-\tilde R^{7,1}(\pm\ell,t)
				\end{aligned}
				\right.
			\end{equation}
			\noindent where:
			\begin{equation}\label{remainder1'}
				\begin{aligned}
					\tilde R^{1,1}&=R^1+\operatorname{div}_{\mathcal{A}^{1}}S_{\mathcal{A}^{1}}(p,u)+\operatorname{div}_{\mathcal{A}^{1}}S_{\mathcal{A}^{1}-\mathcal{A}^{2}}(p^{2},u^{2})+\operatorname{div}_{\mathcal{A}^{1}-\mathcal{A}^{2}}S_{\mathcal{A}^{2}}(p^{2},u^{2}),\\
					\tilde{R}^{2,1}&=-\int_{0}^{t}(\operatorname{div}_{\p_{t}\mathcal{A}^{1}-\p_{t}\mathcal{A}^{2}}{u}^{2})-\int_{0}^{t}\dive_{\p_{t}\mathcal{A}^1}{u}-\int_{0}^{t}\dive_{\mathcal{A}^{1}-\mathcal{A}^{2}}\tilde{u}^{2},\\
					\tilde{R}^{4,1}&=R^4 +\int_{0}^{t}\mu S_{\p_{t}\mathcal{A}^1}({u},{p})\mathcal{N}^1+\int_{0}^{t}\mu S_{\mathcal{A}^{1}}(u,p)\p_{t}\mathcal{N}^{1}+\int_{0}^{t}\mu S_{\p_{t}\mathcal{A}^{1}-\p_{t}\mathcal{A}^{2}}(u^{2},p^{2})\mathcal{N}^{2}\\&\quad +\int_{0}^{t}\mu S_{\mathcal{A}^{1}-\mathcal{A}^{2}}({u}^{2},{p}^{2})\p_{t}\mathcal{N}^{1} +\int_{0}^{t}g{\eta}\p_{t}\mathcal{N}^{1}-\sigma\p_1\left(\frac{\p_1{\eta}}{(1+|\p_1\zeta_0|)^{3/2}}\right)\p_t\mathcal{N}^1\\
					&\quad-\int_{0}^{t}(\mathcal{K}(\eta^{1})-\p_{1}(\mathcal{R}(\p_{1}\zeta_{0},\p_{1}\eta^{1}))\p_{t}(\mathcal{N}^{1}-\mathcal{N}^{2})-(\mathcal{K}(\eta)-\p_{1}\mathcal{R}(\p_{1}\zeta_{0},\p_{1}\eta^{1}
					)+\p_{1}\mathcal{R}(\p_{1}\zeta_{0},\p_{1}\eta^{2}))\p_{t}\mathcal{N}^{2}\\
					&\quad-\sigma \int_{0}^{t}(\mathcal{K}(\tilde{\eta}^{1})-\p_{1}(\mathcal{R}_{z}(\p_{1}\zeta_{0},\p_{1}\eta^{1})\p_{1}\p_{t}\tilde{\eta}^{1}))(\mathcal{N}^{1}-\mathcal{N}^{2}) +\int_{0}^{t}\mu S_{\mathcal{A}^{1}-\mathcal{A}^{2}}(\p_{t}\tilde{u}^{2},\p_{t}\tilde{p}^{2})\mathcal{N}^{2}\\
					\tilde R^{5,1}&=-\int_{0}^{t}S_{\p_{t}\mathcal{A}^{1}}({p},{u})+\mu\int_{0}^{t}S_{\p_{t}\mathcal{A}^{1}-\p_{t}\mathcal{A}^{2}}({u}^{2},p^{2})+\int_{0}^{t}\mu S_{\mathcal{A}^{1}-\mathcal{A}^{2}}(\p_{t}\tilde{u}^{2},\p_{t}\tilde{p}^{2}),\\
					\quad \tilde R^{6,1}=& \int_{0}^{t}\p_{t}\tilde{u}^{2}\cdot(\mathcal{N}^{1}-\mathcal{N}^{2})+\int_{0}^{t}(u_{1}^{1}-u_{1}^{2})\p_{1}\p_{t}\tilde{\eta}+\int_{0}^{t}(u_{1}^{2})(\p_{1}\p_{t}\tilde{\eta}),\\
					\quad \tilde{R}^{7,1}=&R^7.
				\end{aligned}
			\end{equation}
			To apply elliptic theory to \eqref{linear_fix1}, we begin by estimating the elliptic norms of each equation independently.
			
			\subparagraph{\underline{Momentum equation}}
			
			Using integration by part, we can reformulate the first equation of \eqref{linear_fix1} to:
			\begin{align}
				\operatorname{div}_{\mathcal{A}^{1}}S_{\mathcal{A}^{1}}(\tilde{p},\tilde{u})=-\p_{t}\tilde{u}+\int_{0}^{t}\Big(\operatorname{div}_{\p_{t}\mathcal{A}^{1}}S_{\mathcal{A}^{1}}(\tilde{p},\tilde{u})+\operatorname{div}_{\mathcal{A}^{1}}\mathbb{D}_{\p_{t}\mathcal{A}^{1}}\tilde{u} \Big)+\tilde{R}^{1,1}. \label{eq:elliptic_1}
			\end{align}
			
			\noindent We estimate the elliptic nor for each term on the right hand side of equation \eqref{eq:elliptic_1}. We have
			\begin{align}
				\|\p_{t}\tilde{u}\|_{L^{2}L^{q_{-}}}\lesssim \|\p_{t}\tilde{u}\|_{L^{2}H^{1}},
			\end{align}
			\noindent and
			\begin{align}{\label{eq:int_e_1}}
				\left\|\int_{0}^{t}\operatorname{div}_{\p_{t}\mathcal{A}^{1}}S_{\mathcal{A}^{1}}(\tilde{p},\tilde{u}) \right\|_{L^{2}L^{q_{-}}}\lesssim& T\|\p_{t}\bar{\eta}^{1}\|_{L^{\infty}W^{1,\infty}}\|\bar{\eta}^{1}\|_{L^{\infty}W^{2,\frac{2}{1-\varepsilon_{-}}}}(\|\tilde{p}\|_{L^{2}L^{\frac{2}{1-\varepsilon_{-}}}}+\|\tilde{u}\|_{L^{2}W^{1,\frac{2}{1-\varepsilon_{-}}}})\no\\
				\lesssim&T\|\p_{t}{\eta}^{1}\|_{L^{\infty}W^{3-\frac{1}{q_{-}},q_{-}}}\|{\eta}^{1}\|_{L^{\infty}W^{3-\frac{1}{q_{-}},q_{-}}}(\|\tilde{p}\|_{L^{2}W^{1,q_{-}}}+\|\tilde{u}\|_{L^{2}W^{2,q_{-}}})\no\\
				\lesssim& T\delta d(\tilde{u},\tilde{p},\tilde{\eta}),
			\end{align}
			\noindent and similarly
			\begin{align}{\label{eq:int_e_2}}
				&\left \|\int_{0}^{t}(\operatorname{div}_{\mathcal{A}^{1}}\mathbb{D}_{\p_{t}\mathcal{A}^{1}}\tilde{u}) \right\|_{L^{2}L^{q_{-}}}\no\\
				\lesssim&T^{\frac{1}{2}}(1+\|{\eta}^{1}\|_{L^{\infty}W^{3-\frac{1}{q_{-}},q_{-}}})(\|\p_{t}{\eta}^{1}\|_{L^{2}W^{1,\infty}}\|\tilde{u}\|_{L^{2}W^{2,q_{-}}}+\|\p_{t}\eta^{1}\|_{L^{\infty}W^{2,\frac{2}{1-\varepsilon_{-}}}}\|\tilde{u}\|_{L_{t}^{2}W^{1,\frac{2}{1-\varepsilon_{-}}}})\no\\
				\lesssim& T^{\frac{1}{2}}(\|\p_{t}{\eta}^{1}\|_{L^{\infty}W^{3-\frac{1}{q_{-}},q_{-}}}\|\tilde{u}\|_{L^{2}W^{2,q_{-}}}+\|\p_{t}\eta^{1}\|_{L^{\infty}W^{3-\frac{1}{q_{-}},q_{-}}}\|\tilde{u}\|_{L_{t}^{2}W^{2,q_{-}}})\no\\
				\lesssim& T^{\frac{1}{2}}\delta d(\tilde{u},\tilde{p},\tilde{\eta}).
			\end{align}
			\noindent Then it remains to estimate the terms included in $R^{1,1}$. We first discuss the terms in $R^{1}$. By definition of $R^{1}$, we have (For simplicity, we omitted details here; the estimate is similar to the computation done in the proof of Theorem \ref{thm:forcing_+}):
			\begin{align}
				\begin{aligned}
					\|R^{1}\|_{L^{2}L^{q_{-}}}\lesssim& T(1+\|\eta^{1}\|_{L^{\infty}W^{3-\frac{1}{q_{-}},q_{-}}}+\|\eta^{2}\|_{L^{\infty}W^{3-\frac{1}{q_{-}},q_{-}}})(\|u^{1}\|_{L^{\infty}W^{2,q_{-}}})\|u\|_{L^{2}W^{2,q_{-}}}\\&+T(1+\|\eta^{1}\|_{L^{\infty}W^{3-\frac{1}{q_{-}},q_{-}}})\|\p_{t}\eta^{1}\|_{L^{\infty}H^{\frac{3}{2}-\alpha}}\|u\|_{L^{2}W^{2,q_{-}}}\\
					&+T(1+\|\eta^{1}\|_{L^{\infty}W^{3-\frac{1}{q_{-}},q_{-}}})\|\p_{t}\eta^{1}\|_{L^{\infty}H^{\frac{3}{2}-\alpha}}\|u^{2}\|_{L^{\infty}W^{2,q_{-}}}\|\eta\|_{L^{2}W^{3-\frac{1}{q_{-}},q_{-}}}\\
					&+T(1+\|\eta^{2}\|_{L^{\infty}W^{3-\frac{1}{q_{+}},q_+}})\|\p_{t}\eta\|_{L^{2}H^{\frac{3}{2}-\alpha}}\|u^{2}\|_{L^{\infty}W^{2,q_{-}}}\|\eta^{2}\|_{L^{\infty}W^{3-\frac{1}{q_{-}},q_{-}}}\\
					\lesssim&T\delta^{\frac{1}{2}}d(u,p,\eta).
				\end{aligned}
			\end{align}
			\noindent And then for other terms included in $\tilde R^{1,1}$, similar to estimates \eqref{eq:int_e_1} and \eqref{eq:int_e_2}, we have the following estimates
			\begin{align}
				\begin{aligned}
					&\left\|\int_{0}^{t}\operatorname{div}_{\p_{t}\mathcal{A}^{1}}S_{\mathcal{A}^{1}}(p,u) \right\|_{L^{2}L^{q_{-}}}\lesssim  \delta^{\frac{1}{2}}Td(u,p,\eta),
				\end{aligned}
			\end{align}
			\noindent and
			\begin{align}
				\begin{aligned}
					&\left\|\int_{0}^{t}\operatorname{div}_{\mathcal{A}^{1}}S_{\p_{t}\mathcal{A}^{1}}(p,u) \right\|_{L^{2}L^{q_{-}}}\lesssim  \delta^{\frac{1}{2}}T^{\frac{1}{2}}d(u,p,\eta).
				\end{aligned}
			\end{align}
			Then for the remaining terms, using temporal integration by part we have
			\begin{align}
				\begin{aligned}
					&\left\|\int_{0}^{t}\operatorname{div}_{\p_{t}\mathcal{A}^{1}}S_{\mathcal{A}^{1}-\mathcal{A}^{2}}(p^{2},u^{2}) \right\|_{L^{2}L^{q_{-}}}
					\lesssim T^{\frac{1}{2}}\|\p_{t}\eta^{1}\|_{L^{2}W^{3-\frac{1}{q_{-}},q_{-}}}\|\eta\|_{L^{\infty}W^{3-\frac{1}{q_{-}},q_{-}}}(\|u^{2}\|_{L^{2}W^{2,q_{-}}}+\|p^{2}\|_{L^{2}W^{1,q_{-}}})\lesssim T^{\frac{1}{2}}\delta^{\frac{1}{2}}d(u,p,\eta),\\
					&\left\|\int_{0}^{t}\operatorname{div}_{\p_{t}\mathcal{A}^{1}-\p_{t}\mathcal{A}^{2}}S_{\mathcal{A}^{2}}(p^{2},u^{2}) \right\|_{L^{2}L^{q_{-}}}\lesssim T\|\p_{t}\eta\|_{L^{2}H^{\frac{3}{2}-\alpha}}\|\eta^{2}\|_{L^{\infty}W^{3-\frac{1}{q_{+}},q_{+}}}(\|u^{2}\|_{L^{\infty}W^{2,q_{+}}}+\|p^{2}\|_{L^{\infty}W^{1,q_{+}}})\lesssim T\delta^{\frac{1}{2}}d(u,p,\eta),\\
					& \left\|\int_{0}^{t}\operatorname{div}_{\mathcal{A}^{1}}\mathbb{D}_{\p_{t}\mathcal{A}^{1}-\p_{t}\mathcal{A}^{2}}u^{2} \right\|_{L^{2}L^{q_{-}}}\lesssim \left\|\int_{0}^{t}\operatorname{div}_{\p_{t}\mathcal{A}^{1}}\mathbb{D}_{\mathcal{A}^{1}-\mathcal{A}^{2}}u^{2} \right\|_{L^{2}L^{q_{-}}}+ \left\|\int_{0}^{t}\operatorname{div}_{\mathcal{A}^{1}}\mathbb{D}_{\mathcal{A}^{1}-\mathcal{A}^{2}}\p_{t}u^{2} \right\|_{L^{2}L^{q_{-}}}\\
					&\lesssim T^{\frac{1}{2}}\|\p_{t}\eta^{1}\|_{L^{2}W^{3-\frac{1}{q_{-}},q_{-}}}(\|u^{2}\|_{L^{\infty}W^{2,q_{-}}})\|\eta\|_{L_{t}^{2}W^{3-\frac{1}{q_{-}},q_{-}}}+T^{\frac{1}{2}}(\|u^{2}\|_{L^{\infty}W^{2,q_{-}}})\|\eta\|_{L_{t}^{2}W^{3-\frac{1}{q_{-}},q_{-}}}\lesssim \delta^{\frac{1}{2}}T^{\frac{1}{2}}d(u,p,\eta),\\
					& \left \|\int_{0}^{t}\operatorname{div}_{\mathcal{A}^{1}-\mathcal{A}^{2}}\mathbb{D}_{\p_{t}\mathcal{A}^{2}}u^{2} \right\|_{L^{2}L^{q_{-}}}\lesssim T^{\frac{1}{2}}\|\eta\|_{L^{2}W^{3-\frac{1}{q_{-}},q_{-}}}\|\p_{t}\eta^{2}\|_{L^{2}W^{3-\frac{1}{q_{-}},q_{-}}}(\|u^{2}\|_{L^{\infty}W^{2,q_{-}}})\lesssim \delta^{\frac{1}{2}}T^{\frac{1}{2}}d(u,p,\eta),
				\end{aligned}
			\end{align}
			\noindent which finishes the estimate of the momentum equations.
			
			\subparagraph{\underline{Divergence equation}}
			We can rewrite the divergence equation as follows
			\begin{align}{\label{eq:elliptic_2}}
				\dive_{\mathcal{A}^{1}}\tilde{u}=-\int_{0}^{t}\operatorname{div}_{\p_{t}\mathcal{A}^{1}}\tilde{u}+\int_{0}^{t}\dive_{\mathcal{A}^{1}-\mathcal{A}^{2}}\p_{t}\tilde{u}^{2}+\tilde{R}^{2,1}.
			\end{align}
			
			\noindent We then estimate the elliptic norms of each term on the right hand side of \eqref{eq:elliptic_2}. We have:
			\begin{align}
				\|\int_{0}^{t}\operatorname{div}_{\p_{t}\mathcal{A}^{1}}\tilde{u}\|_{L^{2}W^{1,q_{-}}}\lesssim T\|\p_{t}\eta^{1}\|_{L^{\infty}W^{3-\frac{1}{q_{-}},q_{-}}}\|\p_{t}\tilde{u}\|_{L^{2}W^{2,q_{-}}}\lesssim \delta^{\frac{1}{2}}Td(\tilde{u},\tilde{p},\tilde{\eta}),
			\end{align}
			and
			\begin{align}
				\|\int_{0}^{t}\operatorname{div}_{\mathcal{A}^{1}-\mathcal{A}^{2}}\p_{t}\tilde{u}^{2}\|_{L^{2}W^{1,q_{-}}}\lesssim T\|\eta\|_{L^{\infty}W^{3-\frac{1}{q_{-}},q_{-}}}\|\p_{t}\tilde{u}^{2}\|_{L^{2}W^{2,q_{-}}}\lesssim \delta^{\frac{1}{2}}Td(\tilde{u},\tilde{p},\tilde{\eta}).
			\end{align}
			\noindent Then for the terms included in $\tilde{R}^{2,1}$, using integration by part, we have
			\begin{align}
				\begin{aligned}
					\|\int_{0}^{t}\operatorname{div}_{\p_{t}\mathcal{A}^{1}-\p_{t}\mathcal{A}^{2}}{u}^{2}\|_{L^{2}W^{1,q_{-}}}\lesssim&T^{\frac{1}{2}}\|\eta\|_{L^{2}W^{3-\frac{1}{q_{-}},q_{-}}}\|\p_{t}{u}^{2}\|_{L^{2}W^{2,q_{-}}}\lesssim \delta^{\frac{1}{2}}T^{\frac{1}{2}}d({u},{p},{\eta}),\\
					\|\int_{0}^{t}\operatorname{div}_{\p_{t}\mathcal{A}^{1}}{u}\|_{L^{2}W^{1,q_{-}}}\lesssim& T^{\frac{1}{2}}\|\p_{t}\eta^{1}\|_{L^{2}W^{3-\frac{1}{q_{-}},q_{-}}}\|{u}\|_{L^{2}W^{2,q_{-}}}\lesssim \delta^{\frac{1}{2}}T^{\frac{1}{2}}d({u},{p},{\eta}),
				\end{aligned}
			\end{align}
			\noindent which finishes the estimate for divergence equation.
			
			\subparagraph{\underline{Fixed boundary condition}}
			
			The equation on the fixed boundary can be rewritten as:
			\begin{align}{\label{eq:elliptic_3}}
				S_{\mathcal{A}^{1}}(\tilde{u},\tilde{p})\nu \cdot \tau=\beta \tilde{u}\cdot \tau-\int_{0}^{t}\mathbb{D}_{\p_{t}\mathcal{A}^{1}}(\tilde{u})+\tilde{R}^{5,1}.
			\end{align}
			\noindent Then we estimate each term on the right hand side of the equation\eqref{eq:elliptic_3}. We have
			\begin{align}
				\|\tilde{u}\cdot \tau\|_{L^{2}W^{1-\frac{1}{q_{-}},q_{-}}(\Sigma)}\lesssim T^{\frac{1}{2}}\|\tilde{u}\|_{L^{\infty}H^{1}},
			\end{align}
			\noindent and
			\begin{align}
				\|\int_{0}^{t}\mathbb{D}_{\p_{t}\mathcal{A}^{1}}(\tilde{u})\|_{L^{2}W^{1-\frac{1}{q_{-}},q_{-}}}\lesssim T^{\frac{1}{2}}\|\p_{t}\eta^{1}\|_{L^{2}W^{3-\frac{1}{q_{-}},q_{-}}}\|\tilde{u}\|_{L^{2}W^{2,q_{-}}}\lesssim\delta^{\frac{1}{2}}T^{\frac{1}{2}}d(\tilde{u},\tilde{p},\tilde{\eta}),
			\end{align}
			\noindent where we used the following relation:
			\begin{align}
				\|fg\|_{W^{1-\frac{1}{q_{-}},q_{-}}}\lesssim \|f\|_{W^{1,q_{-}}}\|g\|_{W^{1-\frac{1}{q_{-}},q_{-}}}.
			\end{align}
			
			\noindent Then we show the estimate for terms included $\tilde{R}^{5,1}$. It holds that
			\begin{align}
				\begin{aligned}
					\left\|\int_{0}^{t}\mathbb{D}_{\p_{t}\mathcal{A}^{1}}({u}) \right\|_{L^{2}W^{1-\frac{1}{q_{-}},q_{-}}}\lesssim&T^{\frac{1}{2}}\|\p_{t}\eta^{1}\|_{L^{2}W^{3-\frac{1}{q_{-}},q_{-}}}\|{u}\|_{L^{2}W^{2,q_{-}}}\lesssim T^{\frac{1}{2}}\delta^{\frac{1}{2}}d({u},{p},{\eta}),\\
					\left\|\int_{0}^{t}\mathbb{D}_{\p_{t}\mathcal{A}^{1}-\p_{t}\mathcal{A}^{2}}({u}^{2}) \right\|_{L^{2}W^{1-\frac{1}{q_{-}},q_{-}}}\lesssim&T^{\frac{1}{2}}\|\eta\|_{L^{2}W^{3-\frac{1}{q_{-}},q_{-}}}\|\p_{t}{u}^{2}\|_{L^{2}W^{2,q_{-}}}\lesssim T^{\frac{1}{2}}\delta^{\frac{1}{2}}d({u},{p},{\eta}),
				\end{aligned}
			\end{align}
			\noindent which finishes the estimate for the fixed boundary equation.
			
			\subparagraph{\underline{Kinematic boundary condition}}
			
			The kinematic boundary condition can be rewritten as follows
			\begin{align}{\label{eq:elliptic_4}}
				\tilde{u}\cdot \mathcal{N}^{1}=-\p_{t}\tilde{\eta}+\int_{0}^{t}\tilde{u}\cdot \p_{t}\mathcal{N}^{1}+\tilde{R}^{6,1}+\int_{0}^{t}u^{1}\p_{1}\p_{t}\tilde{\eta}+\int_{0}^{t}u\p_{1}\p_{t}\tilde{\eta}^{2}.
			\end{align}
			We now estimate each term on the right hand side of \eqref{eq:elliptic_4}, we have the following computation:
			\begin{align}
				\|\p_{t}\tilde{\eta}\|_{L^{2}W^{2-\frac{1}{q_{-}},q_{-}}}\lesssim \|\p_{t}\tilde{\eta}\|_{L^{2}H^{\frac{3}{2}-\alpha}},
			\end{align}
			\noindent and
			\begin{align}
				\left\|\int_{0}^{t}\tilde{u}\cdot \p_{t}\mathcal{N}^{1} \right\|_{L^{2}W^{2-\frac{1}{q_{-}},q_{-}}}\lesssim T^{\frac{1}{2}}\|\p_{t}\tilde{u}\|_{L^{2}W^{2,q_{-}}}\|\p_{t}\eta^{1}\|_{L^{2}W^{3-\frac{1}{q_{-}},q_{-}}}\lesssim \delta^{\frac{1}{2}}T^{\frac{1}{2}}d(\tilde{u},\tilde{p},\tilde{\eta}),
			\end{align}
			\noindent where we used the trace theorem and the fact that $W^{2-\frac{1}{q_{-}},q_{-}}$ is an algebra. 
			
			Similarly, using temporal integration by part, we have 
			\begin{align}
				\left\|\int_{0}^{t}u^{1}\p_{1}\p_{t}\tilde{\eta} \right\|_{L^{2}W^{2-\frac{1}{q_{-}},q_{-}}}\lesssim& \|u^{1}\p_{1}\tilde{\eta}\|_{L^{2}W^{2-\frac{1}{q_{-}},q_{-}}}+ \left\|\int_{0}^{t}\p_{t}u^{1}\p_{1}\tilde{\eta} \right\|_{L^{2}W^{2-\frac{1}{q_{-}},q_{-}}}+\|\int_{0}^{t}\p_{t}u^{1}\p_{1}\tilde{\eta}\|_{L^{2}W^{2-\frac{1}{q_{-}},q_{-}}}\notag\\
				\lesssim&(\|u^{1}\|_{L^{\infty}W^{2,q_{-}}}+T^{\frac{1}{2}}\|\p_{t}u^{1}\|_{L^{2}W^{2,q_{-}}})\|\tilde{\eta}\|_{L^{2}W^{3-\frac{1}{q_{-}},q_{-}}}\lesssim(1+T^{\frac{1}{2}})\delta^{\frac{1}{2}}d(\tilde{u},\tilde{p},\tilde{\eta})\label{eq:kine1},
			\end{align}
			\noindent and
			\begin{align}
				\begin{aligned}
					\left \|\int_{0}^{t}u\p_{1}\p_{t}\tilde{\eta}^{2} \right\|_{L^{2}W^{2-\frac{1}{q_{-}},q_{-}}}&\lesssim (\|u\|_{L^{2}W^{2,q_{-}}}+\|u\|_{L^{2}W^{2,q_{-}}})(\|\tilde{\eta}^{2}\|_{L^{\infty}W^{3-\frac{1}{q_{-}},q_{-}}}+T^{\frac{1}{2}}\|\p_{t}\tilde{\eta}^{2}\|_{L^{2}W^{3-\frac{1}{q_{-}},q_{-}}})\\
					&\lesssim(1+T^{\frac{1}{2}})\delta^{\frac{1}{2}}d({u},{p},{\eta})\label{eq:kine2}.
				\end{aligned}
			\end{align}
			The estimate for $\tilde{R}^{6,1}$ is similar to \eqref{eq:kine1} and \eqref{eq:kine2}, we omit the details here and only write down the result
			\begin{align}
				\|\tilde{R}^{6,1}\|_{L^{2}W^{2-\frac{1}{q_{-}},q_{-}}}\lesssim \delta^{\frac{1}{2}}(1+T^{\frac{1}{2}})d(u,p,\eta).
			\end{align}
			
			\subparagraph{\underline{Contact point}}
			
			The contact point condition can be rewritten as:
			\begin{align}{\label{eq:elliptic_5}}
				\mp\sigma\frac{\p_1\tilde{\eta}}{(1+|\p_1\zeta_0|^2)^{3/2}}(\pm\ell,t)=&-\kappa(\p_{t}\tilde{\eta})(\pm\ell,t)\mp\sigma\int_{0}^{t}\p_{1}(\mathcal{R}_{z}(\p_{1}\zeta_{0},\p_{1}\eta^{1})\p_{1}\p_{t}\tilde{\eta}\\ &\mp\sigma\int_{0}^{t}\p_{1}((\mathcal{R}_{z}(\p_{1}\zeta_{0},\p_{1}\eta^{2})-\mathcal{R}_{z}(\p_{1}\zeta_{0},\p_{1}\eta^{1}))\p_{1}\p_{t}\tilde{\eta}^{2})-\tilde R^{7,1}.\no
			\end{align}
			We now estimate terms on the right hand side of \eqref{eq:elliptic_5}. Using temporal integration by part, it holds that
			\begin{align}
				\sigma\int_{0}^{t}(\mathcal{R}_{z}(\p_{1}\zeta_{0},\p_{1}\eta^{1})\p_{1}\p_{t}\tilde{\eta}\lesssim T^{\frac{1}{2}}(\|\p_{t}\eta^{1}\|_{L^{2}W^{3-\frac{1}{q_{-}},q_{-}}}+\|\eta^{1}\|_{L^{\infty}W^{3-\frac{1}{q_{-}},q_{-}}})\|\tilde{\eta}\|_{L^{2}W^{3-\frac{1}{q_{-}},q_{-}}}\lesssim T^{\frac{1}{2}}\delta^{\frac{1}{2}}d(\tilde{u},\tilde{p},\tilde{\eta}),
			\end{align}
			\noindent and
			\begin{align}
				\int_{0}^{t}\p_{1}((\mathcal{R}_{z}(\p_{1}\zeta_{0},\p_{1}\eta^{2})-\mathcal{R}_{z}(\p_{1}\zeta_{0},\p_{1}\eta^{1}))\p_{1}\p_{t}\tilde{\eta}^{2})\lesssim &T^{\frac{1}{2}}(\|\eta\|_{L^{2}W^{3-\frac{1}{q_{-}},q_{-}}})\|\p_{t}\tilde{\eta}^{2}\|_{L^{2}W^{3-\frac{1}{q_{-}},q_{-}}}
				\lesssim T^{\frac{1}{2}}\delta^{\frac{1}{2}}d({u},{p},{\eta}).
			\end{align}
			\noindent For the integrated term in $\tilde{R}^{7,1}$, we use an estimate similar  to the one above. We omit the details. Then it remains to show the estimate for $\p_{t}R^{7}$. We have the following:
			
			\begin{align}
				\p_{t}(\hat{\mathcal{W}}(\p_{t}\eta^{1})-\hat{\mathcal{W}}(\p_{t}\eta^{2}))\lesssim \delta^{\frac{1}{2}}[\p_{t}\eta]_{\ell},
			\end{align}
			\noindent which finishes the estimate for the contact point condition.
			
			\subparagraph{\underline{Capillary equation}}
			
			Finally, we show the estimate for the capillary equation. We rewrite the equation as
			\begin{align}{\label{eq:elliptic_6}}
				\begin{aligned}
					S_{\mathcal{A}^1}(\tilde{p},\tilde{u})\mathcal{N}^1=&\int_{0}^{t}S_{\p_{t}\mathcal{A}^{1}}(\tilde{u},\tilde{p})\mathcal{N}^{1}+\int_{0}^{t}S_{\mathcal{A}^{1}}(\tilde{p},\tilde{u})\p_{t}\mathcal{N}^{1}-g(\tilde{\eta})\mathcal{N}^1-\sigma\p_1\left(\frac{\p_1\tilde{\eta}}{(1+|\p_1\zeta_0|)^{3/2}}\right)\mathcal{N}^1\\&-\int_{0}^{t}g(\tilde{\eta})\p_{t}\mathcal{N}^1-\int_{0}^{t}\sigma\p_1\left(\frac{\p_1\tilde{\eta}}{(1+|\p_1\zeta_0|)^{3/2}}\right)\p_{t}\mathcal{N}^1-\sigma\int_{0}^{t}\p_{1}(\mathcal{R}_{z}(\p_{1}\zeta_{0},\p_{1}\eta^{1})\p_{1}\p_{t}\tilde{\eta})\mathcal{N}^{1}(t)\\&-\sigma \int_{0}^{t}\p_{1}((\mathcal{R}_{z}(\p_{1}\zeta_{0},\p_{1}\eta^{2})-\mathcal{R}_{z}(\p_{1}\zeta_{0},\p_{1}\eta^{1}))\p_{1}\p_{t}\tilde{\eta}^{2})\mathcal{N}^{1}(t)+\tilde{R}^{4,1}.
				\end{aligned}
			\end{align}
			For the right hand side of \eqref{eq:elliptic_6}, all of the terms including the integrated terms can be estimated similarly to those of the other equations except for the mean curvature-type terms. We have 
			\begin{align}
				\begin{aligned}
					\left\|\int_{0}^{t}S_{\p_{t}\mathcal{A}^{1}}(\tilde{p},\tilde{u})\mathcal{N}^{1}\right\|_{L_{t}^{2}W^{1-\frac{1}{q_{-}},q_{-}}}\lesssim T^{\frac{1}{2}}\|\p_{t}\eta^{1}\|_{L_{t}^{2}W^{3-\frac{1}{q_{-}},q_{-}}}(\|\tilde{p}\|_{L_{t}^{2}W^{1,q_{-}}}+\|\tilde{u}\|_{L_{t}^{2}W^{2,q_{-}}})\lesssim T^{\frac{1}{2}}\delta^{\frac{1}{2}}d(\tilde{u},\tilde{\eta},\tilde{p}),\\
					\left\|\int_{0}^{t}S_{\mathcal{A}^{1}}(\tilde{p},\tilde{u})\p_{t}\mathcal{N}^{1} \right\|_{L_{t}^{2}W^{1-\frac{1}{q_{-}},q_{-}}}\lesssim T^{\frac{1}{2}}\|\p_{t}\eta^{1}\|_{L_{t}^{2}W^{3-\frac{1}{q_{-}},q_{-}}}(\|\tilde{p}\|_{L_{t}^{2}W^{1,q_{-}}}+\|\tilde{u}\|_{L_{t}^{2}W^{2,q_{-}}})\lesssim T^{\frac{1}{2}}\delta^{\frac{1}{2}}d(\tilde{u},\tilde{\eta},\tilde{p}),
				\end{aligned}
			\end{align}
			and
			\begin{align}
				\left\|\int_{0}^{t}(g\tilde{\eta}-\p_{1}(\frac{\p_{1}\tilde{\eta}}{(1+|\p_{1}\zeta_{0}|^{2})^{\frac{3}{2}}}))\p_{t}\mathcal{N}^{1} \right\|_{L^{2}W^{1-\frac{1}{q_{-}},q_{-}}}\lesssim T\|\p_{t}\eta^{1}\|_{L^{\infty}W^{2,q_{-}}}\|\tilde{\eta}\|_{L^{2}W^{3-\frac{1}{q_{-}},q_{-}}}\lesssim T\delta^{\frac{1}{2}}d(\tilde{u},\tilde{p},\tilde{\eta}).
			\end{align}
			\noindent  Using temporal integration by part, we have
			\begin{align}
				&\left \|\sigma\int_{0}^{t}\p_{1}(\mathcal{R}_{z}(\p_{1}\zeta_{0},\p_{1}\eta^{1})\p_{1}\p_{t}\tilde{\eta})\mathcal{N}^{1}(t) \right\|_{L^{2}W^{1-\frac{1}{q_{-}},q_{-}}}\\
				&\lesssim \|\p_{1}(\mathcal{R}_{z}(\p_{1}\zeta_{0},\p_{1}\eta^{1})\p_{1}\tilde{\eta})\mathcal{N}^{1}(t)\|_{L^{2}W^{1-\frac{1}{q_{-}},q_{-}}}+ \left\|\int_{0}^{t}\p_{1}(\mathcal{R}_{zz}(\p_{1}\zeta_{0},\p_{1}\eta^{1})\p_{1}\p_{t}\eta^{1}\p_{1}\tilde{\eta})\mathcal{N}^{1}(t) \right\|_{L^{2}W^{1-\frac{1}{q_{-}},q_{-}}}\notag\\
				&\quad+ \left\|\int_{0}^{t}\p_{1}(\mathcal{R}_{z}(\p_{1}\zeta_{0},\p_{1}\eta^{1})\p_{1}\tilde{\eta})\p_{t}\mathcal{N}^{1}(t) \right\|_{L^{2}W^{1-\frac{1}{q_{-}},q_{-}}}\notag\\
				&\lesssim T^{\frac{1}{2}}\left(\|\eta^{1}\|_{L^{\infty}W^{3-\frac{1}{q_{-}},q_{-}}}+\|\p_{t}\eta^{1}\|_{L^{2}W^{3-\frac{1}{q_{-}},q_{-}}} \right)\|\tilde{\eta}\|_{L^{2}W^{3-\frac{1}{q_{-}},q_{-}}}\notag\lesssim T^{\frac{1}{2}}\delta^{\frac{1}{2}}d(\tilde{u},\tilde{p},\tilde{\eta}),\no
			\end{align}
			\noindent and similarly
			\begin{align}\label{eq:elliptic_7}
				&\left \|\int_{0}^{t}\p_{1}((\mathcal{R}_{z}(\p_{1}\zeta_{0},\p_{1}\eta^{2})-\mathcal{R}_{z}(\p_{1}\zeta_{0},\p_{1}\eta^{1}))\p_{1}\p_{t}\tilde{\eta}^{2})\mathcal{N}^{1}(t) \right\|_{L^{2}W^{1-\frac{1}{q_{-}},q_{-}}}\notag\\
				&\lesssim T^{\frac{1}{2}}\|\p_{t}\tilde{\eta}^{2}\|_{L^{2}W^{3-\frac{1}{q_{-}},q_{-}}} \|\eta\|_{L^{2}W^{3-\frac{1}{q_{-}},q_{-}}} \lesssim T^{\frac{1}{2}}\delta^{\frac{1}{2}}d(u,\eta,p).
			\end{align}
			
			For terms in $\tilde{R}^{4,1}$, they can be estimated similarly to the terms above, but we omit the details here. The idea is that each time a temporal derivative hits a difference term, we use integration by parts to shift the time derivative to the non-difference term.
			
			Using elliptic theory for system \eqref{linear_fix1} and estimates \eqref{eq:elliptic_1} to \eqref{eq:elliptic_7}, we deduce that
			\begin{equation}\label{est:elliptic_8}
				\begin{aligned}
					&\|\tilde{u}\|_{L^2W^{2,q_-}}^2+\|\tilde{p}\|_{L^2W^{1,q_-}}^2+\|\tilde{\eta}\|_{L^2W^{3-1/q_-,q_-}}^2\lesssim (\delta+T+T^{\frac{1}{2}})(d(u,\eta,p)+d(\tilde{u},\tilde{\eta},\tilde{p}))^2
					+\|\p_t\tilde{\eta}\|_{L^2H^{3/2-\alpha}}^2.
				\end{aligned}
			\end{equation}
			
			\paragraph{\underline{Step 7 -- Fixed Point}}
			Estimates \eqref{est:contract_1}, \eqref{est:contract_3}, and \eqref{est:elliptic_8} yield the bound:
			\begin{equation}\label{est:fixed_point1}
				\begin{aligned}
					d^{2}(\tilde{u},\tilde{p},\tilde{\eta})\leq {C}\exp(\delta^{\frac{1}{2}})(\delta^{\frac{1}{2}}+T)(d(u,p,\eta)+d(\tilde{u},\tilde{p},\tilde{\eta}))^{2},
				\end{aligned}
			\end{equation}
			where $C$ is a universal constant that may change from line to line, and we have used the fact that $Ce^{\delta^{\frac{1}{2}}}(\delta+T)<\frac12$. We then restrict $\delta$ and $T$ to be sufficiently small so that all terms with a tilde superscript on the right-hand side of \eqref{est:fixed_point1} can be absorbed into the left-hand side of inequality \eqref{est:fixed_point1}.  Moreover, we may further restrict $\delta$ to deduce the following inequality:
			\[
			d^{2}(\tilde{u},\tilde{p},\tilde{\eta})\leq \frac{1}{2}d^{2}(u,p,\eta).
			\]
			Consequently, the map $A$ defined in \eqref{def:map_a} is a contraction. Thus, Banach's fixed point theorem implies that the nonlinear system \eqref{eq:quasi_linear} with $(\p_t^ju, \p_t^jp, \p_t^j\eta)$, has a solution $(u,p,\eta)\in S(T,\delta)$ for a small $T$ determined above.
			
			The uniqueness for \eqref{eq:quasi_linear} follows directly from the uniqueness of the solution to the corresponding linear system.
		\end{proof}

		%%%%%%%%%%%%%%%%%%%%%%%%%%%%%%%%%%%%%%%%%%%%%%
		\section{Uniform Bounds and Global Well-Posedness of the Nonlinear System}
		%%%%%%%%%%%%%%%%%%%%%%%%%%%%%%%%%%%%%%%%%%%%%%

		In this section, we study the global well-posedness of the nonlinear system \eqref{eq:geometric} and complete the proof of the main theorem. In light of the local well-posedness results in Theorem \ref{thm: fixed point}, it suffices to derive the global-in-time a priori estimate. 
		
		We denote $\mathcal{A}$, $\mathcal{N}$, $J$ and $K$ as defined in terms of $\eta$. Note that $(\p_t^2u, \p_t^2\eta)$ and $(\p_t^ju, \p_t^jp, \p_t^j\eta)$, $j=0,1$ satisfy the equations
		\begin{equation}\label{eq:geometric1}
			\left\{
			\begin{aligned}
				&\p_t^{j+1}u+\dive_{\mathcal{A}}S_{\mathcal{A}}(\p_t^jp,\p_t^ju)=F^{1,j}\quad&\text{in}&\ \Om,\\
				&\dive_{\mathcal{A}}\p_t^ju=F^{2,j}\quad&\text{in}&\ \Om,\\
				&S_{\mathcal{A}}(\p_t^jp,\p_t^ju)\mathcal{N}=\left[g\p_t^j\eta-\sigma\p_1\left(\frac{\p_1\p_t^j\eta}{(1+|\p_1\zeta_0|^2)^{3/2}}+F^{3,j}\right)\right]\mathcal{N}+F^{4,j}\quad&\text{on}&\ \Sigma,\\
				&\p_t^{j+1}\eta-\p_t^ju\cdot\mathcal{N}=F^{6,j}\quad&\text{on}&\ \Sigma,\\
				&(S_{\mathcal{A}}(\p_t^jp,\p_t^ju)\cdot\nu-\beta\p_t^ju)\cdot\tau=F^{5,j}\quad&\text{on}&\ \Sigma_s,\\
				&\p_t^ju\cdot\nu=0\quad&\text{on}&\ \Sigma_s,\\
				&\kappa\p_t^{j+1}\eta=\mp\sigma\left(\frac{\p_1\p_t^j\eta}{(1+|\p_1\zeta_0|^2)^{3/2}}+F^{3,j}\right)-\kappa F^{7,j}\quad&\text{on}&\ \pm\ell,
			\end{aligned}
			\right.
		\end{equation}
		where the forcing terms are given in the Appendix \ref{sec:dive_forcing}.
		We follow the argument of \cite[Section 10.2]{GT2020} to derive the uniform boundedness of solutions to \eqref{eq:geometric1}.
		
		\begin{proof}[Proof of Theorem \ref{thm:main}.]
			We give only a sketch of the proof.
			
			\paragraph{\underline{Step 1 -- Energy-Dissipation Estimates for \eqref{eq:geometric1}}}
			Using standard arguments, we obtain
			the energy-dissipation equality for $(\p_{t}^{j}u, \p_{t}^{j}\eta,\p_{t}^{j}p)$ with $j=0,1$ in \eqref{eq:geometric1}, and $(\p_{t}^{2}u,\p_{t}^{2}\eta)$
			\begin{equation}\label{eq:e_d1}
				\begin{aligned}
					&\frac{d}{dt} \left(\frac12\int_\Om|u|^2J+\int_{-\ell}^\ell g|\eta|^2+\sigma\frac{|\p_1\eta|^2}{(1+|\p_1\zeta_0|^2)^{3/2}}\right)+\frac{\mu}2\int_\Om|\mathbb{D}_{\mathcal{A}}u|^2J+\beta\int_{\Sigma_s}|u\cdot\tau|^2J+[\p_t\eta]_\ell^2\\
					&=-\sigma (F^{3,0}, \p_1(\p_t\eta))_{L^2(-\ell,\ell)}-[\p_t\eta, F^{7,0}]_\ell,
				\end{aligned}
			\end{equation}
			\begin{equation}\label{eq:e_d2}
				\begin{aligned}
					&\frac{d}{dt}\left(\frac12\int_\Om|\p_tu|^2J+\int_{-\ell}^\ell g|\p_t\eta|^2+\sigma\frac{|\p_1\p_t\eta|^2}{(1+|\p_1\zeta_0|^2)^{3/2}}\right)+\frac{\mu}2\int_\Om|\mathbb{D}_{\mathcal{A}}\p_tu|^2J+\beta\int_{\Sigma_s}|\p_tu\cdot\tau|^2 J+[\p_t^2\eta]_\ell^2\\
					&=\mathcal{F}^{1}(\p_tu, \p_tp, \p_t\eta)+(\p_{t}p,F^{2,1})_{\mathcal{H}^{0}(\Omega)},
				\end{aligned}
			\end{equation}
			\begin{equation}\label{eq:e_d3}
				\begin{aligned}
					&\frac{d}{dt}\left(\frac12\int_\Om|\p_t^{2}u|^2J+\int_{-\ell}^\ell g|\p_t^{2}\eta|^2+\sigma\frac{|\p_1\p_t^{2}\eta|^2}{(1+|\p_1\zeta_0|^2)^{3/2}}-\int_{\Omega}\p_{t}^{2}u \om J\right)+\frac{\mu}2\int_\Om|\mathbb{D}_{\mathcal{A}}\p_t^{2}u|^2J+\beta\int_{\Sigma_s}|\p_t^{2}u\cdot\tau|^2 J+[\p_t^2\eta]_\ell^2\\
					&=\mathcal{F}^{2}(\p_t^{2}u, \p_t^{2}p, \p_t^{2}\eta)+\int_{\Omega}F^{1,2}\om J+\int_{\Omega}\langle JF^{2,2}\rangle \p_{t}^{2}p+\frac{\mu}{2}\int_{\Omega}\mathbb{D}_{\mathcal{A}}\p_{t}^{2}u:\mathbb{D}_{\mathcal{A}}\om J,
				\end{aligned}
			\end{equation}
			where
			\[
			\begin{aligned}
				\mathcal{F}^{1}(\p_tu, \p_tp, \p_t\eta)
				&=\int_\Om (F^{1,1}\cdot\p_tu) J-\int_{\Sigma_s}F^{5,1}(\p_tu\cdot\tau)J+[\p_t^2\eta, F^{6,1}-F^{7,1}]_\ell\\
				&\quad-\int_{-\ell}^\ell\sigma F^{3,2}\p_1(\p_tu\cdot\mathcal{N})+F^{4,1}\cdot\p_tu-\int_{-\ell}^\ell g\p_t\eta F^{6,1}+\sigma\frac{\p_1\p_t\eta}{(1+|\p_1\zeta_0|^2)^{3/2}}\p_1F^{6,1},
			\end{aligned}
			\]
			\[
			\begin{aligned}
				\mathcal{F}^{2}(\p_{t}^{2}u,\p_{t}^{2}p,\p_{t}^{2}\eta)&=\int_\Om (F^{1,2}\cdot\p_t^{2}u) J-\int_{\Sigma_s}F^{5,2}(\p_t^{2}u\cdot\tau)J+[\p_t^3\eta, F^{6,2}-F^{7,2}]_\ell\\
				&\quad-\int_{-\ell}^\ell\sigma F^{3,2}\p_1(\p_{t}^{2}u\cdot\mathcal{N})+F^{4,2}\cdot\p_t^{2}u-\int_{-\ell}^\ell g\p_t^{2}\eta F^{6,2}+\sigma\frac{\p_1\p_t^{2}\eta}{(1+|\p_1\zeta_0|^2)^{3/2}}\p_1F^{6,2}.
			\end{aligned}
			\]
			Here the construction of $\om$ is in \cite[Proposition 10.1]{GT2020}.
			
			Using estimate for the forcing terms in \cite{GT2020} and Section 5, the nonlinear interaction bounds for \eqref{eq:e_d1}, \eqref{eq:e_d2} and \eqref{eq:e_d3} yield the energy-dissipation estimate for $(u, \eta)$:
			\begin{equation}\label{est:e_d1}
				\begin{aligned}
					&\|u(t)\|_{H^0}^2+\|\eta(t)\|_{H^1}^2+\int_s^t(\|u\|_{H^1}^2+\|u\|_{L^2(\Sigma_s)}^2+[\p_t\eta]_\ell^2)\lesssim \|u(s)\|_{H^0}^2+\|\eta(s)\|_{H^1}^2+\int_s^t\sqrt{\mathcal{E}}\mathcal{D},
				\end{aligned}
			\end{equation}
			and the energy-dissipation estimate for $(\p_tu, \p_t\eta)$:
			\begin{equation}\label{est:e_d2}
				\begin{aligned}
					&\|\p_tu(t)\|_{H^0}^2+\|\p_t\eta(t)\|_{H^1}^2+\int_s^t(\|\p_tu\|_{H^1}^2+\|\p_tu\|_{L^2(\Sigma_s)}^2+[\p_t^2\eta]_\ell^2)\lesssim \|\p_tu(s)\|_{H^0}^2+\|\p_t\eta(s)\|_{H^1}^2+\int_s^t\sqrt{\mathcal{E}}\mathcal{D},
				\end{aligned}
			\end{equation}
			as well as the energy-dissipation estimate for $(\p_t^2u, \p_t^2\eta)$:
			\begin{equation}\label{est:e_d3}
				\begin{aligned}
					&\|\p_t^2u(t)\|_{H^0}^2+\|\p_t^2\eta(t)\|_{H^1}^2-(\mathcal{E}(t))^{3/2}+\int_s^t(\|\p_t^2u\|_{H^1}^2+\|\p_t^2u\|_{L^2(\Sigma_s)}^2+[\p_t^3\eta]_\ell^2)\\
					&\lesssim \|\p_t^2u(s)\|_{H^0}^2+\|\p_t^2\eta(s)\|_{H^1}^2+(\mathcal{E}(s))^{3/2}+\int_s^t\sqrt{\mathcal{E}}\mathcal{D}
				\end{aligned}
			\end{equation}
			for all $0\le s\le t\le T$.  We set
			$
			\mathcal{E}_\shortparallel:=\sum_{j=0}^2\|\p_t^ju\|_{L^2}^2+\|\p_t^j\eta\|_{H^1}^2$,
			and
			$
			\mathcal{D}_\shortparallel:=\sum_{j=0}^3\|\p_t^ju\|_{H^1}^2+\|\p_t^ju\|_{L^2(\Sigma_s)}^2+[\p_t^{j+1}\eta]_\ell^2.$
			Then, summing inequalities \eqref{est:e_d1}--\eqref{est:e_d3} yields the estimate
			\begin{equation}\label{est:energy_dissipation1}
				\mathcal{E}_\shortparallel(t)-(\mathcal{E}(t))^{3/2}+\int_s^t\left(\mathcal{D}_\shortparallel\right)\lesssim\mathcal{E}_\shortparallel(s)+(\mathcal{E}(s))^{3/2}+\int_s^t\sqrt{\mathcal{E}}\mathcal{D}
			\end{equation}
			for all $0\le s\le t\le T$.
			
			\paragraph{\underline{Step 2 -- Enhanced Dissipation on the Surface}}
			
			First, from the kinematic boundary condition, it holds that
			\[
			\int_{-\ell}^{\ell}\p_{t}\eta dx=\int_{-\ell}^{\ell}u\cdot \mathcal{N}dx=\int_{\Omega}J\dive_{\mathcal{A}}u =0.
			\]
			Similarly, we have
			\[
			\int_{-\ell}^{\ell}\p_{t}^{j}\eta dx=\p_{t}^{j-1}\int_{-\ell}^{\ell}\p_{t}\eta=0
			\]
			for $j=0,1,2$. Using these results, we obtain
			\[
			\int_{-\ell}^{\ell}\eta(t,x)dx=\int_{0}^{t}\int_{-\ell}^{\ell}\p_{t}\eta(s,x)dsdx=0.
			\]
			
			Using the zero mean value properties for $\p_{t}^{j}\eta$, we apply
			the functional calculus in \cite[Section 8]{GT2020} for the gravity-capillary operator
			\[
			\mathcal{K}: \p_t^j\eta\mapsto g(\p_t^j\eta)-\sigma\p_1\left(\frac{\p_1(\p_t^j\eta)}{(1+|\p_1\zeta_0|^2)^{3/2}}\right)
			\]
			to construct test functions $\psi_{i}=M\nabla \phi_{i}$, where $\phi_{i}$ solves the equation
			\[
			-\Delta\phi_{i}=0~\operatorname{in}~\Omega,~~~\p_{\nu}\phi_{i}=(D_{j}^{s}\p_{t}^{i}{\eta}\zeta_{0})/|\mathcal{N}_{0}|~\operatorname{on}~\Sigma,~~~\p_{\nu}\phi_{i}=0~\operatorname{on}~\Sigma_{s},
			\]
			for $i=\{0,1,2\}$. Plugging this test function $\phi_{i}$ into the $i$-th order weak formulation and using computations similar to those in \cite{GT2020}, step 6 in the proof of Theorem \ref{thm:linear_low} and step 7 in the proof of Theorem  \ref{thm: fixed point}, we obtain that
			\begin{equation}\label{est:enhance_dissipation1}
				\begin{aligned}
					\int_s^t\|\p_t^j\eta\|_{H^{3/2-\alpha}}^2
					&\lesssim \|\p_t^ju(s)\|_{H^0}^2+\|\p_t^j\eta(s)\|_{H^1}^2+\|\p_t^ju(t)\|_{H^0}^2+\|\p_t^j\eta(t)\|_{H^1}^2\\
					&\quad+\int_s^t\left[(\sum_{k=0}^2\|\p_t^ku\|_{H^1}^2+\|\p_t^ku\|_{L^2(\Sigma_s)}^2+[\p_t^{k+1}\eta]_\ell^2)+\mathcal{E}\mathcal{D}\right],
				\end{aligned}
			\end{equation}
			for all $0\le s\le t\le T$ and each $j\in \{0, 1, 2,3\}$.
			The inequalities \eqref{est:energy_dissipation1} and \eqref{est:enhance_dissipation1} give the following result
			\begin{equation}\label{est:energy_dissipation2}
				\begin{aligned}
					&\mathcal{E}_\shortparallel(t)-(\mathcal{E}(t))^{3/2}+\int_s^t\left(\mathcal{D}_\shortparallel+\sum_{j=0}^2\|\p_t\eta\|_{H^{3/2-\alpha}}^2\right)\\
					&\lesssim\mathcal{E}_\shortparallel(s)+\sum_{j=0}^2(\|\p_t^ju(s)\|_{H^0}^2+\|\p_t^j\eta(s)\|_{H^1}^2)+\sum_{j=0}^{1}\|\p_{t}^{j}p(s)\|^{2}_{L^{2}}+(\mathcal{E}(s))^{3/2}+\int_s^t\sqrt{\mathcal{E}}\mathcal{D}.
				\end{aligned}
			\end{equation}
			
			\paragraph{\underline{Step 3 -- Enhanced Dissipation Estimates}}
			{The elliptic theory for \eqref{eq:geometric1}, with $j=0, 1,2$, yields
				\[
				\begin{aligned}
					\int_s^t(\|u\|_{W^{2,q_+}}^2+\|p-\bar{p}\|_{W^{1,q_+}}^2+\|\eta\|_{W^{3-1/q_+,q_+}}^2)
					\lesssim \int_s^t\|\p_tu\|_{H^1}^2+\|\p_t\eta\|_{H^{3/2-\alpha}}^2+\sqrt{\mathcal{E}}\mathcal{D},
				\end{aligned}
				\]
				\[
				\begin{aligned}
					\int_s^t(\|\p_tu\|_{W^{2,q_-}}^2+\|\p_tp-\overline{\p_{t}{p}}\|_{W^{1,q_-}}^2+\|\p_t\eta\|_{W^{3-1/q_-,q_-}}^2)
					&\lesssim \int_s^t\|\p_t^2u\|_{H^1}^2+\|\p_t^2\eta\|_{H^{3/2-\alpha}}^2+(\sqrt{\mathcal{E}})\mathcal{D},
				\end{aligned}
				\]
				with $\overline{\p_{t}^{j}p}=\frac{\int_{\Omega}\p_{t}^{j}p}{|\Omega|}$for all $0\le s\le t\le T$.
				
				We now consider the dissipation at the contact points.  From the transport equation
				$
				\p_t^{j+1}\eta-\p_t^ju\cdot\mathcal{N}=F^{6,j}$,
				and the vanishing of $F^{6,j}$ at $\pm\ell$, we immediately obtain
				$\sum_{j=0}^2[\p_t^ju\cdot\mathcal{N}]_\ell^2=\sum_{j=0}^2[\p_t^{j+1}\eta]_\ell^2$.
				
				The enhanced surface dissipation result then yields
				\[
				\|\p_t^3\eta\|_{H^{1/2-\alpha}}^2\lesssim \|\p_t^2u\|_{H^1}^2+\mathcal{E}\mathcal{D}.
				\]
				
				Consequently, by the result in \eqref{est:energy_dissipation2}, we derive the following estimate
				\begin{equation}\label{est:energy_dissipation3}
					\begin{aligned}
						&\mathcal{E}_\shortparallel(t)-(\mathcal{E}(t))^{3/2}+\int_s^t\left(\mathcal{D}_\shortparallel+\sum_{j=1}^2\|\p_t\eta\|_{H^{3/2-\alpha}}^2\right)\\
						&\quad+\int_s^t\sum_{j=0}^2[\p_t^{j+1}\eta]_\ell^2+\int_s^t\sum_{j=0}^{1}(\|\p_t^{j}u\|_{W^{2,q_-}}^2+\|\p_t^{j}p-\overline{\p_{t}^{j}p}\|_{W^{1,q_-}}^2+\|\p_t^{j}\eta\|_{W^{3-1/q_-,q_-}}^2)\\
						&\lesssim\mathcal{E}_\shortparallel(s)+\sum_{j=0}^2(\|\p_t^ju(s)\|_{H^0}^2+\|\p_t^j\eta(s)\|_{H^1}^2)+\sum_{j=0}^{1}\|\p_{t}^{j}p(s)\|^{2}_{L^{2}}+(\mathcal{E}(s))^{3/2}+\int_s^t\sqrt{\mathcal{E}}\mathcal{D}.
					\end{aligned}
				\end{equation}
			}
			\paragraph{\underline{Step 4 --The Estimate for $\overline{\p_{t}^{j}p}$}}
			
			Using the same computation as in the proof of Theorem \ref{thm:pressure}, we obtain
			\[
			\|\overline{p}\|^{2}_{L_{t}^{2}}\lesssim \|u\|_{L_{t}^{2}H^{1}}+\|\p_{t}u\|^{2}_{L_{t}^{2}L^{2}}+\|[\p_{t}^{2}\eta]_{\ell}\|^{2}_{L_{t}^{2}}+\int_{0}^{t}\sqrt{\mathcal{E}}\mathcal{D}+\mathcal{E}(0),
			\]
			\[ \|\overline{\p_{t}p}\|^{2}_{L_{t}^{2}}\lesssim \|\p_{t}u\|_{L_{t}^{2}H^{1}}+\|\p_{t}^{2}u\|^{2}_{L_{t}^{2}L^{2}}+\|[\p_{t}^{2}\eta]_{\ell}\|^{2}_{L_{t}^{2}}+\int_{0}^{t}\sqrt{\mathcal{E}}\mathcal{D}+\mathcal{E}(0).
			\]
			Hence, combining  Steps 3 and 4, and applying interpolation theory, we obtain that for any $t>0$,
			\[
			\begin{aligned}  &\mathcal{E}_\shortparallel(t)-(\mathcal{E}(t))^{3/2}+\int_s^t\left(\mathcal{D}_\shortparallel+\sum_{j=0}^2\|\p_t^{j}\eta\|_{H^{3/2-\alpha}}^2\right)\\
				&\quad+\int_s^t\sum_{j=0}^2[\p_t^{j+1}\eta]_\ell^2+\int_s^t\sum_{j=0}^{1}(\|\p_t^{j}u\|_{W^{2,q_-}}^2+\|\p_t^{j}p\|_{W^{1,q_-}}^2+\|\p_t^{j}\eta\|_{W^{3-1/q_-,q_-}}^2)\\
				&\lesssim\mathcal{E}_\shortparallel(s)+\sum_{j=0}^2(\|\p_t^ju(s)\|_{H^0}^2+\|\p_t^j\eta(s)\|_{H^1}^2)+\sum_{j=0}^{1}\|\p_{t}^{j}p(s)\|^{2}_{L^{2}}+(\mathcal{E}(s))^{3/2}+\int_s^t\sqrt{\mathcal{E}}\mathcal{D}.
			\end{aligned}\]
			
			\paragraph{\underline{Step 5 -- Enhanced Energy Estimates}}
			Interpolation allows us to enhance the energy estimates:
			\[
			\begin{aligned}
				&\|\p_tu(t)\|_{H^{1+\varepsilon_-/2}}^2\lesssim\|\p_tu(s)\|_{H^{1+\varepsilon_-/2}}^2+\int_s^t\mathcal{D},\\
				&\|\p_t\eta(t)\|_{H^{3/2+(\varepsilon_--\alpha)/2}}^2\lesssim\|\p_t\eta(s)\|_{H^{3/2+(\varepsilon_--\alpha)/2 }}^2+\int_s^t\mathcal{D}.
			\end{aligned}
			\]
			The elliptic theory for \eqref{eq:geometric1} with $j=0,1$ produces the bound:
			\[
			\|u\|_{W^{2,q_+}}^2+\|p\|_{W^{1,q_+}}^2+\|\eta\|_{W^{3-1/q_+,q_+}}^2
			\lesssim \|\p_tu\|_{L^2}^2+\|\p_t\eta\|_{H^{3/2-\alpha}}^2+\mathcal{E}.
			\]
			Theorem \ref{thm:pressure} reveals the estimate
			\[
			\|\p_tp\|_{L^2}\lesssim \|\p_tu\|_{H^1}+\|\p_t^2u\|_{L^2}+\|\p_t^{2}\eta\|_{H^1}.
			\]
			Thus, with the estimate in \eqref{est:energy_dissipation3}, we have
			\begin{equation}\label{est:energy_dissipation4}
				\begin{aligned}
					&\mathcal{E}_\shortparallel(t)+\|u(t)\|_{W^{2,q_+}}^2+\|p(t)\|_{W^{1,q_+}}^2+\|\eta(t)\|_{W^{3-1/q_+,q_+}}^2\\
					&-(\mathcal{E}(t))^{3/2}+\|\p_tp(t)\|_{L^2}^2+\int_s^t\left(\mathcal{D}_\shortparallel+\sum_{j=0}^2\|\p_t^{j}\eta\|_{H^{3/2-\alpha}}^2\right)+\int_s^t\sum_{j=0}^2[\p_t^{j+1}\eta]_\ell^2\\
					&\quad+\int_s^t(\|u\|_{W^{2,q_+}}^2+\|p\|_{W^{1,q_+}}^2+\|\eta\|_{W^{3-1/q_+,q_+}}^2)+\int_s^t(\|\p_tu\|_{W^{2,q_-}}^2+\|\p_tp\|_{W^{1,q_-}}^2+\|\p_t\eta\|_{W^{3-1/q_-,q_-}}^2)\\
					&\lesssim\mathcal{E}(s)+(\mathcal{E}(s))^{3/2}+\int_s^t\sqrt{\mathcal{E}}\mathcal{D}.
				\end{aligned}
			\end{equation}
			
			\paragraph{\underline{Step 6 -- Synthesis}}
			For any $T>0$, we take $s=0$ and restrict $\delta$ to be smaller if necessary, so that if
			\[
			\sup_{0\le t\le T}\mathcal{E}(t)\le\delta \ll 1,
			\]
			then we can absorb the superlinear terms of $\mathcal{E}$ in \eqref{est:energy_dissipation4} to obtain the estimate
			\[
			\sup_{0\le t\le t^\prime}\mathcal{E}(t)+\int_0^{t^\prime}\mathcal{D}(s)\,\mathrm{d}s\lesssim \mathcal{E}(0),
			\]
			for each $t^\prime\le T$, so that
			\begin{equation}\label{est:energy5}
				\mathcal{E}(t)+\int_0^t\mathcal{D}(s)\,\mathrm{d}s\lesssim \mathcal{E}(0)
			\end{equation}
			for all $t\le T$.
			
			Since $\mathcal{E}\lesssim \mathcal{D}$, $\mathcal{E}$ is integrable over $(0,T)$. Gronwall's lemma guarantees that there exists a universal constant $\lambda>0$ such that
			$
			e^{\lambda t}\mathcal{E}(t)\lesssim\mathcal{E}(0)$,
			for all $t<T$. Then the inequality \eqref{est:energy5} also implies
			\[
			\int_0^T\mathcal{D}\lesssim\mathcal{E}(0).
			\] 
		\end{proof}
		
		We note that this theorem is conditional. To ensure it is non-vacuous, we must prove  that the set of initial data satisfying the limiting conditions is nonempty. Thanks to Theorem \ref{thm:initial}, the inclusion of the term $\p_t^2\eta(0)\in W^{2-1/q_+,q_+}(\Sigma)$ in addition to the initial energy$\mathcal{E}(0)$ guarantees that such initial data exist and satisfy the conditions in Appendix \ref{sec:initial}. Consequently, it is sufficient to complete the global well-posedness in Theorem \ref{thm:main}.

		%%%%%%%%%%%%%%%%%%%%%%%%%%%%%%%%%%%%%%%%%%%%%%
		\begin{appendix}
			%%%%%%%%%%%%%%%%%%%%%%%%%%%%%%%%%%%%%%%%%%%%%%

			\makeatletter
			\renewcommand \theequation {%
				A.%
				%\ifnum \c@section>\z@ \@arabic\c@section.%
				%\fi
				% \ifnum\c@subsection>\z@\@arabic\c@subsection.%
				%\fi\ifnum \c@subsubsection>\z@\@arabic\c@subsubsection.
				% \fi
				\@arabic\c@equation} \@addtoreset{equation}{section}
			% \@addtoreset{equation}{subsection}
			\makeatother

			%%%%%%%%%%%%%%%%%%%%%%%%%%%%%%%%%%%%%%%%%%%%%%
			\section{Forcing Terms in the Linear System}\label{sec:dive_forcing}
			%%%%%%%%%%%%%%%%%%%%%%%%%%%%%%%%%%%%%%%%%%%%%%

			Define
			\begin{align}{\label{eq:def1}}
				X:=\{(u,p)|(u,p)\in W^{2,q_{-}}\times W^{1,q_{-}}\},
			\end{align}
			and
			\begin{align}{\label{eq:def2}}
				Y:=\{(F^{1},F^{2},F^{4},F^{4_{+}},F^{5})|(F^{1},F^{2},F^{4},F^{4_{+}},F^{5})\in L^{q_{-}}\times W^{1,q_{-}}\times W^{1-\frac{1}{q_{-}},q_{-}}\times W^{2-\frac{1}{q_{-}},q_{-}}\times W^{1-\frac{
						1
					}{q_{-}},q_{-}}\}.
			\end{align}
			We now present the nonlinear interaction terms for \eqref{eq:quasi_linear}. From \eqref{eq:geometric}, we explicitly set
			\[
			\begin{aligned}
				F^{1}&=\operatorname{div}_{\partial_{t}\mathcal{A}}S_{\mathcal{A}}(p,u)
				+\operatorname{div}_{\mathcal{A}}\mathbb{D}_{\partial_{t}\mathcal{A}}u+\partial_{t}u\cdot \nabla_{\mathcal{A}}u+u\cdot \nabla_{\partial_{t}\mathcal{A}}u+u\cdot \nabla_{\mathcal{A}}\partial_{t}u+\partial_{t}(\partial_{t}\bar{\eta}WK\partial_{2}u),\\F^{2}&=\operatorname{div}_{\partial_{t}\mathcal{A}}u,\\
				F^{4}&=\mu\mathbb{D}_{\partial_{t}\mathcal{A}}u\mathcal{N}+\mu S_{\mathcal{A}}(p,u)\partial_{t}\mathcal{N}+\left[g(\eta)-\sigma\p_1\left(\frac{\p_1\eta}{(1+|\p_1\zeta_0|^2)^{3/2}}\right)-\sigma\partial_{1}\mathcal{R}(\partial_{1}\zeta_{0},\partial_{1}\eta)\right]\p_t\mathcal{N},\\
				F^{5}&=\mu\mathbb{D}_{\p_t\mathcal{A}}u\nu\cdot\tau ,\\
				F^{7}&=\kappa\hat{\mathscr{W}}^\prime(u\cdot \mathcal{N})\p_tu\cdot \mathcal{N}(\pm\ell).
			\end{aligned}
			\]
			
			\makeatletter
			\renewcommand \theequation {%
				B.%
				%\ifnum \c@section>\z@ \@arabic\c@section.%
				%\fi
				% \ifnum\c@subsection>\z@\@arabic\c@subsection.%
				%\fi\ifnum \c@subsubsection>\z@\@arabic\c@subsubsection.
				% \fi
				\@arabic\c@equation} \@addtoreset{equation}{section}
			% \@addtoreset{equation}{subsection}
			\makeatother

			%%%%%%%%%%%%%%%%%%%%%%%%%%%%%%%%%%%%%%%%%%%%%%
			\section{Initial Data for the Nonlinear System}\label{sec:initial}
			%%%%%%%%%%%%%%%%%%%%%%%%%%%%%%%%%%%%%%%%%%%%%%

			Our aim is to use the Galerkin method to establish the existence and uniqueness of solutions to \eqref{eq:geometric}. Therefore, it is crucial to ensure that the initial data of the sequence of approximate solutions converge to the initial data of PDE \eqref{eq:geometric}.
			Now we state our theorem for the construction of initial conditions for nonlinear system \eqref{eq:geometric}.
			\begin{theorem}\label{thm:initial}
				Suppose that the initial data $\p_t^2\eta(0)\in \mathring{H}^1(\Sigma)\cap W^{2-1/q_+, q_+}(\Sigma)$ and $ D_t^2u(0)\in H^0(\Om)$ are given and satisfy $\dive_{\mathcal{A}_0}D_t^2u(0)=0$. Then we can  construct $(u_0, p_0, \eta_0, \p_tu(0), \p_tp(0), \p_t\eta(0))$
				solving the system \eqref{eq:geometric} at $t=0$ such that
				\begin{align}
					\mathcal{E}(0) + \|\p_t^2\eta(0)\|_{W^{2-1/q_+, q_+}}^2\le \delta
				\end{align}
				for a universal constant $\delta>0$ sufficiently small.
			\end{theorem}
			
			\begin{remark}
				It is important to note that the initial data in Theorem \ref{thm:initial} implies the finiteness of $\mathcal{E}(0)$, where $\mathcal{E}$ is defined in \eqref{def:energy}. So, Theorem \ref{thm:initial} guarantees that our initial data for nonlinear system \eqref{eq:geometric} are well defined.
			\end{remark}

			%%%%%%%%%%%%%%%%%%%%%%%%%%%%%%%%%%%%%%%%%%%%%%
			\subsection{Initial Data for the Nonlinear System}\label{sec:initial_linear}
			%%%%%%%%%%%%%%%%%%%%%%%%%%%%%%%%%%%%%%%%%%%%%%
			
			In this section, we construct the initial data for the nonlinear system \eqref{eq:geometric1}. Evaluating the system \eqref{eq:geometric1} at $t=0$, we define the initial forcing terms
			\begin{equation}\label{initial_dt2xi}
				\begin{aligned}
					&F^{1,j}(0)\in L^{q_-}(\Om),\ F^{2,j}(0)\in W^{1,q_{-}}(\Om)\ , F^{3,j}\in W^{2-\frac{1}{q_{-}},q_{-}}(\Sigma), \\
					&F^{4,j}(0)\in W^{1-1/q_-,q_-}(\Sigma), F^{6,j}\in W^{2-\frac{1}{q_{-}},q_{-}}(\Sigma), F^{5,j}(0)\in W^{1-1/q_-,q_-}(\Sigma_s),
				\end{aligned}
			\end{equation}
			are in terms of $(u_0, p_0,\eta_0, \p_tu(0), \p_tp(0), \p_t\eta(0),\p_{t}^{2}u(0),\p_{t}^{2}\eta(0))$, and $\mathcal{A}_0$, $\mathcal{N}(0)$, $J(0)=1/K(0)$ are in terms of $\eta_0$, and $R(0)$ are in terms of $(\eta_0, \p_t\eta(0))$.  We replace $(u,p,\eta)$ with $(v,q,\xi)$ on the left hand side of the system. 
			
			We now use elliptic theory to construct $(D_tu(0), \p_tp(0), \p_t\xi(0))$. First, we define the initial forcing terms $F^{1,1}(0)$, $F^{3,1}(0)$, $F^{4,1}(0)$ and $F^{5,1}(0)$ as
			\begin{equation}\label{def:F_11}
				\begin{aligned}
					&F^{1,1}:=\p_tF^{1,0}+\mathfrak{G}^1(u,p), \quad F^{3,1}(0)=\p_tF^{3,0}(0),\\
					&F^{4,1}:=\p_tF^{4,0}+\mathfrak{G}^4(u,p,\xi),\quad F^{5,1}:=\p_tF^{5,0}+\mathfrak{G}^5(u),
				\end{aligned}
			\end{equation}
			where
			\begin{equation}\label{def:map_g}
				\begin{aligned}
					\mathfrak{G}^1(u,p)&=-\p_tRu -R^2u+\nabla_{\p_t\mathcal{A}}p+\dive_{\mathcal{A}}\left(\mathbb{D}_{\mathcal{A}}(Ru)+\mathbb{D}_{\p_t\mathcal{A}}u)+ \dive_{\p_t\mathcal{A}}\mathbb{D}_{\mathcal{A}}u\right),\\
					\mathfrak{G}^4(u,p,\xi)&=\mu\mathbb{D}_{\mathcal{A}}(Ru)\mathcal{N}-(pI-\mu\mathbb{D}_{\mathcal{A}}u)\p_t\mathcal{N}+\mu\mathbb{D}_{\p_t\mathcal{A}}u\mathcal{N}+\mathcal{K}(\xi)\p_t\mathcal{N}-\sigma\p_1F^3\p_t\mathcal{N},\\
					\mathfrak{G}^5(u)&=(\mu\mathbb{D}_{\mathcal{A}}(Ru)\nu+\mu\mathbb{D}_{\p_t\mathcal{A}}u\nu+\beta Ru)\cdot\tau.
				\end{aligned}
			\end{equation}
			\noindent and
			\begin{equation}{\label{def:forcing_0}}
				F^{1,0}=\p_t\bar{\eta}WK\p_2u-u\cdot\nabla_{\mathcal{A}}u,\quad F^{3,0}=\sigma\mathcal{R}(\p_1\zeta_0,\p_1\eta),\quad F^{4,0}=0,\quad F^{5,0}=0,\quad F^{7,0}=\kappa\hat{\mathscr{W}}(\p_t\eta).
			\end{equation}
			Provided $u$, $p$ and $\xi$ are sufficiently regular, we define
			\begin{equation}\label{def:forcej2}
				F^{1,2}:=\mathfrak{G}^1(D_tu,\p_tp)+D_tF^{1,1},\
				F^{4,2}:=\mathfrak{G}^4(D_tu,\p_tp,\p_t\xi)+\p_tF^{4,1},\
				F^{5,2}:=\mathfrak{G}^5(D_tu)+\p_tF^{5,1}.
			\end{equation}
			
			Then we have the following lemma.
			\begin{lemma}\label{lem:initial_force1}
				The following estimates hold:
				\begin{equation}\label{est:initial_f11}
					\begin{aligned}
						\|F^{1,1}(0)\|_{L^{q_+}(\Om)}&\lesssim \|\p_tF^{1,0}(0)\|_{L^{q_+}}+\|\p_t\eta(0)\|_{W^{3-1/q_+,q_+}}\|F^{1,0}(0)\|_{L^{q_+}}\\
						&\quad +\big(\|\p_t^2\eta(0)\|_{W^{2-1/q_+,q_+}}+\|\p_t\eta(0)\|_{W^{3-1/q_+,q_+}}^2+\|\p_t\eta(0)\|_{W^{3-1/q_+,q_+}}\big)\|u_0\|_{W^{2,q_+}}\\
						&\quad+\|\p_t\eta(0)\|_{W^{3-1/q_+,q_+}}\big(\|p_0\|_{W^{1,q_+}}+\|u_0\|_{W^{2,q_+}}\big),
					\end{aligned}
				\end{equation}
				\begin{equation}\label{est:initial_f41}
					\begin{aligned}
						\|F^{4,1}(0)\|_{W^{1-1/q_+, q_+}(\Sigma)}&\lesssim\|\p_t\eta(0)\|_{W^{3-1/q_+,q_+}}\big(\|p_0\|_{W^{1,q_+}}+\|u_0\|_{W^{2,q_+}}+\|\xi_0\|_{W^{3-1/q_+,q_+}}\\
						&\quad+\|\p_t\xi(0)\|_{W^{3-1/q_+,q_+}}+\|\p_1F^{3,0}(0)\|_{W^{1-1/q_+, q_+}(\Sigma)}\big),
					\end{aligned}
				\end{equation}
				and
				\begin{equation}\label{est:initial_f51}
					\begin{aligned}
						\|F^{5,1}(0)\|_{W^{1-1/q_+, q_+}(\Sigma_s)}\lesssim \|\p_t\eta(0)\|_{W^{3-1/q_+,q_+}}\|u_0\|_{W^{2, q_+}}.
					\end{aligned}
				\end{equation}
				These \eqref{est:initial_f11}-- \eqref{est:initial_f51} imply
				\[
				\begin{aligned}
					&F^{1,1}(0)\in L^{q_+}(\Om),\quad \p_1F^{3,1}(0)=\p_1\p_tF^{3,0}(0)\in W^{1-1/q_+, q_+}(\Sigma),\\
					&F^{4,1}(0)\in W^{1-1/q_+, q_+}(\Sigma),\quad F^{5,1}(0)\in W^{1-1/q_+, q_+}(\Sigma_s).
				\end{aligned}
				\]
				
			\end{lemma}
			\begin{proof}
				By the assumption that $F^{1,0}(0)$, $\p_tF^{1,0}(0)\in L^{q_+}$, and $D_tF^{1,0}(0)=\p_tF^{1,0}(0)+R^1(0)F^{1,0}(0)$, we directly deduce that
				\begin{equation}\label{est:initial_dtf1}
					\begin{aligned}
						\|D_tF^{1,0}(0)\|_{L^{q_+}}&\lesssim \|\p_tF^{1,0}(0)\|_{L^{q_+}}+\|R(0)\|_{L^\infty}\|F^{1,0}(0)\|_{L^{q_+}}\\
						&\lesssim\|\p_tF^{1,0}(0)\|_{L^{q_+}}+\|\p_t\eta(0)\|_{W^{3-1/q_+,q_+}}\|F^{1,0}(0)\|_{L^{q_+}}.
					\end{aligned}
				\end{equation}
				We now estimate $\mathfrak{G}^1(u_0, p_0)$ term by term. By $\|\mathcal{A}_0\|_{L^\infty(\Om)}\lesssim1$, we have
				\begin{equation}\label{est:initial_g11}
					\begin{aligned}
						&\|(R(0)+\p_tJ(0)K(0))\Delta_{\mathcal{A}_0}u_0\|_{L^{q_+}}\\
						&\lesssim\|R(0)+\p_tJ(0)K(0)\|_{L^\infty(\Om)}(\|\nabla^2u_0\|_{L^{q_+}}+\|\nabla\mathcal{A}_0\|_{L^2(\Om)}\|\nabla u_0\|_{L^{2/(1-\varepsilon_+)}})\\
						&\lesssim \|\p_t\eta(0)\|_{W^{3-1/q_+,q_+}}(1+\|\eta_0\|_{W^{3-1/q_+,q_+}})\|u_0\|_{W^{2,q_+}}.
					\end{aligned}
				\end{equation}
				Similarly, we have
				\begin{equation}\label{est:initial_g12}
					\|(\p_tJ(0)K(0)+R(0)+R^T(0))\nabla_{\mathcal{A}_0}p_0\|_{L^{q_+}}\lesssim \|\p_t\eta(0)\|_{W^{3-1/q_+,q_+}}\|p_0\|_{W^{1,q_+}},
				\end{equation}
				and
				\begin{equation}\label{est:initial_g13}
					\begin{aligned}
						&\|\dive_{\mathcal{A}_0}\left(\mathbb{D}_{\mathcal{A}_0}(R(0)u_0)+\mathbb{D}_{\p_t\mathcal{A}(0)}u_0-R(0)\mathbb{D}_{\mathcal{A}_0}u_0\right)\|_{L^{q_+}}\\
						&\lesssim \|\nabla\mathcal{A}_0\|_{L^2(\Om)}\|\nabla(R(0) u_0)\|_{L^{2/(1-\varepsilon_+)}}+\|\nabla^2(R(0) u_0)\|_{L^{q_+}}+\|\p_t\mathcal{A}(0)\|_{L^\infty(\Om)}\|\nabla^2u_0\|_{L^{q_+}}\\
						&\quad+\|\nabla\p_t\mathcal{A}(0)\|_{L^2(\Om)}\|\nabla u_0\|_{L^{2/(1-\varepsilon_+)}}+\|\nabla R(0)\|_{L^2(\Om)}\|\nabla u_0\|_{L^{2/(1-\varepsilon_+)}}+\|R(0)\|_{L^\infty(\Om)}\|\nabla^2 u_0\|_{L^{q_+}}\\
						&\quad+\|R(0)\|_{L^\infty(\Om)}\|\nabla\mathcal{A}_0\|_{L^2(\Om)}\|\nabla u_0\|_{L^{2/(1-\varepsilon_+)}}\\
						&\lesssim \|\nabla\mathcal{A}_0\|_{L^2(\Om)}(\|\nabla R(0)\|_{L^{2/(1-\varepsilon_+)}}\| u_0\|_{L^\infty}+ \|R(0)\|_{L^\infty(\Om)}\|\nabla u_0\|_{L^{2/(1-\varepsilon_+)}})\\
						&\quad+\|\nabla^2 R(0)\|_{L^{q_+}(\Om)}\| u_0\|_{L^\infty}+ \|R(0)\|_{L^\infty(\Om)}\|\nabla u_0\|_{L^{q_+}}+\|\nabla R(0)\|_{L^2(\Om)}\|\nabla u_0\|_{L^{2/(1-\varepsilon_+)}}\\
						&\quad+\|\p_t\mathcal{A}(0)\|_{L^\infty(\Om)}\|\nabla^2u_0\|_{L^{q_+}}+\|\nabla\p_t\mathcal{A}(0)\|_{L^2(\Om)}\|\nabla u_0\|_{L^{2/(1-\varepsilon_+)}}+\|\nabla R(0)\|_{L^2(\Om)}\|\nabla u_0\|_{L^{2/(1-\varepsilon_+)}}\\
						&\quad+\|R(0)\|_{L^\infty(\Om)}\|\nabla^2 u_0\|_{L^{q_+}}+\|R(0)\|_{L^\infty(\Om)}\|\nabla\mathcal{A}_0\|_{L^2(\Om)}\|\nabla u_0\|_{L^{2/(1-\varepsilon_+)}}\\
						&\lesssim\|\p_t\eta(0)\|_{W^{3-1/q_+,q_+}}(1+\|\eta_0\|_{W^{3-1/q_+,q_+}})\|u_0\|_{W^{2,q_+}}.
					\end{aligned}
				\end{equation}
				Then we estimate
				\begin{equation}\label{est:initial_g14}
					\begin{aligned}
						\|\p_tR(0)u_0\|_{L^{q_+}}&\lesssim(\|\nabla\p_t^2\bar{\eta}(0)\|_{L^{q_+}}+\|\nabla\p_t\bar{\eta}(0)\|_{L^{q_+}})\| u_0\|_{L^\infty}\\
						&\lesssim (\|\p_t^2\eta(0)\|_{W^{2-1/q_+,q_+}}+\|\p_t\eta(0)\|_{W^{3-1/q_+,q_+}}^2)\|u_0\|_{W^{2,q_+}}.
					\end{aligned}
				\end{equation}
				The combination of \eqref{est:initial_g11}--\eqref{est:initial_g14} gives the bound of  $\mathfrak{G}^1(v_0, q_0)$, which together with \eqref{est:initial_dtf1} imply \eqref{est:initial_f11}.
				
				We now consider the bounds for $\mathfrak{G}^4(u_0,p_0,\xi_0)$. We first employ Sobolev embedding theory and usual trace theory to deduce that
				\begin{equation}\label{est:initial_g41}
					\begin{aligned}
						\|\mathbb{D}_{\mathcal{A}_0}(R(0)u_0)\mathcal{N}(0)\|_{W^{1-1/q_+, q_+}(\Sigma)}&\lesssim \|\mathbb{D}_{\mathcal{A}_0}(R(0)u_0\|_{W^{1, q_+}(\Om)}\\
						&\lesssim \|\nabla R(0)u_0)\|_{L^{q_+}}+\| R(0)\nabla u_0)\|_{L^{q_+}}+\|\nabla^2 R(0)u_0)\|_{L^{q_+}}\\
						&\quad+\|\nabla R(0)\nabla u_0)\|_{L^{q_+}}+\|R(0)\nabla^2 u_0)\|_{L^{q_+}}\\
						&\lesssim \|\p_t\eta(0)\|_{W^{3-1/q_+,q_+}}\|u_0\|_{W^{2,q_+}}.
					\end{aligned}
				\end{equation}
				The similar argument allows us to deduce
				\begin{equation}\label{est:initial_g42}
					\begin{aligned}
						&\|(p_0I-\mu\mathbb{D}_{\mathcal{A}_0}u_0)\p_t\mathcal{N}(0)\|_{W^{1-1/q_+, q_+}(\Sigma)}\\
						&\lesssim \|p_0I-\mu\mathbb{D}_{\mathcal{A}_0}u_0\|_{W^{1, q_+}(\Om)}\|\p_t\mathcal{N}(0)\|_{W^{1, q_+}(\Sigma)}\\
						&\lesssim (\|p_0\|_{W^{1,q_+}}+\||\mathcal{A}_0|\nabla u_0\|_{L^{q_+}}+\||\nabla\mathcal{A}_0|\nabla u_0\|_{L^{q_+}}+\||\mathcal{A}_0|\nabla^2u_0\|_{L^{q_+}})\|\p_t\eta(0)\|_{W^{3-1/q_+,q_+}}\\
						&\lesssim \|\p_t\eta(0)\|_{W^{3-1/q_+,q_+}}(\|p_0\|_{W^{1,q_+}}+\|u_0\|_{W^{2,q_+}}),
					\end{aligned}
				\end{equation}
				and
				\begin{equation}\label{est:initial_g43}
					\begin{aligned}
						&\|\mathbb{D}_{\p_t\mathcal{A}(0)}u_0\mathcal{N}(0)\|_{W^{1-1/q_+, q_+}(\Sigma)}\lesssim \|\mathbb{D}_{\p_t\mathcal{A}(0)}u_0\|_{W^{1, q_+}(\Om)}\\
						&\lesssim \||\p_t\mathcal{A}(0)|\nabla u_0\|_{L^{q_+}}+\||\nabla\p_t\mathcal{A}(0)|\nabla u_0\|_{L^{q_+}}+\||\p_t\mathcal{A}_0|\nabla^2u_0\|_{L^{q_+}})\lesssim \|\p_t\eta(0)\|_{W^{3-1/q_+,q_+}}\|u_0\|_{W^{2,q_+}}.
					\end{aligned}
				\end{equation}
				For the capillary term, we have
				\begin{equation}\label{est:initial_g44}
					\begin{aligned}
						\|\mathcal{K}(\xi_0)\p_t\mathcal{N}(0)\|_{W^{1-1/q_+, q_+}(\Sigma)}&\lesssim \|\mathcal{K}(\xi_0)\|_{W^{1-1/q_+, q_+}(\Sigma)}\|\p_t\mathcal{N}(0)\|_{W^{1, q_+}(\Sigma)}\\
						&\lesssim  \|\p_t\eta(0)\|_{W^{3-1/q_+,q_+}}(\|\xi_0\|_{W^{3-1/q_+,q_+}}+\|\p_t\xi(0)\|_{W^{3-1/q_+,q_+}}).
					\end{aligned}
				\end{equation}
				The assumption for $\p_1F^3(0)$ shows that
				\begin{equation}\label{est:initial_g45}
					\begin{aligned}
						\|\p_1F^{3,0}(0)\p_t\mathcal{N}\|_{W^{1-1/q_+, q_+}(\Sigma)}&\lesssim \|\p_1F^{3,0}(0)\|_{W^{1-1/q_+, q_+}(\Sigma)}\|\p_t\mathcal{N}(0)\|_{W^{1, q_+}(\Sigma)}\\
						&\lesssim  \|\p_t\eta(0)\|_{W^{3-1/q_+,q_+}}\|\p_1F^{3,0}(0)\|_{W^{1-1/q_+, q_+}(\Sigma)}.
					\end{aligned}
				\end{equation}
				The estimates \eqref{est:initial_g41}--\eqref{est:initial_g45} along with the assumption $\p_tF^4(0)\in W^{1-1/q_+, q_+}(\Sigma)$ imply \eqref{est:initial_f41}.
				
				Finally, we estimate $\mathfrak{G}^5(u_0)$. We first bound that
				\begin{equation}\label{est:initial_g51}
					\begin{aligned}
						&\|\mathbb{D}_{\mathcal{A}_0}(R(0)u_0)\|_{W^{1-1/q_+, q_+}(\Sigma_s)}+\|\mathbb{D}_{\p_t\mathcal{A}(0)}u_0\|_{W^{1-1/q_+, q_+}(\Sigma_s)}\\
						&\lesssim \|\mathbb{D}_{\mathcal{A}_0}(R(0)u_0)\|_{W^{1, q_+}(\Om)}+\|\mathbb{D}_{\p_t\mathcal{A}(0)}u_0\|_{W^{1, q_+}(\Om)}\lesssim \|\p_t\eta(0)\|_{W^{3-1/q_+,q_+}}\|u_0\|_{W^{2,q_+}}.
					\end{aligned}
				\end{equation}
				We then bound the remainder terms:
				\begin{equation}\label{est:initial_g52}
					\begin{aligned}
						&\|\p_tJ(0)K(0)\mathbb{D}_{\mathcal{A}(0)}u_0\|_{W^{1-1/q_+, q_+}(\Sigma_s)}+\|R(0)u_0\|_{W^{1-1/q_+, q_+}(\Sigma_s)}+\| \p_tJ(0)K(0)u_0\|_{W^{1-1/q_+, q_+}(\Sigma_s)}\\
						&\lesssim\|\p_tJ(0)K(0)\|_{W^{1, q_+}(\Sigma_s)}\|\mathbb{D}_{\mathcal{A}(0)}u_0\|_{W^{1-1/q_+, q_+}(\Sigma_s)}+\||R(0)|+|\p_tJ(0)K(0)|\|_{W^{1, q_+}(\Sigma_s)}\|u_0\|_{W^{2, q_+}}\\
						&\lesssim \|\p_t\eta(0)\|_{W^{3-1/q_+,q_+}}\|u_0\|_{W^{2, q_+}}.
					\end{aligned}
				\end{equation}
				Inequalities \eqref{est:initial_g51} and \eqref{est:initial_g52} yield the bound \eqref{est:initial_f51}.
			\end{proof}

			The result for the nonlinear system is given in Theorem \ref{thm:initial}.
			The proof of Theorem \ref{thm:initial} relies on the Banach fixed-point theorem. We first assume that $(u_0, p_0, \eta_0)$, $(D_tu(0), \p_{t}p(0), \p_{t}\eta(0))$, and $(D_t^2u(0), \p_t^2\eta(0))$ are given so that
			$\mathcal{E} (u,p,\eta)(0) + \|\p_t^{2}\eta(0)\|_{W^{2-1/q_+, q_+}}^2\le \delta$. We then show that the forcing terms are well-defined in Lemma \ref{est:initial_force3}. Then we denote the unknowns $(v_0, q_0, \xi_0)$ in \eqref{eq:geometric1} with $j=0$ at time $t=0$ instead of $(u_0, p_0, \xi_0)$, and the unknowns $(\p_tv(0), \p_tq(0), \p_t\xi(0))$ in \eqref{eq:geometric1} with $j=1$ at time $t=0$ instead of $(\p_tu(0), \p_tp(0), \p_t\xi(0))$, and the unknowns $(\p_t^{2}v(0), \p_t^{2}q(0), \p_t^{2}\xi(0))$ in \eqref{eq:geometric1} with $j=2$ at time $t=0$ instead of $(\p_t^{2}u(0), \p_t^{2}p(0), \p_t^{2}\xi(0))$, and obtain the bound $\mathcal{E}(v,q,\xi)(0) + \|\p_t^2\xi(0)\|_{W^{2-1/q_+, q_+}}^2\le \delta$ for nonlinear initial data, subject to the assumption that $D_t^2v(0) = D_t^2u(0)$ and $\p_t^2\xi(0) = \p_t^2 \eta(0)$. Finally, we use a contraction mapping argument to prove the existence of fixed points.
			
			\begin{lemma}\label{est:initial_force3}
				It holds that
				\begin{equation}\label{est:initial_force1}
					\begin{aligned}
						&\|F^{1,0}(0)\|_{L^{q_+}}+\|\p_1F^{1,0}(0)\|_{W^{1-1/q_+,q_+}}
						+[F^{7,0}(0)]_\ell\\
						&\lesssim\|\p_t\eta(0)\|_{W^{3-1/q_+,q_+}}(\|u_0\|_{W^{2,q_+}}+\|\eta_0\|_{W^{3-1/q_+,q_+}}^2)+\|\eta_0\|_{W^{3-1/q_+,q_+}}^2+\|u_0\|_{W^{2,q_+}}^2
					\end{aligned}
				\end{equation}
				and
				\begin{equation}\label{est:initial_force2}
					\begin{aligned}
						&\|\p_tF^{1,0}(0)\|_{L^{q_+}}+\|\p_1\p_tF^{3,0}(0)\|_{W^{1-1/q_+,q_+}}
						+[\p_tF^{7,0}(0)]_\ell\\
						&\lesssim\|\p_t^2\eta(0)\|_{H^1}\|u_0\|_{W^{2,q_+}}+\|u_0\|_{W^{2,q_+}}(\|D_tu(0)\|_{W^{2,q_+}}+\|u_0\|_{W^{2,q_+}}\|\p_t\eta(0)\|_{W^{3-1/q_+,q_+}})\\
						&\quad+\|\p_t\eta(0)\|_{W^{3-1/q_+,q_+}}(\|\eta_0\|_{W^{3-1/q_+,q_+}}+\|\p_t^2\eta(0)\|_{W^{3-1/q_+,q_+}}).
					\end{aligned}
				\end{equation}
			\end{lemma}
			\begin{proof}
				From the definition of $F^1(0)$, we directly use H\"older's inequality to show that
				\begin{equation}\label{est:initial_f111}
					\begin{aligned}
						\|F^{1,0}(0)\|_{L^{q_+}}&\lesssim \|\p_t\bar{\eta}(0)\p_2u_0\|_{L^{q_+}}+\|u_0\cdot\nabla_{\mathcal{A}_0}u_0\|_{L^{q_+}}\lesssim \|\p_t\eta(0)\|_{H^{3/2+(\varepsilon_-\alpha)/2}}\|u_0\|_{W^{2,q_+}}+\|u_0\|_{W^{2,q_+}}^2,
					\end{aligned}
				\end{equation}
				where in the last inequality, we have used the fact $\|\mathcal{A}_0\|\lesssim 1$, and the Sobolev embedding theory and the usual trace theory.
				By definition of $\mathcal{R}$, we have $\p_1F^{3,0}(0)=\p_y\mathcal{R}(\p_1\zeta_0,\p_1\eta_0)\p_1\zeta_0+\p_z\mathcal{R}(\p_1\zeta_0,\p_1\eta_0)\p_1\eta_0$. By the boundedness of $\p_y\mathcal{R}$, $\p_z\mathcal{R}$ and $\p_z^2\mathcal{R}$, we have
				\begin{equation}\label{est:initial_f311}
					\begin{aligned}
						\|\p_1F^{3,0}(0)\|_{W^{1-1/q_+,q_+}}\lesssim \|\eta_0\|_{W^{3-1/q_+,q_+}}^2.
					\end{aligned}
				\end{equation}
				By property $|\hat{\mathscr{W}}(z)|\lesssim|z|^2$, we have
				\begin{equation}\label{est:initial_f711}
					[F^{7,0}(0)]_\ell\lesssim [\p_t\eta(0)]_\ell\|\eta_0\|_{W^{3-1/q_+,q_+}}^2\lesssim\|\p_t\eta(0)\|_{W^{3-1/q_+,q_+}}\|\eta_0\|_{W^{3-1/q_+,q_+}}^2.
				\end{equation}
				The combination of \eqref{est:initial_f111}, \eqref{est:initial_f311} and \eqref{est:initial_f711} implies \eqref{est:initial_force1}.
				
				We estimate $\p_tF^{1,0}(0)$ term by term. First, by the H\"older's inequality, and the Sobolev embedding theory and the usual trace theory, we have that
				\begin{equation}\label{est:initial_dtf111}
					\begin{aligned}
						\|\p_t^2\bar{\eta}(0)K(0)\p_2u_0\|_{L^{q_+}}\lesssim \|\p_t^2\bar{\eta}(0)\|_{L^\infty}\|\p_2u_0\|_{L^{q_+}}\lesssim \|\p_t^2\eta(0)\|_{H^1}\|u_0\|_{W^{2,q_+}},
					\end{aligned}
				\end{equation}
				where we have used $\|K(0)\|_{L^\infty}\lesssim 1$. Similarly,
				\begin{equation}\label{est:initial_dtf112}
					\begin{aligned}
						\|\p_t\bar{\eta}(0)\p_tK(0)W\p_2u_0+\p_t\bar{\eta}(0)K(0)W\p_2D_tu(0)\|_{L^{q_+}}\lesssim \|\p_t\eta(0)\|_{W^{3-1/q_+,q_+}}(\|u_0\|_{W^{2,q_+}}+\|D_tu(0)\|_{W^{2,q_+}}),
					\end{aligned}
				\end{equation}
				and
				\begin{equation}\label{est:initial_dtf113}
					\|\p_t\bar{\eta}(0)K(0)W\p_2(R(0)u_0)\|_{L^{q_+}}\lesssim \|\p_t\eta(0)\|_{W^{3-1/q_+,q_+}}\|R(0)u_0\|_{W^{2,q_+}}\lesssim \|\p_t\eta(0)\|_{W^{3-1/q_+,q_+}}^2\|u_0\|_{W^{2,q_+}}.
				\end{equation}
				Then the lower order terms are bounded by
				\begin{equation}\label{est:initial_dtf114}
					\begin{aligned}
						\|D_tu(0)\nabla_{\mathcal{A}_0}u_0\|_{L^{q_+}}+\|u_0\nabla_{\p_t\mathcal{A}(0)}u_0\|_{L^{q_+}}\lesssim (\|D_tu(0)\|_{W^{2,q_+}}+\|u_0\|_{W^{2,q_+}}\|\p_t\eta(0)\|_{W^{3-1/q_+,q_+}})\|u_0\|_{W^{2,q_+}}.
					\end{aligned}
				\end{equation}
				and
				\begin{equation}\label{est:initial_dtf115}
					\begin{aligned}
						&\|u_0\nabla_{\mathcal{A}_0}D_tu(0)\|_{L^{q_+}}+\|(R(0)u_0)\nabla_{\mathcal{A}_0}u_0\|_{L^{q_+}}+\|u_0\nabla_{\mathcal{A}_0}(R(0)u_0)\|_{L^{q_+}}\\
						&\lesssim \|u_0\|_{W^{2,q_+}}(\|D_tu(0)\|_{W^{2,q_+}}+\|u_0\|_{W^{2,q_+}}\|\p_t\eta(0)\|_{W^{3-1/q_+,q_+}}).
					\end{aligned}
				\end{equation}
				
				We expand
				\[
				\begin{aligned}
					\p_1\p_tF^{3,0}(0)&=\p_{yz}^2\mathcal{R}(\p_1\zeta_0, \p_1\eta_0)\p_1\p_t\eta(0)\p_1^2\zeta_0+\p_{zz}^2\mathcal{R}(\p_1\zeta_0, \p_1\eta_0)\p_1\p_t\eta(0)\p_1^2\eta_0+\p_z\mathcal{R}(\p_1\zeta_0, \p_1\eta_0)\p_1^2\p_t\eta(0),
				\end{aligned}
				\]
				and then use the fact that $W^{1,q_+}(\Sigma)$ is a Banach algebra to estimate
				\begin{equation}\label{est:initial_dtf311}
					\begin{aligned}
						\|\p_1\p_tF^{3,0}(0)\|_{W^{1-1/q_+,q_+}(\Sigma)}&\lesssim \|\p_{yz}^2\mathcal{R}(\p_1\zeta_0, \p_1\eta_0)\|_{W^{1,q_+}}\|\p_1\p_t\eta(0)\|_{W^{1-1/q_+,q_+}}\\
						&\quad+\|\p_{zz}^2\mathcal{R}(\p_1\zeta_0, \p_1\eta_0)\|_{W^{1,q_+}}\|\p_1\p_t\eta(0)\p_1^2\eta_0\|_{W^{1-1/q_+,q_+}}\\
						&\quad+\|\p_z\mathcal{R}(\p_1\zeta_0, \p_1\eta_0)\|_{W^{1,q_+}}\|\p_1^2\p_t\eta_0\|_{W^{1-1/q_+,q_+}}\\
						&\lesssim \|\p_t\eta(0)\|_{W^{3-1/q_+,q_+}}\|\eta+\|_{W^{3-1/q_+,q_+}},
					\end{aligned}
				\end{equation}
				where we have used the facts $|\p_{yz}^2\mathcal{R}(y,z)|\lesssim|z|$ and $|\p_{zz}^2\mathcal{R}(y,z)|\lesssim1$.
				
				The definition of $\hat{\mathscr{W}}$ gives the bound $|\hat{\mathscr{W}}^\prime(z)|\lesssim|z|$ for $|z|\lesssim1$. Then
				\begin{equation}\label{est:initial_dtf711}
					[\p_tF^{7,0}(0)]_\ell\lesssim \max_{\pm\ell}|\p_t\eta(0)||\p_t^2\eta(0)|\lesssim \|\p_t\eta(0)\|_{W^{3-1/q_+,q_+}}\|\p_t^2\eta(0)\|_{W^{3-1/q_+,q_+}}.
				\end{equation}
				
				So all the above estimates from \eqref{est:initial_dtf111}--\eqref{est:initial_dtf711} show the estimate \eqref{est:initial_force2}.
			\end{proof}
			
			To establish the contraction, we denote velocities by $v^i$, $u^i$, pressures by $q^i$,$ p^i$, and surface functions by $\eta^i$, $\xi^i$ for $i=1, 2$. We suppose that for $j=1, 2$, $(u_0^j, p_0^j, \eta_0^j, D_tu^j(0), \p_tp^j(0), \p_t\eta^j(0))$
			as well as $(D_t^2v^1(0), \p_t^2\eta^1(0))=(D_t^2v^2(0), \p_t^2\eta^2(0))$ are two initial data for the nonlinear system \eqref{eq:geometric} with $\xi$ replaced by $\eta$. We then consider the following systems with $j=0, 1$
			\begin{equation}\label{eq:pde_dt1}
				\left\{
				\begin{aligned}
					&(D_t^{j+1}v^i)(0)+\nabla_{\mathcal{A}^i_0}(\p_t^jq^i)(0)-\mu\Delta_{\mathcal{A}^i_0}(D_t^jv^i(0)+2^jR^i(0)D_t^jv^i(0)=\mathfrak{f}^{1,i}_{j}(0)\quad&\text{in}&\quad\Om,\\
					&\dive_{\mathcal{A}^i_0}(D_t^jv^i)(0)=0\quad&\text{in}&\quad\Om,\\
					&S_{\mathcal{A}^i_0}(\p_t^jq^i(0),D_t^jv^i(0))\mathcal{N}^i(0)=\bigg[g(\p_t^j\xi^i(0))-\sigma\p_1\left(\frac{\p_1\p_t^j\xi^i(0)}{(1+|\p_1\zeta_0|^2)^{3/2}}\right)-\p_1 \mathfrak{f}^{3,i}_j(0)\bigg]\mathcal{N}(0)+\mathfrak{f}^{4,i}_j(0)\quad&\text{on}&\quad\Sigma,\\
					&D_t^jv^i(0)\cdot\mathcal{N}^i(0)=\p_t^{j+1}\xi^i(0)\quad&\text{on}&\quad \Sigma,\\
					&(S_{\mathcal{A}^i_0}(\p_t^jq^i(0),D_t^jv^i(0))\nu-\beta D_t^jv^i(0))\cdot\tau=\mathfrak{f}^{5,i}_j(0)\quad&\text{on}&\quad\Sigma_s,\\
					&D_t^jv^i(0)\cdot\nu=0\quad&\text{on}&\quad\Sigma_s,\\
					&\mp\sigma\frac{\p_1\p_t^j\xi^i(0)}{(1+|\p_1\zeta_0|^2)^{3/2}}(\pm\ell)=\kappa(D_t^jv^i(0)\cdot\mathcal{N}^i(0))(\pm\ell)\pm \mathfrak{f}^{3,i}_j(\pm\ell,0)-\mathfrak{f}^{7,i}_j(\pm\ell,0),
				\end{aligned}
				\right.
			\end{equation}
			where $\mathfrak{f}^{k,i}_0=F^{j,0}(u_0^i,\eta_0^i,\p_t\eta^i(0))$ for $i=1, 2$, $k=1, 3, 4, 5$, $\mathfrak{f}^{7,i}_0=F^{7,0}(\p_t\eta^i(0))$ and for $i=1, 2$,
			\[
			\mathfrak{f}^{k,i}_1=\p_tF^{k,0}(u_0^i,\eta_0^i,D_tu^i(0),\p_t\eta^i(0))+\mathfrak{G}^{k}(v^i_0, q^i_0, \xi^i_0),
			\]
			$k=1, 4, 5$, $\mathfrak{f}^{3,i}_1=\p_tF^{3,0}(\eta_0^i,\p_t\eta^i(0))$ and $\mathfrak{f}^{7,i}_1=\p_tF^{7,0}(\p_t\eta^i(0))$.  We define $\mathcal{A}^i$, $\mathcal{N}^i$, etc, in terms of $\eta^i$, $i=1, 2$.
			Set
			\[
			\begin{aligned}
				\widetilde{\mathfrak{E}}_0(u_0^i,p_0^i,\eta_0^i, D_tu^i(0),\p_tp^i(0),\p_t\xi^i(0))=\|D_tu^i(0)\|_{W^{2, q_+}}^2+\|\p_tp^i(0)\|_{W^{1, q_+}}^2+\|\p_t\xi^i(0)\|_{W^{3-1/q_+, q_+}}^2\\
				+\|u_0^i\|_{W^{2, q_+}}^2+\|p_0^i\|_{W^{1, q_+}}^2+\|\xi_0^i\|_{W^{3-1/q_+, q_+}}^2.
			\end{aligned}
			\]

			Lemma  \ref{est:initial_force3} directly shows that
			\begin{lemma}
				The solutions $(v^i_0, q^i_0, \xi^i_0)$ and $(D_tv^i(0), \p_tq^i(0), \p_t\xi^i(0))$ solving \eqref{eq:pde_dt1} are bounded by
				\[
				\begin{aligned}
					&\|v_0^i\|_{W^{2, q_+}}^2+\|q_0^i\|_{W^{1, q_+}}^2+\|\xi_0^i\|_{W^{3-1/q_+, q_+}}^2+\|D_tv(0)\|_{W^{2, q_+}}^2+\|\p_tq(0)\|_{W^{1, q_+}}^2+\|\p_t\xi(0)\|_{W^{3-1/q_+, q_+}}^2\\
					&\lesssim \|D_t^2u(0)\|_{L^2}^2+\|\p_t^2\eta(0)\|_{W^{3-1/q_+, q_+}}^2+\widetilde{\mathfrak{E}}_0(u_0^i,p_0^i,\eta_0^i, D_tu^i(0),\p_tp^i(0),\p_t\eta^i(0)).
				\end{aligned}
				\]
			\end{lemma}
			Therefore, by assuming $\|D_t^2u(0)\|_{L^2}^2+\|\p_t^2\eta(0)\|_{W^{3-1/q_+, q_+}}^2\lesssim \delta_0$, we have
			\begin{equation}\label{eq:small_initial}
				\widetilde{\mathfrak{E}}_0(v_0^i,q_0^i,\xi_0^i, D_tv^i(0),\p_tq^i(0),\p_t\xi^i(0))\lesssim \delta
			\end{equation}
			for a universal constant $\delta$ sufficiently small, $i=1, 2$.
			
			For simplicity,  we might abuse the notation to denote
			\begin{equation}\label{eq:diff_1}
				u_0=u_0^1-u_0^2,\ p_0=p_0^1-p_0^2,\ \xi_0=\xi_0^1-\xi_0^2,\ v_0=v_0^1-v_0^2,\ q_0=q_0^1-q_0^2,\ \eta_0=\eta_0^1-\eta_0^2,
			\end{equation}
			and
			\begin{equation}\label{eq:diff_2}
				\begin{aligned}
					D_tu(0)=D_tu^1(0)-D_tu^2(0),\ \p_tp(0)=\p_tp^1(0)-\p_tp^2(0),\ \p_t\xi(0)=\p_t\xi^1(0)-\p_t\xi^2(0),\\ D_tv(0)=D_tv^1(0)-D_tv^2(0),\ \p_tq=\p_tq(0)^1-\p_tq(0)^2,\ \p_t\eta(0)=\p_t\eta(0)^1-\p_t\eta(0)^2.
				\end{aligned}
			\end{equation}
			\begin{lemma}\label{lem:contract_1}
				The triple $(D_tv(0), \p_tq(0), \p_t\xi(0))$ in \eqref{eq:diff_2} obeys the estimate
				\begin{equation}\label{est:initial_contract_4}
					\begin{aligned}
						\|D_tv(0)\|_{H^1}^2+\|\p_tq(0)\|_{H^0}^2+\|\p_t\xi(0)\|_{H^{3/2}}^2
						\le C\delta\big(\|D_tu(0)\|_{H^1}^2+\|u_0\|_{W^{2,q_+}}^2+\|\eta_0\|_{W^{3-1/q_+,q_+}}^2\\
						\quad+\|\p_t\eta(0)\|_{H^{3/2}}^2\|v_0\|_{W^{2,q_+}}^2+\|q_0\|_{W^{1,q_+}}^2+\|\xi_0\|_{W^{3-1/q_+,q_+}}^2\big)
					\end{aligned}
				\end{equation}
				where $C$ is a universal constant that may change from line to line.
			\end{lemma}
			\begin{proof}
				By $(D_t^2v^1(0), \p_t^2\xi^1(0))=(D_t^2v^2(0), \p_t^2\xi^2(0))$, the difference of \eqref{eq:diff_2} satisfies the equations
				\begin{equation}\label{eq:diff_initial1}
					\left\{
					\begin{aligned}
						&\nabla_{\mathcal{A}_0^1}\p_t^jq-\Delta_{\mathcal{A}_0^1}D_t^jv+2^jR(0)^1D_t^jv=\mathfrak{g}^1_j&\text{in}&\ \Om,\\
						&\dive_{\mathcal{A}_0^1}D_t^jv=\mathfrak{g}^2_j&\text{in}&\ \Om,\\
						&S_{\mathcal{A}_0^1}(\p_t^jq,D_t^jv)\mathcal{N}^1(0)=\bigg[g\p_t^j\xi(0)-\sigma\p_1\bigg(\frac{\p_1\p_t^j\xi(0)}{(1+|\p_1\zeta_0|^2)^{3/2}}\bigg)-\p_1 \mathfrak{g}^3_j\bigg]\mathcal{N}^1(0)+\mathfrak{g}^4_j&\text{on}&\ \Sigma,\\
						&D_t^ju\cdot\mathcal{N}^1(0)=\mathfrak{g}^6_j\quad&\text{on}&\ \Sigma,\\
						&(S_{\mathcal{A}_0^1}(\p_t^jq,D_t^jv)\nu-\beta D_t^jv)\cdot\tau=\mathfrak{g}^5_j&\text{on}&\ \Sigma_s,\\
						&D_t^jv\cdot\nu=0&\text{on}&\ \Sigma_s,\\
						&\mp\sigma\frac{\p_1\p_t^j\xi(0)}{(1+|\p_1\zeta_0|^2)^{3/2}}(\pm\ell)=\kappa D_t^jv\cdot\mathcal{N}^1(0)(\pm\ell)\pm \mathfrak{g}^3_j(\pm\ell,0)-\mathfrak{g}^7_j(\pm\ell,0)
					\end{aligned}
					\right.
				\end{equation}
				with $j=1$,  where the forcing terms $\mathfrak{g}^k_1$, $k=1, \ldots, 7$, are presented in Appendix \ref{sec:initial_forcing}.
				
				We multiply the first equation of \eqref{eq:diff_initial1} by $J^1(0)D_tv(0)$, then integrate over $\Om$ and integrate by parts to obtain
				\[
				\begin{aligned}
					&\frac\mu2\int_\Om|\mathbb{D}_{\mathcal{A}^1_0}D_tv(0)|^2J^1(0)+\beta\int_{\Sigma_s}|D_tv(0)|^2J^1(0)=\int_\Om\mathfrak{g}^1\cdot D_tv(0)J^1(0)+\int_\Om \p_tq(0)\mathfrak{g}^2J^1(0)-(\p_t\xi(0), \mathfrak{g}^6)_{1,\Sigma}\\
					&\qquad\qquad\qquad\qquad-\int_\Sigma\mathfrak{g}^3\p_1(\mathfrak{g}^6)-\int_\Sigma\mathfrak{g}^4\cdot(D_tv(0))-\int_{\Sigma_s}\mathfrak{g}^5(D_tv(0))\cdot\tau J^1(0),
				\end{aligned}
				\]
				where we have used the fact $\mathfrak{g}^6(\pm\ell)=0$ on the contact points. By the assumption \eqref{eq:small_initial} and integration by parts, we obtain the bounds
				\begin{equation}\label{est:icontract_1}
					\begin{aligned}
						&\left|\int_\Om\mathfrak{g}^1\cdot(D_tv(0))J^1(0)-\int_\Sigma\mathfrak{g}^4\cdot(D_tv(0))-\int_{\Sigma_s}\mathfrak{g}^5(D_tv(0))\cdot\tau J^1(0)\right|\\
						&\lesssim \delta^{1/2}(\|\p_t\eta(0)\|_{H^{3/2}}+\|\eta_0\|_{W^{3-1/q_+,q_+}}+\|D_tu(0)\|_{H^1}+\|u_0\|_{W^{2,q_+}}\\
						&\quad+\|v_0\|_{W^{2,q_+}}+\|q_0\|_{W^{1,q_+}}+\|\xi_0\|_{W^{3-1/q_+,q_+}})\|D_tv(0)\|_{H^1}.
					\end{aligned}
				\end{equation}
				Similarly, we have
				\begin{equation}\label{est:diff_dt_q}
					\begin{aligned}
						&\left|\int_\Om \p_tq(0)\mathfrak{g}^2J^1(0)-(\p_t\xi(0), \mathfrak{g}^6)_{1,\Sigma}-\int_\Sigma\mathfrak{g}^3\p_1(\mathfrak{g}^6)\right|\\
						&\lesssim \delta^{1/2}(\|\p_tq(0)\|_{H^0}+\|\p_t\xi(0)\|_{H^{3/2}}+\|\eta_0\|_{W^{3-1/q_+,q_+}}+\|\p_t\eta(0)\|_{H^{3/2}})\|\eta_0\|_{W^{3-1/q_+,q_+}}.
					\end{aligned}
				\end{equation}
				By Cauchy's inequality, we have the bounds
				\begin{equation}\label{est:diff_dt_u}
					\begin{aligned}
						\|D_tv(0)\|_{H^1}^2&\le C\delta(\|D_tu(0)\|_{H^1}^2+\|u_0\|_{W^{2,q_+}}^2+\|\eta_0\|_{W^{3-1/q_+,q_+}}^2+\|\p_t\eta(0)\|_{H^{3/2}}^2)\\
						&\quad+C\delta(\|v_0\|_{W^{2,q_+}}^2+\|q_0\|_{W^{1,q_+}}^2+\|\xi_0\|_{W^{3-1/q_+,q_+}}^2+\|\p_tq(0)\|_{H^0}^2+\|\p_t\xi(0)\|_{H^{3/2}}^2).
					\end{aligned}
				\end{equation}
				
				It is observed from \eqref{est:diff_dt_q} and \eqref{est:diff_dt_u}, we have to estimate the difference $\p_tq(0)$ and $\p_t\xi(0)$. We now estimate the difference $\p_tq(0)$. We first define $w=M^1(0)\nabla\psi$, where $\psi\in H^2(\Om)$ solves
				\[
				-\Delta\psi=\p_tq(0) \quad \text{in}\ \Om,\quad
				\psi=0 \quad \text{on}\ \Sigma,\quad
				\p_\nu\psi=0\quad \text{on}\ \Sigma_s,
				\]
				which, by the elliptic theory on the $\Om$ with convex corners, enjoys the estimate
				$\|\psi\|_{H^2}\lesssim \|\p_tq(0)\|_{H^0}$.
				Then we have
				\begin{equation}\label{eq:div_2}
					\dive_{\mathcal{A}^1_0}w=\dive_{\mathcal{A}^1_0}M^1(0)\nabla\psi=K^1(0)[\p_tq(0)],\quad \text{satisfying} \quad \|w\|_{H^1}\lesssim \|\p_tq(0)\|_{H^0}.
				\end{equation}
				
				We now multiply the first equation of \eqref{eq:diff_initial1} by $J^1(0)w$, and then integrate over $\Om$ and integrate by parts, together with \eqref{eq:div_2}, to deduce that
				\begin{equation}\label{eq:contract_dtq}
					\begin{aligned}
						-\|\p_tq(0)\|_{H^0}^2+\frac\mu2\int_\Om(\mathbb{D}_{\mathcal{A}_0^1}(D_tv(0)):\mathbb{D}_{\mathcal{A}_0^1}w)J^1(0)=\int_\Om \mathfrak{g}^1\cdot J^1(0)w-\int_{\Sigma_s}\mathfrak{g}^5(w\cdot\tau)J^1(0)\\
						-\int_{-\ell}^\ell \left[g\p_t\xi(0)-\sigma\p_1\left(\frac{\p_1\p_t\xi(0)}{(1+|\p_1\zeta_0|^2)^{3/2}}+\mathfrak{g}^3\right)\right]w\cdot\mathcal{N}^1(0)+\mathfrak{g}^4\cdot w.
					\end{aligned}
				\end{equation}
				Then by \eqref{eq:div_2}, we estimate
				\begin{equation}\label{est:icontract_2}
					\begin{aligned}
						|\int_\Om(\mathbb{D}_{\mathcal{A}_0^1}(D_tv(0)):\mathbb{D}_{\mathcal{A}_0^1}w)J^1(0)|\lesssim \|D_tv(0)\|_{H^1}\|w\|_{H^1}\lesssim\|D_tv(0)\|_{H^1}\|\p_tq(0)\|_{H^0}.
					\end{aligned}
				\end{equation}
				The same argument as \eqref{est:icontract_1} implies that
				\begin{equation}
					\begin{aligned}
						&|\int_\Om \mathfrak{g}^1\cdot J^1(0)w-\int_{\Sigma_s}\mathfrak{g}^5(w\cdot\tau)J^1(0)-\int_{-\ell}^\ell \mathfrak{g}^4\cdot w|\\
						&\lesssim \delta^{1/2}(\|\p_t\eta(0)\|_{H^{3/2-\alpha}}+\|\eta_0\|_{W^{3-1/q_+,q_+}}+\|D_tu(0)\|_{H^1}+\|u_0\|_{W^{2,q_+}}\\
						&\quad+\|v_0\|_{W^{2,q_+}}+\|q_0\|_{W^{1,q_+}}+\|\xi_0\|_{W^{3-1/q_+,q_+}})\|w\|_{H^1}\\
						&\lesssim\delta^{1/2}(\|\p_t\eta(0)\|_{H^{3/2}}+\|\eta_0\|_{W^{3-1/q_+,q_+}}+\|D_tu(0)\|_{H^1}+\|u_0\|_{W^{2,q_+}}\\
						&\quad+\|v_0\|_{W^{2,q_+}}+\|q_0\|_{W^{1,q_+}}+\|\xi_0\|_{W^{3-1/q_+,q_+}})\|\p_tq(0)\|_{H^0}.
					\end{aligned}
				\end{equation}
				We directly estimate
				\begin{equation}\label{est:icontract_3}
					\begin{aligned}
						|\int_{-\ell}^\ell \left[g\p_t\xi(0)-\sigma\p_1\left(\frac{\p_1\p_t\xi(0)}{(1+|\p_1\zeta_0|^2)^{3/2}}+\mathfrak{g}^3\right)\right]w\cdot\mathcal{N}^1(0)|\lesssim (\|\p_t\xi(0)\|_{H^{3/2}}+\|\mathfrak{g}^3\|_{H^{1/2}})\|w\|_{H^1}\\
						\lesssim [\|\p_t\xi(0)\|_{H^{3/2}}+\delta^{1/2}(\|\p_t\eta(0)\|_{H^{3/2}}+\|\eta_0\|_{W^{3-1/q_+,q_+}})]\|\p_tq(0)\|_{H^0}.
					\end{aligned}
				\end{equation}
				So \eqref{eq:contract_dtq}, together with \eqref{est:icontract_2}--\eqref{est:icontract_3} implies that
				\begin{equation}\label{est:diff_dtq}
					\begin{aligned}
						\|\p_tq(0)\|_{H^0}&\lesssim\|D_tv(0)\|_{H^1}+\|\p_t\xi(0)\|_{H^{3/2}}+\delta^{1/2}(\|\p_t\eta(0)\|_{H^{3/2}}+\|\eta_0\|_{W^{3-1/q_+,q_+}}\\
						&\quad+\|D_tu(0)\|_{H^1}+\|u_0\|_{W^{2,q_+}}+\|v_0\|_{W^{2,q_+}}+\|q_0\|_{W^{1,q_+}}+\|\xi_0\|_{W^{3-1/q_+,q_+}}).
					\end{aligned}
				\end{equation}
				
				We now estimate the difference $\p_t\xi(0)$. The weak formulation for \eqref{eq:diff_initial1} might be given as
				\begin{equation}\label{eq:weak_initial_dtxi}
					\begin{aligned}
						&(\p_t\xi(0), \varphi\cdot\mathcal{N}^1(0))_{1,\Sigma}+\frac\mu2\int_\Om(\mathbb{D}_{\mathcal{A}_0^1}(D_tu(0)):\mathbb{D}_{\mathcal{A}_0^1}w)J^1(0)\\
						&=\int_\Om \mathfrak{g}^1\cdot J^1(0)\varphi-\int_{-\ell}^\ell \mathfrak{g}^3\p_1(\varphi\cdot\mathcal{N}^1(0))+\mathfrak{g}^4\cdot \varphi-\int_{\Sigma_s}\mathfrak{g}^5(\varphi\cdot\tau)J^1(0)
					\end{aligned}
				\end{equation}
				for each $\varphi\in \mathcal{W}_\sigma(\Om)$. Then the dissipation theory for the surface function developed in \cite{GT18} can  be directly applied to \eqref{eq:weak_initial_dtxi} to obtain the estimate
				\begin{equation}\label{est:diff_dtxi}
					\begin{aligned}
						\|\p_t\xi(0)\|_{H^{3/2}}&\lesssim \|D_tv(0)\|_{H^1}+\delta^{1/2}(\|\p_t\eta(0)\|_{H^{3/2-\alpha}}+\|\eta_0\|_{W^{3-1/q_+,q_+}}+\|D_tu(0)\|_{H^1}\\
						&\quad+\|u_0\|_{W^{2,q_+}}+\|v_0\|_{W^{2,q_+}}+\|q_0\|_{W^{1,q_+}}+\|\xi_0\|_{W^{3-1/q_+,q_+}}).
					\end{aligned}
				\end{equation}
				
				In order to obtain the estimate \eqref{est:initial_contract_4}, we plug \eqref{est:diff_dtq} and \eqref{est:diff_dtxi} into \eqref{est:diff_dt_u} to obtain
				\begin{equation}\label{est:diff_dt_u1}
					\begin{aligned}
						\|D_tv(0)\|_{H^1}^2&\le C\delta(\|D_tu(0)\|_{H^1}^2+\|u_0\|_{W^{2,q_+}}^2+\|\eta_0\|_{W^{3-1/q_+,q_+}}^2\\
						&\quad+\|\p_t\eta(0)\|_{H^{3/2}}^2)+C\delta(\|v_0\|_{W^{2,q_+}}^2+\|q_0\|_{W^{1,q_+}}^2+\|\xi_0\|_{W^{3-1/q_+,q_+}}^2).
					\end{aligned}
				\end{equation}
				for $\delta$ sufficiently small. Then we plunge \eqref{est:diff_dt_u1} into \eqref{est:diff_dtxi} to bound $(\p_t\xi^1(0)-\p_t\xi^2(0))$ by the right-hand side of inequality in \eqref{est:initial_contract_4}. Finally, we plunge \eqref{est:diff_dt_u1} and the resulting estimate for $(\p_t\xi^1(0)-\p_t\xi^2(0))$ into \eqref{est:diff_dtq} to bound the $\p_tq^1(0)-\p_tq^2(0)$ to obtain the estimate  \eqref{est:initial_contract_4}.
			\end{proof}
			
			In order to get the contraction result, we have to estimate the difference of \eqref{eq:diff_1}.
			\begin{lemma}\label{lem:contract_2}
				The triple $(v_0, q_0, \xi_0)$ in \eqref{eq:diff_1} obeys the estimate
				\[
				\begin{aligned}
					&\|v_0\|_{W^{2,q_+}}^2+\|q_0\|_{W^{1,q_+}}^2+\|\xi_0\|_{W^{3-1/q_+,q_+}}^2\lesssim \|D_tv(0)\|_{L^2}^2+\delta(\|\eta_0\|_{W^{3-1/q_+,q_+}}^2+\|\p_t\eta(0)\|_{H^{3/2}}^2+\|u_0\|_{W^{2,q_+}}^2).
				\end{aligned}
				\]
			\end{lemma}
			\begin{proof}
				The differences defined in \eqref{eq:diff_1} satisfy equations \eqref{eq:diff_initial1} for j=0, where the forcing terms $\mathfrak{g}_{0}^{j}$ are given below. To apply the elliptic theory in \cite[Theorem 4.9]{GT2020}, we need to guarantee that the assumptions in the elliptic estimates are meaningful. We employ the H\"older's inequality, Sobolev embedding theory and usual trace theory to deduce the following estimates.
				\begin{equation}\label{est:diff_u1}
					\begin{aligned}
						\|(R(0)^1-R(0)^2)v_0^2\|_{L^{q_+}}\lesssim \delta^{1/2}\|\p_t\eta(0)\|_{H^{3/2+(\varepsilon_--\alpha)/2}},
					\end{aligned}
				\end{equation}
				\begin{equation}
					\begin{aligned}
						\|\nabla_{\mathcal{A}_0^1-\mathcal{A}_0^2}q_0^2\|_{L^{q_+}}+\|\dive_{\mathcal{A}_0^1-\mathcal{A}_0^2}\mathbb{D}_{\mathcal{A}_0^1}v_0^2\|_{L^{q_+}}\lesssim \delta^{1/2}\|\eta_0\|_{W^{3-1/q_+,q_+}},
					\end{aligned}
				\end{equation}
				\begin{equation}
					\|\dive_{\mathcal{A}_0^2}\mathbb{D}_{\mathcal{A}_0^1-\mathcal{A}_0^2}u_0^2\|_{L^{q_+}}\lesssim \delta^{1/2}\|\eta_0\|_{W^{3-1/q_+,q_+}},
				\end{equation}
				\begin{equation}
					\|\dive_{\mathcal{A}_0^1-\mathcal{A}_0^2}v_0^2\|_{W^{1,q_+}}\lesssim \delta^{1/2}\|\eta_0\|_{W^{3-1/q_+,q_+}},
				\end{equation}
				\begin{equation}
					\begin{aligned}
						\|\mathfrak{f}^1\|_{L^{q_+}}\lesssim \delta^{1/2} (\|\eta_0\|_{W^{3-1/q_+,q_+}}+\|\p_t\eta(0)\|_{H^{3/2}}
						+\|u_0\|_{W^{2,q_+}}),
					\end{aligned}
				\end{equation}
				\begin{equation}\label{est:diff_u2}
					\|v_0^2\cdot[\mathcal{N}^1(0)-\mathcal{N}^2(0)]\|_{W^{2-1/q_+, q_+}(\Sigma)}\lesssim\delta^{1/2}\|\eta_0\|_{W^{3-1/q_+,q_+}}.
				\end{equation}
				Then by the elliptic theory and the above estimates \eqref{est:diff_u1}--\eqref{est:diff_u2}, with $\delta$ sufficiently small,
				\[
				\begin{aligned}
					&\|v_0\|_{W^{2,q_+}}^2+\|q_0\|_{W^{1,q_+}}^2+\|\xi_0\|_{W^{3-1/q_+,q_+}}^2\lesssim \|D_tv(0)\|_{L^2}^2+\delta(\|\eta_0\|_{W^{3-1/q_+,q_+}}^2+\|\p_t\eta(0)\|_{H^{3/2}}^2+\|u_0\|_{W^{2,q_+}}^2).
				\end{aligned}
				\]
				
			\end{proof}
			
			The Lemmas \ref{lem:contract_1} and \ref{lem:contract_2} give the initial data for $\varepsilon$- nonlinear system \eqref{eq:geometric}.
			
			\begin{proof}[Proof of Theorem \ref{thm:initial}.]
				The results in Lemma \ref{lem:contract_1} and Lemma \ref{lem:contract_2} imply the estimates
				\[
				\begin{aligned}
					&\|v_0^1-v_0^2\|_{W^{2,q_+}}^2+\|q_0^1-q_0^2\|_{W^{1,q_+}}^2+\|\xi_0^1-\xi_0^2\|_{W^{3-1/q_+,q_+}}^2+\|D_tv^1(0)-D_tv^2(0)\|_{H^1}^2\\
					&\quad+\|\p_tq^1(0)-\p_tq^2(0)\|_{H^0}^2+\|\p_t\xi^1(0)-\p_t\xi^2(0)\|_{H^{3/2}}^2\\
					&\le \frac12(\|D_tu^1(0)-D_tu^2(0)\|_{H^1}^2+\|u_0^1-u_0^2\|_{W^{2,q_+}}^2+\|\eta_0^1-\eta_0^2\|_{W^{3-1/q_+,q_+}}^2+\|\p_t\eta^1(0)-\p_t\eta^2(0)\|_{H^{3/2-\alpha}}^2)
				\end{aligned}
				\]
				for $\delta$ sufficiently small.
				
				Then by the Banach's fixed point theory, given $D_t^2u(0)$ and $\p_t^2\eta(0)$ satisfying the conditions in assumptions in Theorem \ref{thm:initial}, there exist one unique $(u_0, p_0, \eta_0, D_tu(0), \p_tp(0), \p_t\eta(0))$ satisfying the full initial conditions.
			\end{proof}
			
			Denote 
			\[
			\begin{aligned}
				\mathscr{F}_j(\psi)=\int_{\Om}F^{1,j}(0)\cdot \psi J(0)-\int_{-\ell}^{\ell} F^{3,j}(0)\p_1(\psi\cdot\mathcal{N}(0))+F^{4,j}(0)\cdot\psi -\int_{\Sigma_s}F^{5,j}(0)(\psi\cdot\tau)J(0)\\-[F^{7,j}(0),\psi\cdot\mathcal{N}(0)]_\ell
			\end{aligned}
			\]
			for $j=0, 1$.
			The weak solution to \eqref{eq:geometric} at $t=0$ for j=0 satisfies
			\begin{equation}\label{eq:couple5}
				\begin{aligned}
					(D_tu(0),\psi)_{\mathcal{H}^0}+((u_0,\psi))+[u_0\cdot\mathcal{N},\psi\cdot\mathcal{N}]_\ell-(p_0, \dive_{\mathcal{A}_0}\psi)_{\mathcal{H}^0}
					+(Ru_0,\psi)_{\mathcal{H}^0}=\mathscr{F}_1(\psi)-(\xi(0),\psi\cdot\mathcal{N})_{1,\Sigma}
				\end{aligned}
			\end{equation}
			for each $\psi\in\mathcal{W}(0)$.
			
			So, the pressureless weak solutions to \eqref{eq:geometric} at $t=0$ with $j=0$ are supposed to satisfy the equation
			\begin{equation}\label{eq:initial_p}
				\begin{aligned}
					(D_tu(0),\psi)_{\mathcal{H}^0}+((u_0,\psi))+[u_0\cdot\mathcal{N},\psi\cdot\mathcal{N}]_\ell
					+(Ru_0,\psi)_{\mathcal{H}^0}=\mathscr{F}_1(\psi)-(\xi(0),\psi\cdot\mathcal{N})_{1,\Sigma}
				\end{aligned}
			\end{equation}
			for each $\psi\in\mathcal{W}_\sigma(0)$.
			The pressureless weak formulation at $t=0$ with $j=1$  is expressed as
			\begin{equation}\label{eq:couple2}
				\begin{aligned}
					(D_t^2u(0),\psi)_{\mathcal{H}^0}
					+((D_tu(0),\psi))+[D_tu(0)\cdot\mathcal{N}(0),\psi\cdot\mathcal{N}(0)]_\ell-(\p_tp(0), \dive_{\mathcal{A}_0}\psi)_{\mathcal{H}^0}\\
					+2(RD_tu(0),\psi)_{\mathcal{H}^0}=\mathscr{F}_2(\psi)-(\p_t\xi(0),\psi\cdot\mathcal{N}(0))_{1,\Sigma}
				\end{aligned}
			\end{equation}
			for each $\psi\in\mathcal{W}(0)$.
			So, the pressureless weak solutions to \eqref{eq:geometric} at $t=0$ with $j=1$ satisfy
			\begin{equation}\label{eq:initial_p2}
				\begin{aligned}
					(D_t^2u(0),\psi)_{\mathcal{H}^0}
					+((D_tu(0),\psi))+[D_tu(0)\cdot\mathcal{N}(0),\psi\cdot\mathcal{N}(0)]_\ell\\
					+2(RD_tu(0),\psi)_{\mathcal{H}^0}=\mathscr{F}_2(\psi)-(u_0\cdot\mathcal{N}(0),\psi\cdot\mathcal{N}(0))_{1,\Sigma}
				\end{aligned}
			\end{equation}
			for each $\psi\in\mathcal{W}_\sigma(0)$.

			%%%%%%%%%%%%%%%%%%%%%%%%%%%%%%%%%%%%%%%%%%%%%%
			\subsection{Initial Forcing Terms for Contraction}\label{sec:initial_forcing}
			%%%%%%%%%%%%%%%%%%%%%%%%%%%%%%%%%%%%%%%%%%%%%%

			We give the forcing terms in the system of \eqref{eq:diff_initial1}. For $k=0$,
			\[
			\begin{aligned}
				\mathfrak{g}^1_0&=-(R(0)^1-R(0)^2)v_0^2-\nabla_{\mathcal{A}_0^1-\mathcal{A}_0^2}q_0^2+\dive_{\mathcal{A}_0^1-\mathcal{A}_0^2}\mathbb{D}_{\mathcal{A}_0^1}v_0^2+\dive_{\mathcal{A}_0^2}\mathbb{D}_{\mathcal{A}_0^1-\mathcal{A}_0^2}v_0^2+\p_tF^1(\eta^1_0, u^1_0)- \p_tF^1(\eta^2_0, u^2_0),\\
				\mathfrak{g}^2_0&=-\dive_{\mathcal{A}_0^1-\mathcal{A}_0^2}u^2(0),\\
				\mathfrak{g}^3_0&=\mathcal{R}(\p_1\zeta_0, \p_1\eta^1(0))-\mathcal{R}(\p_1\zeta_0, \p_1\eta^2(0)),\\
				\mathfrak{g}^4_0&=-\mu\mathbb{D}_{\mathcal{A}_0^1-\mathcal{A}_0^2}v_0^2\mathcal{N}^1(0)-\mu\mathbb{D}_{\mathcal{A}_0^2}v_0^2(\mathcal{N}^1(0)-\mathcal{N}^2(0))\\
				&\quad+\bigg[g\xi_0^2-\sigma\p_1\bigg(\frac{\p_1\xi_0^2}{(1+|\p_1\zeta_0|^2)^{3/2}}\bigg)-\p_1 \mathcal{R}(\p_1\zeta_0,\p_1\eta^2(0))\bigg][\mathcal{N}^1(0)-\mathcal{N}^2(0)],\\
				\mathfrak{g}^5_0&=\mu\mathbb{D}_{\mathcal{A}_0^1-\mathcal{A}_0^2}u^2(0)\nu\cdot\tau,\\
				\mathfrak{g}^6_0&=-u(0)\cdot[\mathcal{N}^1(0)-\mathcal{N}^1(0)],\\
				\mathfrak{g}^7_0&=v_0^2\cdot(\mathcal{N}^1(0)-\mathcal{N}^2(0))-\kappa[\hat{\mathscr{W}}(\p_t\eta^1)-\hat{\mathscr{W}}(\p_t\eta^2)].
			\end{aligned}
			\]
			For $k=1$,
			\[
			\begin{aligned}
				\mathfrak{g}^1_1&=-2(R^1(0)-R^2(0))D_tu^2(0)+\mu\dive_{\mathcal{A}^1(0)}\mathbb{D}_{\mathcal{A}^1(0)-\mathcal{A}^2(0)}D_tu^2(0)+\mu\dive_{\mathcal{A}^1(0)-\mathcal{A}^2(0)}\mathbb{D}_{\mathcal{A}^2_0}D_tu^2(0)\\
				&\quad-\mu\dive_{\mathcal{A}^1(0)}\mathbb{D}_{\mathcal{A}^1_0}R^1(0)u_0-\mu\dive_{\mathcal{A}^1(0)}\mathbb{D}_{\mathcal{A}^1_0}[R^1(0)-R^2(0)]u^2_0-\mu\dive_{\mathcal{A}^1_0}\mathbb{D}_{\mathcal{A}^1_0-\mathcal{A}^2_0}R^2(0)u^2_0\\
				&\quad-\mu\dive_{\mathcal{A}^1_0-\mathcal{A}^2_0}\mathbb{D}_{\mathcal{A}^2_0}R^2(0)u^2_0+\mu\dive_{\p_t\mathcal{A}^1(0)}(\mathbb{D}_{\mathcal{A}^1_0}u_0)+\mu\dive_{\p_t\mathcal{A}^1(0)}(\mathbb{D}_{\mathcal{A}^1_0-\mathcal{A}^2_0}u^2_0)\\
				&\quad+\mu\dive_{(\p_t\mathcal{A}^1(0)-\p_t\mathcal{A}^2(0))}(\mathbb{D}_{\mathcal{A}^2_0}u^2_0)+\mu\dive_{(\mathcal{A}^1_0-\mathcal{A}^2_0)}(\mathbb{D}_{\p_t\mathcal{A}^2(0)}u^2_0)+\mu\dive_{\mathcal{A}^1_0}(\mathbb{D}_{\p_t\mathcal{A}^1(0)-\p_t\mathcal{A}^2(0)}Du^2(0))\\
				&\quad+\mu\dive_{\mathcal{A}^1_0}(\mathbb{D}_{\p_t\mathcal{A}^2(0)}u_0)-\nabla_{\p_t\mathcal{A}^1(0)}p_0-\nabla_{(\p_t\mathcal{A}^1(0)-\p_t\mathcal{A}^2(0))}p^2_0-\nabla_{(\mathcal{A}^1-\mathcal{A}^2)}\p_tp^2(0)-R^1(0))^2u_0\\
				&\quad-[R^1(0)-R^2(0)][R^1(0)+R^2(0)]u^2_0+(\p_t^2\bar{\eta}^1(0)-\p_t^2\bar{\eta}^2(0))K^1(0)W\p_2u^1_0\\
				&\quad+\p_t^2\bar{\eta}^2(0)(K^1(0)-K^2(0))W\p_2v^1_0+\p_t^2\bar{\eta}^2(0)K^2(0)W\p_2v_0+(\p_t\bar{\eta}^1(0)-\p_t\bar{\eta}^2(0))\p_tK^1(0)W\p_2v^1_0\\
				&\quad+\p_t\bar{\eta}^2(0)(\p_tK^1(0)-\p_tK^2(0))W\p_2v^1_0+\p_t\bar{\eta}^2(0)\p_tK^2(0)W\p_2v_0+\p_t\bar{\eta}^2(0)K^2(0)W\p_2D_tv(0)\\
				&\quad+\p_t\bar{\eta}^2(0)(K^1(0)-K^2(0))W\p_2D_tv^1(0)+(\p_t\bar{\eta}^1(0)-\p_t\bar{\eta}^2(0))K^1(0)W\p_2D_tv^1(0)\\
				&\quad+(\p_t\bar{\eta}^1(0)-\p_t\bar{\eta}^2(0))K^1(0)W\p_2(R^1(0)v^1_0)+\p_t\bar{\eta}^2(0)(K^1(0)-K^2(0))W\p_2(R^1(0)v^1_0)\\
				&\quad+\p_t\bar{\eta}^2(0)K^2(0)W\p_2(R^1(0)v_0)+\p_t\bar{\eta}^2(0)K^2(0)W\p_2[(R^1(0)-R^2(0))v^2_0]\\
				&\quad-D_tv(0)\cdot\nabla_{\mathcal{A}^1_0}v^1_0- D_tv^2(0)\cdot\nabla_{\mathcal{A}^1_0-\mathcal{A}^2_0}v^1_0-D_tv^2(0)\cdot\nabla_{\mathcal{A}^2_0}v_0-R^1(0)v_0\cdot\nabla_{\mathcal{A}^1_0}v^1_0- R^2(0)v^2_0\cdot\nabla_{\mathcal{A}^1_0-\mathcal{A}^2_0}v^1_0\\
				&\quad-R^2(0)v^2_0\cdot\nabla_{\mathcal{A}^2_0}v_0-[R^1(0)-R^2(0)]v^2_0\cdot\nabla_{\mathcal{A}^1_0}v^1_0-v_0\cdot\nabla_{\p_t\mathcal{A}^1(0)}v^1_0- v^2_0\cdot\nabla_{\p_t\mathcal{A}^1(0)-\p_t\mathcal{A}^2(0)}v^1_0\\
				&\quad-v^2_0\cdot\nabla_{\p_t\mathcal{A}^2(0)}v_0-v_0\cdot\nabla_{\mathcal{A}^1(0)}D_tv^1(0)- v^2_0\cdot\nabla_{\mathcal{A}^1(0)-\mathcal{A}^2(0)}D_tv^1(0)-u^2_0\cdot\nabla_{\mathcal{A}^2(0)}D_tv(0)\\
				&\quad-v_0\cdot\nabla_{\mathcal{A}^1_0}(R^1(0)v^1_0)- v^2_0\cdot\nabla_{\mathcal{A}^1_0-\mathcal{A}^2_0}(R^1(0)v^1_0)-u^2_0\cdot\nabla_{\mathcal{A}^2_0}(R^1(0)v_0)-v^2_0\cdot\nabla_{\mathcal{A}^1_0}[R^1(0)-R^2(0)]v^2_0.
			\end{aligned}
			\]
			\[
			\begin{aligned}
				\mathfrak{g}^2_1&=-\dive_{\mathcal{A}_0^1-\mathcal{A}_0^2}D_tu^2(0),\\
				\mathfrak{g}^3_1&=\p_z\mathcal{R}(\p_1\zeta_0,\p_1\eta^1(0))\p_1\p_t\eta^1(0)-\p_z\mathcal{R}(\p_1\zeta_0,\p_1\eta^2(0))\p_1\p_t\eta^2(0),
			\end{aligned}
			\]
			\[
			\begin{aligned}
				\mathfrak{g}^4_1&=-S_{\mathcal{A}^1(0)}(\p_tp^2(0), D_tu^2(0))[\mathcal{N}^1(0)-\mathcal{N}^2(0)]+\mu\mathbb{D}_{\mathcal{A}_0^1}D_tu^2(0)(\mathcal{N}^1(0)-\mathcal{N}^2(0))\\
				&\quad+\mu\mathbb{D}_{\mathcal{A}_0^1-\mathcal{A}_0^2}D_tu^2(0)\mathcal{N}^2(0)+\mu\mathbb{D}_{\mathcal{A}_0^1}R^1(0)u_0\mathcal{N}^1(0)+\mu\mathbb{D}_{\mathcal{A}_0^1}[R^1(0)-R^2(0)]u_0^2\mathcal{N}^1(0)\\
				&\quad+\mu\mathbb{D}_{\mathcal{A}_0^1-\mathcal{A}_0^2}R^2(0)u^2_0\mathcal{N}^2(0)+\mu\mathbb{D}_{\mathcal{A}_0^1}R^2(0)u^2_0(\mathcal{N}^1(0)-\mathcal{N}^2(0))+\mu\mathbb{D}_{\p_t\mathcal{A}^1(0)}u_0\mathcal{N}^1(0)\\
				&\quad+\mu\mathbb{D}_{\p_t\mathcal{A}^1(0)-\p_t\mathcal{A}^2(0)}u^2_0\mathcal{N}^2(0)+\mu\mathbb{D}_{\p_t\mathcal{A}^1(0)}u^2_0(\mathcal{N}^1(0)-\mathcal{N}^2(0))-S_{\mathcal{A}^1_0}(p_0, u_0)\p_t\mathcal{N}^1(0)\\
				&\quad-S_{\mathcal{A}^1_0}(p^2_0, u^2_0)[\p_t\mathcal{N}^1(0)-\p_t\mathcal{N}^2(0)]+\mu\mathbb{D}_{\mathcal{A}_0^1-\mathcal{A}_0^2}u^2_0\p_t\mathcal{N}^2(0)\\
				&\quad+\bigg[g\p_t\xi^2(0)-\sigma\p_1\bigg(\frac{\p_1\p_t\xi^2(0)}{(1+|\p_1\zeta_0|^2)^{3/2}}\bigg)-\p_1 \p_z\mathcal{R}\p_1\p_t\eta^2(0)\bigg][\mathcal{N}^1(0)-\mathcal{N}^2(0)]\\
				&\quad+\bigg[g\xi^2(0)-\sigma\p_1\bigg(\frac{\p_1\xi^2(0)}{(1+|\p_1\zeta_0|^2)^{3/2}}\bigg)-\p_1 \mathcal{R}\bigg][\p_t\mathcal{N}^1(0)-\p_t\mathcal{N}^2(0)]\\
				&\quad+\bigg[g\xi_0-\sigma\p_1\bigg(\frac{\p_1\xi_0}{(1+|\p_1\zeta_0|^2)^{3/2}}\bigg)-\p_1 \mathcal{R}(\p_1\zeta_0,\p_1\eta^1_0+\p_1\mathcal{R}(\p_1\zeta_0,\p_1\eta^2_0)\bigg]\p_t\mathcal{N}^1(0),
			\end{aligned}
			\]
			\[
			\begin{aligned}
				\mathfrak{g}^5_1&=\mu\mathbb{D}_{\mathcal{A}_0^1-\mathcal{A}_0^2}D_tu^2(0)\nu\cdot\tau+\mu\mathbb{D}_{\mathcal{A}_0^1}R^1(0)u_0\nu\cdot\tau+\mu\mathbb{D}_{\mathcal{A}_0^1}[R^1(0)-R^2(0)]u^2_0\nu\cdot\tau\\
				&\quad+\mu\mathbb{D}_{\mathcal{A}_0^1-\mathcal{A}_0^2}R^2(0)u_0^2\nu\cdot\tau+\beta R^1(0)u_0\cdot\tau+\beta[R^1(0)-R^2(0)]u^2_0\cdot\tau,\\
				\mathfrak{g}^6_1&=-D_tu^2(0)\cdot[\mathcal{N}^1(0)-\mathcal{N}^1(0)],\\
				\mathfrak{g}^7_1&=\kappa D_tu^2(0)\cdot[\mathcal{N}^1(0)-\mathcal{N}^2(0)](\pm\ell)+\kappa\p_t\hat{\mathscr{W}}(\p_t\eta^1)(\pm\ell)-\kappa\p_t\hat{\mathscr{W}}(\p_t\eta^2)(\pm\ell).
			\end{aligned}
			\]
			
			\makeatletter
			\renewcommand \theequation {%
				C.%
				%\ifnum \c@section>\z@ \@arabic\c@section.%
				%\fi
				% \ifnum\c@subsection>\z@\@arabic\c@subsection.%
				%\fi\ifnum \c@subsubsection>\z@\@arabic\c@subsubsection.
				% \fi
				\@arabic\c@equation} \@addtoreset{equation}{section}
			% \@addtoreset{equation}{subsection}
			\makeatother
			
			\section{Construction of Smooth Sequences}\label{sec:smooth}
			
			In this section, we construct the approximating sequence in the linear estimate. 
			Recall that in equation \eqref{eq:quasi_linear_{s}}, we only smooth the functions on the free surface.
			
			We construct the extension map
			$
			E: H^s(-\ell,\ell)\to H^{s}(\mathbb{R})$,
			explicitly as follows
			\[
			(Eu):=
			\begin{cases}
				u(x)~~~~~&|x|\leq\ell,\\
				\eta(x-\ell)(6u(2\ell-x)-8u(3\ell-2x)+6u(4\ell-3x))~~~~~~&x>\ell,\\
				\eta(x+\ell)(6u(-2\ell-x)-8u(-3\ell-2x)+6u(-4\ell-3x))~~~~~&x<-\ell,
			\end{cases}
			\]
			where the function $\eta\in C_{c}^{\infty}(-\frac{\ell}{8},\frac{\ell}{8})$ satisfying
			$
			\eta(x)=1\ \text{for}\ |x|\leq \frac{\ell}{24}$,
			$\eta(x)=0~\operatorname{for}~|x|\geq \frac{\ell}{12}$,
			and
			$\eta(x)\in [0,1]$ for any $x$.
			\noindent The definition of $E$ above provides a bounded extension operator that is independent of time $t$ for all $0\le s\le 3$. 
			
			Let $K(x)$ be a smooth function defined by
			\[
			K(x):=
			\begin{cases}
				C_{K}\exp(\frac{1}{x^{2}-1})\quad\quad\quad\quad&|x|<1,\\
				0\quad\quad\quad\quad\quad\quad\quad\quad\quad&|x|\geq1.
			\end{cases}
			\]
			\noindent where $C_{K}$ is a constant such that $
			\int_{\mathbb{R}}K(x)dx=1$.
			
			Using the extension operator and mollifier introduced above, we define the following smooth approximating sequence for any $T>0$ and each function $f(t,x)\in L_{t}^{2}(0,T;L^{2}(\Sigma))$
			\begin{align}{\label{def:smooth_0}}
				\begin{aligned}
					f^{\varepsilon}(t,x):=\frac{C_{K}}{\varepsilon}\int_{-\infty}^{\infty}K\big(\frac{x-y}{\varepsilon}\big)Ef(t,y)dy.
				\end{aligned}
			\end{align}
			
			By the standard properties of mollification and the boundedness of the extension operator, if $f\in L^{p}(0,T; W^{i,q}(\Sigma))$ for any $p,q\in [1,\infty]$ and $i\in (1,\infty)$, we have the following convergence results:
			\begin{align}{\label{eq:convergence_e}}
				\begin{aligned}
					\lim_{\varepsilon\rightarrow 0}f^{\varepsilon}(t,x)=f(t,x)~~\operatorname{in}~~L^{p}(0,T; W^{i,q}(\Sigma)).
				\end{aligned}
			\end{align}
			\noindent Moreover, the following regularity properties hold for the smooth modification function $f^{\varepsilon}$
			\begin{align}{\label{eq:smooth_e}}
				\begin{aligned}
					f^{\varepsilon} \in& L_{t}^{\infty}(0,T; C^{\infty}(\Sigma)),
				\end{aligned}
			\end{align}
			
			\makeatletter
			\renewcommand \theequation {%
				D.%
				%\ifnum \c@section>\z@ \@arabic\c@section.%
				%\fi
				% \ifnum\c@subsection>\z@\@arabic\c@subsection.%
				%\fi\ifnum \c@subsubsection>\z@\@arabic\c@subsubsection.
				% \fi
				\@arabic\c@equation} \@addtoreset{equation}{section}
			% \@addtoreset{equation}{subsection}
			\makeatother
			
			\section{Initial Data for the Linear System}{\label{sec:initial_l}}
			
			\subsection{Initial Data for System with Smooth Surface Function \texorpdfstring{$\eta^{n}$}{}}
			For any prescribed $n$, we establish the initial data for $(v,\xi,q)$ which is a solution to the following system
			
			\begin{equation}{\label{eq:quasi_linear_{sn}}}
				\begin{cases}
					\partial_{t}^{2}v+\operatorname{div}_{\mathcal{A}^{n}}S_{\mathcal{A}^{n}}(v,q)=F^{1}(u,p,\eta)~~~&\operatorname{in}~~\Omega,\\
					\operatorname{div}_{\mathcal{A}^{n}}D_{t}v=0~~~&\operatorname{in}~~\Omega,\\
					S_{\mathcal{A}^{n}}(\p_{t}q,\p_{t}v)\mathcal{N}^{n}=g\xi_{d}\mathcal{N}^{n}-\sigma\partial_{1}(\frac{1}{(1+\vert \partial_{1}\zeta_{0}\vert^{2})^{3/2}}\partial_{1}\xi_{d})\mathcal{N}^{n}
					+\sigma\partial_{1}(\mathcal{R}_{z}(\partial_{1}\zeta_{0},\partial_{1}\eta)(\partial_{1}\p_{t}\xi))\mathcal{N}^{n}\\\quad\quad\quad\quad\quad\quad\quad+F^{4}(u,p,\eta)~~~&\operatorname{on}~~\Sigma,\\
					(S_{\mathcal{A}^{n}}(\p_{t}q,\p_{t}v)\nu-\beta \p_{t}v)\cdot \tau=F^{5}(u,\eta,p)~~~&\operatorname{on}~~\Sigma_{s},\\
					\p_{t}v\cdot \nu=0~~~&\operatorname{on}~~\Sigma_{s},\\
					\partial_{t}^{2}\xi=\p_{t}v\cdot \mathcal{N}^{n}+u_{1}\cdot\partial_{t}\p_{1}\xi~~~&\operatorname{on}~~\Sigma\\
					\sigma(\mp \frac{\partial_{1}\p_{t}\xi}{(1+\vert \partial_{1}\zeta_{0}\vert^{2})^{\frac{3}{2}}}\pm \mathcal{R}_{z}(\partial_{1}\zeta_{0},\partial_{1}\eta)\partial_{t}\partial_{1}\xi)(\pm \ell))(\pm\ell)=\kappa (v_{d}\cdot \mathcal{N}^{n})(\pm \ell)-{F}^{7}.
				\end{cases}
			\end{equation}
			
			For any prescribed $n$, let $\p_{t}^{2}\xi(0)=\p_{t}^{2}\eta(0)$ and $D_{t}^{2}v(0)=D_{t}^{2}u(0)$ which are the initial functions constructed in Appendix \ref{sec:initial}. We construct the initial conditions by solving the following system. Let $(v(0),\xi(0),q(0))$, and $(D_{t}v(0),\p_{t}\xi(0),\p_{t}q(0))$ be the solution to the following systems
			\begin{align}{\label{eq:smooth_n0}}
				\begin{cases}
					D_{t}v_{0}+R^{n}(0)v_{0}+\operatorname{div}_{\mathcal{A}^{n}(0)}S_{\mathcal{A}^{n}(0)}(v_{0},q_{0})=F^{1,0}(u_{0},p_{0},\eta_{0})~~~&\operatorname{in}~~\Omega,\\
					\operatorname{div}_{\mathcal{A}^{n}(0)}v_{0}=0~~~&\operatorname{in}~~\Omega,\\
					S_{\mathcal{A}^{n}(0)}(q_{0},v_{0})\mathcal{N}^{n}(0)=g\xi_{0}\mathcal{N}^{n}(0)-\sigma\partial_{1}(\frac{1}{(1+\vert \partial_{1}\zeta_{0}\vert^{2})^{3/2}}\partial_{1}\xi_{0})\mathcal{N}^{n}(0)
					+\sigma\partial_{1}(\mathcal{R}(\partial_{1}\zeta_{0},\partial_{1}\eta_{0}))\mathcal{N}^{n}(0)~~~&\operatorname{on}~~\Sigma,\\
					(S_{\mathcal{A}^{n}(0)}(q_{0},v_{0})\nu-\beta v_{0})\cdot \tau=0~~~&\operatorname{on}~~\Sigma_{s},\\
					v_{0}\cdot \nu=0~~~&\operatorname{on}~~\Sigma_{s},\\
					\partial_{t}\xi_{0}=v_{0}\cdot \mathcal{N}^{n}(0)~~~&\operatorname{on}~~\Sigma,\\
					\sigma(\mp \frac{\partial_{1}\xi_{0}}{(1+\vert \partial_{1}\zeta_{0}\vert^{2})^{\frac{3}{2}}}\pm \mathcal{R}_{z}(\partial_{1}\zeta_{0},\partial_{1}\eta_{0})\partial_{1}\xi_{0})(\pm \ell))(\pm\ell)=\kappa (v_{0}\cdot \mathcal{N}^{n}(0))(\pm \ell)-{F}^{7,0},
				\end{cases}
			\end{align}
			\noindent and
			\begin{align}{\label{eq:smooth_n1}}
				\begin{cases}
					D_{t}^{2}v_{0}+\operatorname{div}_{\mathcal{A}^{n}(0)}S_{\mathcal{A}^{n}(0)}(D_{t}v_{0}+R^{n}(0)v_{0},\p_{t}q_{0})\\
					=\p_{t}F^{1,0}(u_{0},p_{0},\eta_{0})+\mathfrak{G}^{1}(u_{0},p_{0})-R^{n}(0)D_{t}v_{0}-\p_{t}(R^{n}(0)v_{0})~~~&\operatorname{in}~~\Omega,\\
					\operatorname{div}_{\mathcal{A}^{n}(0)}D_{t}v_{0}=0~~~&\operatorname{in}~~\Omega,\\
					S_{\mathcal{A}^{n}(0)}(\p_{t}q_{0},D_{t}v_{0}+R^{n}(0)v_{0})\mathcal{N}^{n}(0)=g\p_{t}\xi_{0}\mathcal{N}^{n}-\sigma\partial_{1}(\frac{1}{(1+\vert \partial_{1}\zeta_{0}\vert^{2})^{3/2}}\partial_{1}\p_{t}\xi_{0})\mathcal{N}^{n}(0)
					\\
					~~~~~~~~~~~~~~~~~~~~~~~~~~~~~~~~~~~~~~~~~~~~~~~~+\sigma\partial_{1}(\mathcal{R}(\partial_{1}\zeta_{0},\partial_{1}\eta_{0})(\partial_{1}\p_{t}\xi_{0}))\mathcal{N}^{n}(0)+\mathfrak{G}^{4}(u_{0},p_{0},\eta_{0})~~~&\operatorname{on}~~\Sigma,\\
					(S_{\mathcal{A}^{n}(0)}(\p_{t}q_{0},D_{t}v_{0}+R^{n}(0)v_{0})\nu-\beta \p_{t}v_{0})\cdot \tau=\mathfrak{G}^{5}(u_{0},p_{0},\eta_{0})~~~&\operatorname{on}~~\Sigma_{s},\\
					D_{t}v_{0}\cdot \nu=0~~~&\operatorname{on}~~\Sigma_{s},\\
					\partial_{t}^{2}\xi_{0}=D_{t}v_{0}\cdot \mathcal{N}^{n}(0)+R^{n}(0)v_{0}\cdot \mathcal{N}^{n}(0)-u_{0}\p_{t}\p_{1}\xi_{0}-\frac{1}{2\ell}\int_{-\ell}^{\ell}(R^{n}(0)v_{0}\cdot \mathcal{N}^{n}(0)-u_{0}\p_{t}\p_{1}\xi_{0})~~~&\operatorname{on}~~\Sigma,\\
					\sigma(\mp \frac{\partial_{1}\p_{t}\xi_{0}}{(1+\vert \partial_{1}\zeta_{0}\vert^{2})^{\frac{3}{2}}}\pm \mathcal{R}_{z}(\partial_{1}\zeta_{0},\partial_{1}\eta_{0})\partial_{1}\p_{t}\xi_{0})ds)(\pm\ell)=\kappa (D_{t}v_{0}\cdot \mathcal{N}^{n}(0))(\pm \ell)-\p_{t}{F}^{7,0}.
				\end{cases}
			\end{align}
			
			Remark: Because the initial data of the smooth sequence do not satisfy the original compatibility conditions, we add a new integral to the kinematic boundary condition so that the modified system remains compatible. 
			This integral converges to $0$ when $n,k\rightarrow \infty$
			
			Following the arguments in Appendix \ref{sec:initial}, given $\eta_{0},u_{0},p_{0},\eta_{0}^{n},\p_{t}\eta_{0}^{n},\p_{t}^{2}\eta_{0}^{n}$. the system above has a solution. We state the following theorem with respect to this result.
			\begin{theorem}{\label{thm:initial_n}}
				$(u_{0},\eta_{0},p_{0})$ are as constructed in Appendix \ref{sec:initial}. $\eta_{0}^{n},\p_{t}\eta_{0}^{n},\p_{t}^{2}\eta_{0}^{n}$ are the initial data of $\eta^{n},\p_{t}\eta^{n},\p_{t}^{2}\eta^{n}$ as constructed in Appendix \ref{sec:smooth}. Then for any given $(D_{t}^{i}u_{0},\p_{t}^{i}\eta_{0},\p_{t}^{i}p_{0})$ with $i=0,1,2$, and $\eta_{0}^{n},\p_{t}\eta_{0}^{n},\p_{t}^{2}\eta_{0}^{n}$. Equation \eqref{eq:smooth_n0} and equation \eqref{eq:smooth_n1} admits a solution $(v^{n}_{0},\xi_{0}^{n},q_{0}^{n},D_{t}v_{0}^{n},\p_{t}\xi_{0}^{n},\p_{t}q_{0}^{n})$. Moreover, when $n\rightarrow \infty$, it holds that
				\[
				D_{t}^{i}v_{0}^{n}\rightarrow D_{t}^{i}u_{0}~~~\operatorname{in}~~~W^{2,q_{+}},\
				\p_{t}^{i}\xi_{0}^{n}\rightarrow \p_{t}^{i}\eta_{0}~~~\operatorname{in}~~~W^{3-\frac{1}{q_{+}},q_{+}},\
				\p_{t}^{i}q_{0}^{n}\rightarrow \p_{t}^{i}q_{0}~~~\operatorname{in}~~~W^{1,q_{+}}
				\]
				for $i=0,1$.
			\end{theorem}
			\begin{proof}
				To prove existence, it suffices to verify the compatibility conditions. Applying the same construction as in Appendix \ref{sec:initial} then yields existence.
				
				Since $\int_{-\ell}^{\ell}\p_{t}^{i}\xi^{n}(0)=0$, integrating both sides of the kinematic boundary condition in \eqref{eq:smooth_n0}, we obtain
				\[
				\int_{-\ell}^{\ell}v^{n}_{0}\cdot \mathcal{N}^{n}(0)=\int_{-\ell}^{\ell}\p_{t}\xi^{n}(0)=0
				\]
				\noindent Moreover, $v_{0}^{n}\cdot \nu=0$ on $\Sigma_{s}$ implies that
				$\int_{\Sigma_{s}}v_{0}^{n}\cdot \nu =0$.
				Therefore,
				\[
				\int_{\Omega}\operatorname{div}_{\mathcal{A}^{n}(0)}v^{n}_{0}=0.
				\]
				This guarantees the compatibility for the zero-order equation.
				
				We now show the compatibility condition for the first-order system \eqref{eq:smooth_n1}. From the kinematic boundary condition in \eqref{eq:smooth_n1}, we have
				\[
				\begin{aligned}
					\int_{-\ell}^{\ell} D_{t}v^{n}_{0}\cdot \mathcal{N}^{n}(0)&=\int_{-\ell}^{\ell}\p_{t}^{2}\xi^{n}_0-\int_{-\ell}^{\ell}(R^{n}(0)v^{n}_{0}\cdot \mathcal{N}^{n}(0)+u_{0}\p_{t}\p_{1}\xi^{n}_{0})+\int_{-\ell}^{\ell}\frac{1}{2\ell}\int_{-\ell}^{\ell}(R^{n}(0)v^{n}_{0}\cdot \mathcal{N}^{n}(0)+u^{n}_{0}\p_{t}\p_{1}\xi^{n}_{0})\\
					&=\int_{-\ell}^{\ell}\p_{t}^{2}\xi^{n}_{0}=0.
				\end{aligned}
				\]
				\noindent Since $D_{t}v^{n}_{0}\cdot \nu =0$ on $\Sigma_{s}$, we have
				\[
				\int_{\Omega}\operatorname{div}_{\mathcal{A}^{n}(0)}D_{t}v^{n}_{0}=\int_{-\ell}^{\ell} D_{t}v_{0}^{n}\cdot\mathcal{N}^{n}(0)+\int_{\Sigma_{s}}D_{t}v_{0}^{n}\cdot \nu =0,
				\]
				which guarantees the compatibility for the first-order equation.
				
				We now prove the convergence results. By the elliptic theory developed in \cite{GT2020}, we have
				\[
				\begin{aligned}
					&\|v^{n}_{0}\|_{W^{2,q_{+}}}+\|\xi^{n}_{0}\|_{W^{3-\frac{1}{q_{+}},q_{+}}}+\|q^{n}_{0}\|_{W^{1,q_{+}}}+\|\p_{t}v^{n}_{0}\|_{W^{2,q_{+}}}+\|\p_{t}\xi^{n}_{0}\|_{W^{3-\frac{1}{q_{+}},q_{+}}}+\|\p_{t}q^{n}_{0}\|_{W^{1,q_{+}}}\\
					&\lesssim \|u_{0}\|_{W^{3-\frac{1}{q_{+}},q_{+}}}+\|\eta_{0}\|_{W^{3-\frac{1}{q_{+}},q_{+}}}+\|q_{0}\|_{W^{1,q_{+}}}+\|\p_{t}u_{0}\|_{W^{3-\frac{1}{q_{+}},q_{+}}}+\|\p_{t}\eta_{0}\|_{W^{3-\frac{1}{q_{+}},q_{+}}}+\|\p_{t}q_{0}\|_{W^{1,q_{+}}}\\
					&\quad+\|\p_{t}^{2}\eta_{0}\|_{W^{2-\frac{1}{q_{+}},q_{+}}}+\|\p_{t}^{2}u_{0}\|_{L^{2}}.
				\end{aligned}
				\]
				
				Subtracting \eqref{eq:smooth_n0}, \eqref{eq:smooth_n1} by \eqref{eq:geometric} and \eqref{eq:quasi_linear}, respectively, we obtain the following system for $(v^{n}(0)-u(0),\xi^{n}(0)-\eta(0),q^{n}(0)-p(0),D_{t}v^{n}-D_{t}u(0),\p_{t}\xi^{n}(0)-\p_{t}\eta(0),\p_{t}q^{n}(0)-\p_tp(0))$
				\begin{align}{\label{eq:converge_n0}}
					\begin{cases}
						D_{t}(v_{0}^{n}-u_{0})+\operatorname{div}_{\mathcal{A}(0)}S_{\mathcal{A}(0)}(v_{0}^{n}-u_{0},q_{0}^{n}-p_{0})+R(0)(v_{0}^{n}-v_{0})\\ \qquad\qquad\qquad=\operatorname{div}_{\mathcal{A}(0)}S_{\mathcal{A}(0)-\mathcal{A}^{n}(0)}(v_{0}^{n},q_{0}^{n})+\operatorname{div}_{\mathcal{A}(0)-\mathcal{A}^{n}(0)}S_{\mathcal{A}^{n}(0)}(v_{0}^{n},q_{0}^{n})~~~&\operatorname{in}~~\Omega,\\
						\operatorname{div}_{\mathcal{A}(0)}(v_{0}^{n}-u_{0})=\operatorname{div}_{\mathcal{A}(0)-\mathcal{A}^{n}(0)}v_{0}^{n}~~~&\operatorname{in}~~\Omega,\\
						S_{\mathcal{A}(0)}(q_{0}^{n}-p_{0},v_{0}^{n}-u_{0})\mathcal{N}(0)=g(\xi^{n}(0)-\eta_{0})\mathcal{N}(0)-\sigma\partial_{1}(\frac{1}{(1+\vert \partial_{1}\zeta_{0}\vert^{2})^{3/2}}\partial_{1}(\xi^{n}(0)-\eta_{0}))(\mathcal{N}(0))
						\\
						\quad\quad\quad\quad\quad\quad\quad\quad\quad+\sigma\partial_{1}(\mathcal{R}(\partial_{1}\zeta_{0},\partial_{1}\eta_{0}))(\mathcal{N}^{n}(0)-\mathcal{N}(0))+S_{\mathcal{A}(0)-\mathcal{A}^{n}(0)}(q_{0}^{n},v_{0}^{n})\mathcal{N}(0)\\
						\quad\quad\quad\quad\quad\quad\quad\quad\quad+S_{\mathcal{A}^{n}(0)}(q_{0}^{n},v_{0}^{n})(\mathcal{N}^{n}(0)-\mathcal{N}(0))-\mathcal{K}(\xi_{0}^{n})(\mathcal{N}^{n}(0)-\mathcal{N}(0))~~~&\operatorname{on}~~\Sigma,\\
						(S_{\mathcal{A}(0)}(q_{0}^{n}-p_{0},v_{0}^{n}-u_{0})\nu-\beta (v_{0}^{n}-u_{0}))\cdot \tau=S_{\mathcal{A}(0)-\mathcal{A}^{n}(0)}(q_{0}^{n},v_{0}^{n})~~~&\operatorname{on}~~\Sigma_{s},\\
						(v_{0}^{n}-u_{0})\cdot \nu=0~~~&\operatorname{on}~~\Sigma_{s},\\
						\partial_{t}\xi_{0}^{n}-\p_{t}\eta_{0}=(v_{0}^{n}-u_{0})\cdot \mathcal{N}(0)+v_{0}^{n}(\mathcal{N}^{n}(0)-\mathcal{N}(0))~~~&\operatorname{on}~~\Sigma,\\
						\sigma(\mp \frac{\partial_{1}(\xi_{0}^{n}-\eta_{0})}{(1+\vert \partial_{1}\zeta_{0}\vert^{2})^{\frac{3}{2}}}\pm \mathcal{R}_{z}(\partial_{1}\zeta_{0},\partial_{1}\eta_{0})\partial_{1}(\xi_{0}^{n}-\eta_{0}))ds)(\pm\ell)\\=\kappa ((v_{0}^{n}-u_{0})\cdot \mathcal{N}(0))(\pm \ell)+\kappa v_{0}^{n}\cdot (\mathcal{N}^{n}(0)-\mathcal{N}(0)),
					\end{cases}
				\end{align}
				\noindent and
				\begin{align}{\label{eq:converge_n1}}
					\begin{cases}
						\operatorname{div}_{\mathcal{A}(0)}S_{\mathcal{A}(0)}(D_{t}(v_{0}^{n}-u_{0})+R(0)(v_{0}^{n}-u_{0}),\p_{t}(q_{0}^{n}-p_{0}))=-R^{n}(0)(D_{t}v_{0}^{n}-D_{t}u_{0})\\
						\quad\quad\quad\quad\quad-(R(0)-R^{n}(0))(D_{t}v_{0}^{n})-\p_{t}(R^{n}(0)(v_{0}^{n}-u_{0}))-\p_{t}((R(0)-R^{n}(0))v_{0}^{n})+R(0)(v_{0}^{n}),\p_{t}(q_{0}^{n}))\\
						\quad\quad\quad\quad\quad+\operatorname{div}_{\mathcal{A}(0)}S_{\mathcal{A}(0)-\mathcal{A}^{n}(0)}(D_{t}(v_{0}^{n})+R(0)(v_{0}^{n}),\p_{t}(q_{0}^{n}))+\operatorname{div}_{\mathcal{A}(0)-\mathcal{A}^{n}(0)}S_{\mathcal{A}^{n}(0)}(D_{t}(v_{0}^{n})~~~&\operatorname{in}~~\Omega,\\
						\operatorname{div}_{\mathcal{A}(0)}(D_{t}v_{0}^{n}-D_{t}u_{0})=\operatorname{div}_{\mathcal{A}(0)-\mathcal{A}^{n}(0)}(D_{t}v_{0}^{n} )~~~&\operatorname{in}~~\Omega,\\
						S_{\mathcal{A}(0)}(\p_{t}q_{0}^{n}-\p_{t}p_{0},(D_{t}v_{0}^{n}-D_{t}u_{0})+R(0)(v_{0}^{n}-u_{0}))\mathcal{N}(0)=g(\p_{t}\xi_{0}^{n}-\p_{t}\eta_{0})\mathcal{N}(0)\\
						\quad\quad\quad\quad\quad\quad\quad\quad-\sigma\partial_{1}(\frac{1}{(1+\vert \partial_{1}\zeta_{0}\vert^{2})^{3/2}}\partial_{1}(\p_{t}\xi_{0}^{n}-\p_{t}\eta_{0}))\mathcal{N}(0)+\sigma\partial_{1}(\mathcal{R}_{z}(\partial_{1}\zeta_{0},\partial_{1}\eta_{0})(\partial_{1}(\p_{t}\xi_{0}^{n}-\p_{t}\eta_{0})))\mathcal{N}(0)
						\\
						\quad\quad\quad\quad\quad\quad\quad\quad+S_{\mathcal{A}(0)}(\p_{t}q_{0}^{n},D_{t}v_{0}^{n}+R(v_{0}^{n}))(\mathcal{N}^{n}(0)-\mathcal{N}(0))\\
						\quad\quad\quad\quad\quad\quad\quad\quad+S_{\mathcal{A}^{n}(0)-\mathcal{A}(0)}(\p_{t}q_{0}^{n},D_{t}v_{0}^{n}+R(0)(v_{0}^{n}))(\mathcal{N}^{n}(0))+S_{\mathcal{A}^{n}}(0,(R^{n}(0)-R(0))(v_{0}^{n}))(\mathcal{N}^{n}(0))\\
						\quad\quad\quad\quad\quad\quad\quad\quad+(\mathcal{K}(\p_{t}\xi_{0}^{n})+\sigma\p_{1}(\mathcal{R}_{z}(\p_{1}\zeta_{0},\p_{1}\eta)\p_{1}\p_{t}\xi_{0}^{n}))(\mathcal{N}^{n}(0)-\mathcal{N}(0))~~~&\operatorname{on}~~\Sigma,\\
						(S_{\mathcal{A}(0)}(\p_{t}q_{0}^{n}-\p_{t}p_{0},(D_{t}v_{0}^{n}-D_{t}u_{0})+R(0)(v_{0}^{n}-u_{0}))\nu-\beta (D_{t}v_{0}-D_{t}u_{0}))\cdot \tau\\=S_{\mathcal{A}(0)-\mathcal{A}^{n}(0)}(\p_{t}q_{0}^{n},D_{t}v_{0}^{n}+Rv_{0}^{n})\cdot \tau+S_{\mathcal{A}^{n}(0)}(0,(R^{n}-R)v_{0}^{n})\cdot \tau~~~&\operatorname{on}~~\Sigma_{s},\\
						(D_{t}v_{0}^{n}-D_{t}u_{0})\cdot \nu=0~~~&\operatorname{on}~~\Sigma_{s},\\
						(D_{t}v_{0}^{n}-D_{t}u_{0})\cdot \mathcal{N}(0)+R(0)(v_{0}^{n}-u_{0})\cdot \mathcal{N}(0)\\-u_{0}\p_{1}(\p_{t}\xi_{0}^{n}-\p_{t}\eta_{0})-{\frac{1}{2\ell}\int_{-\ell}^{\ell}(R(0)(v_{0}^{n}-v_{0})\cdot \mathcal{N}(0)-u_{0}\p_{1}(\p_{t}\xi_{0}^{n}-\p_{t}\eta_{0}))}\\
						=D_{t}v_{0}^{n}\cdot (\mathcal{N}(0)-\mathcal{N}^{n}(0))+(R^{n}(0))v_{0}^{n}\cdot \mathcal{N}(0)+(R(0)-R^{n}(0))v_{0}^{n}\cdot (\mathcal{N}(0)-\mathcal{N}^{n}(0))\\-\frac{1}{2\ell}\int_{-\ell}^{\ell}(R(0)-R^{n}(0))v_{0}^{n}\cdot \mathcal{N}(0)+{R}^n(0)v_{0}^{n}\cdot (\mathcal{N}(0)-\mathcal{N}^{n}(0))~~&\operatorname{on}~~\Sigma,\\
						\sigma(\mp \frac{\partial_{1}(\p_{t}\xi_{0}^{n}-\p_{t}\eta_{0})}{(1+\vert \partial_{1}\zeta_{0}\vert^{2})^{\frac{3}{2}}}\pm \mathcal{R}_{z}(\partial_{1}\zeta_{0},\partial_{1}\eta_{0})\partial_{1}(\p_{t}\xi_{0}^{n}-\p_{t}\eta_{0}))(\pm \ell))(\pm\ell)=\kappa ((D_{t}v_{0}^{n}-D_{t}u_{0})\cdot \mathcal{N}(0))(\pm \ell).
					\end{cases}
				\end{align}
			\end{proof}
			\noindent By the similar elliptic estimate used in Appendix \ref{sec:initial}, we have the following estimate
			\[
			\begin{aligned}
				\|v_{0}^{n}-u_{0}\|_{W^{2,q_{+}}}+\|\xi_{0}^{n}-\eta_{0}\|_{W^{3-\frac{1}{q_{+}},q_{+}}}+\|q_{0}^{n}-p_{0}\|_{W^{1,q_{+}}}&+\|D_{t}v_{0}^{n}-D_{t}u_{0}\|_{W^{2,q_{+}}}+\|\p_{t}\xi_{0}^{n}-\p_{t}\eta_{0}\|_{W^{3-\frac{1}{q_{+}},q_{+}}}\\
				+\|\p_{t}q_{0}^{n}-\p_{t}p_{0}\|_{W^{1,q_{+}}}&\lesssim \|\eta_{0}^{n}-\eta_{0}\|_{W^{3-\frac{1}{q_{+}},q_{+}}}+\|\p_{t}\eta_{0}^{n}-\p_{t}\eta_{0}\|_{W^{3-\frac{1}{q_{+}},q_{+}}}.
			\end{aligned}
			\]
			This directly implies the convergence results as $n\rightarrow \infty$.
			
			\subsection{Initial Data for Smooth Given Data Problem with \texorpdfstring{$\eta^{k},\eta^{n},v_{l}^{k},\xi_{l}^{k},u^{k}$}{}}
			
			In this subsection, we consider the initial data for the following system
			
			\begin{equation}
				\begin{cases}
					\partial_{t}v_{d}+\operatorname{div}_{\mathcal{A}^{n}}S_{\mathcal{A}^{n}}(v_{d},q_{d})+\dive_{\mathcal{A}^{n}}\nabla_{\mathcal{A}}(R^{n}v_{l})+\p_{t}(R^{n}v_{l})=F^{1}(u,p,\eta)~~~&\operatorname{in}~~\Omega,\\
					\operatorname{div}_{\mathcal{A}^{n}}v_{d}=0~~~&\operatorname{in}~~\Omega,\\
					S_{\mathcal{A}^{n}}(q_{d},v_{d})\mathcal{N}^{n}+\nabla_{\mathcal{A}^{n}}(R^{n}v_{l})\mathcal{N}^{n}=g\xi_{d}\mathcal{N}^{n}-\sigma\partial_{1}(\frac{\p_1\xi_{d0}}{(1+\vert \partial_{1}\zeta_{0}\vert^{2})^{3/2}})\mathcal{N}^{n}+\p_{1}I_{1}^{n}\mathcal{N}^{n}\\
					\quad\quad\quad\quad\quad\quad\quad+\sigma\partial_{1}(\int_{0}^{t}\mathcal{R}_{z}(\partial_{1}\zeta_{0},\partial_{1}\eta^{k})(\partial_{1}\p_{t}\xi_{l}^{k})\p_{1}\p_{t}\eta^{k})\mathcal{N}^{n}+\sigma\partial_{1}(\int_{0}^{t}\mathcal{R}_{z}(\partial_{1}\zeta_{0},\partial_{1}\eta^{k})\partial_{1}\partial_{t}\xi_{d})\mathcal{N}^{n}\\
					\quad\quad\quad\quad\quad\quad\quad+F^{4}(u,p,\eta)~~~&\operatorname{on}~~\Sigma,\\
					(S_{\mathcal{A}^{n}}(q_{d},v_{d})\nu+\nabla_{\mathcal{A}^{n}}(R^{n}v_{l})\nu-\beta v_{d})\cdot \tau=F^{5}(u,\eta,p)~~~&\operatorname{on}~~\Sigma_{s},\\
					v_{d}\cdot \nu=0~~~&\operatorname{on}~~\Sigma_{s},\\
					\partial_{t}\xi_{d}=v_{d}\cdot \mathcal{N}^{n}+(R^{n}v_{l}^{k})\cdot \mathcal{N}^{n}+\int_{0}^{t}(\partial_{t}u_{1}^{k}\cdot\partial_{t}\p_{1}\xi_{l}^{k}+u_{1}^{k}\partial_{1}\partial_{t}\xi_{d})+I_{2}^{n}~~~&\operatorname{on}~~\Sigma,\\
					\sigma(\mp \frac{\partial_{1}\xi_{d}}{(1+\vert \partial_{1}\zeta_{0}\vert^{2})^{\frac{3}{2}}}\pm \int_{0}^{t}\mathcal{R}_{z}(\partial_{1}\zeta_{0},\partial_{1}\eta^{k})\partial_{t}\partial_{1}\xi_{d})(\pm \ell)\pm\int_{0}^{t}\mathcal{R}_{z}(\p_{1}\zeta_{0},\p_{1}\eta^{k})(\p_{t}\p_{1}\xi_{l}^{k}\p_{1}\p_{t}\eta^{k})\pm I_{1}^{n})\\
					\quad\quad\quad\quad\quad\quad\quad=\kappa (v_{d}\cdot \mathcal{N}^{n})(\pm \ell)-{F}^{7},
				\end{cases}
			\end{equation}
			where $v_{l}^{k},\eta^{k},\xi_{l}^{k},u^{k},\eta^{n}$ are as defined in section \ref{sec:smooth}. For simplicity, we give the following definition
			\[
			f^{k}(x):=\frac{C_{K}}{\frac{1}{k}}\int_{-\ell}^{\ell}K(\frac{x-y}{\frac{1}{k}})f(y)dy.
			\]
			
			We construct the initial data for $(v_{d},\xi_{d},q_{d})$ via solving the following system
			\begin{equation}{\label{initial_nk}}
				\begin{cases}
					D_{t}v_{d0}+R^{n}(0)v_{d0}+\operatorname{div}_{\mathcal{A}^{n}(0)}S_{\mathcal{A}^{n}(0)}(v_{d0},q_{d0})+\dive_{\mathcal{A}^{n}(0)}\nabla_{\mathcal{A}^{n}(0)}(R^{n}(0)v_{l0})+\p_{t}(R^{n}(0)v_{l0})\\
					\quad\quad\quad\quad\quad\quad\quad~ =F^{1}(u_{0},p_{0},\eta_{0})~~~&\operatorname{in}~~\Omega,\\
					\operatorname{div}_{\mathcal{A}^{n}(0)}v_{d0}=0~~~&\operatorname{in}~~\Omega,\\
					S_{\mathcal{A}^{n}(0)}(q_{d0},v_{d0})\mathcal{N}^{n}(0)+\nabla_{\mathcal{A}^{n}(0)}(R^{n}(0)v_{l0})\mathcal{N}^{n}(0)=g\xi_{d0}\mathcal{N}^{n}(0)-\sigma\partial_{1}(\frac{\p_1\xi_{d0}}{(1+\vert \partial_{1}\zeta_{0}\vert^{2})^{3/2}})\mathcal{N}^{n}(0)\\
					\quad\quad\quad\quad\quad\quad\quad~~~~~~~~~~~~~~~~~~~~~~~~~~~~~~~~~~~~~~~~~~~+\p_{1}I_{1}^{n,k}\mathcal{N}^{n}(0)+F^{4}(u_{0},p_{0},\eta_{0})~~~&\operatorname{on}~~\Sigma,\\
					(S_{\mathcal{A}^{n}(0)}(q_{d0},v_{d0})\nu+\nabla_{\mathcal{A}^{n}(0)}(R^{n}(0)v_{l0})\nu-\beta v_{d0})\cdot \tau=F^{5}(u_{0},\eta_{0},p_{0})~~~&\operatorname{on}~~\Sigma_{s},\\
					v_{d0}\cdot \nu=0~~~&\operatorname{on}~~\Sigma_{s},\\
					\partial_{t}\xi_{d0}=v_{d0}\cdot \mathcal{N}^{n}(0)+(R^{n}(0)v_{l0}^{k})\cdot \mathcal{N}^{n}(0)+I_{2}^{n,k}-\frac{1}{2\ell}\int_{-\ell}^{\ell}(I^{n,k}_{2}+(R^{n}v_{l0}^{k})\cdot \mathcal{N}^{n}(0))~~~&\operatorname{on}~~\Sigma,\\
					\sigma(\mp \frac{\partial_{1}\xi_{d0}}{(1+\vert \partial_{1}\zeta_{0}\vert^{2})^{\frac{3}{2}}}\pm I_{1}^{n,k})
					=\kappa (v_{d0}\cdot \mathcal{N}^{n}(0))(\pm \ell)-{F}^{7}(u_{0},q_{0},\eta_{0}),
				\end{cases}
			\end{equation}
			
			\noindent where
			\[
			\begin{aligned}
				u^{k}(0)=(u_{0})^{k},~~~D_{t}v_{d0}=D_{t}^{2}u(0),~~~~~~\p_{t}\xi_{d0}=\p_{t}^{2}\eta(0),\\
				v_{l0}^{k}=(v^{n}(0))^{k},~~~~v_{l0}=(v^{n}(0)),~~~\p_{t}v_{l0}=(\p_{t}v^{n}(0)),~~~\eta^{k}(0)=(\eta_{0})^{k}.
			\end{aligned}
			\]
			
			By an elliptic argument identical to that in Appendix \ref{sec:initial}, there exists a unique solution $(v_{d0}^{n,k},\xi^{n,k}_{d0},q^{n,k}_{d0})$ of system \eqref{initial_nk}. Moreover, when $k\rightarrow +\infty$, we have
			\[
			\begin{aligned}
				v_{d0}^{n,k}\rightarrow D_{t}v_{0}^{n}~~~\operatorname{in}~~~W^{2,q_{+}},\
				\xi_{d0}^{n,k}\rightarrow \p_{t}\xi_{0}^{n}~~~\operatorname{in}~~~W^{3-\frac{1}{q_{+}},q_{+}},\
				q_{d0}^{n,k}\rightarrow \p_{t}q_{0}^{n}~~~\operatorname{in}~~~W^{1,q_{+}}.
			\end{aligned}
			\]
			
		\end{appendix}

		% \bibliographystyle{siam}
		% \bibliography{Reference}

\begin{thebibliography}{50}
			
			\bibitem{Bo06} S. Bodea. \emph{The motion of a fluid in an open channel}. Ann. Sc. Norm. Super. Pisa Cl. Sci. (5) {\bf 5} (2006), no. 1, 77--105.
			
			\bibitem{boyer_fabrie} F. Boyer, P. Fabrie. \emph{Mathematical tools for the study of the incompressible Navier-Stokes equations and related models}. Applied Mathematical Sciences, 183. Springer, New York, 2013.
			
			\bibitem{C1986} R. G. Cox. {\it The dynamics of the spreading of liquids on a solid surface. part 1. viscous flow}. Journal of Fluid Mechanics, 168(1986), 195--220.
			
			\bibitem{dG} P. de Gennes. {\em Wetting: statics and dynamics}. Rev. Mod. Phys. 57(1985), no. 3, 827--863.
			
			\bibitem{Evans} L. C. Evans. {\em Partial differential equations}, 2nd ed. Graduate Studies in Mathematics 19, American  Mathematical Society, Providence, RI, 2010.
			
			\bibitem{Finn} R. Finn. {\em Equilibrium capillary surfaces}. Grundlehren der mathematischen Wissenschaften, 284. Springer-Verlag, New York, 1986.
			
			\bibitem{GM70} S. I. Grossman and R. K. Miller, {\it Perturbation theory for Volterra integrodifferential systems}. J. Diff. Equ., 8(1970), 457--474.
			
			\bibitem{GT1} Y. Guo, I. Tice. {\em Local well-posedness of the viscous surface wave problem without surface tension}.
			Anal PDE, 6, 287--369, 2013.
			
			\bibitem{GT18} Y. Guo, I. Tice. {\em  Stability of contact lines in fluids: 2D Stokes flow}.  Arch. Ration. Mech. Anal. 227 (2018), no. 2, 767--854.
			
			\bibitem{GT2020} Y. Guo, I. Tice. {\em  Stability of contact lines in fluids: 2D Navier-Stokes flow}. J. Eur. Math. Soc. 26 (2024), no. 4, 1445--1557.
			
			\bibitem{Jin05} B. Jin. \emph{ Free boundary problem of steady incompressible flow with contact angle $\pi/2$}. J. Diff. Equ. {\bf 217} (2005), no. 1, 1--25.
			
			\bibitem{Kr87} D. Kr\"oner. \emph{ The flow of a fluid with a free boundary and a dynamic contact angle}. Z.
			Angew. Math. Mech. {\bf 5} (1987), 304--306.
			
			\bibitem{KM15} H. Kn\"upfer, N. Masmoudi. \emph{ Darcy's flow with prescribed contact angle: well- posedness and lubrication approximation}. Arch. Ration. Mech. Anal. {\bf 218} (2015), no. 2, 589--646.
			
			\bibitem{N} J. Nitsche. {\em On Korn's second inequality}. RAIRO Anal. Numr. 15(1981),3, 237--248.
			
			\bibitem{RE} W. Ren, W. E. {\em Boundary conditions for the moving contact line problem}. Phys. Fluids 19(2007).
			
			\bibitem{S97} B. Schweizer. \emph{Free boundary fluid systems in a semigroup approach and oscillatory be-
				havior}. SIAM J. Math. Anal. {\bf 28} (1997), 1135--1157.
			
			\bibitem{S01} B. Schweizer. {\em A well- posed model for dynamic contact angles}. Nonlinear Anal. {\bf 43} (2001), no. 1, 109--125.
			
			\bibitem{Soco93} J. Socolosky. \emph{ The solvability of a free boundary value problem for the stationary
				Navier-Stokes equations with a dynamic contact line}. Nonlinear Anal. {\bf 21} (1993), 763--784.
			
			\bibitem{Sol95}V. A. Solonnikov. \emph{ On some free boundary problems for the Navier-Stokes equations with
				moving contact points and lines}. Math. Ann. {\bf 302} (1995), 743--772.
			
			\bibitem{Sol98} V. A. Solonnikow. \emph{ Solvability of two dimensional free boundary value problem for the
				Navier - Stokes equations for limiting values of contact angle}. In: ''Recent Developments in
			Partial Differential Equation''. Rome: Aracne. Quad. Mat. 2, 1998, 163--210.
			
			\bibitem{Te} R. Temam. {\em Navier-Stokes Equations. Theory and Numerical Analysis}. Reprint of the 1984 edition. AMS Chelsea Publishing, Providence, RI, 2001.
			
			\bibitem{TW21} I. Tice, L. Wu. \emph{ Dynamics and stability of sessile drops with contact points}.  J. Diff. Equ. {\bf 272} (2021) 648--731.
			
			\bibitem{ZhT17} I. Tice, Y. Zheng. {\it Local well posedness of the near-equilibrium contact line problem in 2-dimensional Stokes flow}. SIAM J. Math. Anal. {\bf 49 } (2017), no. 2, 899--953.
			
			\bibitem{Wa} A. Wazwaz. {\em Linear and nonlinear integral equations: methods and applications}. Higher Education Press, Beijing and Springer-Verlag Berlin Heidelberg, 2011.
			
			\bibitem{WL} J. V. Wehausen and E. V. Laitone, {\it Surface waves}. 1960 Handbuch der Physik, Vol. {\bf 9}, Part 3 pp. 446-778 Springer-Verlag, Berlin.
			
			\bibitem{Yo} T. Young. {\em An essay on the cohesion of fluids}. Philos. Trans. R. Soc. London 95(1805), 65--87.
			
		\end{thebibliography}

	\end{document}